\documentclass[11pt,reqno,a4paper]{amsart}
\usepackage[utf8]{inputenc}
\usepackage{mathtools,amsmath,amssymb}
\usepackage[T1]{fontenc}
\usepackage{lmodern}
\usepackage{microtype}
\usepackage{hyperref}
\usepackage{xspace}
\usepackage{bibentry}
\usepackage{easybmat}
\usepackage{mathdots}

\usepackage{cite}

\usepackage[usenames,dvipsnames,svgnames,table]{xcolor}
\definecolor{todocolor}{HTML}{D7E1E5}
\usepackage[textsize=tiny,
            backgroundcolor=todocolor,
            bordercolor=todocolor,
            linecolor=todocolor,
            textwidth=3.9cm]{todonotes}

\usepackage{comment}
\usepackage{rotating}

\usepackage{subcaption}
\usepackage{bibentry}
\usepackage{tikz}
\usepackage{dynkin-diagrams}
\usepackage[margin=2.4cm]{geometry}

\usepackage{tikz}
\usetikzlibrary{decorations.pathreplacing}
\numberwithin{equation}{section}

\usepackage{amsmath}
\usepackage{amsthm}

\makeatletter
\let\@wraptoccontribs\wraptoccontribs
\makeatother

\theoremstyle{plain}
\newtheorem{Thm}[equation]{Theorem}
\newtheorem{Cor}[equation]{Corollary}
\newtheorem{Lem}[equation]{Lemma}
\newtheorem{Prop}[equation]{Proposition}

\newtheorem*{Thm-non}{Theorem}

\theoremstyle{definition}

\theoremstyle{remark}

\newtheorem{Rem}[equation]{Remark}

\newcommand{\ol}{\overline}
\newcommand\iso{\,\vphantom{j^{X^2}}\smash{\overset{\sim}{\vphantom{\rule{0pt}{0.20em}}\smash{\longrightarrow}}}\,}

\newmuskip\pFqmuskip
\newcommand*\pFq[6][8]{%
    \begingroup
    \pFqmuskip=#1mu\relax
    \mathcode`\,=\string"8000
    \begingroup\lccode`\~=`\,
    \lowercase{\endgroup\let~}\pFqcomma
    {}_{#2}F_{#3}{\left(\genfrac..{0pt}{}{#4}{#5};#6\right)}%
    \endgroup}
\newcommand{\pFqcomma}{\mskip\pFqmuskip}

\def\be{\begin{eqnarray}}
\def\ee{\end{eqnarray}}

\newcommand{\sfrac}[2]{{\textstyle\frac{#1}{#2}}}

\newcommand{\Pop}{\mathrm{P}}
\newcommand{\ID}{\mathrm{I}}

\newcommand{\rtt}{\mathrm{rtt}}
\newcommand{\End}{\mathrm{End}}

\newcommand{\fg}{\mathfrak{g}}
\newcommand{\fh}{\mathfrak{h}}

\newcommand{\gl}{\mathfrak{gl}}
\newcommand{\ssl}{\mathfrak{sl}}

\newcommand{\BZ}{\mathbb{Z}}
\newcommand{\BC}{\mathbb{C}}

\makeatletter
\g@addto@macro\bfseries{\boldmath}
\makeatother

\makeatletter
\newcases{rrcases}{\quad}{%
  \hfil$\m@th\displaystyle{##}$}{$\m@th\displaystyle{##}$\hfil}{\lbrace}{.}
\makeatother

\newcommand{\fosp}{\mathfrak{osp}}
\newcommand{\fso}{\mathfrak{so}}
\newcommand{\fsp}{\mathfrak{sp}}
\newcommand{\sfv}{{\mathsf{v}}}
\newcommand{\sfe}{{\mathsf{e}}}
\newcommand{\sff}{{\mathsf{f}}}
\newcommand{\sfh}{{\mathsf{h}}}
\newcommand{\sfr}{{\mathsf{r}}}
\newcommand{\Parity}{\Upsilon_V}
\newcommand{\parity}{\Upsilon_{\mathsf{V}}}
\newcommand{\Gr}{\mathrm{gr}}
\newcommand{\unl}{\underline}

\newcommand{\wt}{\widetilde}
\newcommand{\VV}{{\mathsf{V}}}
\newcommand{\sfc}{{\mathsf{c}}}
\newcommand{\sz}{{\mathfrak{z}}}
\newcommand{\sft}{{\mathsf{t}}}
\newcommand{\sfT}{{\mathsf{T}}}
\newcommand{\sfR}{{\mathsf{R}}}

\newcommand{\sfE}{{\mathsf{E}}}
\newcommand{\sfF}{{\mathsf{F}}}
\newcommand{\sfH}{{\mathsf{H}}}
\newcommand{\sfx}{{\mathsf{x}}}
\newcommand{\lt}{z}

\begin{document}
\title[Orthosymplectic Yangians]{\large{\textbf{Orthosymplectic Yangians}}}

\author{Rouven Frassek}
\address{R.F.: University of Modena and Reggio Emilia, FIM, Via G.~Campi~213/b, 41125 Modena, Italy}
\email{rouven.frassek@unimore.it}

\author{Alexander Tsymbaliuk}
\address{A.T.: Purdue University, Department of Mathematics, West Lafayette, IN 47907, USA}
\email{sashikts@gmail.com}

\begin{abstract}
We study the RTT orthosymplectic super Yangians and present their Drinfeld realizations for any parity sequence,
generalizing the results of~\cite{jlm} for non-super case,~\cite{m} for a standard parity sequence, and~\cite{p,t}
for the super $A$-type.
\end{abstract}

\maketitle
\tableofcontents



\section{Introduction}
\label{sec:intro}


\subsection{Summary}
\label{ssec:summary}
\

The original definition of Yangians $Y(\fg)$ associated to any simple Lie algebra $\fg$ is due to~\cite{d1},
where these algebras are realized as Hopf algebras with a finite set of generators (known as the
\emph{$J$-realization}). The representation theory of such algebras is best developed using their alternative
\emph{(new) Drinfeld realization} (also known as the \emph{current realization}) proposed in~\cite{d2}, though
the Hopf algebra structure is much more involved in this presentation (for example, a proof of the coproduct
formula was given only recently in \cite{gnw}).

For $\fg=\gl_n$, a closely related algebra $Y^\rtt(\gl_n)$ was studied earlier in the work of Faddeev's school
on the \emph{quantum inverse scattering method}, see e.g.~\cite{frt}, where the algebra generators were encoded
by an $n\times n$ square matrix $T(u)$ subject to a single \emph{RTT relation}
\begin{equation}\label{eq:rtt intro}
  R(u-v)T_1(u)T_2(v) = T_2(v)T_1(u)R(u-v)
\end{equation}
involving Yang's $R$-matrix $R(u)$ satisfying the \emph{Yang-Baxter equation} with a spectral parameter
\begin{equation}\label{eq:YB intro}
   R_{12}(u)R_{13}(u+v)R_{23}(v) = R_{23}(v)R_{13}(u+v)R_{12}(u) \,.
\end{equation}
We note that the $\ssl_n$-version $Y^\rtt(\ssl_n)$ is recovered by imposing an extra relation
\begin{equation}\label{eq:sl-vs-gl}
  \mathrm{qdet}\, T(u)=1 \,.
\end{equation}
The Hopf algebra structure on both $Y^\rtt(\gl_n)$ and $Y^\rtt(\ssl_n)$ is extremely simple with the coproduct
\begin{equation}\label{eq:coproduct-T}
  \Delta\colon T(u)\mapsto T(u)\otimes T(u) \,.
\end{equation}
This \emph{RTT realization} is well suited for the development of both the representation theory and the
corresponding integrable systems (involving Bethe subalgebras on the mathematical side).

An explicit isomorphism from the new Drinfeld to the RTT realizations of type $A$ Yangians is constructed
using the Gauss decomposition of $T(u)$, a complete proof been provided in~\cite{bk} (curiously enough the
trigonometric version of this result was established a decade earlier in~\cite{df}). A similar explicit
isomorphism for the remaining classical $BCD$-types was obtained only a decade later in~\cite{jlm}, where it
was again constructed using the Gauss decomposition of the generating matrices $T(u)$ which are subject to
the RTT relations~\eqref{eq:rtt intro} with the rational solutions of~\eqref{eq:YB intro} first discovered
in~\cite{zz}. An implicit existence of such an isomorphism for any $\fg$ was noted by Drinfeld back in the 80's,
while a detailed proof of his result was established only recently in~\cite{w}.

Finally, we note that the RTT realization of the (antidominantly) shifted Yangians $Y_\mu(\fg)$ from~\cite{bfn}
was recently obtained in~\cite{fpt,ft} for classical $\fg$. This significantly simplifies some of their basic
structures such as the coproduct homomorphisms
  $\Delta\colon Y_{\mu_1+\mu_2}(\fg)\to Y_{\mu_1}(\fg)\otimes Y_{\mu_2}(\fg)$,
cf.~\eqref{eq:coproduct-T}, and allows to introduce integrable systems on the corresponding quantized Coulomb
branches of $3d$ $\mathcal{N}=4$ quiver gauge theories. An important aspect of this setup in $A$-type is that
the central series $\mathrm{qdet}\, T(u)$ encodes all \emph{masses} of the corresponding physical theory,
cf.~\eqref{eq:sl-vs-gl}.

\medskip
The theory of Yangians associated with Lie superalgebras is still far from a full development. In particular,
there is no uniform $J$- or Drinfeld realizations of those. The cases studied mostly up to date involve rather
the RTT realization. The general linear RTT Yangians $Y^\rtt(\gl(n|m))$ and the orthosymplectic RTT Yangians
$Y^\rtt(\fosp(N|2m))$ first appeared in~\cite{n} and~\cite{aacfr}, respectively, using the super-analogs of
the Yang's and Zamolodchikov-Zamolodchikov's rational $R$-matrices.

A novel feature of Lie superalgebras is that they admit several non-isomorphic Dynkin diagrams. The isomorphism
of the Lie superalgebras corresponding to different Dynkin diagrams of the same finite/affine type was obtained
by Serganova in the Appendix to~\cite{lss}. Likewise, one may define various quantizations of the universal
enveloping superalgebras starting from different Dynkin diagrams, and establishing isomorphisms among those is
quite a non-trivial task. In the case of quantum finite/affine superalgebras in their Drinfeld-Jimbo realization,
this was accomplished by Yamane in~\cite{y} two decades ago.

Despite the absence of the definition of super Yangians, the rational setup admits some benefits. As an example,
the RTT realization of $Y^\rtt(\gl(n|m))$ manifestly provides an isomorphism between these algebras corresponding
to different Dynkin diagrams, which is far from being obvious when considering their Drinfeld realizations as
developed in~\cite{p,t}. We note however that the \emph{positive subalgebras} in the Drinfeld realization do
essentially depend on a choice of the Dynkin~diagram.

\medskip
One of the major objectives of the present note is to generalize~\cite[\S2]{t} to the orthosymplectic Yangians.
To this end, we study the RTT Yangians $Y^\rtt(\fosp(N|2m))$ and their extended versions $X^\rtt(\fosp(N|2m))$
associated to an arbitrary Dynkin diagram. Alike the aforementioned $\gl(n|m)$-type, these algebras are manifestly
isomorphic, while their Drinfeld realizations look quite different. In fact, one of our key results is the Drinfeld
realization of these algebras for all Dynkin~diagrams. We note that the case of $N\geq 3$ and the standard
Dynkin diagram was recently treated in~\cite{m}.

Our approach is quite straightforward, generalizing~\cite{bk} for $A$-type, \cite{jlm} for $BCD$-types,
and~\cite{m} for the distinguished Dynkin diagram. The above crucially used the rank reduction embeddings
that are compatible with the Gauss decompositions. Let us emphasize that while the proof of the existence
of such embeddings solely utilized the RTT formalism in non-super case of~\cite{jlm}, this approach is not
fully applicable in the present setup (due to the possible singularity of $R(u)$ at $u=1$), and we rather
use an update of the corresponding core computation from~\cite{m}. With the help of these embeddings, the
quadratic relations in the Drinfeld presentation of orthosymplectic Yangians are derived from the super
$A$-type analogue and rank $\leq 2$ cases handled by brute force. Additionally, we also have Serre relations
(standardly deduced from their Lie-theoretic counterparts). The Drinfeld realization of $\fosp(1|2)$-Yangians
previously appeared in~\cite{acfr}, where many details were missing and an opposite Gauss decomposition was used.

We note that the orthosymplectic type simultaneously resembles all three classical types $B, C, D$.
In the sequel note~\cite{ft2}, we construct orthosymplectic Lax matrices generalizing our orthogonal
and symplectic Lax matrices from~\cite{ft,fkt}.

\medskip
While we were preparing the present note and~\cite{ft2}, the work~\cite{mr} appeared that independently treats
the $N=1$ case. The arguments of~\emph{loc.cit.} are quite similar to ours and also crucially rely on the
Drinfeld realization of $X^\rtt(\fosp(1|2))$, thus filling in the aforementioned gaps of~\cite{acfr}.


\subsection{Outline}
\label{ssec:outline}
\

The structure of the present paper is the following:

\medskip
\noindent
$\bullet$
In Section~\ref{sec:Lie orthosymplectic}, we recall basic results on the orthosymplectic Lie superalgebras
$\fosp(V)$. We recover their Dynkin diagrams of \cite{fss} from the \emph{parity sequences}
  $\Parity\in \{\bar{0},\bar{1}\}^{\lfloor \dim(V)/2 \rfloor}$
(see Subsection~\ref{ssec:Dynkin diagrams}) as well as recall their Serre-type presentations from~\cite{z}
highlighting the presence of the higher order Serre relations of orders $3$, $4$, $6$, or $7$ for specific
parity sequences $\Parity$ (see Subsection~\ref{ssec:Serre-classical}).

\medskip
\noindent
$\bullet$
In Section~\ref{sec:RTT orthosymplectic}, we introduce the RTT (extended) Yangians $X^\rtt(\fosp(V)), Y^\rtt(\fosp(V))$
and establish their basic properties. We emphasize that both algebras $X^\rtt(\fosp(V))$ and $Y^\rtt(\fosp(V))$
depend (up to isomorphism) only on the total number of $\bar{0}$'s and $\bar{1}$'s in $\Parity$, according to
Lemma~\ref{lem:isomorphism of RTT yangians} and Corollary~\ref{cor:isom-RTT-yangians}. Thus, all of them are
isomorphic to the (extended) orthosymplectic Yangians $X^\rtt(\fosp(N|2m))$ and $Y^\rtt(\fosp(N|2m))$ of~\cite{aacfr},
which correspond to the \emph{standard parity} case (where all $\bar{0}$'s are placed after all $\bar{1}$'s).
This observation allows us to generalize some of the basic structural results of~\cite{aacfr}, such as the tensor
product decomposition~\eqref{eq:extended-vs-nonextended} and the PBW-type results of Proposition~\ref{prop:assoc.graded}
and Corollary~\ref{cor:pbw-thm}, to arbitrary parity sequences $\Parity$.

The rest of this note is devoted to the Gauss decomposition~\eqref{eq:gauss-osp} of the generator matrix~$T(u)$.
To this end, we first establish our key technical tool of \emph{rank reduction} in Theorem~\ref{thm:embedding}
(the proof of which closely follows the arguments of~\cite[\S3]{m}). The latter implies the commutativity of some
of the generating currents, see Corollary~\ref{cor:commutativity}. We also establish Lemma~\ref{lem:ef-com-tl} that
significantly simplifies several computations in the rest of the note. Finally, we recall the defining relations
among the generating currents of the super $A$-type Yangians $Y^\rtt(\gl(\VV))$ in Theorem~\ref{thm:Drinfeld-A},
and deduce the corresponding relations for the currents of $X^\rtt(\fosp(V))$ with $\Parity=\parity$, see
Corollaries~\ref{cor:A-type relations},~\ref{cor:other-A-type relations}.

\medskip
\noindent
$\bullet$
In Section~\ref{sec:Gauss and Serre}, we recover explicit formulas for all entries of the matrices $E(u),F(u),H(u)$
from the Gauss decomposition~\eqref{eq:gauss-osp} in terms of the generating currents $e_i(u),f_i(u),h_i(u)$
from~(\ref{eq:HFE-matrices},~\ref{eq:hfe-generating}). We also derive a factorized formula for the central series
$c_V(u)$ of~\eqref{eq:central-c} in Lemmas~\ref{lem:cseries-evenN-even},~\ref{lem:cseries-evenN-odd},~\ref{lem:cseries-oddN}.
In Subsection~\ref{ssec:Serre-Yangian}, we establish some higher order relations generalizing those from
Subsection~\ref{ssec:Serre-classical}.

\medskip
\noindent
$\bullet$
In Section~\ref{sec:rank 1-2}, we establish quadratic relations between the generating currents
$e_i(u),f_i(u),h_\imath(u)$ of $X^\rtt(\fosp(V))$ in rank $\leq 2$. The arguments are straightforward (though tedious)
and we present them in a uniform way (eliminating the smaller rank reduction of~\cite{jlm} for non-super types).

\medskip
\noindent
$\bullet$
In Section~\ref{sec:Drinfeld orthosymplectic}, we present Drinfeld realizations of RTT (extended) orthosymplectic
super Yangians $X^\rtt(\fosp(V))$ and $Y^\rtt(\fosp(V))$, associated with any parity sequence, see
Theorems~\ref{thm:Main-Theorem-ext}~and~\ref{thm:Main-Theorem-nonext}. The corresponding relations
follow from those for $Y^\rtt(\gl(\VV))$ and $Y^\rtt(\ssl(\VV))$ through Corollaries
\ref{cor:A-type relations},~\ref{cor:other-A-type relations}, the commutativity of
Corollary~\ref{cor:commutativity}, the Serre relations (the higher order ones generalize those from
Subsection~\ref{ssec:Serre-Yangian}), and the quadratic relations in rank $\leq 2$ as established in
Section~\ref{sec:rank 1-2}. To~prove the sufficiency of these relations, we use the standard argument
(originating from~\cite{bk}) of passing through the associated graded algebras and utilize the PBW result
of Corollary~\ref{cor:pbw-thm}.

\medskip
\noindent
$\bullet$
In Appendix~\ref{sec:6-fold fusion}, we recall the isomorphisms $X^\rtt(\fso_3) \simeq Y^\rtt(\gl_2)$,
$Y^\rtt(\fso_3) \simeq Y^\rtt(\ssl_2)$ of~\cite{amr}, see Proposition~\ref{prop:so3=gl2}, whose proof is based on
the important \emph{6-fold $R$-matrix fusion} of Lemma~\ref{lem:amr-fusion}. We then establish similar isomorphisms
$X^\rtt(\fso_6) \simeq Y^\rtt(\gl_4)$, $Y^\rtt(\fso_6) \simeq Y^\rtt(\ssl_4)$ in Proposition~\ref{prop:so6=gl4},
the proof of which is based on the analogous 6-fold $R$-matrix fusion of Lemma~\ref{lem:new-fusion}. Finally,
we explain in Remark~\ref{rem:sop22=gl12-Lie} why applying the above $R$-matrix fusion approach to $Y^\rtt(\gl(1|2))$
recovers an algebra that looks surprisingly different from $X^\rtt(\fosp(2|2))$, despite $\fosp(2|2)\simeq \ssl(1|2)$.
We conclude by matching the resulting two $16\times 16$ $R$-matrices with those of~\cite{rm},
see Remark~\ref{rem:comparison-to-RM}.


\subsection{Acknowledgement}
\label{ssec:acknowl}
\

A.T.\ is grateful to M.~Finkelberg for stimulating discussions on orthosymplectic quantum groups; to A.~Molev
for a correspondence regarding~\cite{mr} (which has a partial overlap with the $N=1$ case of the present note)
and a discussion of the isomorphism between $\fosp(2|2)$ and $\ssl(1|2)$ Yangians.
We are very grateful to the referees for their useful suggestions that improved the exposition.

The work of R.F. was partially supported in part by the INFN grant Gauge and String Theory (GAST), by the
``INdAM–GNFM Project'' codice CUP-E53C22001930001, by the FAR UNIMORE project CUP-E93C23002040005, and by
the PRIN project ``2022ABPBEY'' CUP-E53D23002220006.
A.T.\ gratefully acknowledges NSF Grants DMS-$2037602$ and DMS-$2302661$.


\section{Orthosymplectic Lie superalgebras}
\label{sec:Lie orthosymplectic}

In this section, we recall the basic results on orthosymplectic Lie superalgebras. We recover their
various Dynkin diagrams from the parity sequences, and discuss their Serre-type presentations.


\subsection{Setup and notations}
\label{ssec:setup}
\

Fix $N\geq 1, m\geq 0$, and consider the set $\mathbb{I}:=\{1,2,\ldots,N+2m\}$ equipped with
an involution $'$ :
\begin{equation}\label{eq:index-inv}
  i':=N+2m+1-i \,.
\end{equation}
Consider a superspace $V=V_{\bar{0}}\oplus V_{\bar{1}}$ with a $\BC$-basis $v_1,\ldots,v_{N+2m}$
such that each $v_i$ is either \emph{even} (that is, $v_i\in V_{\bar{0}}$) or \emph{odd}
(that is, $v_i\in V_{\bar{1}}$), the dimensions are $\dim(V_{\bar{0}})=N, \dim(V_{\bar{1}})=2m$,
and the vectors $v_i,v_{i'}$ have the same parity for any $i$ (in particular,
$v_{(N+1)/2+m}\in V_{\bar{0}}$ for odd $N$), cf.~\eqref{eq:index-inv}. The latter condition means that
\begin{equation}\label{eq:parity-sym}
  \ol{i}=\ol{i'} \,,
\end{equation}
where for $i\in \mathbb{I}$, we define its $\BZ_2$-parity $\ol{i}\in \BZ_2$ via:
\begin{equation}\label{eq:index-parity}
  \ol{i}=
  \begin{cases}
    \bar{0} & \text{if } v_i\in V_{\bar{0}} \\
    \bar{1} & \text{if } v_i\in V_{\bar{1}}
  \end{cases} \,.
\end{equation}
We also define the sequence $\theta_V:=(\theta_1,\theta_2,\ldots,\theta_{N+2m})$ of $\pm 1$'s via:
\begin{equation}\label{eq:theta}
  \theta_i=1 \qquad \mathrm{and} \qquad \theta_{i'}=(-1)^{\ol{i}}
  \qquad \mathrm{for\ any} \qquad 1\leq i\leq \lceil \sfrac{N}{2} \rceil+m \,.
\end{equation}
It implies that
\begin{equation}\label{eq:theta-symmetry}
  \theta_{i'}=(-1)^{\ol{i}}\theta_i \qquad \forall\, i\in \mathbb{I} \,.
\end{equation}

For a superalgebra $A$ and its two homogeneous elements $x$ and $x'$, we define
\begin{equation}\label{eq:super commutator}
  [x,x']=\mathrm{ad}_{x} (x'):=xx'-(-1)^{|x|\cdot |x'|}x'x
    \qquad \mathrm{and} \qquad
  \{x,x'\}:=xx'+(-1)^{|x|\cdot |x'|}x'x \,,
\end{equation}
where $|x|$ denotes the $\BZ_2$-grading of $x$ (i.e.\ $x\in A_{|x|}$)
and we use conventions $(-1)^{\bar{0}}=1, (-1)^{\bar{1}}=-1$.

Given two superspaces $A=A_{\bar{0}}\oplus A_{\bar{1}}$ and $B=B_{\bar{0}}\oplus B_{\bar{1}}$,
their tensor product $A\otimes B$ is also a superspace with
  $(A\otimes B)_{\bar{0}}=A_{\bar{0}}\otimes B_{\bar{0}}\oplus A_{\bar{1}}\otimes B_{\bar{1}}$
and
  $(A\otimes B)_{\bar{1}}=A_{\bar{0}}\otimes B_{\bar{1}}\oplus A_{\bar{1}}\otimes B_{\bar{0}}$.
Furthermore, if $A$ and $B$ are superalgebras, then $A\otimes B$ is made into a superalgebra, the
\emph{graded tensor product} of the superalgebras $A$ and $B$, via the following multiplication:
\begin{equation}\label{eq:graded tensor product}
  (x\otimes y)(x'\otimes y')=(-1)^{|y|\cdot |x'|} (xx')\otimes (yy')
  \quad \mathrm{for\ any}\ x\in A_{|x|}, x'\in A_{|x'|}, y\in B_{|y|},y'\in B_{|y'|} \,.
\end{equation}
We will use only the graded tensor products of superalgebras throughout this paper.


\subsection{Orthosymplectic Lie superalgebras}
\label{ssec:osp-definition}
\

A standard basis of the general linear Lie superalgebra $\gl(V)$ is formed by the elements
$E_{ij}$ $(1\leq i,j\leq N+2m)$ of parity $\ol{i}+\ol{j}$ with the commutation relations
\begin{equation*}
  [E_{ij},E_{k\ell}] =
  \delta_{kj} E_{i \ell} - \delta_{\ell i}(-1)^{(\ol{i}+\ol{j})(\ol{k}+\ol{\ell})}\, E_{kj} \,.
\end{equation*}
Consider a bilinear form $B_G\colon V\times V\to \BC$ defined by the anti-diagonal matrix
\begin{equation*}
  G=(g_{ij})_{i,j=1}^{N+2m} \qquad \mathrm{with} \qquad g_{ij}=\delta_{ij'}\theta_i \,.
\end{equation*}
We regard the orthosymplectic Lie superalgebra $\fosp(V)$ associated with the bilinear form
$B_G$ as the Lie subalgebra of $\gl(V)$ spanned by the elements
\begin{equation}\label{eq:F-elements}
  F_{ij} = E_{ij} - (-1)^{\ol{i}\cdot \ol{j}+\ol{i}} \theta_i\theta_j\, E_{j'i'}
  \qquad  \forall\ 1\leq i,j\leq N+2m \,.
\end{equation}
We note that $F_{j'i'}=-(-1)^{\ol{i}\cdot \ol{j}+\ol{i}} \theta_i\theta_j\cdot F_{ij}$.
Furthermore, the elements
\begin{equation}\label{eq:F-basis}
  \Big\{ F_{ij} \,\Big|\, i+j<N+2m+1 \Big\}\bigcup
  \Big\{ F_{ii'} \,\Big|\, |v_i|=\bar{1} \,, 1\leq i\leq \sfrac{N}{2}+m \Big\}
\end{equation}
form a basis of $\fosp(V)$. In what follows, we shall also need the explicit commutation relations:
\begin{equation}\label{eq:F-commutation}
  [F_{ij},F_{k\ell}] =
  \delta_{kj} F_{i \ell} - \delta_{\ell i} (-1)^{(\ol{i}+\ol{j})(\ol{k}+\ol{\ell})}\, F_{kj} -
  \delta_{k i'} (-1)^{\ol{i}\cdot \ol{j}+\ol{i}}\theta_i\theta_j\, F_{j' \ell} +
  \delta_{\ell j'} (-1)^{\ol{i}\cdot \ol{k} + \ol{\ell}\cdot \ol{k}} \theta_{i'}\theta_{j'}\, F_{k i'} \,.
\end{equation}

\medskip
The Lie superalgebra $\fosp(V)$ is $\BZ_2$-graded: $\fosp(V)=\fosp(V)_{\bar{0}} \oplus \fosp(V)_{\bar{1}}$.
We choose the Cartan subalgebra $\fh$ of $\fosp(V)$ (which by definition is just a Cartan subalgebra of
$\fosp(V)_{\bar{0}}$) to consist of all diagonal matrices. Thus, $\fh$ has a basis $\{F_{ii}\}_{i=1}^{\sfr}$
with $\sfr=\lfloor N/2\rfloor+m$. Let $\{e^*_i\}_{i=1}^{\sfr}$ denote the dual basis of $\fh^*$. We consider
the \emph{root space decomposition} $\fosp(V)=\fh\oplus\bigoplus_{\alpha\in \Delta} \fosp(V)_\alpha$, where
$\Delta\subset \fh^*$ is the root system. We also have a decomposition $\Delta=\Delta_{0}\cup \Delta_{1}$
into \emph{even} and \emph{odd}~roots.


\subsection{Dynkin diagrams with labels via parity sequences}
\label{ssec:Dynkin diagrams}
\

In this subsection, we explain how various Dynkin diagrams (with labels) of the orthosymplectic
Lie superalgebras $\fosp(V)$ can be easily read off the corresponding \textbf{parity sequence}
\begin{equation}\label{eq:parity sequence}
  \Parity:=
  \left(|v_1|,\ldots,|v_{\sfr}|\right) = \left(\ol{1},\ldots,\ol{\sfr}\, \right) \in \big\{\bar{0},\bar{1}\big\}^{\sfr}
  \qquad \mathrm{where} \quad \sfr = \lfloor N/2 \rfloor+m \,.
\end{equation}

\medskip
Following~\cite[\S2.1]{z} (cf.~\cite[\S2.2]{fss}), let us first recall the construction of the Cartan
matrices and Dynkin diagrams for the orthosymplectic Lie superalgebras $\fosp(V)$. To this end, we
consider the non-degenerate invariant bilinear form $(\cdot,\cdot)\colon \fosp(V)\times \fosp(V)\to \BC$
defined via
  $$(X,Y)=\frac{1}{2}\mathrm{sTr}(XY) \,,$$
that is the \emph{supertrace} form associated with the natural action $\fosp(V)\curvearrowright V$.
Its restriction to the Cartan subalgebra $\fh$ of $\fosp(V)$ is non-degenerate, thus giving rise to an
identification $\fh\simeq \fh^*$ and inducing a bilinear form $(\cdot,\cdot)\colon \fh^*\times \fh^*\to \BC$.
Explicitly, we have (for $1\leq i,j\leq \sfr$):
\begin{equation}\label{eq:form-dual-Cartan}
  (e^*_i,e^*_j)=\delta_{ij}(-1)^{\ol{i}}=\delta_{ij}\cdot
   \begin{cases}
     1 & \mbox{if } v_i\in V_{\bar{0}} \\
    -1 & \mbox{if } v_i\in V_{\bar{1}}
   \end{cases} \,.
\end{equation}

\begin{Rem}
We note that~\cite{fss} used $\{\epsilon_k\}_{k=1}^{\lfloor N/2\rfloor} \cup \{\delta_i\}_{i=1}^m$ with
$(\epsilon_k,\epsilon_l)=\mp \delta_{kl}$, $(\delta_i,\delta_j)=\pm \delta_{ij}$. Our uniform choice
of $\{e^*_i\}_{i=1}^{\lfloor N/2\rfloor+m}$ with the pairing~\eqref{eq:form-dual-Cartan} is better suited
for the discussions~below.
\end{Rem}

A root $\beta\in \Delta$ is called \emph{isotropic} if $(\beta,\beta)=0$ (in particular, $\beta\in \Delta_1$).
In what follows, we need
\begin{equation}\label{eq:weird-constant}
  l^2_{\min}:=\min\big\{|(\beta,\beta)| \,\big|\, \beta\in \Delta, \mathrm{not\ isotropic}\big\} \,.
\end{equation}
Let $\Pi=\{\alpha_1,\ldots,\alpha_\sfr\}$ be the set of simple roots of $\Delta$, relative to a Borel subalgebra of
$\fosp(V)$ (that is, the maximal solvable subalgebra of $\fosp(V)$ containing a Borel subalgebra of $\fosp(V)_{\bar{0}}$).
Define the \emph{symmetrized Cartan matrix} of $\fosp(V)$ associated with the choice $\Pi$ of simple roots via
\begin{equation}\label{eq:symm-Cartan-matrix}
  B=(b_{ij})_{i,j=1}^\sfr  \quad \mathrm{with} \quad b_{ij}=(\alpha_i,\alpha_j) \,.
\end{equation}
We also define the diagonal matrix $D=\mathrm{diag}(d_1,\ldots,d_\sfr)$ via (cf.~\eqref{eq:weird-constant})
\begin{equation}\label{eq:Cartan-symmetrizers}
  d_i=
  \begin{cases}
    \frac{(\alpha_i,\alpha_i)}{2} & \mbox{if } (\alpha_i,\alpha_i)\ne 0 \\
    l^2_{\min}/2^{\varkappa} & \mbox{if } (\alpha_i,\alpha_i)=0
  \end{cases} \,,
    \qquad \mathrm{where} \quad
  \varkappa=
  \begin{cases}
    0 & \mbox{if } N\ \mathrm{is\ odd} \\
    1 & \mbox{if } N\ \mathrm{is\ even}
  \end{cases} \,.
\end{equation}
Finally, we define the \emph{Cartan matrix} of $\fosp(V)$ associated with the choice $\Pi$ of simple roots via
\begin{equation}\label{eq:nonsymm-Cartan-matrix}
  A=D^{-1}B=(a_{ij})_{i,j=1}^\sfr \,.
\end{equation}

\medskip
Let us now recall a construction of the \emph{Dynkin diagram} of $\fosp(V)$ from the Cartan matrix $A$.
It is a graph with $\sfr$ vertices, colored in one of the three colors: vertex $i$ is colored white
  $\begin{tikzpicture}
    \begin{scope}[shift={(0.7,0)}]
     \draw[fill=white!50] (5,0.75) circle (0.3cm);
    \end{scope}
   \end{tikzpicture}$
if $\alpha_i$ is an even root, gray
  $\begin{tikzpicture}
    \begin{scope}[shift={(0.7,0)}]
      \draw[fill=gray!50] (5,0.75) circle (0.3cm);
    \end{scope}
   \end{tikzpicture}$
if $\alpha_i$ is an odd isotropic root, black
  $\begin{tikzpicture}
    \begin{scope}[shift={(0.7,0)}]
      \draw[fill=black] (5,0.75) circle (0.3cm);
    \end{scope}
   \end{tikzpicture}$
if $\alpha_i$ is an odd not isotropic root.
We join $i$-th and $j$-th vertices with $n_{ij}$ lines, where:
\begin{equation}\label{eq:Dynkin-edges}
  n_{ij}=
  \begin{cases}
    \max\{|a_{ij}|,|a_{ji}|\} & \mbox{if } a_{ii}+a_{jj}\geq 2 \\
    |a_{ij}| & \mbox{if } a_{ii}=a_{jj}=0
  \end{cases} \,.
\end{equation}
Finally, if the $i$-th vertex is not gray and is connected by more than one edge to the $j$-th vertex, then
we orient them from $i$-th towards $j$-th if $a_{ij}=-1$, and from $j$-th towards $i$-th if $a_{ij}<-1$.

\medskip
In the discussions below, we follow the notations of~\cite{fss}:
\begin{itemize}

\item[--]
use a small black dot
  $\begin{tikzpicture}
    \begin{scope}[shift={(0.7,0)}]
      \draw[fill=black] (0.3,0) circle (0.2cm);
    \end{scope}
   \end{tikzpicture}$
in a Dynkin diagram to represent a white or gray vertex

\item[--]
use an integer $K$ to denote the number of gray vertices among those small black dots.

\end{itemize}
The corresponding Lie superalgebras form four classical series, which we now treat case-by-case.

\medskip
\noindent
$\bullet$ $N=2n$ with $n>1$ (which corresponds to the so-called $D(n,m)$-series).

In this case, the root system is (cf.~\cite[(2.9)]{fss}):
\begin{equation*}
  \Delta =
  \Big\{ \pm e^*_i \pm e^*_j \,\Big|\, 1\leq i < j\leq n+m \Big\} \bigcup
  \Big\{ \pm 2e^*_i \,\Big|\, v_i\in V_{\bar{1}} \,, 1\leq i\leq n+m \Big\} \,.
\end{equation*}
The latter follows from the explicit description of the basis~\eqref{eq:F-basis}, in particular, $2e^*_i$
correspond to nonzero $F_{ii'}$. The choice of simple roots crucially depends on the $\BZ_2$-parity of
the vector $v_{n+m}$:
\begin{itemize}

\item[(1)]
If $v_{n+m}\in V_{\bar{0}}$, then the simple positive roots are the same as in the $D_{n+m}$-type:
\begin{equation*}
  \alpha_1=e^*_1-e^*_2 \,,\, \alpha_2=e^*_2-e^*_3 \,,\, \ldots \,,\,
  \alpha_{n+m-1}=e^*_{n+m-1}-e^*_{n+m} \,,\, \alpha_{n+m}=e^*_{n+m-1}+e^*_{n+m} \,;
\end{equation*}

\item[(2)]
If $v_{n+m}\in V_{\bar{1}}$, then the simple positive roots are as follows:
\begin{equation*}
  \alpha_1=e^*_1-e^*_2 \,,\, \alpha_2=e^*_2-e^*_3 \,,\, \ldots \,,\,
  \alpha_{n+m-1}=e^*_{n+m-1}-e^*_{n+m} \,,\, \alpha_{n+m}=2e^*_{n+m} \,,
\end{equation*}
since we have $e^*_{n+m-1}+e^*_{n+m}=(e^*_{n+m-1}-e^*_{n+m})+2e^*_{n+m}$.

\end{itemize}
Likewise, the highest root $\theta$ depends on the $\BZ_2$-parity of the vector $v_1$:
\begin{itemize}

\item[(A)]
If $v_{1}\in V_{\bar{0}}$, then $\theta=e^*_1+e^*_2$ as in the $D_{n+m}$-type;

\item[(B)]
If $v_{1}\in V_{\bar{1}}$, then $\theta=2e^*_1$.

\end{itemize}

Let us now use the above to read off the Dynkin diagrams of~\cite{fss} together with their labels
$\{a_i\}_{i=1}^{n+m}$, the latter defined as the coefficients of the highest root in the basis of
simple roots
  \[ \theta=\sum_{i=1}^{n+m} a_i\alpha_i \,. \]

\noindent
\underline{Case 1}: $\Parity=(\bar{1} , \ast , \ldots , \ast , \bar{1} , \bar{0})$
with each $\ast$ being either $\bar{0}$ or $\bar{1}$.

In this case, we get the following diagram with labels from~\cite[Table~2]{fss}:
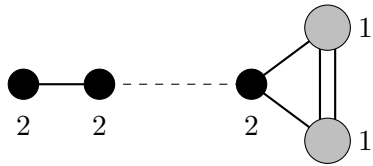
\begin{equation*}
\begin{tikzpicture}
\foreach \a in {1} {
    \begin{scope}[shift={(0.7*\a,0)}]
      \draw[fill=black] (0.3*\a,0) circle (0.2cm);
      \draw[black,thick] (0.3*\a+0.2,0)--(0.3*\a+0.8,0);
    \end{scope}
  }
      \draw[black,dashed] (2.1,0)--(4,0);
      \draw[black,thick] (4,0)--(5,0.75);
      \draw[black,thick] (4,0)--(5,-0.75);
      \draw[black,thick] (4.9,1)--(4.9,-1);
      \draw[black,thick] (5.1,1)--(5.1,-1);
      \draw[fill=black] (2,0) circle (0.2cm);
      \draw[fill=black] (4,0) circle (0.2cm);
      \draw[fill=gray!50] (5,0.75) circle (0.3cm);
      \draw[fill=gray!50] (5,-0.75) circle (0.3cm);
       \node [below] at (2,-0.3) {$2$};
       \node [below] at (1,-0.3) {$2$};
       \node [below] at (4,-0.3) {$2$};
       \node [] at (5.5,-0.75) {$1$};
       \node [] at (5.5,0.75) {$1$};
\end{tikzpicture}
\end{equation*}
Indeed, we have $(\alpha_{n+m-1},\alpha_{n+m-1})=(\alpha_{n+m},\alpha_{n+m})=0$ and
$(\alpha_{n+m-1},\alpha_{n+m})=-2\ne 0$. The number $K$ of gray dots among
$\begin{tikzpicture}
\foreach \a in {1} {
    \begin{scope}[shift={(0.7*\a,0)}]
      \draw[fill=black] (0.3*\a,0) circle (0.2cm);
    \end{scope}
  }
\end{tikzpicture}$
is \unl{even} since it equals the number of $1\leq i\leq n+m-2$ such that $\ol{i}\ne \ol{i+1}$
and $\ol{1}=\ol{n+m-1}$. Finally, the labels on the diagram are read off the equality:
\begin{equation*}
  2e^*_1=2(e^*_1-e^*_2)+\cdots+2(e^*_{n+m-2}-e^*_{n+m-1})+(e^*_{n+m-1}-e^*_{n+m})+(e^*_{n+m-1}+e^*_{n+m}) \,.
\end{equation*}

\noindent
\underline{Case 2}: $\Parity=(\bar{0} , \ast , \ldots , \ast , \bar{1} , \bar{0})$
with each $\ast$ being either $\bar{0}$ or $\bar{1}$.

In this case, we get the following diagram with labels from~\cite[Table~2]{fss}:
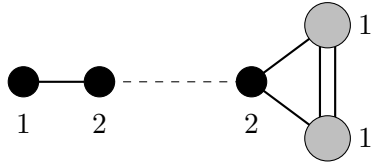
\begin{equation*}
\begin{tikzpicture}
\foreach \a in {1} {
    \begin{scope}[shift={(0.7*\a,0)}]
      \draw[fill=black] (0.3*\a,0) circle (0.2cm);
      \draw[black,thick] (0.3*\a+0.2,0)--(0.3*\a+0.8,0);
    \end{scope}
  }
      \draw[black,dashed] (2.1,0)--(4,0);
      \draw[black,thick] (4,0)--(5,0.75);
      \draw[black,thick] (4,0)--(5,-0.75);
      \draw[black,thick] (4.9,1)--(4.9,-1);
      \draw[black,thick] (5.1,1)--(5.1,-1);
      \draw[fill=black] (2,0) circle (0.2cm);
      \draw[fill=black] (4,0) circle (0.2cm);
      \draw[fill=gray!50] (5,0.75) circle (0.3cm);
      \draw[fill=gray!50] (5,-0.75) circle (0.3cm);
       \node [below] at (2,-0.3) {$2$};
       \node [below] at (1,-0.3) {$1$};
       \node [below] at (4,-0.3) {$2$};
       \node [] at (5.5,-0.75) {$1$};
       \node [] at (5.5,0.75) {$1$};
\end{tikzpicture}
\end{equation*}
This is analogous to the Case 1, except that now $K$ is \unl{odd} and the labels on the diagram
are rather read off the following equality:
\begin{equation*}
  e^*_1+e^*_2 =
  (e^*_1-e^*_2)+2(e^*_2-e^*_3)+\cdots+2(e^*_{n+m-2}-e^*_{n+m-1})+(e^*_{n+m-1}-e^*_{n+m})+(e^*_{n+m-1}+e^*_{n+m}) \,.
\end{equation*}

\noindent
\underline{Case 3}: $\Parity=(\bar{0} , \ast , \ldots , \ast , \bar{0} , \bar{0})$
with each $\ast$ being either $\bar{0}$ or $\bar{1}$.

In this case, we get the following diagram with labels from~\cite[Table~2]{fss}:
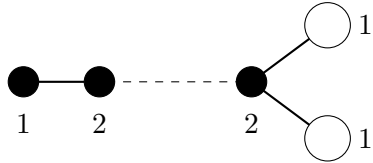
\begin{equation*}
\begin{tikzpicture}
\foreach \a in {1} {
    \begin{scope}[shift={(0.7*\a,0)}]
      \draw[fill=black] (0.3*\a,0) circle (0.2cm);
      \draw[black,thick] (0.3*\a+0.2,0)--(0.3*\a+0.8,0);
    \end{scope}
  }
      \draw[black,dashed] (2.1,0)--(4,0);
      \draw[black,thick] (4,0)--(5,0.75);
      \draw[black,thick] (4,0)--(5,-0.75);
      \draw[fill=black] (2,0) circle (0.2cm);
      \draw[fill=black] (4,0) circle (0.2cm);
      \draw[fill=white!50] (5,0.75) circle (0.3cm);
      \draw[fill=white!50] (5,-0.75) circle (0.3cm);
       \node [below] at (2,-0.3) {$2$};
       \node [below] at (1,-0.3) {$1$};
       \node [below] at (4,-0.3) {$2$};
       \node [] at (5.5,-0.75) {$1$};
       \node [] at (5.5,0.75) {$1$};
\end{tikzpicture}
\end{equation*}
Indeed, we have $(\alpha_{n+m-1},\alpha_{n+m-1})=(\alpha_{n+m},\alpha_{n+m})=2$,
$(\alpha_{n+m-1},\alpha_{n+m})=0$. The number $K$ of gray dots among
$\begin{tikzpicture}
\foreach \a in {1} {
    \begin{scope}[shift={(0.7*\a,0)}]
      \draw[fill=black] (0.3*\a,0) circle (0.2cm);
    \end{scope}
  }
\end{tikzpicture}$
is \unl{even} since it equals the number of $1\leq i\leq n+m-2$ such that $\ol{i}\ne \ol{i+1}$
and $\ol{1}=\ol{n+m-1}$. The labels on the diagram are read off the same equality as in Case 2:
\begin{equation*}
  e^*_1+e^*_2 =
  (e^*_1-e^*_2)+2(e^*_2-e^*_3)+\cdots+2(e^*_{n+m-2}-e^*_{n+m-1})+(e^*_{n+m-1}-e^*_{n+m})+(e^*_{n+m-1}+e^*_{n+m}) \,.
\end{equation*}

\noindent
\underline{Case 4}: $\Parity=(\bar{1} , \ast , \ldots , \ast , \bar{0} , \bar{0})$
with each $\ast$ being either $\bar{0}$ or $\bar{1}$.

In this case, we get the following diagram with labels from~\cite[Table~2]{fss}:
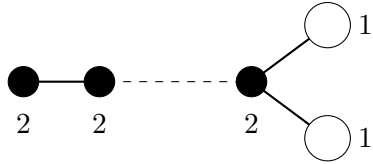
\begin{equation*}
\begin{tikzpicture}
\foreach \a in {1} {
    \begin{scope}[shift={(0.7*\a,0)}]
      \draw[fill=black] (0.3*\a,0) circle (0.2cm);
      \draw[black,thick] (0.3*\a+0.2,0)--(0.3*\a+0.8,0);
    \end{scope}
  }
      \draw[black,dashed] (2.1,0)--(4,0);
      \draw[black,thick] (4,0)--(5,0.75);
      \draw[black,thick] (4,0)--(5,-0.75);
      \draw[fill=black] (2,0) circle (0.2cm);
      \draw[fill=black] (4,0) circle (0.2cm);
      \draw[fill=white!50] (5,0.75) circle (0.3cm);
      \draw[fill=white!50] (5,-0.75) circle (0.3cm);
       \node [below] at (2,-0.3) {$2$};
       \node [below] at (1,-0.3) {$2$};
       \node [below] at (4,-0.3) {$2$};
       \node [] at (5.5,-0.75) {$1$};
       \node [] at (5.5,0.75) {$1$};
\end{tikzpicture}
\end{equation*}
This is analogous to the Case 3, except that now $K$ is \unl{odd} and the labels on the diagram
are rather read off the same equality as in Case 1:
\begin{equation*}
  2e^*_1 =
  2(e^*_1-e^*_2)+2(e^*_2-e^*_3)+\cdots+2(e^*_{n+m-2}-e^*_{n+m-1})+(e^*_{n+m-1}-e^*_{n+m})+(e^*_{n+m-1}+e^*_{n+m}) \,.
\end{equation*}

\noindent
\underline{Case 5}: $\Parity=(\bar{1} , \ast , \ldots , \ast , \bar{1})$
with each $\ast$ being either $\bar{0}$ or $\bar{1}$.

In this case, we get the following diagram with labels from~\cite[Table~2]{fss}:
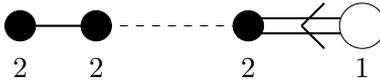
\begin{equation*}
\begin{tikzpicture}
\foreach \a in {1} {
    \begin{scope}[shift={(0.7*\a,0)}]
      \draw[fill=black] (0.3*\a,0) circle (0.2cm);
      \draw[black,thick] (0.3*\a+0.2,0)--(0.3*\a+0.8,0);
    \end{scope}
  }
      \draw[black,dashed] (2.1,0)--(4,0);
      \draw[black,thick] (4,-0.1)--(5.5,-0.1);
      \draw[black,thick] (4,0.1)--(5.5,0.1);
      \draw[black,thick] (4.7,0)--(5,0.3);
      \draw[black,thick] (4.7,0)--(5,-0.3);
      \draw[fill=black] (2,0) circle (0.2cm);
      \draw[fill=black] (4,0) circle (0.2cm);
      \draw[fill=white] (5.5,0) circle (0.3cm);
       \node [below] at (2,-0.3) {$2$};
       \node [below] at (1,-0.3) {$2$};
       \node [below] at (4,-0.3) {$2$};
       \node [below] at (5.5,-0.3) {$1$};
\end{tikzpicture}
\end{equation*}
Indeed, we have $\alpha_{n+m}=2e^*_{n+m}$ so that $(\alpha_{n+m},\alpha_{n+m})=-4$
and $(\alpha_{n+m-1},\alpha_{n+m})=2$. The number $K$ of gray dots among
$\begin{tikzpicture}
\foreach \a in {1} {
    \begin{scope}[shift={(0.7*\a,0)}]
      \draw[fill=black] (0.3*\a,0) circle (0.2cm);
    \end{scope}
  }
\end{tikzpicture}$
is \unl{even} since it equals the number of $1\leq i\leq n+m-1$ such that $\ol{i}\ne \ol{i+1}$
and $\ol{1}=\ol{n+m}$. The labels on the diagram are read off the following equality:
\begin{equation*}
  2e^*_1=2(e^*_1-e^*_2)+2(e^*_2-e^*_3)+\cdots+2(e^*_{n+m-1}-e^*_{n+m})+(2e^*_{n+m}) \,.
\end{equation*}

\noindent
\underline{Case 6}: $\Parity=(\bar{0} , \ast , \ldots , \ast , \bar{1})$
with each $\ast$ being either $\bar{0}$ or $\bar{1}$.

In this case, we get the following diagram with labels from~\cite[Table~2]{fss}:
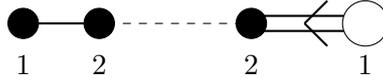
\begin{equation*}
\begin{tikzpicture}
\foreach \a in {1} {
    \begin{scope}[shift={(0.7*\a,0)}]
      \draw[fill=black] (0.3*\a,0) circle (0.2cm);
      \draw[black,thick] (0.3*\a+0.2,0)--(0.3*\a+0.8,0);
    \end{scope}
  }
      \draw[black,dashed] (2.1,0)--(4,0);
      \draw[black,thick] (4,-0.1)--(5.5,-0.1);
      \draw[black,thick] (4,0.1)--(5.5,0.1);
      \draw[black,thick] (4.7,0)--(5,0.3);
      \draw[black,thick] (4.7,0)--(5,-0.3);
      \draw[fill=black] (2,0) circle (0.2cm);
      \draw[fill=black] (4,0) circle (0.2cm);
      \draw[fill=white] (5.5,0) circle (0.3cm);
       \node [below] at (2,-0.3) {$2$};
       \node [below] at (1,-0.3) {$1$};
       \node [below] at (4,-0.3) {$2$};
       \node [below] at (5.5,-0.3) {$1$};
\end{tikzpicture}
\end{equation*}
This is analogous to the Case 5, except that now $K$ is \unl{odd} and the labels on
the diagram are rather read off the following equality:
\begin{equation*}
  e^*_1+e^*_2=(e^*_1-e^*_2)+2(e^*_2-e^*_3)+\cdots+2(e^*_{n+m-1}-e^*_{n+m})+(2e^*_{n+m}) \,.
\end{equation*}

\medskip
\noindent
$\bullet$ $N=2$ (which corresponds to the so-called $C(m+1)$-series.\footnote{We warn the reader not to
confuse this with the symplectic $C_{m+1}$-series, corresponding to $\fosp(0|2m+2)$.})

The descriptions of simple roots $\{\alpha_i\}_{i=1}^{m+1}$ and the highest root $\theta$ are the same as for
even $N>2$. The corresponding parity sequence $\Upsilon_V$ consists of a single $\bar{0}$ and $m$ $\bar{1}$'s,
hence, the following cases:

1) For $\Upsilon_V=(\bar{0} , \bar{1} , \ldots , \bar{1})$, one clearly obtains the (labelled) Dynkin diagram
of~\cite[p.~463]{fss}:\footnote{We note that this choice actually differs from the standard choice made
in~\cite{m} for even $N>2$, cf.~\eqref{eq:distinguished=standard}.}
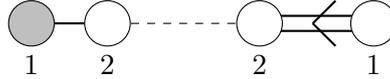
\begin{equation*}
\begin{tikzpicture}
\foreach \a in {1} {
    \begin{scope}[shift={(0.7*\a,0)}]
      \draw[fill=gray!50] (0.3*\a,0) circle (0.3cm);
      \draw[black,thick] (0.3*\a+0.3,0)--(0.3*\a+0.7,0);
    \end{scope}
  }
      \draw[black,dashed] (2.1,0)--(4,0);
      \draw[black,thick] (4,-0.1)--(5.5,-0.1);
      \draw[black,thick] (4,0.1)--(5.5,0.1);
      \draw[black,thick] (4.7,0)--(5,0.3);
      \draw[black,thick] (4.7,0)--(5,-0.3);
      \draw[fill=white] (2,0) circle (0.3cm);
      \draw[fill=white] (4,0) circle (0.3cm);
      \draw[fill=white] (5.5,0) circle (0.3cm);
       \node [below] at (2,-0.3) {$2$};
       \node [below] at (1,-0.3) {$1$};
       \node [below] at (4,-0.3) {$2$};
       \node [below] at (5.5,-0.3) {$1$};
\end{tikzpicture}
\end{equation*}

2) For $\Upsilon_V=(\bar{1} , \ldots , \bar{1} , \bar{0} , \bar{1} , \ldots , \bar{1})$, one obtains the
following (labelled) Dynkin diagram with two consecutive black dots being gray and the rest being white:
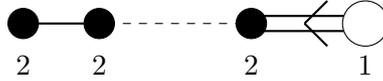
\begin{equation*}
\begin{tikzpicture}
\foreach \a in {1} {
    \begin{scope}[shift={(0.7*\a,0)}]
      \draw[fill=black] (0.3*\a,0) circle (0.2cm);
      \draw[black,thick] (0.3*\a+0.2,0)--(0.3*\a+0.8,0);
    \end{scope}
  }
      \draw[black,dashed] (2.1,0)--(4,0);
      \draw[black,thick] (4,-0.1)--(5.5,-0.1);
      \draw[black,thick] (4,0.1)--(5.5,0.1);
      \draw[black,thick] (4.7,0)--(5,0.3);
      \draw[black,thick] (4.7,0)--(5,-0.3);
      \draw[fill=black] (2,0) circle (0.2cm);
      \draw[fill=black] (4,0) circle (0.2cm);
      \draw[fill=white] (5.5,0) circle (0.3cm);
       \node [below] at (2,-0.3) {$2$};
       \node [below] at (1,-0.3) {$2$};
       \node [below] at (4,-0.3) {$2$};
       \node [below] at (5.5,-0.3) {$1$};
\end{tikzpicture}
\end{equation*}

3) For $\Upsilon_V=(\bar{1} , \ldots , \bar{1} , \bar{0})$, one obtains the following (labelled) Dynkin diagram:
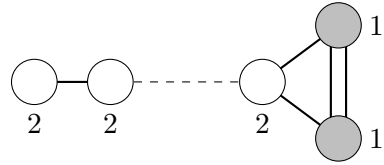
\begin{equation*}
\begin{tikzpicture}
\foreach \a in {1} {
    \begin{scope}[shift={(0.7*\a,0)}]
      \draw[fill=white] (0.3*\a,0) circle (0.3cm);
      \draw[black,thick] (0.3*\a+0.3,0)--(0.3*\a+0.7,0);
    \end{scope}
  }
      \draw[black,dashed] (2.1,0)--(4,0);
      \draw[black,thick] (4,0)--(5,0.75);
      \draw[black,thick] (4,0)--(5,-0.75);
      \draw[black,thick] (4.9,1)--(4.9,-1);
      \draw[black,thick] (5.1,1)--(5.1,-1);
      \draw[fill=white] (2,0) circle (0.3cm);
      \draw[fill=white] (4,0) circle (0.3cm);
      \draw[fill=gray!50] (5,0.75) circle (0.3cm);
      \draw[fill=gray!50] (5,-0.75) circle (0.3cm);
       \node [below] at (2,-0.3) {$2$};
       \node [below] at (1,-0.3) {$2$};
       \node [below] at (4,-0.3) {$2$};
       \node [] at (5.5,-0.75) {$1$};
       \node [] at (5.5,0.75) {$1$};
\end{tikzpicture}
\end{equation*}

\medskip
\noindent
$\bullet$ $N=2n+1$ with $n\geq 1$ (which corresponds to the so-called $B(n,m)$-series).

In this case, the root system is (cf.~\cite[(2.7)]{fss}):
\begin{equation*}
  \Delta =
  \Big\{ \pm e^*_i \pm e^*_j \,\Big|\, 1\leq i < j\leq n+m \Big\} \bigcup
  \Big\{ \pm e^*_i \,\Big|\, 1\leq i\leq n+m \Big\} \bigcup
  \Big\{ \pm 2e^*_i \,\Big|\, v_i\in V_{\bar{1}} \,, 1\leq i\leq n+m \Big\} \,.
\end{equation*}
The latter follows from the explicit description of the basis~\eqref{eq:F-basis}.
In contrast to the case of even $N$, the simple roots are uniformly given by:
\begin{equation*}
  \alpha_1=e^*_1-e^*_2 \,,\, \alpha_2=e^*_2-e^*_3 \,,\, \ldots \,,\,
  \alpha_{n+m-1}=e^*_{n+m-1}-e^*_{n+m} \,,\, \alpha_{n+m}=e^*_{n+m} \,.
\end{equation*}
Similarly to the case of even $N$, the highest root $\theta$ depends on the parity of $v_1$:
\begin{equation*}
  \theta=
  \begin{cases}
    e^*_1+e^*_2 & \mbox{if } v_{1}\in V_{\bar{0}} \\
    2e^*_1 & \mbox{if } v_{1}\in V_{\bar{1}}
  \end{cases} \,.
\end{equation*}
We shall now match the (labelled) Dynkin diagrams of~\cite[Table 2]{fss} with the parity sequences.

\medskip
\noindent
\underline{Case 1}: $\Parity=(\bar{0} , \ast , \ldots , \ast , \bar{0})$
with each $\ast$ being either $\bar{0}$ or $\bar{1}$.

In this case, we get the following diagram with labels from~\cite[Table~2]{fss}:
\begin{equation*}
\begin{tikzpicture}
\foreach \a in {1} {
    \begin{scope}[shift={(0.7*\a,0)}]
      \draw[fill=black] (0.3*\a,0) circle (0.2cm);
      \draw[black,thick] (0.3*\a+0.2,0)--(0.3*\a+0.8,0);
    \end{scope}
  }
      \draw[black,dashed] (2.1,0)--(4,0);
      \draw[black,thick] (4,-0.1)--(5.5,-0.1);
      \draw[black,thick] (4,0.1)--(5.5,0.1);
      \draw[black,thick] (4.7,0.3)--(5,0);
      \draw[black,thick] (4.7,-0.3)--(5,-0);
      \draw[fill=black] (2,0) circle (0.2cm);
      \draw[fill=black] (4,0) circle (0.2cm);
      \draw[fill=white] (5.5,0) circle (0.3cm);
       \node [below] at (2,-0.3) {$2$};
       \node [below] at (1,-0.3) {$1$};
       \node [below] at (4,-0.3) {$2$};
       \node [below] at (5.5,-0.3) {$2$};
\end{tikzpicture}
\end{equation*}
The number $K$ of gray dots among
$\begin{tikzpicture}
\foreach \a in {1} {
    \begin{scope}[shift={(0.7*\a,0)}]
      \draw[fill=black] (0.3*\a,0) circle (0.2cm);
    \end{scope}
  }
\end{tikzpicture}$
is \unl{even} since it equals the number of $1\leq i\leq n+m-1$ such that $\ol{i}\ne \ol{i+1}$
and $\ol{1}=\ol{n+m}$, while the labels on the diagram are read off the following equality:
\begin{equation*}
  e^*_1+e^*_2=(e^*_1-e^*_2)+2(e^*_2-e^*_3)+\cdots+2(e^*_{n+m-1}-e^*_{n+m})+2(e^*_{n+m}) \,.
\end{equation*}

\noindent
\underline{Case 2}: $\Parity=(\bar{1} , \ast , \ldots , \ast , \bar{0})$
with each $\ast$ being either $\bar{0}$ or $\bar{1}$.

In this case, we get the following diagram with labels from~\cite[Table~2]{fss}:
\begin{equation*}
 \begin{tikzpicture}
\foreach \a in {1} {
    \begin{scope}[shift={(0.7*\a,0)}]
      \draw[fill=black] (0.3*\a,0) circle (0.2cm);
      \draw[black,thick] (0.3*\a+0.2,0)--(0.3*\a+0.8,0);
    \end{scope}
  }
      \draw[black,dashed] (2.1,0)--(4,0);
      \draw[black,thick] (4,-0.1)--(5.5,-0.1);
      \draw[black,thick] (4,0.1)--(5.5,0.1);
      \draw[black,thick] (4.7,0.3)--(5,0);
      \draw[black,thick] (4.7,-0.3)--(5,-0);
      \draw[fill=black] (2,0) circle (0.2cm);
      \draw[fill=black] (4,0) circle (0.2cm);
      \draw[fill=white] (5.5,0) circle (0.3cm);
       \node [below] at (2,-0.3) {$2$};
       \node [below] at (1,-0.3) {$2$};
       \node [below] at (4,-0.3) {$2$};
       \node [below] at (5.5,-0.3) {$2$};
\end{tikzpicture}
\end{equation*}
The number $K$ of gray dots among
$\begin{tikzpicture}
\foreach \a in {1} {
    \begin{scope}[shift={(0.7*\a,0)}]
      \draw[fill=black] (0.3*\a,0) circle (0.2cm);
    \end{scope}
  }
\end{tikzpicture}$
is \unl{odd} since it equals the number of $1\leq i\leq n+m-1$ such that $\ol{i}\ne \ol{i+1}$
and $\ol{1}\ne \ol{n+m}$, while the labels on the diagram are read off the following equality:
\begin{equation*}
  2e^*_1=2(e^*_1-e^*_2)+2(e^*_2-e^*_3)+\cdots+2(e^*_{n+m-1}-e^*_{n+m})+2(e^*_{n+m}) \,.
\end{equation*}

\noindent
\underline{Case 3}: $\Parity=(\bar{1} , \ast , \ldots , \ast , \bar{1})$
with each $\ast$ being either $\bar{0}$ or $\bar{1}$.

In this case, we get the following diagram with labels from~\cite[Table~2]{fss}:
\begin{equation*}
\begin{tikzpicture}
\foreach \a in {1} {
    \begin{scope}[shift={(0.7*\a,0)}]
      \draw[fill=black] (0.3*\a,0) circle (0.2cm);
      \draw[black,thick] (0.3*\a+0.2,0)--(0.3*\a+0.8,0);
    \end{scope}
  }
      \draw[black,dashed] (2.1,0)--(4,0);
      \draw[black,thick] (4,-0.1)--(5.5,-0.1);
      \draw[black,thick] (4,0.1)--(5.5,0.1);
      \draw[black,thick] (4.7,0.3)--(5,0);
      \draw[black,thick] (4.7,-0.3)--(5,-0);
      \draw[fill=black] (2,0) circle (0.2cm);
      \draw[fill=black] (4,0) circle (0.2cm);
      \draw[fill=black] (5.5,0) circle (0.3cm);
       \node [below] at (2,-0.3) {$2$};
       \node [below] at (1,-0.3) {$2$};
       \node [below] at (4,-0.3) {$2$};
       \node [below] at (5.5,-0.3) {$2$};
\end{tikzpicture}
\end{equation*}
The number $K$ of gray dots among
$\begin{tikzpicture}
\foreach \a in {1} {
    \begin{scope}[shift={(0.7*\a,0)}]
      \draw[fill=black] (0.3*\a,0) circle (0.2cm);
    \end{scope}
  }
\end{tikzpicture}$
is \unl{even} since it equals the number of $1\leq i\leq n+m-1$ such that $\ol{i}\ne \ol{i+1}$
and $\ol{1}=\ol{n+m}$, while the labels are read off the same equality as in Case 2:
\begin{equation*}
  2e^*_1=2(e^*_1-e^*_2)+2(e^*_2-e^*_3)+\cdots+2(e^*_{n+m-1}-e^*_{n+m})+2(e^*_{n+m}) \,.
\end{equation*}

\noindent
\underline{Case 4}: $\Parity=(\bar{0} , \ast , \ldots , \ast , \bar{1})$
with each $\ast$ being either $\bar{0}$ or $\bar{1}$.

In this case, we get the following diagram with labels from~\cite[Table~2]{fss}:
\begin{equation*}
\begin{tikzpicture}
\foreach \a in {1} {
    \begin{scope}[shift={(0.7*\a,0)}]
      \draw[fill=black] (0.3*\a,0) circle (0.2cm);
      \draw[black,thick] (0.3*\a+0.2,0)--(0.3*\a+0.8,0);
    \end{scope}
  }
      \draw[black,dashed] (2.1,0)--(4,0);
      \draw[black,thick] (4,-0.1)--(5.5,-0.1);
      \draw[black,thick] (4,0.1)--(5.5,0.1);
      \draw[black,thick] (4.7,0.3)--(5,0);
      \draw[black,thick] (4.7,-0.3)--(5,-0);
      \draw[fill=black] (2,0) circle (0.2cm);
      \draw[fill=black] (4,0) circle (0.2cm);
      \draw[fill=black] (5.5,0) circle (0.3cm);
       \node [below] at (2,-0.3) {$2$};
       \node [below] at (1,-0.3) {$1$};
       \node [below] at (4,-0.3) {$2$};
       \node [below] at (5.5,-0.3) {$2$};
\end{tikzpicture}
\end{equation*}
The number $K$ of gray dots among
$\begin{tikzpicture}
\foreach \a in {1} {
    \begin{scope}[shift={(0.7*\a,0)}]
      \draw[fill=black] (0.3*\a,0) circle (0.2cm);
    \end{scope}
  }
\end{tikzpicture}$
is \unl{odd} since it equals the number of $1\leq i\leq n+m-1$ such that $\ol{i}\ne \ol{i+1}$
and $\ol{1}\ne \ol{n+m}$, while the labels are read off the same equality as in Case 1:
\begin{equation*}
  e^*_1+e^*_2=(e^*_1-e^*_2)+2(e^*_2-e^*_3)+\cdots+2(e^*_{n+m-1}-e^*_{n+m})+2(e^*_{n+m}) \,.
\end{equation*}

\medskip
\noindent
$\bullet$ $N=1$ (which corresponds to the so-called $B(0,m)$-series).

In this case, there is only one parity sequence $\Upsilon_V=(\bar{1} \,,\, \ldots \,,\, \bar{1})$,
that is, $|v_1|=\ldots=|v_{m}|=\bar{1}$. The corresponding root system is (cf.~\cite[(2.8)]{fss}):
\begin{equation*}
  \Delta =
  \Big\{ \pm e^*_i \pm e^*_j \,\Big|\, 1\leq i < j\leq m \Big\} \bigcup
  \Big\{ \pm e^*_i \,\Big|\, 1\leq i\leq m \Big\} \bigcup
  \Big\{ \pm 2e^*_i \,\Big|\, 1\leq i\leq m \Big\} \,,
\end{equation*}
with simple roots given by
\begin{equation*}
  \alpha_1=e^*_1-e^*_2 \,,\, \alpha_2=e^*_2-e^*_3 \,,\, \ldots \,,\,
  \alpha_{m-1}=e^*_{m-1}-e^*_{m} \,,\, \alpha_{m}=e^*_{m} \,,
\end{equation*}
and the highest root
  \[ \theta=2e^*_1 \,. \]
This obviously corresponds to the (labelled) Dynkin diagram of~\cite[p.~463]{fss}:
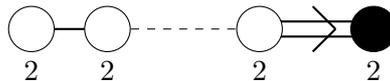
\begin{equation*}
\begin{tikzpicture}
\foreach \a in {1} {
    \begin{scope}[shift={(0.7*\a,0)}]
      \draw[fill=white] (0.3*\a,0) circle (0.3cm);
      \draw[black,thick] (0.3*\a+0.3,0)--(0.3*\a+0.7,0);
    \end{scope}
  }
      \draw[black,dashed] (2.1,0)--(4,0);
      \draw[black,thick] (4,-0.1)--(5.5,-0.1);
      \draw[black,thick] (4,0.1)--(5.5,0.1);
      \draw[black,thick] (4.7,0.3)--(5,0);
      \draw[black,thick] (4.7,-0.3)--(5,-0);
      \draw[fill=white] (2,0) circle (0.3cm);
      \draw[fill=white] (4,0) circle (0.3cm);
      \draw[fill=black] (5.5,0) circle (0.3cm);
       \node [below] at (2,-0.3) {$2$};
       \node [below] at (1,-0.3) {$2$};
       \node [below] at (4,-0.3) {$2$};
       \node [below] at (5.5,-0.3) {$2$};
\end{tikzpicture}
\end{equation*}

\begin{Rem}
We note the following uniform formula for the first label:
  $a_1=
   \begin{cases}
     1 & \mbox{if } |v_1|=\bar{0} \\
     2 & \mbox{if } |v_1|=\bar{1}
  \end{cases}$.
\end{Rem}

The parity sequence~\eqref{eq:parity sequence} is called \textbf{standard} if
\begin{equation}\label{eq:distinguished=standard}
  \Parity=(\bar{1},\ldots,\bar{1},\bar{0},\ldots,\bar{0}) \,.
\end{equation}


\subsection{Chevalley-Serre type presentation}
\label{ssec:Serre-classical}
\

We conclude this section with the Chevalley-Serre type presentation of the orthosymplectic Lie superalgebras.
This result is a partial case of such a presentation for all simple contragredient Lie superalgebras, established
in~\cite[Main Theorem]{z}. Let $A=(a_{ij})_{i,j}$ be the Cartan matrix of~\eqref{eq:nonsymm-Cartan-matrix}.

\begin{Thm}\cite{z}\label{thm:Chevalley-Serre-Zhang}
The Lie superalgebra $\fosp(V)$ is generated by $\{e_i,f_i,h_i\}_{i=1}^{\sfr}$, with the $\BZ_2$-grading
\begin{equation}\label{eq:Z2-grading}
  |e_i|=|f_i|=
  \begin{cases}
     \bar{0} & \mbox{if } \alpha_i\in \Delta_0 \\
     \bar{1} & \mbox{if } \alpha_i\in \Delta_1
  \end{cases} \,, \quad
  |h_i|=\bar{0} \,,
\end{equation}
subject to the quadratic \emph{Chevalley relations}
\begin{equation}\label{eq:Chevalley-relations}
\begin{split}
  & [h_i,h_j]=0 \,, \\
  & [h_i,e_j]=a_{ij}e_j \,, \quad [h_i,f_j]=-a_{ij}f_j \,, \\
  & [e_i,f_j]=\delta_{ij}h_i \,,
\end{split}
\end{equation}
the \emph{standard Serre relations}
\begin{equation}\label{eq:Lie-Serre-standard}
\begin{split}
  & (\mathrm{ad}_{e_i})^{1-a_{ij}}(e_j)=0=(\mathrm{ad}_{f_i})^{1-a_{ij}}(f_j)
    \qquad \mathrm{for} \ i\ne j \,,\ \mathrm{with}\ a_{ii}\ne 0 \ \mathrm{or}\ a_{ij}=0 \,, \\
  & [e_i,e_i]=0=[f_i,f_i] \qquad \mathrm{if} \ a_{ii}=0 \,,
\end{split}
\end{equation}
and the \emph{higher order Serre relations}
(\ref{eq:Lie-Serre-superA},~\ref{eq:Lie-Serre-new3},~\ref{eq:Lie-Serre-new6},~\ref{eq:Lie-Serre-new7})
that are described in details below.
\end{Thm}

\medskip
We shall now specify the aforementioned \emph{higher order Serre relations} of degrees $4$, $3$, $6$, or $7$.

\medskip
\noindent
$\bullet$ For the sub-diagram
\begin{equation}\label{eq:pic-Serre-new4a}
  \begin{tikzpicture}
  \foreach \a in {1} {
    \begin{scope}[shift={(0.7*\a,0)}]
      \draw[fill=black] (0.3*\a,0) circle (0.2cm);
      \draw[black,thick] (0.3*\a+0.2,0)--(0.3*\a+1.3,0);
    \end{scope}
  }
      \draw[black,dashed] (2.1,0)--(2.5,0);
      \draw[black,thick] (2.5,0)--(4,0);
      \draw[fill=gray] (2.5,0) circle (0.3cm);
      \draw[fill=black] (4,0) circle (0.2cm);
       \node [below] at (1,-0.3) {$j$};
       \node [below] at (2.5,-0.3) {$t$};
       \node [below] at (4,-0.3) {$k$};
  \end{tikzpicture}
    \quad \mathrm{with} \quad
  (\alpha_j,\alpha_t)\cdot (\alpha_t,\alpha_k)<0
\end{equation}
or one of the following sub-diagrams
\begin{equation}\label{eq:pic-Serre-new4b}
  \begin{tikzpicture}
  \foreach \a in {1} {
    \begin{scope}[shift={(0.7*\a,0)}]
      \draw[fill=black] (0.3*\a,0) circle (0.2cm);
      \draw[black,thick] (0.3*\a+0.2,0)--(0.3*\a+1.3,0);
    \end{scope}
  }
      \draw[black,dashed] (2.1,0)--(2.5,0);
      \draw[black,thick] (2.5,-0.1)--(4,-0.1);
      \draw[black,thick] (2.5,0.1)--(4,0.1);
      \draw[black,thick] (3.1,0.3)--(3.4,0);
      \draw[black,thick] (3.1,-0.3)--(3.4,0);
      \draw[fill=gray] (2.5,0) circle (0.3cm);
      \draw[fill=white] (4,0) circle (0.3cm);
       \node [below] at (1,-0.3) {$j$};
       \node [below] at (2.5,-0.3) {$t$};
       \node [below] at (4,-0.3) {$k$};
  \end{tikzpicture}
    \qquad \mathrm{or} \qquad
  \begin{tikzpicture}
  \foreach \a in {1} {
    \begin{scope}[shift={(0.7*\a,0)}]
      \draw[fill=black] (0.3*\a,0) circle (0.2cm);
      \draw[black,thick] (0.3*\a+0.2,0)--(0.3*\a+1.3,0);
    \end{scope}
  }
      \draw[black,dashed] (2.1,0)--(2.5,0);
      \draw[black,thick] (2.5,-0.1)--(4,-0.1);
      \draw[black,thick] (2.5,0.1)--(4,0.1);
      \draw[black,thick] (3.1,0.3)--(3.4,0);
      \draw[black,thick] (3.1,-0.3)--(3.4,0);
      \draw[fill=gray] (2.5,0) circle (0.3cm);
      \draw[fill=black] (4,0) circle (0.3cm);
       \node [below] at (1,-0.3) {$j$};
       \node [below] at (2.5,-0.3) {$t$};
       \node [below] at (4,-0.3) {$k$};
  \end{tikzpicture}
\end{equation}
the associated higher order Serre relations are:
\begin{equation}\label{eq:Lie-Serre-superA}
\begin{split}
  & \big[[e_j,e_t],[e_t,e_k]\big]=0 \,, \\
  & \big[[f_j,f_t],[f_t,f_k]\big]=0 \,.
\end{split}
\end{equation}

\begin{Rem}
(a) We note that the relations $\Big[e_t,\big[e_j,[e_t,e_k]\big]\Big]=0$ and
$\Big[f_t,\big[f_j,[f_t,f_k]\big]\Big]=0$ of~\cite[\S3.2.1(1,2,3)]{z} are equivalent
to~\eqref{eq:Lie-Serre-superA}, due to the relations $[e_t,e_t]=0$ and $[f_t,f_t]=0$.

\medskip
\noindent
(b) The relations~\eqref{eq:Lie-Serre-superA} also hold for the analogues
of~(\ref{eq:pic-Serre-new4a},~\ref{eq:pic-Serre-new4b}) with the white $t$-th vertex.
\end{Rem}

In our setup, sub-diagrams~\eqref{eq:pic-Serre-new4b} occur only if $N=2n+1$, $n+m\geq 3$, $\ol{n+m}\ne \ol{n+m-1}$
(and $k=n+m,t=n+m-1,j=n+m-2$). Likewise, sub-diagrams~\eqref{eq:pic-Serre-new4a} occur either if
$t<\lfloor N/2 \rfloor + m - 1$ and $\ol{t}\ne \ol{t+1}$ (with $j=t-1,k=t+1$) or $N=2n,j=n+m-3,t=n+m-2,k=n+m$ and
$\Parity=(\ast , \ldots , \ast , \bar{1} , \bar{0}, \bar{0})$ where each $\ast$ is either $\bar{0}$ or $\bar{1}$.

\begin{Rem}
As noted in~\cite[\S2.2]{z}, the condition $(\alpha_j,\alpha_t)\cdot (\alpha_t,\alpha_k)<0$ in
sub-diagrams~\eqref{eq:pic-Serre-new4a} cannot be ignored. In our setup, that excludes the corresponding
sub-diagrams for $N=2n, n+m\geq 3$, $t=n+m-2,j=n+m-1,k=n+m$, and
$\Parity=(\ast , \ldots , \ast , \bar{1} , \bar{0}, \bar{0})$ where each $\ast\in \{\bar{0},\bar{1}\}$.
\end{Rem}

\medskip
\noindent
$\bullet$ For the sub-diagram
\begin{equation}\label{eq:pic-Serre-new3}
\begin{tikzpicture}
  \foreach \a in {1} {
    \begin{scope}[shift={(0.7*\a,0)}]
      \draw[fill=black] (0.3*\a,0) circle (0.2cm);
      \draw[black,thick] (0.3*\a,0)--(0.3*\a,0);
    \end{scope}
  }
      \draw[black,thick] (1,0)--(2,0.75);
      \draw[black,thick] (1,0)--(2,-0.75);
      \draw[black,thick] (1.9,1)--(1.9,-1);
      \draw[black,thick] (2.1,1)--(2.1,-1);
      \draw[fill=gray!50] (2,0.75) circle (0.3cm);
      \draw[fill=gray!50] (2,-0.75) circle (0.3cm);
       \node [below] at (1,-0.3) {$i$};
       \node [] at (2.5,-0.75) {$s$};
       \node [] at (2.5,0.75) {$t$};
\end{tikzpicture}
\end{equation}
the associated higher order Serre relations are:
\begin{equation}\label{eq:Lie-Serre-new3}
\begin{split}
  & \big[e_t,[e_s,e_i]\big]-\big[e_s,[e_t,e_i]\big]=0 \,, \\
  & \big[f_t,[f_s,f_i]\big]-\big[f_s,[f_t,f_i]\big]=0 \,,
\end{split}
\end{equation}
cf.~\cite[\S3.2.1(6)]{z}. In our setup, that occurs only if $N=2n$, $n+m\geq 3$, and the parity sequence
is $\Parity=(\ast , \ldots , \ast , \bar{1} , \bar{0})$ where each $\ast$ is either $\bar{0}$ or $\bar{1}$
(and $i=n+m-2,t=n+m-1,s=n+m$).

\medskip
\noindent
$\bullet$ For the sub-diagram
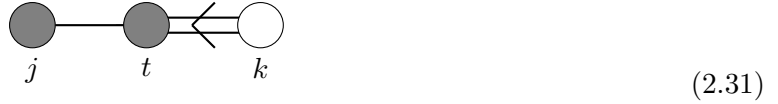
\begin{equation}\label{eq:pic-Serre-new6}
\begin{tikzpicture}
  \foreach \a in {1} {
    \begin{scope}[shift={(0.7*\a,0)}]
      \draw[fill=gray] (0.3*\a,0) circle (0.3cm);
      \draw[black,thick] (0.3*\a+0.3,0)--(0.3*\a+1.3,0);
    \end{scope}
  }
      \draw[black,dashed] (2.1,0)--(2.5,0);
      \draw[black,thick] (2.5,-0.1)--(4,-0.1);
      \draw[black,thick] (2.5,0.1)--(4,0.1);
      \draw[black,thick] (3.1,0)--(3.4,0.3);
      \draw[black,thick] (3.1,0)--(3.4,-0.3);
      \draw[fill=gray] (2.5,0) circle (0.3cm);
      \draw[fill=white] (4,0) circle (0.3cm);
       \node [below] at (1,-0.3) {$j$};
       \node [below] at (2.5,-0.3) {$t$};
       \node [below] at (4,-0.3) {$k$};
\end{tikzpicture}
\end{equation}
the associated higher order Serre relations are:
\begin{equation}\label{eq:Lie-Serre-new6}
\begin{split}
  & \Big[[e_j,e_t],\big[[e_j,e_t],[e_t,e_k]\big]\Big]=0 \,, \\
  & \Big[[f_j,f_t],\big[[f_j,f_t],[f_t,f_k]\big]\Big]=0 \,,
\end{split}
\end{equation}
cf.~\cite[\S3.2.1(4)]{z}. In our setup, that occurs only if $N=2n$, $n+m\geq 3$, and
the parity sequence is $\Parity=(\ast , \ldots , \ast , \bar{1} , \bar{0} , \bar{1})$
where each $\ast$ is either $\bar{0}$ or $\bar{1}$ (and $j=n+m-2,t=n+m-1,k=n+m$).

\medskip
\noindent
$\bullet$ For the sub-diagram
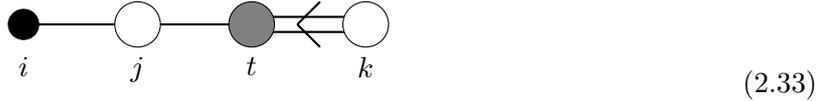
\begin{equation}\label{eq:pic-Serre-new7}
\begin{tikzpicture}
  \foreach \a in {1} {
    \begin{scope}[shift={(0.7*\a,0)}]
      \draw[fill=black] (0.3*\a,0) circle (0.2cm);
      \draw[black,thick] (0.3*\a+0.2,0)--(0.3*\a+1.2,0);
    \end{scope}
  }
      \draw[black,thick] (2.5,0)--(4,0);
      \draw[black,thick] (4,-0.1)--(5.5,-0.1);
      \draw[black,thick] (4,0.1)--(5.5,0.1);
      \draw[black,thick] (4.6,0)--(4.9,0.3);
      \draw[black,thick] (4.6,0)--(4.9,-0.3);
      \draw[fill=white] (2.5,0) circle (0.3cm);
      \draw[fill=gray] (4,0) circle (0.3cm);
      \draw[fill=white] (5.5,0) circle (0.3cm);
       \node [below] at (2.5,-0.3) {$j$};
       \node [below] at (1,-0.3) {$i$};
       \node [below] at (4,-0.3) {$t$};
       \node [below] at (5.5,-0.3) {$k$};
\end{tikzpicture}
\end{equation}
the associated higher order Serre relations are:
\begin{equation}\label{eq:Lie-Serre-new7}
\begin{split}
  & \Big[\big[e_i,[e_j,e_t]\big],\big[[e_j,e_t],[e_t,e_k]\big]\Big]=0 \,, \\
  & \Big[\big[f_i,[f_j,f_t]\big],\big[[f_j,f_t],[f_t,f_k]\big]\Big]=0 \,,
\end{split}
\end{equation}
cf.~\cite[\S3.2.1(5)]{z}. In our setup, that occurs only if $N=2n$, $n+m\geq 4$, and
$\Parity=(\ast , \ldots , \ast , \bar{0} , \bar{0} , \bar{1})$ where each $\ast$ is
either $\bar{0}$ or $\bar{1}$ (and $i=n+m-3,j=n+m-2,t=n+m-1,k=n+m$).

\begin{Rem}
(a) For odd $N$ or even $N$ but with the parity sequences $\Parity$ ending in $\bar{0}\bar{0}$
or $\bar{1}\bar{1}$, there may be only degree $4$ higher order Serre relations.

\medskip
\noindent
(b) For even $N\geq 6-2m$ and parity sequences $\Parity$ ending in $\bar{1}\bar{0}$,
we get new degree $3$ Serre relations.

\medskip
\noindent
(c) For even $N\geq 8-2m$ and parity sequences $\Parity$ ending in $\bar{0}\bar{1}$,
we get new degree $6$ or $7$ Serre relations.
\end{Rem}

\begin{Rem}\label{rem:Serre-parities}
(a) The degree $6$ Serre relations~\eqref{eq:Lie-Serre-new6} as well as the degree $7$ Serre
relations~\eqref{eq:Lie-Serre-new7} always hold in $\fosp(V)$  with $N=2n$ for any parity sequence $\Parity$.

\medskip
\noindent
(b) The degree $3$ Serre relations~\eqref{eq:Lie-Serre-new3} hold in $\fosp(V)$ with $N=2n$
iff $v_{n+m}$ is even (we note that for odd $v_{n+m}$ the corresponding Dynkin diagram does
not have the $s\leftrightarrow t$ $\BZ_2$-symmetry either).
\end{Rem}


\section{RTT orthosymplectic Yangians}
\label{sec:RTT orthosymplectic}

In this section, we recall the definition of the RTT (extended) Yangians of $\fosp(V)$ and their basic properties.
We establish the key rank-reduction result in Theorem~\ref{thm:embedding}, prove Lemma~\ref{lem:ef-com-tl},
and explain the relevance of the defining relations for super $A$-type Yangians to the present setup.


\subsection{RTT extended orthosymplectic super Yangian}
\label{ssec:RTT osp-Yangian}
\

Let $P\colon V\otimes V\to V\otimes V$ be the permutation operator defined by
\begin{equation}\label{eq:P}
  P\, = \sum_{i,j=1}^{N+2m} (-1)^{\ol{j}}\, e_{ij}\otimes e_{ji} \,,
\end{equation}
whose action is explicitly given by:
\begin{equation*}
  P(v_j\otimes v_i)=(-1)^{\ol{i}\cdot \ol{j}}\, v_i\otimes v_j \,.
\end{equation*}
Evoking the definition~\eqref{eq:theta}, we also consider the operator $Q\colon V\otimes V\to V\otimes V$
defined by
\begin{equation}\label{eq:Q}
  Q \, = \sum_{i,j=1}^{N+2m} (-1)^{\ol{i}\cdot \ol{j}}\theta_i\theta_j\, e_{ij}\otimes e_{i'j'} \,,
\end{equation}
whose action is explicitly given by:
\begin{equation*}
  Q(v_a\otimes v_b)=
  \begin{cases}
    0 & \mbox{if } b\ne a' \\
    \sum_{i=1}^{N+2m} \theta_i \, v_i\otimes v_{i'}
      & \mbox{if } b=a' \,,\, a>\lceil \sfrac{N}{2} \rceil+m \\
    (-1)^{\ol{a}}\sum_{i=1}^{N+2m} \theta_i\, v_i\otimes v_{i'}
      & \mbox{if } b=a' \,,\, a\leq \lceil \sfrac{N}{2} \rceil+m
  \end{cases} \,.
\end{equation*}
We also introduce a constant $\kappa$ via:
\begin{equation}\label{eq:kappa}
  \kappa=\sfrac{N}{2}-m-1 \,.
\end{equation}

Consider the \emph{rational} $R$-matrix (a super-version of the one considered in~\cite{zz}):
\begin{equation}\label{eq:osp-Rmatrix}
  R(u)=\ID-\frac{P}{u}+\frac{Q}{u-\kappa} \in \End\,V \otimes \End\,V \,.
\end{equation}
According to~\cite{aacfr}\footnote{While~\cite[Theorem 2.5]{aacfr} established~\eqref{eq:YBE} only for the
standard parity sequence~\eqref{eq:distinguished=standard}, the general case follows immediately by using
the $S(\lfloor \frac{N}{2}\rfloor +m)$-symmetry as in our proof of
Lemma~\ref{lem:isomorphism of RTT yangians}.}, it satisfies the famous Yang-Baxter equation with
a spectral parameter:
\begin{equation}\label{eq:YBE}
  R_{12}(u)R_{13}(u+v)R_{23}(v)=R_{23}(v)R_{13}(u+v)R_{12}(u) \,.
\end{equation}
Following~\cite[\S III]{aacfr}, we define the \emph{RTT extended Yangian of $\fosp(V)$},
denoted by $X^\rtt(\fosp(V))$, to be the associative $\BC$-superalgebra generated by
$\{t^{(r)}_{ij}\}_{1\leq i,j\leq N+2m}^{r\geq 1}$ with the $\BZ_2$-grading
$|t_{ij}^{(r)}|=\ol{i}+\ol{j}$ and subject to the following defining relation
(commonly called the \emph{RTT relation}, see~\eqref{eq:rtt intro}):
\begin{equation}\label{eq:RTT relation}
  R(u-v)T_1(u)T_2(v)=T_2(v)T_1(u)R(u-v) \,,
\end{equation}
viewed as an equality in $\End\,V \otimes \End\,V \otimes X^\rtt(\fosp(V))$. Here, $T(u)$ is the
series in $u^{-1}$ with coefficients in the algebra $\End\,V \otimes X^\rtt(\fosp(V))$, defined by:
\begin{equation}\label{eq:Tmatrix}
  T(u)=\sum_{i,j=1}^{N+2m} (-1)^{\ol{i}\cdot \ol{j} + \ol{j}}\, e_{ij}\otimes t_{ij}(u)
    \qquad \mathrm{with} \qquad
  t_{ij}(u):=\delta_{ij}+\sum_{r\geq 1} t^{(r)}_{ij}u^{-r} \,.
\end{equation}
Therefore,
  $T_1(u)=\sum_{i,j=1}^{N+2m} (-1)^{\ol{i}\cdot \ol{j} + \ol{j}}\, e_{ij} \otimes 1\otimes t_{ij}(u)$
and
  $T_2(v)=\sum_{i,j=1}^{N+2m} (-1)^{\ol{i}\cdot \ol{j} + \ol{j}}\, 1\otimes e_{ij} \otimes t_{ij}(v)$.

\begin{Rem}\label{sign explanation}
We identify the operator $\sum_{i,j=1}^{N+2m} (-1)^{\ol{i}\cdot \ol{j} + \ol{j}}\, e_{ij}\otimes t_{ij}(u)$
with the matrix $(t_{ij}(u))_{i,j=1}^{N+2m}$. Evoking the multiplication~\eqref{eq:graded tensor product}
for the graded tensor products, we see that the extra sign $(-1)^{\ol{i}\cdot \ol{j} + \ol{j}}$ ensures
that the product of matrices is calculated in the usual way.
\end{Rem}

Henceforth, for $A\in \End\,V \otimes \End\,V$ we shall often use the notation
$\langle v_i\otimes v_k | A | v_j\otimes v_\ell \rangle$ to denote the coefficient of $v_i\otimes v_k$ in
$A(v_j\otimes v_\ell)$. In particular,
comparing the matrix coefficients $\langle v_i\otimes v_k | \cdots | v_j\otimes v_\ell \rangle$ of both sides
of the defining relation~\eqref{eq:RTT relation}, it is straightforward to see that the latter is equivalent
to the following system of relations:
\begin{equation}\label{eq:RTT-termwise}
\begin{split}
  & [t_{ij}(u),t_{k\ell}(v)] =
    \frac{(-1)^{\ol{i}\cdot \ol{j}+\ol{i}\cdot \ol{k}+\ol{j}\cdot \ol{k}}}{u-v}
    \Big(t_{kj}(u)t_{i\ell}(v)-t_{kj}(v)t_{i\ell}(u)\Big) - \frac{1}{u-v-\kappa}\, \times \\
  & \left(\delta_{ki'}\sum_{p=1}^{N+2m} (-1)^{\ol{i}+\ol{i}\cdot \ol{j}+\ol{j}\cdot \ol{p}}
            \theta_i\theta_p\, t_{pj}(u)t_{p'\ell}(v) -
          \delta_{\ell j'}\sum_{p=1}^{N+2m} (-1)^{\ol{i}\cdot \ol{k}+\ol{j}\cdot \ol{k}+\ol{i}\cdot \ol{p}}
            \theta_{j'}\theta_{p'}\, t_{kp'}(v)t_{ip}(u) \right)
\end{split}
\end{equation}
for all $1\leq i,j,k,\ell\leq N+2m$. Here, we only use~\eqref{eq:parity-sym},~\eqref{eq:graded tensor product},
and the property~\eqref{eq:theta-symmetry}.

\begin{Rem}\label{rem:tau-antiautom}
As follows from the direct verification using~\eqref{eq:RTT-termwise}, the assignment (cf.~\cite[(2.9)]{m})
\begin{equation}\label{eq:tau-t}
  \tau\colon t_{ij}(u)\mapsto (-1)^{\ol{i}\cdot \ol{j} + \ol{j}}\, t_{ji}(u)
  \qquad \forall\ 1\leq i,j\leq N+2m
\end{equation}
gives rise to an anti-automorphism $\tau$ of the superalgebra $X^\rtt(\fosp(V))$, that is, we have:
\begin{equation*}
  \tau(xy)=(-1)^{|x|\cdot |y|} \tau(y)\tau(x)
  \quad \mathrm{for\ any\ homogeneous} \quad x,y\in X^\rtt(\fosp(V)) \,.
\end{equation*}
\end{Rem}

In the particular case of the standard parity sequence~\eqref{eq:distinguished=standard}, corresponding
to the case $v_1,\ldots,v_m\in V_{\bar{1}}$, we recover the RTT extended Yangian $X^\rtt(\fosp(N|2m))$.
The latter was introduced in~\cite{aacfr} and revised more recently in~\cite{m}; in particular, the
relation~\eqref{eq:RTT-termwise} recovers~\cite[(3.3)]{aacfr}, cf.~\cite[(2.8)]{m}. Meanwhile, for a
general parity sequence we actually get isomorphic superalgebras, due to the following simple result:

\begin{Lem}\label{lem:isomorphism of RTT yangians}
The superalgebra $X^\rtt(\fosp(V))$ depends only on $\dim(V_{\bar{0}}),\dim(V_{\bar{1}})$, up to an isomorphism.
Thus, $X^\rtt(\fosp(V))$ is isomorphic to the RTT extended Yangian $X^\rtt(\fosp(N|2m))$.
\end{Lem}

\begin{proof}
Let $U$ be another superspace with a $\BC$-basis $u_1,\ldots,u_{N+2m}$ such that each $u_i$ is \emph{even} or
\emph{odd}, $|u_i|=|u_{i'}|$, and $\dim(V_{\bar{0}})=\dim(U_{\bar{0}}), \dim(V_{\bar{1}})=\dim(U_{\bar{1}})$.
Pick a permutation $\sigma\in S(\lfloor \sfrac{N}{2} \rfloor + m)$ such that $v_i\in V$ and $u_{\sigma(i)}\in U$
have the same parity for all $1\leq i\leq \lfloor \sfrac{N}{2} \rfloor+m$\footnote{We abstain from using $\sfr$
instead of $\lfloor \sfrac{N}{2} \rfloor+m$ in this section, since we now have a similar looking index $r\geq 1$.}.
We then extend $\sigma$ to a permutation $\sigma\in S(N+2m)$ by
\begin{equation}\label{eq:sigma-extended}
  \sigma(i')=\sigma(i)' \quad \forall\ 1\leq i\leq \lfloor \sfrac{N}{2} \rfloor+m \,, \qquad
  \sigma(\sfrac{N+1}{2}+m)=\sfrac{N+1}{2}+m \qquad \mathrm{for\ odd}\ N \,.
\end{equation}
Then, the assignment
\begin{equation}\label{eq:parity-symmetry}
  t^{(r)}_{ij}\mapsto t^{(r)}_{\sigma(i),\sigma(j)} \qquad \forall\ i,j\in \mathbb{I} \,,\, r\geq 1
\end{equation}
is compatible with~\eqref{eq:RTT-termwise}, thus giving rise to an isomorphism
$X^\rtt(\fosp(V)) \,\iso\, X^\rtt(\fosp(U))$.
\end{proof}


\subsection{RTT orthosymplectic super Yangian}
\label{ssec:center}
\

Consider the matrix \emph{supertransposition} $t$ defined by
  $(A^t)_{ij}=(-1)^{\ol{i}\cdot \ol{j} + \ol{j}}\theta_i\theta_j\, (A)_{j'i'}$.
In particular:
\begin{equation}\label{eq:T-transposed}
  T^t(u)_{ij}=(-1)^{\ol{i}\cdot \ol{j} + \ol{j}}\theta_i\theta_j\, t_{j'i'}(u) \,.
\end{equation}
As shown in~\cite{aacfr}\footnote{While~\cite[Theorem 3.1]{aacfr} established this only for the
standard parity sequence~\eqref{eq:distinguished=standard}, the general case follows immediately
by utilizing the $S(\lfloor \frac{N}{2}\rfloor +m)$-symmetry as in our proof of
Lemma~\ref{lem:isomorphism of RTT yangians} above.}, the product $T(u-\kappa)T^t(u)$ is a scalar matrix:
\begin{equation}\label{eq:central-c}
  T(u-\kappa)T^t(u)=c_V(u)\cdot \mathrm{Id} \,,
\end{equation}
where $c_V(u)=1+\sum_{r\geq 1}c_r u^{-r}$ with all $c_r$ belonging to the center $ZX^\rtt(\fosp(V))$
of $X^\rtt(\fosp(V))$ and, in fact, freely generating $ZX^\rtt(\fosp(V))$, which can be shown as
in~\cite{amr} for non-super case.

For any formal power series $f(u)\in 1+u^{-1}\BC[[u^{-1}]]$, the assignment
\begin{equation}\label{eq:mu-automorphims}
  \mu_f\colon T(u)\mapsto f(u)T(u)
\end{equation}
gives rise to a superalgebra automorphism $\mu_f$ of $X^\rtt(\fosp(V))$. Following~\cite{aacfr,m} for
the standard parity and~\cite{jlm} for non-super case, we define the \emph{RTT Yangian of $\fosp(V)$},
denoted by $Y^{\rtt}(\fosp(V))$, as the following $\BC$-subalgebra of $X^\rtt(\fosp(V))$:
\begin{equation}\label{eq:RTT-non-extended}
  Y^\rtt(\fosp(V)):=\Big\{y\in X^\rtt(\fosp(V)) \,\Big|\, \mu_f(y)=y
  \quad  \forall\, f(u)\in 1+u^{-1}\BC[[u^{-1}]] \Big\} \,.
\end{equation}
Similarly to~\cite[(2.7)]{m} for the standard parity~\eqref{eq:distinguished=standard},
cf.~\cite{amr} for non-super case, we have the tensor product decomposition
\begin{equation}\label{eq:extended-vs-nonextended}
  X^\rtt(\fosp(V)) \simeq ZX^\rtt(\fosp(V)) \otimes Y^\rtt(\fosp(V)) \,.
\end{equation}
Thus, the RTT Yangian $Y^\rtt(\fosp(V))$ can be also realized as a quotient of $X^\rtt(\fosp(V))$:
\begin{equation}\label{eq:non-extended-quotient}
  Y^\rtt(\fosp(V)) \simeq X^\rtt(\fosp(V))/(c_V(u)-g(u))
  \qquad  \forall\, g(u)\in 1+u^{-1}\BC[[u^{-1}]] \,.
\end{equation}

\begin{Rem}\label{rem:tau-generators}
There is a unique series $\sz_V(u)=1+\sum_{r\geq 1} \sz_ru^{-r}$ with $\sz_r\in \BC[c_1,c_2,\ldots]$ satisfying
\begin{equation}\label{eq:c-squareroot}
  \sz_V(u-\kappa)\sz_V(u)=c_V(u) \,,
\end{equation}
cf.~\cite[Theorem 3.1]{amr}. According to~\eqref{eq:central-c}, the automorphisms $\mu_f$ of~\eqref{eq:mu-automorphims}
map $c_V(u)$ to $f(u)f(u-\kappa)c_V(u)$, hence, $\mu_f(\sz_V(u))=f(u)\sz_V(u)$.
Therefore, the series $\{\tau_{ij}(u)\}_{i,j\in \mathbb{I}}$ defined by
\begin{equation}\label{eq:tau-series}
  \delta_{ij}+\sum_{r\geq 1} \tau^{(r)}_{ij}u^{-r}=\tau_{ij}(u):=\sz_V(u)^{-1}t_{ij}(u)
\end{equation}
are $\mu_f$-invariant, and so their coefficients $\{\tau^{(r)}_{ij}\}_{i,j\in \mathbb{I}}^{r\geq 1}$
belong to $Y^\rtt(\fosp(V))$ of~\eqref{eq:RTT-non-extended}. The corresponding matrix
$\mathcal{T}(u)=(\tau_{ij}(u))_{i,j=1}^{N+2m}$ satisfies the RTT relation~\eqref{eq:RTT relation}
and $\mathcal{T}(u-\kappa)\mathcal{T}^t(u)=\mathrm{Id}$. This clarifies why~\eqref{eq:non-extended-quotient}
is usually stated for $g(u)=1$, cf.~\cite[Corollary 3.2]{amr} for non-super case.
\end{Rem}

Evoking Lemma~\ref{lem:isomorphism of RTT yangians}, we thus immediately obtain:

\begin{Cor}\label{cor:isom-RTT-yangians}
The superalgebra $Y^\rtt(\fosp(V))$ depends only on $\dim(V_{\bar{0}}),\dim(V_{\bar{1}})$, up to an isomorphism.
In particular, $Y^\rtt(\fosp(V))$ is isomorphic to $Y^\rtt(\fosp(N|2m))$ of~\cite{aacfr,m}.
\end{Cor}

\begin{Rem}\label{rem:super-to-nonsuper}
For $m=0$, the assignment $T(u)\mapsto T(u)$ gives rise to isomorphisms
  $$X^\rtt(\fosp(N|0)) \iso X^\rtt(\fso_N) \qquad \mathrm{and} \qquad
    Y^\rtt(\fosp(N|0)) \iso Y^\rtt(\fso_N) \,.$$
For $N=0$, the assignment $T(u)\mapsto T(-u)$ gives rise to isomorphisms
  $$X^\rtt(\fosp(0|2m)) \iso X^\rtt(\fsp_{2m}) \qquad \mathrm{and} \qquad
    Y^\rtt(\fosp(0|2m))\simeq Y^\rtt(\fsp_{2m}) \,.$$
Thus, the orthosymplectic setup generalizes classical $BCD$-types all at once.
\end{Rem}


\subsection{Relation to Lie superalgebras and PBW theorem}
\label{ssec:quasiclassical}
\

For $i,j\in \mathbb{I}$, define $\hat{t}^{(1)}_{ij}:=(-1)^{\ol{i}}\, t^{(1)}_{ij}$.
Their commutation relations
\begin{equation*}
  [\hat{t}^{(1)}_{ij},\hat{t}^{(1)}_{k\ell}] =
  \delta_{kj} \hat{t}^{(1)}_{i \ell} -
  \delta_{\ell i} (-1)^{(\ol{i}+\ol{j})(\ol{k}+\ol{\ell})}\, \hat{t}^{(1)}_{kj} -
  \delta_{k i'} (-1)^{\ol{i}\cdot \ol{j}+\ol{i}}\theta_i\theta_j\, \hat{t}^{(1)}_{j' \ell} +
  \delta_{\ell j'} (-1)^{\ol{i}\cdot \ol{k} + \ol{\ell}\cdot \ol{k}} \theta_{i'}\theta_{j'}\, \hat{t}^{(1)}_{k i'}
\end{equation*}
follow immediately by evaluating the $u^{-1}v^{-1}$-coefficients in the defining relation~\eqref{eq:RTT-termwise}.
On the other hand, comparing the $(i,j)$ matrix coefficients of both sides of~\eqref{eq:central-c}, we also obtain:
\begin{equation*}
  \hat{t}^{(1)}_{j'i'} =
  -(-1)^{\ol{i}\cdot \ol{j}+\ol{i}} \theta_i\theta_j\, \hat{t}^{(1)}_{ij} \quad \forall\, i\ne j \,, \qquad
  \hat{t}^{(1)}_{i'i'}+\hat{t}^{(1)}_{ii}=(-1)^{\ol{i}} c_1 \,,
\end{equation*}
where $c_1$ is the coefficient of $u^{-1}$ in $c_V(u)$ from~\eqref{eq:central-c}.
Thus, we get an algebra homomorphism
\begin{equation}\label{eq:iota}
  \iota\colon
  U(\fosp(V)\oplus \BC\cdot \sfc) \longrightarrow X^\rtt(\fosp(V)) \quad \mathrm{given\ by} \quad
  \sfc\mapsto c_1 \,,\, F_{ij}\mapsto \hat{t}^{(1)}_{ij} - \sfrac{(-1)^{\ol{i}}}{2}\delta_{ij}c_1 \,.
\end{equation}
In fact, the homomorphism $\iota$ of~\eqref{eq:iota} is a superalgebra embedding, due to the
Poincar\'{e}-Birkhoff-Witt (PBW) theorem for the RTT extended orthosymplectic Yangians that we recall next.

To this end, let us endow the RTT extended Yangian $X^\rtt(\fosp(V))$ with a filtration defined via
\begin{equation}\label{eq:gradings}
  \deg\, t^{(r)}_{ij}=r-1 \qquad \forall\ i,j\in \mathbb{I} \,,\, r\geq 1 \,.
\end{equation}
Let $\Gr\, X^\rtt(\fosp(V))$ denote the associated graded algebra with respect to this filtration.
For any element $x\in X^\rtt(\fosp(V))$, we use $\wt{x}$ to denote its image in $\Gr\, X^\rtt(\fosp(V))$.
In particular, $\wt{t}^{(r)}_{ij}$ and $\wt{c}_r$ will be the images of $t^{(r)}_{ij}$ and $c_r$
(the coefficient of $u^{-r}$ in $c_V(u)$ from~\eqref{eq:central-c}) in the $(r-1)$-th component of
$\Gr\, X^\rtt(\fosp(V))$. Due to~\eqref{eq:RTT-termwise}, we have a superalgebra homomorphism
\begin{equation}\label{eq:pi-map}
\begin{split}
  & \pi\colon \Gr\, X^\rtt(\fosp(V)) \longrightarrow U(\fosp(V)[\lt])\otimes \BC[\sfc_1,\sfc_2,\dots] \\
  & \quad \mathrm{given\ by} \quad
    \wt{t}^{(r)}_{ij} \mapsto (-1)^{\ol{i}}\, F_{ij} \lt^{r-1} + \sfrac{1}{2}\delta_{ij}\sfc_r
\end{split}
\end{equation}
with $\pi(\wt{c}_r)=\sfc_r$. The following result was stated first in~\cite{aacfr} and proved recently in~\cite{gk}:

\begin{Prop}\label{prop:assoc.graded}
(a) The homomorphism $\pi$ of~\eqref{eq:pi-map} is actually an isomorphism, that is
\begin{equation}\label{eq:quasiclassical-limit-1}
  \Gr\, X^\rtt(\fosp(V)) \simeq U(\fosp(V)[\lt])\otimes \BC[\sfc_1,\sfc_2,\dots] \,.
\end{equation}

\noindent
(b) Endowing the subalgebra $Y^\rtt(\fosp(V))$ of $X^\rtt(\fosp(V))$ with the induced filtration, we have
\begin{equation}\label{eq:quasiclassical-limit-nonextended-2}
  \Gr\, Y^\rtt(\fosp(V)) \,\iso\, U(\fosp(V)[\lt]) \quad \mathrm{via}\ \quad
  \wt{\tau}^{(r)}_{ij} \mapsto (-1)^{\ol{i}}\, F_{ij} \lt^{r-1} \,.
\end{equation}
\end{Prop}

\begin{Rem}
Considering $Y^\rtt(\fosp(V))$ rather as the filtered quotient $X^\rtt(\fosp(V))/(c_1,c_2,\dots)$,
see Remark~\ref{rem:tau-generators}, we can recast~\eqref{eq:quasiclassical-limit-nonextended-2}
in the following form (which does not involve $\tau$-generators):
\begin{equation}\label{eq:quasiclassical-limit-nonextended}
  \Gr\, Y^\rtt(\fosp(V)) \,\iso\, U(\fosp(V)[\lt]) \quad \mathrm{via}\ \quad
  \wt{t}^{(r)}_{ij} \mapsto (-1)^{\ol{i}}\, F_{ij} \lt^{r-1} \,.
\end{equation}
\end{Rem}

As a direct corollary, one obtains the PBW theorem for the RTT orthosymplectic Yangians:

\begin{Cor}\label{cor:pbw-thm}
The algebra $X^\rtt(\fosp(V))$ (respectively $Y^\rtt(\fosp(V))$) is generated by the elements $t^{(r)}_{ij}$
and $c_r$ (respectively elements $\tau^{(r)}_{ij}$) with the conditions $i+j\leq N+2m+|v_i|$, $r\geq 1$.
Moreover, given any total order on the set of these generators, the ordered monomials, with the powers of odd
generators not exceeding $1$, form a basis of the algebra $X^\rtt(\fosp(V))$ (respectively $Y^\rtt(\fosp(V))$).
\end{Cor}


\subsection{Gauss decomposition and rank reduction}
\label{ssec:key-embedding}
\

To derive the Drinfeld realization of $X^\rtt(\fosp(V))$ and subsequently of $Y^\rtt(\fosp(V))$,
we consider the Gauss decomposition of the generator matrix $T(u)$ from~\eqref{eq:Tmatrix}:
\begin{equation}\label{eq:gauss-osp}
  T(u)=F(u)\cdot H(u)\cdot E(u) \,.
\end{equation}
Here, $H(u),F(u),E(u)$ are the diagonal, lower-triangular, and upper-triangular matrices
\begin{equation}\label{eq:HFE-matrices}
\begin{split}
  & H(u)=\mathrm{diag}\Big(h_1(u),h_2(u),\ldots,h_{2'}(u),h_{1'}(u)\Big) \,, \\
  & F(u)=\left(\begin{array}{cccc}
        1 & 0 & \cdots & 0 \\
        f_{21}(u) & 1 & \cdots & 0 \\
        \vdots & \vdots & \ddots & \vdots \\
        f_{1'1}(u) & f_{1'2}(u) & \cdots  & 1
       \end{array}\right) \,, \qquad
    E(u)=\left(\begin{array}{cccc}
        1 & e_{12}(u) & \cdots & e_{11'}(u) \\
        0 & 1 & \cdots & e_{21'}(u) \\
        \vdots & \vdots & \ddots & \vdots \\
        0 & 0 & \cdots & 1
       \end{array}\right) \,,
\end{split}
\end{equation}
with $h_{\imath}(u),f_{ji}(u),e_{ij}(u)\in X^\rtt(\fosp(V))[[u^{-1}]]$ for $1 \leq \imath\leq N+2m$
and $1\leq i<j\leq N+2m$, cf.~Remark~\ref{sign explanation}. Define the elements
  $\{h^{(r)}_\imath, e^{(r)}_{ij}, f^{(r)}_{ji}\}_{1\leq \imath,i,j\leq N+2m, i<j}^{r\geq 1}$
of $X^\rtt(\fosp(V))$ via
\begin{equation*}
  e_{ij}(u)=\sum_{r\geq 1} e^{(r)}_{ij} u^{-r} \,,\quad
  f_{ji}(u)=\sum_{r\geq 1} f^{(r)}_{ji} u^{-r} \,,\quad
  h_{\imath}(u)=1+\sum_{r\geq 1} h^{(r)}_{\imath} u^{-r} \,.
\end{equation*}
In particular, we have $h_1(u)=t_{11}(u)$, $f_{i1}(u)=t_{i1}(u)t_{11}(u)^{-1}$,
$e_{1i}(u)=t_{11}(u)^{-1}t_{1i}(u)$ for $i>1$.

\begin{Rem}\label{rem:tau-efh}
Completely analogously to~\cite[Lemma 4.1]{m}, one proves by induction that
\begin{equation}\label{eq:tau-efh}
  \tau\colon
  e_{ij}(u) \mapsto (-1)^{\ol{i}\cdot \ol{j} + \ol{j}}\, f_{ji}(u) \,,\quad
  f_{ji}(u) \mapsto (-1)^{\ol{i}\cdot \ol{j} + \ol{i}}\, e_{ij}(u) \,,\quad
  h_{\imath}(u) \mapsto h_\imath(u)
\end{equation}
for $1\leq i<j\leq 1'$, $1\leq \imath\leq 1'$, where $\tau$ is the
anti-automorphism of $X^\rtt(\fosp(V))$ given by~\eqref{eq:tau-t}.
\end{Rem}

One of our \underline{main results} is the Drinfeld realization of $X^\rtt(\fosp(V))$, with the generators
\begin{equation}\label{eq:rtt-generators}
  \Big\{ h_\imath^{(r)}, e_i^{(r)}, f_i^{(r)} \,\Big|\,
  1\leq i\leq \lfloor \sfrac{N}{2} \rfloor+m,
  1\leq \imath \leq \lfloor \sfrac{N}{2} \rfloor+m+1, r\geq 1 \Big\}
\end{equation}
and an explicit collection of the defining relations, where:
\begin{equation*}
\begin{split}
  & e^{(r)}_i=e^{(r)}_{i,i+1} \,,\quad f^{(r)}_i=f^{(r)}_{i+1,i}
    \qquad \mathrm{for}\quad 1\leq i < \lfloor \sfrac{N}{2} \rfloor+m \, \\
  & \begin{cases}
      e^{(r)}_{n+m}=e^{(r)}_{n+m-1,n+m+1} \,,\, f^{(r)}_{n+m}=f^{(r)}_{n+m+1,n+m-1}
        & \mathrm{if}\ N=2n \,,\, \ol{n+m}=\bar{0} \\
      e^{(r)}_{n+m}=e^{(r)}_{n+m,n+m+1} \,,\, f^{(r)}_{n+m}=f^{(r)}_{n+m+1,n+m}
        & \mathrm{if}\ N=2n+1\ \mathrm{or}\ N=2n \,,\, \ol{n+m}=\bar{1}
    \end{cases} \,.
\end{split}
\end{equation*}
We shall use the corresponding generating series $e_i(u),f_i(u)$ defined via
\begin{equation}\label{eq:hfe-generating}
  e_i(u)=\sum_{r\geq 1} e_i^{(r)} u^{-r} \,, \qquad  f_i(u)=\sum_{r\geq 1} f_i^{(r)} u^{-r}
  \qquad \forall\ 1\leq i\leq \lfloor \sfrac{N}{2} \rfloor+m\,.
\end{equation}
The fact that the elements above generate $X^\rtt(\fosp(V))$ is straightforward, see explicit formulas in
Subsections~\ref{ssec:e-currents}--\ref{ssec:c-current}. The aforementioned relations will be read off from
the super $A$-type of~\cite{p,t} (recalled in Subsection~\ref{ssec:RTT gl-Yangian}) as well as rank $\leq 2$
cases, carried out case-by-case in Subsections~\ref{ssec:rank-1}--\ref{ssec:rank-2}. Finally, the proof that
these relations are indeed defining will proceed in the standard way by passing through the associated graded
algebras, see the proof of Theorem~\ref{thm:Main-Theorem-ext}.

\medskip
Let us now introduce the \underline{key ingredient} that will be used through the rest of this paper:
\begin{equation*}
  \mathrm{\textbf{rank\ reduction}\ embeddings} \quad
  \psi_{V,s}\colon X^\rtt(\fosp(V^{[s]})) \hookrightarrow X^\rtt(\fosp(V)) \,.
\end{equation*}
For $1\leq s\leq \lfloor \sfrac{N-1}{2} \rfloor+m$, let $V^{[s]}$ denote the following subspace
of the superspace $V$:
\begin{equation}\label{eq:V truncated}
  V^{[s]}=\mathrm{span}\, \big\{ v_i \,\big|\, s<i<s' \big\} \,.
\end{equation}
Let $X^\rtt(\fosp(V^{[s]}))$ denote the corresponding RTT extended orthosymplectic Yangian, defined via
the RTT relation using the corresponding $R$-matrix $R^{[s]}(u)$, cf.~\eqref{eq:osp-Rmatrix}. To define
the latter, we use the operators $P^{[s]},Q^{[s]}\in \End\,V^{[s]}\otimes \End\,V^{[s]}$ given by the
formulas alike~(\ref{eq:P},~\ref{eq:Q}) but with the indices $s<i,j<s'$ in the summations, while the
associated constant $\kappa^{[s]}$ is easily seen to be related to $\kappa$ of~\eqref{eq:kappa} via:
\begin{equation}\label{eq:kappa truncated}
  \kappa^{[s]}=\kappa-\sum_{i=1}^{s} (-1)^{\ol{i}} \,.
\end{equation}
We also consider the following $(N+2m-2s)\times (N+2m-2s)$ submatrices of~\eqref{eq:HFE-matrices}:
\begin{equation}\label{eq:diagonal truncated}
  H^{[s]}(u)=
  \left(\begin{array}{cccc}
        h_{s+1}(u) & 0 & \cdots & 0 \\
        0 & h_{s+2}(u) & \cdots & 0 \\
        \vdots & \vdots & \ddots & \vdots \\
        0 & 0 & \cdots & h_{(s+1)'}(u)
  \end{array}\right) \,,
\end{equation}
\begin{equation}\label{eq:lower truncated}
  F^{[s]}(u)=
  \left(\begin{array}{cccc}
        1 & 0 & \cdots & 0 \\
        f_{s+2,s+1}(u) & 1 & \cdots & 0 \\
        \vdots & \vdots & \ddots & \vdots \\
        f_{(s+1)',s+1}(u) & f_{(s+1)',s+2}(u) & \cdots & 1
  \end{array}\right) \,,
\end{equation}
\begin{equation}\label{eq:upper truncated}
  E^{[s]}(u)=
  \left(\begin{array}{cccc}
        1 & e_{s+1,s+2}(u) & \cdots & e_{s+1,(s+1)'}(u) \\
        0 & 1 & \cdots & e_{s+2,(s+1)'}(u) \\
        \vdots & \vdots & \ddots & \vdots \\
        0 & 0 & \cdots & 1
  \end{array}\right) \,,
\end{equation}
and define
\begin{equation}\label{eq:T truncated}
  T^{[s]}(u):=F^{[s]}(u)\cdot H^{[s]}(u)\cdot E^{[s]}(u) \,.
\end{equation}
Accordingly, the entries of the matrix $T^{[s]}(u)$ will be denoted by $t^{[s]}_{ij}(u)$ with $s<i,j<s'$.

Generalizing~\cite[Theorem 3.1, Proposition 4.1]{jlm} for non-super case (RTT extended
orthogonal/symplectic Yangians) and~\cite[Theorem~3.1, Proposition~4.2]{m} for $N\geq 3$ and the standard
parity sequence~\eqref{eq:distinguished=standard}, we have the following powerful \emph{rank reduction}:

\begin{Thm}\label{thm:embedding}
The assignment $T_{V^{[s]}}(u)\mapsto T^{[s]}_V(u)$ gives rise to a superalgebra embedding
\begin{equation}\label{eq:psi-embed}
  \psi_{V,s}\colon  X^\rtt(\fosp(V^{[s]})) \hookrightarrow X^\rtt(\fosp(V)) \,,
\end{equation}
where we use indices $V^{[s]}$ and $V$ solely to distinguish the corresponding generator $T$-matrices.
\end{Thm}

\begin{Rem}\label{rem:psi-comment}
(a) First, we note that all $\psi_{V,s}$ can be constructed as compositions of various $\psi_{V^{[?]},1}$.
This is based on the following natural compatibility between the maps~\eqref{eq:psi-embed}:
\begin{equation}\label{eq:psi-tower}
  \psi_{V,s} \circ \psi_{V^{[s]},t} = \psi_{V,s+t} \colon
  X^\rtt(\fosp(V^{[s+t]})) \longrightarrow X^\rtt(\fosp(V)) \,.
\end{equation}

\noindent
(b) The proof of~\cite[Theorem 3.1]{m} establishes Theorem~\ref{thm:embedding} for odd $v_1$
(we note that while the author considers the standard parity~\eqref{eq:distinguished=standard},
the proof of~\cite[Theorem 3.1]{m} only uses $|v_1|=\bar{1}$).

\medskip
\noindent
(c) As noted in~\cite{m}, the proof for the RTT extended orthogonal/symplectic Yangians from~\cite{jlm}
cannot be fully extended to the present setup since the value $R(1)$ is not always well-defined.
\end{Rem}

\begin{proof}[Proof of Theorem~\ref{thm:embedding}]
As follows from Remark~\ref{rem:psi-comment}(a), it suffices to show that $\psi_{V,1}$ is a
superalgebra embedding. The key is to show that it is a superalgebra homomorphism
(to verify its injectivity, it suffices to show that the associated graded
  $\Gr\, \psi_{V,1}\colon \Gr\, X^\rtt(\fosp(V^{[1]})) \to \Gr\, X^\rtt(\fosp(V))$
is injective, which follows from Proposition~\ref{prop:assoc.graded}(a) as in~\cite[Proof of Theorem 3.1]{jlm}).

To prove that $T_{V^{[1]}}(u)\mapsto T^{[1]}_V(u)$ gives rise to a superalgebra homomorphism we consider
two cases depending on the first element of the parity sequence $\Parity$. If $v_1$ is odd (i.e.\ $\Parity$
starts with~$\bar{1}$), then the proof is already contained in~\cite{m}, see Remark~\ref{rem:psi-comment}(b).
The case of even $v_1$ is treated completely similarly, so we shall only identify the key changes in
the respective formulas of~\cite{m}:
\begin{itemize}

\item[$\circ$]
The $\ol{R}(u)$ of~\cite{m} is now given by $\ol{R}(u)=1-\frac{\ol{P}}{u}+\frac{\ol{Q}}{u-\kappa+1}$,
where $\ol{P}=P^{[1]}$ and $\ol{Q}=Q^{[1]}$.

\item[$\circ$]
The operators $K^\pm, \check{K}^\pm \in \End\,V\otimes \End\,V$ of~\cite{m} are now defined as follows:
\begin{equation*}
  K^+=\sum_{i=2}^{2'} \theta_i\, e_{i1}\otimes e_{i'1'} \,,\,
  \check{K}^+=\sum_{i=2}^{2'} \theta_i\, e_{1i}\otimes e_{1'i'} \,,\,
  K^-=\sum_{i=2}^{2'} \theta_i\, e_{i1'}\otimes e_{i'1} \,,\,
  \check{K}^-=\sum_{i=2}^{2'} \theta_i\, e_{1'i}\otimes e_{1i'} \,.
\end{equation*}
Then, the operators $K=K^+ + K^-$ and $\check{K}=\check{K}^+ + \check{K}^-$
still satisfy~\cite[(3.7)--(3.8)]{m}.

\item[$\circ$]
The formula after (3.8) in~\cite{m} shall now read as
\begin{multline*}
  KT_1(u)T_2(v)=
  -\frac{1}{u-v-\kappa+1}\,\ol{Q}T_1(u)T_2(v)+\frac{(u-v+1)(u-v-\kappa)}{(u-v)(u-v-\kappa+1)}\,K^-T_1(u)T_2(v) \, + \\
  \frac{u-v-\kappa}{u-v-\kappa+1}\,K^+T_2(v)T_1(u)R(u-v) \,,
\end{multline*}
while its companion will be
\begin{multline*}
  T_2(v)T_1(u)\check{K}=
  -\frac{1}{u-v-\kappa+1}\,T_2(v)T_1(u)\ol{Q}+\frac{(u-v+1)(u-v-\kappa)}{(u-v)(u-v-\kappa+1)}\,T_2(v)T_1(u)\check{K}^- + \\
  \frac{u-v-\kappa}{u-v-\kappa+1}\,R(u-v)T_1(u)T_2(v)\check{K}^+ \,.
\end{multline*}
Plugging these formulas into~\cite[(3.8)]{m} and rearranging terms, we get the same formula
as in the middle of p.~9 in~\cite{m}, but with $u-v-\kappa+1$ used instead of $u-v-\kappa-1$.

\item[$\circ$]
Using the equalities
  $I_1I_2K^\pm=K^\pm$, $\check{K}^\pm I_1I_2=\check{K}^\pm$, $I_1I_2\wt{P}=\ol{P}=\wt{P}I_1I_2$,
as well as
  $$K^-T_1(u)T_2(v)J_1J_2\ol{T}_2(v)^{-1}\ol{T}_1(u)^{-1}I_1I_2=0 \,,\quad
    I_1I_2\ol{T}_1(u)^{-1}\ol{T}_2(v)^{-1}J_1J_2T_2(v)T_1(u)\check{K}^-=0 \,,$$
we see that the expression of~\cite[(3.9)]{m} still equals that of~\cite[(3.10)]{m},
but with $u-v-\kappa+1$ in place of $u-v-\kappa-1$.

\item[$\circ$]
The expression of~\cite[(3.10)]{m} can be written in the same way using~\cite[(3.12)]{m}
and its companion. Thus, the expression from~\cite[(3.9)]{m} equals
  $-\frac{G_1(u)G_2(v)WG_2(v)G_1(u)}{(u-v-\kappa)(u-v-\kappa+1)}$
with $W$ as in~\cite{m}, so that the only difference is in using $u-v-\kappa+1$ instead of $u-v-\kappa-1$.

\item[$\circ$]
Arguing as in~\cite{m}, we get:
\begin{equation*}
  W=K^+[t_{11}(u),h_{1'}(v)]\check{K}^+=K^+[h_1(u),h_{1'}(v)]\check{K}^+=0 \,.
\end{equation*}
Here, the last equality follows from the identity $h_{1'}(v)=c_V(v+\kappa)h_1(v+\kappa)^{-1}$ established
in~\eqref{eq:mat.coef.4} below and the commutativity $[h_1(u),h_1(v)]=0$ which is a direct consequence of
the formula~\eqref{eq:RTT-termwise} applied to $[t_{11}(u),t_{11}(v)]$ (alternatively, it can be derived
from the super $A$-type reduction of Subsection~\ref{ssec:RTT gl-Yangian}).
Therefore, the expression of~\cite[(3.9)]{m} vanishes.
\end{itemize}

This completes the proof for the case of even $v_1$.
\end{proof}

\begin{Rem}
We note that the main technical difference between the above formulas and those of~\cite[Proof of Theorem 3.1]{m}
is that $\frac{1}{u-v-\kappa-1}$ is replaced with $\frac{1}{u-v-\kappa+1}$ everywhere. One can unify these cases
by using $\frac{1}{u-v-\kappa^{[1]}}$.
\end{Rem}

We shall often use the following consequence of Theorem~\ref{thm:embedding}, verified as
its non-super counterpart of~\cite[Corollary 3.10]{jlm} (cf.~\cite[Corollary~3.3]{m} for
the standard parity sequence~\eqref{eq:distinguished=standard}):

\begin{Cor}\label{cor:commutativity}
For any $1\leq a,b\leq \ell$ and $\ell<i,j<\ell'$, we have the following commutativity:
\begin{equation}\label{eq:important commutativity}
  [t_{ab}(u),t^{[\ell]}_{ij}(v)]=0 \,.
\end{equation}
In particular, $\big\{h_a(u),e_{ab}(u),f_{ba}(u)\big|1\leq a,b\leq \ell\big\}$ commute with
$\big\{h_\imath(v),e_{\imath\jmath}(v),f_{\jmath\imath}(v)\big|\ell<\imath,\jmath<\ell'\big\}$.
\end{Cor}

As the embeddings $\psi_{V,s}$ of~\eqref{eq:psi-embed} commute with the automorphisms $\mu_f$
of~\eqref{eq:mu-automorphims}, we obtain:

\begin{Cor}
The restriction of $\psi_{V,s}$ to the subalgebra $Y^\rtt(\fosp(V^{[s]}))$ of $X^\rtt(\fosp(V^{[s]}))$
defines a superalgebra embedding $\psi_{V,s}\colon  Y^\rtt(\fosp(V^{[s]})) \hookrightarrow Y^\rtt(\fosp(V))$.
\end{Cor}


\subsection{Useful lemma}
\label{ssec:useful-lemma}
\

The following result generalizes~\cite[Lemma 4.3]{jlm} for non-super case (cf.~\cite[Lemma 4.3]{m}, where
a different proof is provided for $N\geq 3$ and the standard parity sequence~\eqref{eq:distinguished=standard}):

\begin{Lem}\label{lem:ef-com-tl}
For $\ell < i,j,k < \ell'$ with $k\ne j'$, the following relations hold in $X^\rtt(\fosp(V))$:
\begin{equation}\label{eq:useful-commutator-1}
  [e_{\ell k}(u),t^{[\ell]}_{ij}(v)]=\frac{(-1)^{\ol{\ell}\cdot \ol{i}+\ol{\ell}\cdot \ol{k}+\ol{i}\cdot \ol{k}}}{u-v}
  t^{[\ell]}_{ik}(v) \Big(e_{\ell j}(v) - e_{\ell j}(u)\Big) \,,
\end{equation}
\begin{equation}\label{eq:useful-commutator-2}
  [f_{k\ell}(u),t^{[\ell]}_{ji}(v)]=\frac{(-1)^{\ol{\ell}\cdot \ol{j}+\ol{\ell}\cdot \ol{k}+\ol{j}\cdot \ol{k}}}{u-v}
  \Big(f_{j\ell}(u) - f_{j\ell}(v)\Big) t^{[\ell]}_{ki}(v) \,.
\end{equation}
\end{Lem}


\noindent
To prevent the confusion with the generator $\bar{1}\in \BZ_2$, we use $|v_1|$ instead of $\ol{1}$ in the proof below.

\begin{proof}
It suffices to verify both relations for $\ell=1$, as the general case then follows immediately
from Theorem~\ref{thm:embedding}. Let us verify~\eqref{eq:useful-commutator-1} for $\ell=1$
(the relation~\eqref{eq:useful-commutator-2} follows by applying the anti-automorphism $\tau$ of $X^\rtt(\fosp(V))$
given by~\eqref{eq:tau-t} to~\eqref{eq:useful-commutator-1} and using the formulas~\eqref{eq:tau-efh}).

First, we note that
\begin{equation}\label{eq:eft-1}
  t^{[1]}_{ij}(v)=t_{ij}(v)-f_{i1}(v)h_1(v)e_{1j}(v)=t_{ij}(v)-t_{i1}(v)t_{11}(v)^{-1}t_{1j}(v) \,.
\end{equation}
Thus, the defining relation
  $[t_{1k}(u),t_{ij}(v)]=
  \frac{(-1)^{|v_1|(\ol{i}+\ol{k})+\ol{i}\cdot \ol{k}}}{u-v}
  \big(t_{ik}(u)t_{1j}(v)-t_{ik}(v)t_{1j}(u)\big)$
of~\eqref{eq:RTT-termwise}, which uses $i\ne 1'$ and $k\ne j'$, can be written in the following form:
\begin{multline}\label{eq:eft-2}
  [t_{1k}(u),t^{[1]}_{ij}(v)] + [t_{1k}(u),f_{i1}(v)h_1(v)e_{1j}(v)] =
  \frac{(-1)^{|v_1|(\ol{i}+\ol{k}) + \ol{i}\cdot \ol{k}}}{u-v}
    \Big(t^{[1]}_{ik}(u)t_{1j}(v)-t^{[1]}_{ik}(v)t_{1j}(u)\Big) \, + \\
  \frac{(-1)^{|v_1|(\ol{i}+\ol{k}) + \ol{i}\cdot \ol{k}}}{u-v}
    \Big(f_{i1}(u)h_1(u)e_{1k}(u)t_{1j}(v)-f_{i1}(v)h_1(v)e_{1k}(v)t_{1j}(u)\Big) \,.
\end{multline}
Let us evaluate the second summand in the left-hand side of~\eqref{eq:eft-2}:
\begin{multline}\label{eq:eft-3}
  [t_{1k}(u),f_{i1}(v)h_1(v)e_{1j}(v)] =
  [t_{1k}(u),t_{i1}(v)]e_{1j}(v)+(-1)^{(|v_1|+\ol{i})(|v_1|+\ol{k})} t_{i1}(v) [t_{1k}(u),t_{11}(v)^{-1}t_{1j}(v)] = \\
  [t_{1k}(u),t_{i1}(v)]e_{1j}(v)-(-1)^{(|v_1|+\ol{i})(|v_1|+\ol{k})}
  \Big(f_{i1}(v) [t_{1k}(u),t_{11}(v)]e_{1j}(v) - f_{i1}(v) [t_{1k}(u),t_{1j}(v)]\Big) = \\
  \frac{(-1)^{|v_1|(\ol{i}+\ol{k})+\ol{i}\cdot \ol{k}}}{u-v}
  \Big(t_{ik}(u)t_{11}(v)e_{1j}(v)-t_{ik}(v)t_{11}(u)e_{1j}(v)\Big) \, - \\
  \frac{(-1)^{|v_1|(\ol{i}+\ol{k})+\ol{i}\cdot \ol{k}}}{u-v}
  \Big(f_{i1}(v)t_{1k}(u)t_{11}(v)e_{1j}(v)-f_{i1}(v)t_{1k}(v)t_{11}(u)e_{1j}(v)\Big) \, + \\
  \frac{(-1)^{|v_1|(\ol{i}+\ol{k})+\ol{i}\cdot \ol{k}}}{u-v}
  \Big(f_{i1}(v)t_{1k}(u)t_{1j}(v)-f_{i1}(v)t_{1k}(v)t_{1j}(u)\Big) =
  \frac{(-1)^{|v_1|(\ol{i}+\ol{k})+\ol{i}\cdot \ol{k}}}{u-v} \, \times \\
  \Big(t_{ik}(u)h_1(v)e_{1j}(v)-t_{ik}(v)h_1(u)e_{1j}(v)+f_{i1}(v)h_1(v)e_{1k}(v)h_1(u)\big(e_{1j}(v)-e_{1j}(u)\big)\Big) \,,
\end{multline}
where we used $[t_{1k}(u),t_{11}(v)^{-1}]=-t_{11}(v)^{-1}[t_{1k}(u),t_{11}(v)]t_{11}(v)^{-1}$
in the second equality and applied~\eqref{eq:RTT-termwise} three times in the third equality.
Combining~\eqref{eq:eft-2} and~\eqref{eq:eft-3}, we thus obtain:
\begin{multline}\label{eq:eft-4}
  [t_{1k}(u),t^{[1]}_{ij}(v)] = \frac{(-1)^{|v_1|(\ol{i}+\ol{k})+\ol{i}\cdot \ol{k}}}{u-v}
  \Big(t_{ik}(v)h_1(u)e_{1j}(v)-f_{i1}(v)h_1(v)e_{1k}(v)h_1(u)e_{1j}(v)-t^{[1]}_{ik}(v)t_{1j}(u)\Big) \\
  = \frac{(-1)^{|v_1|(\ol{i}+\ol{k})+\ol{i}\cdot \ol{k}}}{u-v}
    t^{[1]}_{ik}(v)h_1(u) \Big(e_{1j}(v)-e_{1j}(u)\Big) \,.
\end{multline}
As $t_{1k}(u)=h_1(u)e_{1k}(u)$ and $h_1(u)$ commutes with both $t^{[1]}_{ij}(v),t^{[1]}_{ik}(v)$
by Corollary~\ref{cor:commutativity}, we get:
\begin{equation*}
  [e_{1k}(u),t^{[1]}_{ij}(v)] =
  \frac{(-1)^{|v_1|(\ol{i}+\ol{k})+\ol{i}\cdot \ol{k}}}{u-v}
  t^{[1]}_{ik}(v) \Big(e_{1j}(v)-e_{1j}(u)\Big)
\end{equation*}
which is precisely~\eqref{eq:useful-commutator-1} for $\ell=1$.
\end{proof}


\subsection{RTT Yangian in super A-type: revision}\label{ssec:RTT gl-Yangian}
\

Fix $n,m\geq 0$ and consider a superspace $\VV=\VV_{\bar{0}}\oplus \VV_{\bar{1}}$ with a $\BC$-basis
$\sfv_1,\ldots,\sfv_{n+m}$ such that each $\sfv_i$ is either \emph{even} or \emph{odd} and
$\dim(\VV_{\bar{0}})=n, \dim(\VV_{\bar{1}})=m$. We define the corresponding parity sequence
$\parity:=(|\sfv_1|,\ldots,|\sfv_{n+m}|)\in \{\bar{0},\bar{1}\}^{n+m}$. Let
$\Pop\colon \VV\otimes \VV\to \VV\otimes \VV$ be the permutation operator defined via
$\Pop = \sum_{i,j=1}^{n+m} (-1)^{\ol{j}}\, e_{ij}\otimes e_{ji}$, cf.~\eqref{eq:P}.
Consider the \emph{rational} $R$-matrix:
\begin{equation}\label{eq:gl-Rmatrix}
  \sfR(u)=\ID-\frac{\Pop}{u} \in \End\,\VV \otimes \End\,\VV \,,
\end{equation}
which satisfies the Yang-Baxter equation with a spectral parameter, cf.~\eqref{eq:YBE}:
\begin{equation}\label{eq:YBE-gl}
  \sfR_{12}(u)\sfR_{13}(u+v)\sfR_{23}(v) = \sfR_{23}(v)\sfR_{13}(u+v)\sfR_{12}(u) \,.
\end{equation}
The \emph{RTT Yangian of $\gl(\VV)$}, denoted by $Y^\rtt(\gl(\VV))$, is defined as the associative
$\BC$-superalgebra generated by $\{\sft^{(r)}_{ij}\}_{1\leq i,j\leq n+m}^{r\geq 1}$ with the
$\BZ_2$-grading $|\sft_{ij}^{(r)}|=\ol{i}+\ol{j}$ and subject to the following defining RTT relation,
cf.~(\ref{eq:rtt intro},~\ref{eq:RTT relation}):
\begin{equation}\label{eq:RTT relation-gl}
  \sfR(u-v)\sfT_1(u)\sfT_2(v)=\sfT_2(v)\sfT_1(u)\sfR(u-v) \,,
\end{equation}
viewed as an equality in $\End\,\VV \otimes \End\,\VV \otimes Y^\rtt(\gl(\VV))$. Here, $\sfT(u)$ is
the series in $u^{-1}$ with coefficients in the algebra $\End\, \VV \otimes Y^\rtt(\gl(V))$, defined by:
\begin{equation}\label{eq:Tmatrix-gl}
  \sfT(u)=\sum_{i,j=1}^{n+m} (-1)^{\ol{i}\cdot \ol{j} + \ol{j}}\, e_{ij}\otimes \sft_{ij}(u)
    \qquad \mathrm{with} \qquad
  \sft_{ij}(u):=\delta_{ij}+\sum_{r\geq 1} \sft^{(r)}_{ij}u^{-r} \,.
\end{equation}
The relation~\eqref{eq:RTT relation-gl} is equivalent to the following system of relations:
\begin{equation}\label{eq:RTT-termwise-gl}
  [\sft_{ij}(u),\sft_{k\ell}(v)]=
  \frac{(-1)^{\ol{i}\cdot \ol{j}+\ol{i}\cdot \ol{k}+\ol{j}\cdot \ol{k}}}{u-v}
  \Big(\sft_{kj}(u)\sft_{i\ell}(v)-\sft_{kj}(v)\sft_{i\ell}(u)\Big)
\end{equation}
for all $1\leq i,j,k,\ell\leq n+m$, cf.~\eqref{eq:RTT-termwise}.

For any formal power series $f(u)\in 1+u^{-1}\BC[[u^{-1}]]$, the assignment
\begin{equation}\label{eq:mu-automorphims-gl}
  \mu_f\colon \sfT(u)\mapsto f(u)\sfT(u)
\end{equation}
gives rise to a superalgebra automorphism $\mu_f$ of $Y^\rtt(\gl(\VV))$, cf.~\eqref{eq:mu-automorphims}.
The \emph{RTT Yangian of $\ssl(\VV)$}, denoted by $Y^{\rtt}(\ssl(\VV))$, is defined as the following
subalgebra of $Y^\rtt(\gl(\VV))$:
\begin{equation}\label{eq:RTT-non-extended-gl}
  Y^\rtt(\ssl(\VV)):=\Big\{y\in Y^\rtt(\gl(\VV)) \,\Big|\, \mu_f(y)=y
          \quad \forall\, f(u)\in 1+u^{-1}\BC[[u^{-1}]] \Big\} \,.
\end{equation}

\begin{Rem}
In contrast to~\eqref{eq:extended-vs-nonextended}, we note that we have the tensor product decomposition
$Y^\rtt(\gl(\VV)) \simeq ZY^\rtt(\gl(\VV)) \otimes Y^\rtt(\ssl(\VV))$ only for $n\ne m$, while for $n=m$
the center $ZY^\rtt(\gl(\VV))$ of $Y^\rtt(\gl(\VV))$ actually belongs to $Y^\rtt(\ssl(\VV))$,
see~\cite[Theorem 2.48]{t} (generalizing~\cite{g}).
\end{Rem}

For the parity sequence $\parity=(\bar{0},\ldots,\bar{0},\bar{1},\ldots,\bar{1})$,
reverse to~\eqref{eq:distinguished=standard}, that is:
\begin{equation*}
  |\sfv_1|=\ldots=|\sfv_n|=\bar{0} \qquad \mathrm{and} \qquad |\sfv_{n+1}|=\ldots=|\sfv_{n+m}|=\bar{1} \,,
\end{equation*}
we recover the RTT Yangians $Y^\rtt(\gl(n|m))$, $Y^\rtt(\ssl(n|m))$. By~\cite[Lemmas 2.24, Corollary 2.38]{t},
we have $Y^\rtt(\gl(\VV))\simeq Y^\rtt(\gl(n|m))$ and $Y^\rtt(\ssl(\VV))\simeq Y^\rtt(\ssl(n|m))$,
cf.~Lemma~\ref{lem:isomorphism of RTT yangians}, Corollary~\ref{cor:isom-RTT-yangians}.

In what follows, we shall use the Drinfeld realization of $Y^\rtt(\gl(\VV))$ established in~\cite{t}
(cf.~\cite{p}), generalizing~\cite{g}. To this end, we consider the Gauss decomposition of the matrix
$\sfT(u)$ from~\eqref{eq:Tmatrix-gl}:
\begin{equation*}
  \sfT(u)=\sfF(u)\cdot \sfH(u)\cdot \sfE(u) \,,
\end{equation*}
where $\sfH(u),\sfF(u),\sfE(u)$ are the diagonal, lower-triangular, and upper-triangular matrices
with matrix coefficients $\sfh_{\imath}(u), \sff_{ji}(u), \sfe_{ij}(u)$, as in~\eqref{eq:HFE-matrices}.
The coefficients of the series $\sfe_i(u)=\sfe_{i,i+1}(u)$, $\sff_i(u)=\sff_{i+1,i}(u)$, $\sfh_\imath(u)$
with $1\leq i<n+m, 1\leq \imath\leq n+m$ generate $Y^\rtt(\gl(\VV))$. Furthermore, one can specify
all the defining relations (thus recovering the Drinfeld realization of $Y^\rtt(\gl(\VV))$):

\begin{Thm}\cite[Theorem 2.32]{t}\label{thm:Drinfeld-A}
The algebra $Y^\rtt(\gl(\VV))$ is isomorphic to the $\BC$-superalgebra $Y(\gl(\VV))$ generated by
$\{\sfe_i^{(r)}, \sff_i^{(r)},\sfh_\imath^{(r)} \,|\, 1\leq i<n+m, 1\leq \imath\leq n+m, r\geq 1\}$
with the $\BZ_2$-grading $|\sfe_i^{(r)}|=|\sff_i^{(r)}|=\ol{i}+\ol{i+1}$, $|\sfh_\imath^{(r)}|=\bar{0}$,
and subject to the following defining relations:
\begin{equation}\label{eq:Atype-hh}
  [\sfh_\imath(u),\sfh_\jmath(v)]=0 \,,
\end{equation}
\begin{equation}\label{eq:Atype-eh}
  [\sfh_\imath(u),\sfe_j(v)]=(-1)^{\ol{\imath}} (\delta_{\imath,j+1}-\delta_{\imath j})\,
  \frac{\sfh_\imath(u)\big(\sfe_j(u)-\sfe_j(v)\big)}{u-v} \,,
\end{equation}
\begin{equation}\label{eq:Atype-fh}
  [\sfh_\imath(u),\sff_j(v)]=(-1)^{\ol{\imath}}(\delta_{\imath j}-\delta_{\imath,j+1})\,
  \frac{\big(\sff_j(u)-\sff_j(v)\big)\sfh_\imath(u)}{u-v} \,,
\end{equation}
\begin{equation}\label{eq:Atype-ef}
  [\sfe_i(u),\sff_j(v)]=(-1)^{\ol{i+1}}\delta_{ij}\,
  \frac{\sfh_i(u)^{-1}\sfh_{i+1}(u)-\sfh_i(v)^{-1}\sfh_{i+1}(v)}{u-v} \,,
\end{equation}
\begin{equation}\label{eq:Atype-ee-1}
  \begin{cases}
    [\sfe_i(u),\sfe_i(v)]=0 & \mathrm{if}\  \ol{i}\ne \ol{i+1} \\
    [\sfe_i(u),\sfe_i(v)]=(-1)^{\ol{i}}\, \frac{(\sfe_i(u)-\sfe_i(v))^2}{u-v} & \mathrm{if}\ \ol{i}=\ol{i+1}
  \end{cases} \,,
\end{equation}
\begin{equation}\label{eq:Atype-ff-1}
  \begin{cases}
    [\sff_i(u),\sff_i(v)]=0 & \mathrm{if}\ \ol{i}\ne \ol{i+1} \\
    [\sff_i(u),\sff_i(v)]=-(-1)^{\ol{i}}\, \frac{(\sff_i(u)-\sff_i(v))^2}{u-v} & \mathrm{if}\ \ol{i}=\ol{i+1}
  \end{cases} \,,
\end{equation}
\begin{equation}\label{eq:Atype-ee-2}
  u[\sfe^\circ_i(u),\sfe_j(v)]-v [\sfe_i(u),\sfe^\circ_j(v)]=(-1)^{\ol{j}}\delta_{j,i+1}\sfe_i(u)\sfe_j(v)
  \quad \mathrm{for} \ i<j \,,
\end{equation}
\begin{equation}\label{eq:Atype-ff-2}
  u[\sff^\circ_i(u),\sff_j(v)]-v [\sff_i(u),\sff^\circ_j(v)]=-(-1)^{\ol{j}}\delta_{j,i+1}\sff_j(v)\sff_i(u)
  \quad \mathrm{for} \ i<j \,,
\end{equation}
degree $2$ Serre relations
\begin{equation}\label{eq:deg2-serre}
  \big[\sfe_i(u),\sfe_{j}(v)]=0 \,,\quad \big[\sff_i(u),\sff_{j}(v)]=0
  \qquad \mathrm{if} \quad i\ne j,j\pm 1
\end{equation}
as well as degree $3$ Serre relations
\begin{equation}\label{eq:Atype-serre}
  \begin{cases}
    \big[\sfe_i(u_1),[\sfe_i(u_2),\sfe_{i\pm 1}(v)]\big] +
    \big[\sfe_i(u_2),[\sfe_i(u_1),\sfe_{i\pm 1}(v)]\big] = 0  \\
    \big[\sff_i(u_1),[\sff_i(u_2),\sff_{i\pm 1}(v)]\big] +
    \big[\sff_i(u_2),[\sff_i(u_1),\sff_{i\pm 1}(v)]\big] = 0
  \end{cases}
  \qquad \mathrm{if}\quad \ol{i}=\ol{i+1}
\end{equation}
and degree $4$ Serre relations
\begin{multline}\label{eq:Atype-superserre}
  \begin{cases}
    \big[[\sfe_{i-1}(u),\sfe_i(v_1)],[\sfe_i(v_2),\sfe_{i+1}(w)]\big] +
    \big[[\sfe_{i-1}(u),\sfe_i(v_2)],[\sfe_i(v_1),\sfe_{i+1}(w)]\big] = 0  \\
    \big[[\sff_{i-1}(u),\sff_i(v_1)],[\sff_i(v_2),\sff_{i+1}(w)]\big] +
    \big[[\sff_{i-1}(u),\sff_i(v_2)],[\sff_i(v_1),\sff_{i+1}(w)]\big] = 0
  \end{cases} \ \ \mathrm{if} \ \ \ol{i}\ne \ol{i+1}
\end{multline}
where
\begin{equation*}
\begin{split}
  & \sfe_i(u)=\sum_{r\geq 1} \sfe_i^{(r)} u^{-r} \,,\quad
    \sff_i(u)=\sum_{r\geq 1} \sff_i^{(r)} u^{-r} \,,\quad
    \sfh_\imath(u)=1+\sum_{r\geq 1} \sfh_\imath^{(r)} u^{-r} \,, \\
  & \sfe^\circ_i(u)=\sum_{r\geq 2} \sfe_i^{(r)} u^{-r} \,,\quad
    \sff^\circ_i(u)=\sum_{r\geq 2} \sff_i^{(r)} u^{-r} \,.
\end{split}
\end{equation*}
\end{Thm}

Let us record an important consequence of the relations~(\ref{eq:Atype-eh},~\ref{eq:Atype-fh})
that we shall often use:

\begin{Cor}\label{cor:A-resonance}
The following relations hold in $Y^\rtt(\gl(\VV))$:
\begin{align}
  \sfh_i(u)\sfe_i(u) & = \sfe_i\big(u-(-1)^{\ol{i}}\big)\sfh_i(u) \,,
    \label{eq:h-to-e-1}\\
  \sfh_{i+1}(u)\sfe_i(u) & = \sfe_i\big(u+(-1)^{\ol{i+1}}\big)\sfh_{i+1}(u) \,,
    \label{eq:h-to-e-2} \\
  \sff_i(u)\sfh_i(u) & = \sfh_i(u)\sff_i\big(u-(-1)^{\ol{i}}\big) \,,
    \label{eq:h-to-f-1} \\
  \sff_i(u)\sfh_{i+1}(u) & = \sfh_{i+1}(u)\sff_i\big(u+(-1)^{\ol{i+1}}\big)
    \label{eq:h-to-f-2}
\end{align}
for any $1\leq i\leq n+m-1$.
\end{Cor}

\begin{proof}
Let us rewrite $\imath=j=i$ case of~\eqref{eq:Atype-eh} in the following form:
\begin{equation}\label{eq:mat.coef.7}
  \left(u-v-(-1)^{\ol{i}}\right)\sfh_i(u)\sfe_i(v)+(-1)^{\ol{i}}\, \sfh_i(u)\sfe_i(u) =
  (u-v)\sfe_i(v)\sfh_i(u) \,.
\end{equation}
Plugging $v=u-(-1)^{\ol{i}}$ above, we obtain~\eqref{eq:h-to-e-1}.
The other three relations are proved similarly.
\end{proof}

Let us finally explain the relevance of the above super $A$-type to the present orthosymplectic setup.
To this end, we fix $V$ with $N=2n$ or $N=2n+1$ and set $\VV=\mathrm{span}\, \{v_i\}_{i=1}^{n+m}$.
In particular, $V$ and $\VV$ have the same parity sequences: $\Parity=\parity$. Then, the defining
relations~\eqref{eq:RTT-termwise} for $1\leq i,j,k,\ell\leq n+m$ coincide with~\eqref{eq:RTT-termwise-gl}.
Therefore, we have a superalgebra homomorphism
\begin{equation}\label{eq:gl-to-osp}
  Y^\rtt(\gl(\VV))\longrightarrow X^\rtt(\fosp(V)) \quad \mathrm{given\ by} \quad
  \sft_{ij}(u)\mapsto t_{ij}(u) \quad \forall\ 1\leq i,j\leq n+m \,,
\end{equation}
which is injective due to the PBW theorems for $Y^\rtt(\gl(\VV))$ and $X^\rtt(\fosp(V))$,
see Corollary~\ref{cor:pbw-thm}. Combining this with Theorem~\ref{thm:Drinfeld-A}, we obtain:

\begin{Cor}\label{cor:A-type relations}
For $N=2n$ or $N=2n+1$, the currents $\{e_i(u),f_i(u),h_\imath(u)\}_{i<n+m}^{\imath\leq n+m}$
of~(\ref{eq:HFE-matrices},~\ref{eq:hfe-generating}) satisfy the relations from Theorem~\ref{thm:Drinfeld-A}.
\end{Cor}

Likewise, the submatrix $T'(u)=(t_{ij}(u))_{i,j\in \mathbb{I}'}$ of $T(u)$ with
$\mathbb{I}'=\{1,2\ldots,n+m-1,n+m+1\}$ also defines an embedding
$Y^\rtt(\gl(\VV))\hookrightarrow X^\rtt(\fosp(V))$ via $\sfT(u)\mapsto T'(u)$.
Moreover, if $N=2n$ and $|v_{n+m}|=\bar{0}$, then we have the following important equalities
(which follow from~\eqref{eq:ef-vanishing-rk1}):
\begin{equation}\label{eq:ef-vanishing}
  e_{n+m,n+m+1}(u)=0=f_{n+m+1,n+m}(u) \,.
\end{equation}
Thus, in this case the Gauss decomposition of the submatrix $T'(u)$ is formed by the corresponding submatrices
of $F(u),H(u),E(u)$ from~\eqref{eq:gauss-osp}. Combining this with Theorem~\ref{thm:Drinfeld-A}, we obtain:

\begin{Cor}\label{cor:other-A-type relations}
The currents
  $\{e_{i+\delta_{i,n+m-1}}(u), f_{i+\delta_{i,n+m-1}}(u),
     h_{\imath+\delta_{\imath,n+m}}(u)\}_{i<n+m}^{\imath\leq n+m}$
satisfy the relations from Theorem~\ref{thm:Drinfeld-A}, if $N=2n$ and $|v_{n+m}|=\bar{0}$.
\end{Cor}

Due to the two corollaries above, it thus remains to determine the quadratic relations between the currents
$\{e_i(u),f_i(u),h_\imath(u)\}$ where at least one of the indices is $i=n+m$ or $\imath=n+m+1$, as well as
Serre relations. The latter is partially accomplished in Subsection~\ref{ssec:Serre-Yangian}  (the full treatment
being provided in Section~\ref{sec:Drinfeld orthosymplectic}, see Remark~\ref{rem:Serre-extosp-series-general}),
while the former is essentially reduced to the rank $\leq 2$ cases (due to Corollary~\ref{cor:commutativity})
which are treated case-by-case in Subsections~\ref{ssec:rank-1}--\ref{ssec:rank-2}. But first of all, we shall
provide explicit formulas for all entries of $E(u),F(u),H(u)$ and a factorized formula for the central
series $c_V(u)$ in Subsections~\ref{ssec:e-currents}--\ref{ssec:c-current}.


\section{Explicit Gauss decomposition and higher order relations}
\label{sec:Gauss and Serre}

In this section, we recover explicit formulas for all entries of the matrices $E(u),F(u),H(u)$ in the
Gauss decomposition~\eqref{eq:gauss-osp} as well as a factorized formula for the central series $c_V(u)$
of~\eqref{eq:central-c}. We also establish the higher order relations generalizing those from
Subsection~\ref{ssec:Serre-classical}.


\subsection{Upper triangular matrix explicitly}
\label{ssec:e-currents}
\

In this subsection, we derive explicit formulas for all entries of the matrix $E(u)$
from~(\ref{eq:gauss-osp},~\ref{eq:HFE-matrices}) in terms of the generators $e_i^{(r)}$. We consider
three cases ($N=2n$ and $|v_{n+m}|=\bar{0}$, $N=2n$ and $|v_{n+m}|=\bar{1}$, $N=2n+1$), for which
the formulas resemble those of~\cite{ft} for the $D$-type, $C$-type, and $B$-type, respectively.

\medskip
\noindent
$\bullet$ $N=2n$ and $|v_{n+m}|=\bar{0}$.

This case generalizes (from $m=0$ case) the $D_{n}$-type formulas of~\cite[Lemmas~2.79,~2.80]{ft}:

\begin{Lem}\label{lem:E-entries-evenN}
The following relations hold in $X^\rtt(\fosp(V))$:

\medskip
\noindent
(a) $e_{n+m,n+m+1}(u)=0$.

\medskip
\noindent
(b) $e_{i,j+1}(u)=(-1)^{\ol{j}}\, [e_{ij}(u),e_{j,j+1}^{(1)}]$ for $i<j<i'-1$ and $j\ne n+m$.

\medskip
\noindent
(c) $e_{i,n+m+1}(u)=(-1)^{\ol{n+m-1}}\, [e_{i,n+m-1}(u),e_{n+m}^{(1)}]$ for $1\leq i\leq n+m-2$.

\medskip
\noindent
(d) $e_{(i+1)',i'}(u)=-(-1)^{\ol{i+1}+\ol{i}\cdot \ol{i+1}}\, e_i\big(u+\kappa-\sum_{k=1}^i(-1)^{\ol{k}}\big)$
for $1\leq i\leq n+m-1$.

\medskip
\noindent
(e) $e_{(i+1)',j'}(u) = -(-1)^{\ol{j}\cdot \ol{j+1}}\, [e_{(i+1)',(j+1)'}(u),e_{j}^{(1)}]$
for $1\leq j<i\leq n+m-1$.

\medskip
\noindent
(f) $e_{i i'}(u) = -(-1)^{\ol{i+1}+\ol{i}\cdot \ol{i+1}}\, e_{i}(u)e_{i,(i+1)'}(u) -
     (-1)^{\ol{i}\cdot \ol{i+1}}\, [e_{i,(i+1)'}(u),e_{i}^{(1)}]$
for $1\leq i\leq n+m-1$.

\medskip
\noindent
(g) $e_{i+1,i'}(u) = (-1)^{\ol{i+1}+\ol{i}\cdot \ol{i+1}}\, e_{i}(u)e_{i+1,(i+1)'}(u) -
     (-1)^{\ol{i+1}+\ol{i}\cdot \ol{i+1}}\, e_{i,(i+1)'}(u) - \\
     (-1)^{\ol{i}\cdot \ol{i+1}}\, [e_{i+1,(i+1)'}(u),e_{i}^{(1)}]$
for $1\leq i\leq n+m-2$.

\medskip
\noindent
(h) $e_{i j'}(u)=-(-1)^{\ol{j}\cdot \ol{j+1}}\, [e_{i,(j+1)'}(u),e_j^{(1)}]$ for $1\leq j\leq i-2\leq n+m-2$.

\medskip
\noindent
(i) $e_{n+m,n+m+2}(u)=-e_{n+m}(u)$.
\end{Lem}

\begin{proof}
(a) follows from its validity for the $n=1,m=0$ case as established in~\eqref{eq:ef-vanishing-rk1} below
and Theorem~\ref{thm:embedding}.

\medskip
(b) is similar to~\cite[Lemma 2.79(d,e)]{ft}, cf.~\cite[Lemma 5.15]{jlm}. Due to Theorem~\ref{thm:embedding},
it suffices to establish it for $i=1$ and $1<j<2', j\ne n+m$. To this end, evaluating the $v^{-1}$-coefficients
in the defining relation
\begin{equation}\label{eq:mat.coef.1}
  [t_{1 j}(u),t_{j,j+1}(v)]=\frac{(-1)^{\ol{j}}}{u-v} \Big( t_{j j}(u)t_{1,j+1}(v) - t_{j j}(v)t_{1,j+1}(u) \Big) \,,
\end{equation}
we obtain $[t_{1 j}(u),t^{(1)}_{j,j+1}]=(-1)^{\ol{j}}\, t_{1,j+1}(u)$. As $t_{1 j}(u)=h_1(u)e_{1 j}(u)$,
$t_{1,j+1}(u)=h_1(u)e_{1,j+1}(u)$, $h_1(u)$ commutes with $e^{(1)}_{j,j+1}$ by Corollary~\ref{cor:commutativity},
and $h_1(u)$ is invertible, we get the desired relation:
\begin{equation*}
  e_{1,j+1}(u)=(-1)^{\ol{j}} [e_{1 j}(u),e^{(1)}_{j,j+1}] \,.
\end{equation*}
We note that $e_{j,j+1}^{(1)}=e^{(1)}_j$ for $j<n+m$, and
$e^{(1)}_{j,j+1}=-(-1)^{\ol{j}+\ol{j}\cdot \ol{j+1}}\,e^{(1)}_{(j+1)'}$ for $j>n+m$ by~(d).

\medskip
(c) is completely analogous to part (b), but we replace~\eqref{eq:mat.coef.1} rather with
\begin{multline}\label{eq:mat.coef.2}
  [t_{1,n+m-1}(u),t_{n+m-1,n+m+1}(v)] = \\
  \frac{(-1)^{\ol{n+m-1}}}{u-v} \, \Big( t_{n+m-1,n+m-1}(u)t_{1,n+m+1}(v) - t_{n+m-1,n+m-1}(v)t_{1,n+m+1}(u) \Big) \,.
\end{multline}

(d) is similar to~\cite[(5.18)]{jlm}. Due to the equality
  $\kappa^{[i-1]}-(-1)^{\ol{i}}=\kappa-\sum_{k=1}^i (-1)^{\ol{k}}$,
cf.~\eqref{eq:kappa truncated}, and Theorem~\ref{thm:embedding}, it suffices to establish this
relation for $i=1$. To this end, we rewrite the relation~\eqref{eq:central-c} in the form:
\begin{equation}\label{eq:T-reln-used}
  T^t(u+\kappa)=T(u)^{-1}c_V(u+\kappa) \,.
\end{equation}
Here, we note that $T(u)^{-1}=E(u)^{-1}H(u)^{-1}F(u)^{-1}$. In particular, comparing
the $(1',1')$ matrix coefficients of both sides of~\eqref{eq:T-reln-used}, we find:
\begin{equation}\label{eq:mat.coef.4}
  h_1(u+\kappa)=h_{1'}(u)^{-1}c_V(u+\kappa) \,.
\end{equation}
Likewise, comparing the $(2',1')$ matrix coefficients of both sides of~\eqref{eq:T-reln-used}, we get:
\begin{equation}\label{eq:mat.coef.5}
  (-1)^{\ol{1}+\ol{1}\cdot \ol{2}}\theta_{1'}\theta_{2'}\, t_{1 2}(u+\kappa)=
  -e_{2' 1'}(u)h_{1'}(u)^{-1}c_V(u+\kappa) \,.
\end{equation}
Evoking~\eqref{eq:mat.coef.4} and the equality
  $(-1)^{\ol{1}+\ol{1}\cdot \ol{2}}\theta_{1'}\theta_{2'}=(-1)^{\ol{2}+\ol{1}\cdot \ol{2}}$,
we can rewrite~\eqref{eq:mat.coef.5} as follows:
\begin{equation}\label{eq:mat.coef.6}
  (-1)^{\ol{2}+\ol{1}\cdot \ol{2}}\, h_1(u+\kappa)e_{1 2}(u+\kappa)=
  -e_{2' 1'}(u)h_1(u+\kappa) \,.
\end{equation}
Applying $h_1(u+\kappa)e_1(u+\kappa)=e_1(u+\kappa-(-1)^{\ol{1}})h_1(u+\kappa)$, which follows
from~\eqref{eq:h-to-e-1} and Corollary~\ref{cor:A-type relations}, to the left-hand side
of~\eqref{eq:mat.coef.6} and multiplying both sides by $h_1(u+\kappa)^{-1}$ on the right,
we obtain the desired relation:
\begin{equation*}
  e_{2' 1'}(u)=-(-1)^{\ol{2}+\ol{1}\cdot \ol{2}}\, e_1\big(u+\kappa-(-1)^{\ol{1}}\big) \,.
\end{equation*}

(e) follows from yet another super $A$-type reduction, similar to that of~\cite[Proposition 5.6]{jlm}.
Namely, multiplying the bottom-right $(n+m)\times (n+m)$ submatrices of $F(u),H(u),E(u)$ provides
an $(n+m)\times (n+m)$ matrix satisfying the RTT relation~\eqref{eq:RTT relation-gl} of $A$-type
(with the parity sequence $(\ol{n+m} , \ol{n+m-1} , \ldots , \ol{1})$ which is reverse to $\Parity$).
Therefore, part (e) now follows from part (b) and the equality
$e_{(j+1)',j'}^{(1)}=-(-1)^{\ol{j+1}+\ol{j}\cdot \ol{j+1}}\, e_j^{(1)}$ due to part (d).

\medskip
(f) is similar to~\cite[Lemma~2.80(a)]{ft}. Due to Theorem~\ref{thm:embedding}, it suffices
to establish this relation for $i=1$. Applying the reasoning of part~(b) to $j=2'$, we obtain
  $[t_{1 2'}(u),e^{(1)}_{2' 1'}]=(-1)^{\ol{2}}t_{1 1'}(u)$.
According to part~(d), we have
  $e^{(1)}_{2' 1'}=-(-1)^{\ol{1}\cdot \ol{2}+\ol{2}}\, e^{(1)}_{12}$.
Thus, the above equality reads:
\begin{equation}\label{eq:mat.coef.3}
  [h_1(u)e_{1 2'}(u),e^{(1)}_1]=-(-1)^{\ol{1}\cdot \ol{2}}\, h_1(u)e_{1 1'}(u) \,.
\end{equation}
But evaluating the $v^{-1}$-coefficients in the equality
$[h_1(u),e_1(v)]=-(-1)^{\ol{1}}h_1(u)\frac{e_1(u)-e_1(v)}{u-v}$, which follows from~\eqref{eq:Atype-eh}
and Corollary~\ref{cor:A-type relations}, we obtain $[h_1(u),e^{(1)}_1]=(-1)^{\ol{1}}h_1(u)e_1(u)$.
Plugging this into~\eqref{eq:mat.coef.3}, and multiplying both sides by $h_1(u)^{-1}$ on the left,
we get the desired relation:
\begin{equation}\label{eq:mat.coef.12}
  e_{1 1'}(u) =
  -(-1)^{\ol{2}+\ol{1}\cdot \ol{2}}\, e_{1}(u)e_{1 2'}(u) -
  (-1)^{\ol{1}\cdot \ol{2}}\, [e_{1 2'}(u),e_{1}^{(1)}] \,.
\end{equation}

(g) is similar to~\cite[Lemma~2.80(b)]{ft}. Due to Theorem~\ref{thm:embedding}, it suffices to establish
this relation for $i=1$. To this end, let us compare the $v^{-1}$-coefficients in the defining relation
\begin{equation*}
  [t_{2 2'}(u),t_{2' 1'}(v)]=
  \frac{(-1)^{\ol{2}}}{u-v}\Big(t_{2' 2'}(u)t_{2 1'}(v)-t_{2' 2'}(v)t_{2 1'}(u)\Big) -
  \frac{\sum_{p=1}^{N+2m} (-1)^{\ol{2}\cdot \ol{p}} \theta_p\, t_{p 2'}(u)t_{p' 1'}(v)}{u-v-\kappa}
\end{equation*}
of~\eqref{eq:RTT-termwise}, which together with the equality
  $t_{2' 1'}^{(1)}=e_{2' 1'}^{(1)}=-(-1)^{\ol{2}+\ol{1}\cdot \ol{2}}e^{(1)}_{1 2}$ due to part~(d) implies:
\begin{equation}\label{eq:mat.coef.8}
  [t_{2 2'}(u),e_{1 2}^{(1)}]=-(-1)^{\ol{1}\cdot \ol{2}}\, t_{2 1'}(u)-(-1)^{\ol{2}}\, t_{1 2'}(u) \,.
\end{equation}
Note that
\begin{equation}\label{eq:mat.coef.9}
  t_{2 2'}(u)=h_2(u)e_{2 2'}(u)+f_{21}(u)h_1(u)e_{1 2'}(u) \,.
\end{equation}
Comparing the $v^{-1}$-coefficients of both sides of the equality
$[h_2(u),e_1(v)]=(-1)^{\ol{2}}h_2(u)\frac{e_1(u)-e_1(v)}{u-v}$
from~\eqref{eq:Atype-eh} and Corollary~\ref{cor:A-type relations}, we find
$[h_2(u),e^{(1)}_1]=-(-1)^{\ol{2}}h_2(u)e_1(u)$, so that
\begin{equation}\label{eq:mat.coef.10}
  [h_2(u)e_{2 2'}(u),e_{1}^{(1)}]=
  h_2(u)\left(-(-1)^{\ol{2}}e_{1}(u)e_{2 2'}(u)+[e_{2 2'}(u),e_{1}^{(1)}]\right) \,.
\end{equation}
Comparing the $v^{-1}$-coefficients of both sides of
  $[t_{21}(u),t_{12}(v)]=(-1)^{\ol{1}}\, \frac{t_{11}(u)t_{22}(v)-t_{11}(v)t_{22}(u)}{u-v}$,
we get $[f_{21}(u)h_1(u),e_{1}^{(1)}]=-(-1)^{\ol{1}} (t_{11}(u)-t_{22}(u))$, so that:
\begin{equation}\label{eq:mat.coef.11}
  [f_{21}(u)h_1(u)e_{1 2'}(u),e_{1}^{(1)}]=
  -(-1)^{\ol{2}} \Big(h_1(u)-t_{22}(u)\Big) e_{1 2'}(u)+t_{21}(u)[e_{1 2'}(u),e_{1}^{(1)}] \,.
\end{equation}

Combining~\eqref{eq:mat.coef.12}--\eqref{eq:mat.coef.11}, we immediately obtain the desired equality:
\begin{equation}\label{eq:mat.coef.13}
  e_{2 1'}(u) = (-1)^{\ol{2}+\ol{1}\cdot \ol{2}}e_{1}(u)e_{2 2'}(u) -
  (-1)^{\ol{2}+\ol{1}\cdot \ol{2}}\,e_{1 2'}(u) -
  (-1)^{\ol{1}\cdot \ol{2}}\, [e_{2 2'}(u),e_{1}^{(1)}]  \,.
\end{equation}

(h) is similar to~\cite[Lemma~2.80(c)]{ft}. Due to Theorem~\ref{thm:embedding}, it suffices to establish it
for $j=1$. We shall proceed by induction on $i$. Comparing the $v^{-1}$-coefficients in the defining relation
\begin{equation}\label{eq:mat.coef.14}
  [t_{i 2'}(u),t_{2' 1'}(v)]=\frac{(-1)^{\ol{2}}}{u-v} \Big(t_{2' 2'}(u)t_{i 1'}(v)-t_{2' 2'}(v)t_{i 1'}(u)\Big)
\end{equation}
and evoking the aforementioned equality
  $t_{2' 1'}^{(1)}=e^{(1)}_{2' 1'}=-(-1)^{\ol{2}+\ol{1}\cdot \ol{2}}\,e_{1}^{(1)}$,
we obtain:
\begin{equation}\label{eq:mat.coef.15}
  [t_{i 2'}(u),e_{1}^{(1)}]=-(-1)^{\ol{1}\cdot \ol{2}}\,t_{i 1'}(u) \,.
\end{equation}
Note that the series featuring in~\eqref{eq:mat.coef.15} are explicitly given by:
\begin{equation}\label{eq:mat.coef.16}
\begin{split}
  & t_{i 1'}(u)=h_i(u)e_{i 1'}(u)+\sum_{j=1}^{i-1} f_{ij}(u)h_j(u)e_{j 1'}(u) \,, \\
  & t_{i 2'}(u)=h_i(u)e_{i 2'}(u)+\sum_{j=1}^{i-1} f_{ij}(u)h_j(u)e_{j 2'}(u) \,.
\end{split}
\end{equation}
Comparing the $v^{-1}$-coefficients of both sides of
  $[t_{i1}(u),t_{12}(v)]=\frac{(-1)^{\ol{1}}}{u-v} \big(t_{11}(u)t_{i2}(v)-t_{11}(v)t_{i2}(u)\big)$,
we obtain
  $[t_{i1}(u),e_{1}^{(1)}] = (-1)^{\ol{1}}\, t_{i2}(u) =
   (-1)^{\ol{1}}\, f_{i2}(u)h_2(u)+(-1)^{\ol{1}}\, f_{i1}(u)h_1(u)e_{1}(u)$,
so that:
\begin{multline}\label{eq:mat.coef.17}
  [f_{i1}(u)h_1(u)e_{1 2'}(u),e_{1}^{(1)}] = \\
  f_{i1}(u)h_1(u) \Big([e_{1 2'}(u),e_{1}^{(1)}]+(-1)^{\ol{2}}e_{1}(u)e_{1 2'}(u)\Big) +
  (-1)^{\ol{2}}\, f_{i2}(u)h_2(u)e_{1 2'}(u) \,.
\end{multline}
For $j=2$, we have $[f_{i2}(u),e_{1}^{(1)}]=0$ (which follows from $[f_i(u),e^{(1)}_1]=0$ for $2\leq i\leq n+m$,
see Subsection~\ref{ssec:RTT gl-Yangian}) as well as $[h_2(u),e_{1}^{(1)}]=-(-1)^{\ol{2}}\, h_2(u)e_{1}(u)$
(see the proof of~\eqref{eq:mat.coef.10}), so that:
\begin{equation}\label{eq:mat.coef.18}
  [f_{i2}(u)h_2(u)e_{2 2'}(u),e_{1}^{(1)}]=
  f_{i2}(u)h_2(u) \Big([e_{2 2'}(u),e_{1}^{(1)}]-(-1)^{\ol{2}}e_{1}(u)e_{2 2'}(u)\Big) \,.
\end{equation}
For $2<j\leq i-1$, we similarly have $[f_{ij}(u),e_{1}^{(1)}]=0=[h_j(u),e_{1}^{(1)}]$ by
Corollary~\ref{cor:commutativity}, so that:
\begin{equation}\label{eq:mat.coef.19}
  [f_{ij}(u)h_j(u)e_{j 2'}(u),e_{1}^{(1)}]=f_{ij}(u)h_j(u)[e_{j 2'}(u),e_{1}^{(1)}]=
  -(-1)^{\ol{1}\cdot \ol{2}}\, f_{ij}(u)h_j(u)e_{j 1'}(u) \,,
\end{equation}
with the last equality due to the induction assumption.

Combining $[h_i(u)e_{i 2'}(u),e^{(1)}_1]=h_i(u)[e_{i 2'}(u),e^{(1)}_1]$ with the
formulas~(\ref{eq:mat.coef.12},~\ref{eq:mat.coef.13},~\ref{eq:mat.coef.15}--\ref{eq:mat.coef.19}),
we immediately obtain the desired equality:
\begin{equation}\label{eq:mat.coef.20}
  e_{i 1'}(u) = -(-1)^{\ol{1}\cdot \ol{2}}\, [e_{i 2'}(u),e_{1}^{(1)}]
  \qquad \mathrm{for}\quad 3\leq i\leq n+m \,.
\end{equation}

(i) is similar to part~(d). Due to Theorem~\ref{thm:embedding}, it suffices to establish this relation
for $n+m=2$. Comparing the $(3',1')$ matrix coefficients of both sides of~\eqref{eq:T-reln-used}, we obtain:
\begin{equation}\label{eq:mat.coef.21}
  (-1)^{\ol{1}+\ol{1}\cdot \ol{3}}\theta_{1'}\theta_{3'}\, t_{1 3}(u+\kappa) = (T(u)^{-1})_{24}\cdot c_V(u+\kappa) \,.
\end{equation}
Note that $(T(u)^{-1})_{24}=(E(u)^{-1})_{24}h_{1'}(u)^{-1}=-e_{24}(u)h_{1'}(u)^{-1}$, where we use $e_{23}(u)=0$
due to part~(a). Evoking~\eqref{eq:mat.coef.4}, we can thus bring~\eqref{eq:mat.coef.21} to the following form:
\begin{equation}\label{eq:mat.coef.22}
  h_1(u+\kappa)e_{13}(u+\kappa) = - e_{24}(u)h_1(u+\kappa) \,.
\end{equation}
Multiplying both sides of the defining relation
  $[t_{11}(u),t_{13}(v)]=\frac{(-1)^{\ol{1}}}{u-v} (t_{11}(u)t_{13}(v)-t_{11}(v)t_{13}(u))$
by $(u-v)h_1(v)^{-1}$ on the left and plugging $v=u-(-1)^{\ol{1}}$, one gets (cf.~\eqref{eq:h-to-e-1}):
\begin{equation}\label{eq:mat.coef.23}
  h_1(u)e_{13}(u)=e_{13}(u-(-1)^{\ol{1}})h_1(u) \,.
\end{equation}
Thus, the relation~\eqref{eq:mat.coef.22} implies
  $e_{13}(u+\kappa-(-1)^{\ol{1}})h_1(u+\kappa) = - e_{24}(u)h_1(u+\kappa)$.
It remains to note that $\kappa-(-1)^{\ol{1}}=0$ as $\ol{2}=\bar{0}$. Therefore, we obtain the desired equality:
\begin{equation}\label{eq:mat.coef.24}
  e_{24}(u)=-e_{13}(u) \,.
\end{equation}

\medskip
This completes our proof of Lemma~\ref{lem:E-entries-evenN}.
\end{proof}

\noindent
$\bullet$ $N=2n$ and $|v_{n+m}|=\bar{1}$.

This case generalizes (from $n=0$ case) the $C_{m}$-type formulas of~\cite[Lemmas~3.11,~3.12]{ft}:

\begin{Lem}\label{lem:E-entries-evenN-2}
The following relations hold in $X^\rtt(\fosp(V))$:

\medskip
\noindent
(a) $e_{i,j+1}(u)=(-1)^{\ol{j}}\, [e_{ij}(u),e_{j,j+1}^{(1)}]$ for $i<j<i'-1$ and $j\ne n+m$.

\medskip
\noindent
(b) $e_{i,n+m+1}(u)=-\sfrac{1}{2} [e_{i,n+m}(u),e_{n+m}^{(1)}]$ for $1\leq i\leq n+m-1$.

\medskip
\noindent
(c) $e_{(i+1)',i'}(u)=-(-1)^{\ol{i+1}+\ol{i}\cdot \ol{i+1}}\, e_i\big(u+\kappa-\sum_{k=1}^i(-1)^{\ol{k}}\big)$
for $1\leq i\leq n+m-1$.

\medskip
\noindent
(d) $e_{(i+1)',j'}(u) = -(-1)^{\ol{j}\cdot \ol{j+1}}\, [e_{(i+1)',(j+1)'}(u),e_{j}^{(1)}]$
for $1\leq j<i\leq n+m-1$.

\medskip
\noindent
(e) $e_{i i'}(u) = -(-1)^{\ol{i+1}+\ol{i}\cdot \ol{i+1}}\, e_{i}(u)e_{i,(i+1)'}(u) -
     (-1)^{\ol{i}\cdot \ol{i+1}}\, [e_{i,(i+1)'}(u),e_{i}^{(1)}]$
for $1\leq i\leq n+m-1$.

\medskip
\noindent
(f) $e_{i+1,i'}(u) = (-1)^{\ol{i+1}+\ol{i}\cdot \ol{i+1}}\, e_{i}(u)e_{i+1,(i+1)'}(u) -
     (-1)^{\ol{i+1}+\ol{i}\cdot \ol{i+1}}\, e_{i,(i+1)'}(u) - \\
     (-1)^{\ol{i}\cdot \ol{i+1}}\, [e_{i+1,(i+1)'}(u),e_{i}^{(1)}]$
for $1\leq i\leq n+m-1$.

\medskip
\noindent
(g) $e_{i j'}(u)=-(-1)^{\ol{j}\cdot \ol{j+1}}\, [e_{i,(j+1)'}(u),e_j^{(1)}]$ for $1\leq j\leq i-2\leq n+m-2$.
\end{Lem}

\begin{proof}
The proof is completely analogous to that of Lemma~\ref{lem:E-entries-evenN}.
\end{proof}

\noindent
$\bullet$ $N=2n+1$.

This case generalizes (from $m=0$ case) the $B_{n}$-type formulas of~\cite[Lemmas~4.10,~4.11]{ft}:

\begin{Lem}\label{lem:E-entries-oddN}
The following relations hold in $X^\rtt(\fosp(V))$:

\medskip
\noindent
(a) $e_{i,j+1}(u)=(-1)^{\ol{j}}\, [e_{ij}(u),e_{j,j+1}^{(1)}]$ for $i<j<i'-1$.

\medskip
\noindent
(b) $e_{(i+1)',i'}(u)=-(-1)^{\ol{i+1}+\ol{i}\cdot \ol{i+1}}\, e_i\big(u+\kappa-\sum_{k=1}^i(-1)^{\ol{k}}\big)$
for $1\leq i\leq n+m$.

\medskip
\noindent
(c) $e_{(i+1)',j'}(u) = -(-1)^{\ol{j}\cdot \ol{j+1}}\, [e_{(i+1)',(j+1)'}(u),e_{j}^{(1)}]$
for $1\leq j<i\leq n+m-1$.

\medskip
\noindent
(d) $e_{i i'}(u) = -(-1)^{\ol{i+1}+\ol{i}\cdot \ol{i+1}}\, e_{i}(u)e_{i,(i+1)'}(u) -
     (-1)^{\ol{i}\cdot \ol{i+1}}\, [e_{i,(i+1)'}(u),e_{i}^{(1)}]$
for $1\leq i\leq n+m$.

\medskip
\noindent
(e) $e_{i+1,i'}(u) = (-1)^{\ol{i+1}+\ol{i}\cdot \ol{i+1}}\, e_{i}(u)e_{i+1,(i+1)'}(u) -
     (-1)^{\ol{i+1}+\ol{i}\cdot \ol{i+1}}\, e_{i,(i+1)'}(u) - \\
     (-1)^{\ol{i}\cdot \ol{i+1}}\, [e_{i+1,(i+1)'}(u),e_{i}^{(1)}]$
for $1\leq i\leq n+m-1$.

\medskip
\noindent
(f) $e_{i j'}(u)=-(-1)^{\ol{j}\cdot \ol{j+1}}\, [e_{i,(j+1)'}(u),e_j^{(1)}]$ for $1\leq j\leq i-2\leq n+m-1$.
\end{Lem}

\begin{proof}
The proof is completely analogous to that of Lemma~\ref{lem:E-entries-evenN}.
\end{proof}


\subsection{Lower triangular matrix explicitly}
\label{ssec:f-currents}
\

Similarly to the subsection above, we derive explicit formulas for all entries of the matrix $F(u)$
in terms of the generators $f_i^{(r)}$, treating three cases that resemble $BCD$-type formulas of~\cite{ft}.
The following lemmas can be deduced by applying the anti-automorphism $\tau$ of $X^\rtt(\fosp(V))$ given
by~\eqref{eq:tau-t} to the relations in Lemmas~\ref{lem:E-entries-evenN},~\ref{lem:E-entries-evenN-2},~\ref{lem:E-entries-oddN},
respectively, and using the formulas~\eqref{eq:tau-efh}.

\medskip
\noindent
$\bullet$ $N=2n$ and $|v_{n+m}|=\bar{0}$.

This case generalizes (from $m=0$ case) the $D_{n}$-type formulas of~\cite[Lemmas~2.96,~2.97]{ft}:

\begin{Lem}\label{lem:F-entries-evenN}
The following relations hold in $X^\rtt(\fosp(V))$:

\medskip
\noindent
(a) $f_{n+m+1,n+m}(u)=0$.

\medskip
\noindent
(b) $f_{j+1,i}(u)=(-1)^{\ol{j}}\, [f_{j+1,j}^{(1)},f_{ji}(u)]$ for $i<j<i'-1$ and $j\ne n+m$.

\medskip
\noindent
(c) $f_{n+m+1,i}(u)=(-1)^{\ol{n+m-1}}\, [f_{n+m}^{(1)},f_{n+m-1,i}(u)]$ for $1\leq i\leq n+m-2$.

\medskip
\noindent
(d) $f_{i',(i+1)'}(u)=-(-1)^{\ol{i}+\ol{i}\cdot \ol{i+1}}\, f_{i}\big(u+\kappa-\sum_{k=1}^i (-1)^{\ol{k}}\big)$
for $1\leq i\leq n+m-1$.

\medskip
\noindent
(e) $f_{j',(i+1)'}(u) = -(-1)^{\ol{j}+\ol{j+1}+\ol{j}\cdot \ol{j+1}}\, [f_j^{(1)},f_{(j+1)',(i+1)'}(u)]$
for $1\leq j<i\leq n+m-1$.

\medskip
\noindent
(f) $f_{i' i}(u) = -(-1)^{\ol{i}+\ol{i}\cdot \ol{i+1}}\, f_{(i+1)',i}(u)f_{i}(u) -
     (-1)^{\ol{i}+\ol{i+1}+\ol{i}\cdot \ol{i+1}}\, [f_{i}^{(1)},f_{(i+1)',i}(u)]$
for $1\leq i\leq n+m-1$.

\medskip
\noindent
(g) $f_{i',i+1}(u) = (-1)^{\ol{i}+\ol{i}\cdot \ol{i+1}}\, f_{(i+1)',i+1}(u)f_i(u)-
     (-1)^{\ol{i}+\ol{i}\cdot \ol{i+1}}\, f_{(i+1)',i}(u) - \\
     (-1)^{\ol{i}+\ol{i+1}+\ol{i}\cdot \ol{i+1}}\, [f_{i}^{(1)},f_{(i+1)',i+1}(u)]$
for $1\leq i\leq n+m-2$.

\medskip
\noindent
(h) $f_{j' i}(u) = - (-1)^{\ol{j}+\ol{j+1}+\ol{j}\cdot \ol{j+1}}\, [f_{j}^{(1)},f_{(j+1)',i}(u)]$
for $1\leq j\leq i-2 \leq n+m-2$.

\medskip
\noindent
(i) $f_{n+m+2,n+m}(u) = - (-1)^{\ol{n+m-1}}f_{n+m}(u)$.
\end{Lem}


\noindent
$\bullet$ $N=2n$ and $|v_{n+m}|=\bar{1}$.

This case generalizes (from $n=0$ case) the $C_{m}$-type formulas of~\cite[Lemmas~3.11,~3.12]{ft}:

\begin{Lem}\label{lem:F-entries-evenN-2}
The following relations hold in $X^\rtt(\fosp(V))$:

\medskip
\noindent
(a) $f_{j+1,i}(u)=(-1)^{\ol{j}}\, [f_{j+1,j}^{(1)},f_{ji}(u)]$ for $i<j<i'-1$ and $j\ne n+m$.

\medskip
\noindent
(b) $f_{n+m+1,i}(u)=-\sfrac{1}{2}\, [f_{n+m}^{(1)},f_{n+m,i}(u)]$ for $1\leq i\leq n+m-1$.

\medskip
\noindent
(c) $f_{i',(i+1)'}(u)=-(-1)^{\ol{i}+\ol{i}\cdot \ol{i+1}}\, f_{i}\big(u+\kappa-\sum_{k=1}^i (-1)^{\ol{k}}\big)$
for $1\leq i\leq n+m-1$.

\medskip
\noindent
(d) $f_{j',(i+1)'}(u) = -(-1)^{\ol{j}+\ol{j+1}+\ol{j}\cdot \ol{j+1}}\, [f_j^{(1)},f_{(j+1)',(i+1)'}(u)]$
for $1\leq j<i\leq n+m-1$.

\medskip
\noindent
(e) $f_{i' i}(u) = -(-1)^{\ol{i}+\ol{i}\cdot \ol{i+1}}\, f_{(i+1)',i}(u)f_{i}(u) -
     (-1)^{\ol{i}+\ol{i+1}+\ol{i}\cdot \ol{i+1}}\, [f_{i}^{(1)},f_{(i+1)',i}(u)]$
for $1\leq i\leq n+m-1$.

\medskip
\noindent
(f) $f_{i',i+1}(u) = (-1)^{\ol{i}+\ol{i}\cdot \ol{i+1}}\, f_{(i+1)',i+1}(u)f_i(u)-
     (-1)^{\ol{i}+\ol{i}\cdot \ol{i+1}}\, f_{(i+1)',i}(u) - \\
     (-1)^{\ol{i}+\ol{i+1}+\ol{i}\cdot \ol{i+1}}\, [f_{i}^{(1)},f_{(i+1)',i+1}(u)]$
for $1\leq i\leq n+m-1$.

\medskip
\noindent
(g) $f_{j' i}(u) = - (-1)^{\ol{j}+\ol{j+1}+\ol{j}\cdot \ol{j+1}}\, [f_{j}^{(1)},f_{(j+1)',i}(u)]$
for $1\leq j\leq i-2 \leq n+m-2$.
\end{Lem}


\noindent
$\bullet$ $N=2n+1$.

This case generalizes (from $m=0$ case) the $B_{n}$-type formulas of~\cite[Lemmas~4.10,~4.11]{ft}:

\begin{Lem}\label{lem:F-entries-oddN}
The following relations hold in $X^\rtt(\fosp(V))$:

\medskip
\noindent
(a) $f_{j+1,i}(u)=(-1)^{\ol{j}}\, [f_{j+1,j}^{(1)},f_{ji}(u)]$ for $i<j<i'-1$.

\medskip
\noindent
(b) $f_{i',(i+1)'}(u)=-(-1)^{\ol{i}+\ol{i}\cdot \ol{i+1}}\, f_{i}\big(u+\kappa-\sum_{k=1}^i (-1)^{\ol{k}}\big)$
for $1\leq i\leq n+m$.

\medskip
\noindent
(c) $f_{j',(i+1)'}(u) = -(-1)^{\ol{j}+\ol{j+1}+\ol{j}\cdot \ol{j+1}}\, [f_j^{(1)},f_{(j+1)',(i+1)'}(u)]$
for $1\leq j<i\leq n+m-1$.

\medskip
\noindent
(d) $f_{i' i}(u) = -(-1)^{\ol{i}+\ol{i}\cdot \ol{i+1}}\, f_{(i+1)',i}(u)f_{i}(u) -
     (-1)^{\ol{i}+\ol{i+1}+\ol{i}\cdot \ol{i+1}}\, [f_{i}^{(1)},f_{(i+1)',i}(u)]$
for $1\leq i\leq n+m$.

\medskip
\noindent
(e) $f_{i',i+1}(u) = (-1)^{\ol{i}+\ol{i}\cdot \ol{i+1}}\, f_{(i+1)',i+1}(u)f_i(u)-
     (-1)^{\ol{i}+\ol{i}\cdot \ol{i+1}}\, f_{(i+1)',i}(u) - \\
     (-1)^{\ol{i}+\ol{i+1}+\ol{i}\cdot \ol{i+1}}\, [f_{i}^{(1)},f_{(i+1)',i+1}(u)]$
for $1\leq i\leq n+m-1$.

\medskip
\noindent
(f) $f_{j' i}(u) = - (-1)^{\ol{j}+\ol{j+1}+\ol{j}\cdot \ol{j+1}}\, [f_{j}^{(1)},f_{(j+1)',i}(u)]$
for $1\leq j\leq i-2 \leq n+m-1$.
\end{Lem}



\subsection{Diagonal matrix and central current explicitly}
\label{ssec:c-current}
\

In this subsection, we derive explicit formulas for all entries of the matrix $H(u)$ in terms of the
generators $h_\imath^{(r)}$ and the factorized formula for the central current $c_V(u)$ of~\eqref{eq:central-c}.
We consider the same three cases for which the formulas resemble the $BCD$-type formulas of~\cite{ft,jlm}
and generalize~\cite[Proposition~5.1, Theorem~5.3]{m} for $N\geq 3$ and the standard parity
sequence~\eqref{eq:distinguished=standard}, though our approach is different from that used in~\cite[\S5]{m}.

\medskip
\noindent
$\bullet$ $N=2n$ and $|v_{n+m}|=\bar{0}$.

The following generalizes (from $m=0$ case) the $D_{n}$-type formula of~\cite[Theorem 5.8]{jlm}:

\begin{Lem}\label{lem:cseries-evenN-even}
The central series $c_V(u)$ from~\eqref{eq:central-c} can be factorized as follows:
\begin{equation}\label{eq:c-evenN-even}
  c_V(u) \, =
  \prod_{i=1}^{n+m-1} \frac{h_i(u-\sum_{k=1}^{i-1} (-1)^{\ol{k}})}{h_i(u-\sum_{k=1}^{i} (-1)^{\ol{k}})}
  \cdot h_{n+m}(u-n+m+1) h_{n+m+1}(u-n+m+1) \,.
\end{equation}
\end{Lem}

\begin{proof}
Comparing the $(2',2')$ matrix coefficients of both sides of~\eqref{eq:T-reln-used}, we get:
\begin{equation}\label{eq:c-derive-1}
  t_{22}(u+\kappa)=\Big(h_{2'}(u)^{-1}+e_{2' 1'}(u)h_{1'}(u)^{-1}f_{1' 2'}(u)\Big)c_V(u+\kappa) \,.
\end{equation}
Evoking $h_1(u+\kappa)=h_{1'}(u)^{-1}c_V(u+\kappa)$ of~\eqref{eq:mat.coef.4} and
the fact that $c_V(u+\kappa)$ is central, the relation~\eqref{eq:c-derive-1} can be written as:
\begin{equation}\label{eq:c-derive-2}
  h_{2'}(u)^{-1}c_V(u+\kappa) =
  h_2(u+\kappa) + f_{21}(u+\kappa)h_1(u+\kappa)e_{12}(u+\kappa) - e_{2' 1'}(u)h_1(u+\kappa)f_{1' 2'}(u) \,.
\end{equation}
Applying Lemmas~\ref{lem:E-entries-evenN}(d) and~\ref{lem:F-entries-evenN}(d) to the last summand, we obtain:
\begin{multline}\label{eq:c-derive-3}
  h_{2'}(u)^{-1}c_V(u+\kappa) = h_2(u+\kappa) + f_{21}(u+\kappa)h_1(u+\kappa)e_{12}(u+\kappa) \, - \\
  (-1)^{\ol{1}+\ol{2}}\, e_{12}\big(u+\kappa-(-1)^{\ol{1}}\big)h_1(u+\kappa)f_{21}\big(u+\kappa-(-1)^{\ol{1}}\big) \,.
\end{multline}
According to Corollaries~\ref{cor:A-resonance},~\ref{cor:A-type relations}, we have
  $h_1(u+\kappa)e_{12}(u+\kappa)=e_{12}(u+\kappa-(-1)^{\ol{1}})h_1(u+\kappa)$
and
  $h_1(u+\kappa)f_{21}(u+\kappa-(-1)^{\ol{1}})=f_{21}(u+\kappa)h_1(u+\kappa)$.
Plug these formulas into~\eqref{eq:c-derive-3} to get:
\begin{equation}\label{eq:c-derive-4}
  h_{2'}(u)^{-1}c_V(u+\kappa) =
  h_2(u+\kappa) + [f_{21}(u+\kappa),e_{12}(u+\kappa-(-1)^{\ol{1}})] h_1(u+\kappa) \,.
\end{equation}
But $[f_{21}(v),e_{12}(u)]=-\frac{(-1)^{\ol{1}}}{u-v}\left(\frac{h_2(u)}{h_1(u)}-\frac{h_2(v)}{h_1(v)}\right)$,
due to~\eqref{eq:Atype-ef} and Corollary~\ref{cor:A-type relations}, so that:
\begin{equation}\label{eq:ef-A}
  \Big[f_{21}(u+\kappa),e_{12}(u+\kappa-(-1)^{\ol{1}})\Big] =
  \frac{h_2(u+\kappa-(-1)^{\ol{1}})}{h_1(u+\kappa-(-1)^{\ol{1}})} -
  \frac{h_2(u+\kappa)}{h_1(u+\kappa)} \,.
\end{equation}
Plugging~\eqref{eq:ef-A} into the right-hand side of~\eqref{eq:c-derive-4}, we thus get:
\begin{equation}\label{eq:c-derive-5}
  h_{2'}(u)^{-1}c_V(u+\kappa) = \frac{h_1(u+\kappa)}{h_1(u+\kappa-(-1)^{\ol{1}})} h_2(u+\kappa-(-1)^{\ol{1}}) \,,
\end{equation}
which can be rewritten in the form
\begin{equation}\label{eq:c-derive-6}
  c_V(u+\kappa) = \frac{h_1(u+\kappa)}{h_1(u+\kappa-(-1)^{\ol{1}})} \cdot h_{2'}(u)h_2(u+\kappa-(-1)^{\ol{1}}) \,.
\end{equation}

Combining Theorem~\ref{thm:embedding} with~\eqref{eq:mat.coef.4} and
$\kappa-(-1)^{\ol{1}}=\kappa^{[1]}$ of~\eqref{eq:kappa truncated}, we note that
\begin{equation*}
  h_{2'}(u)h_2(u+\kappa-(-1)^{\ol{1}})=\psi_{V,1}(c_{V^{[1]}}(u+\kappa^{[1]})) \,.
\end{equation*}
Therefore, the equality~\eqref{eq:c-derive-6} can be expressed as follows:
\begin{equation}\label{eq:c-derive-7}
  c^{[0]}_{V}(u) = \frac{h_1(u)}{h_1(u-(-1)^{\ol{1}})} \cdot c^{[1]}_V(u-(-1)^{\ol{1}}) \,,
\end{equation}
where we introduce $c^{[k]}_{V}(u)$ for $0\leq k< n+m$ via
\begin{equation}\label{eq:c-truncated}
  c^{[k]}_{V}(u):=\psi_{V,k}(c_{V^{[k]}}(u)) \,.
\end{equation}
Applying the formula~\eqref{eq:c-derive-7} iteratively and using~\eqref{eq:psi-tower}, we obtain:
\begin{equation}\label{eq:c-derive-8}
  c_V(u) \, =
  \prod_{i=1}^{n+m-1} \frac{h_i(u-\sum_{k=1}^{i-1} (-1)^{\ol{k}})}{h_i(u-\sum_{k=1}^{i} (-1)^{\ol{k}})}
  \cdot c^{[n+m-1]}_{V}\left(u \, - \sum_{k=1}^{n+m-1} (-1)^{\ol{k}}\right) \,.
\end{equation}
According to~\eqref{eq:mat.coef.4} and the equality $\kappa^{[n+m-1]}=0$, we have
$c^{[n+m-1]}_{V}(u)=h_{n+m}(u)h_{n+m+1}(u)$. Plugging this equality into~\eqref{eq:c-derive-8}
recovers precisely the desired formula~\eqref{eq:c-evenN-even}.
\end{proof}

The following result generalizes (from $m=0$ case) the $D_{n}$-type formula of~\cite[Lemma 2.77]{ft}:

\begin{Lem}\label{lem:h-recover-1}
For $1\leq i<n+m$, we have
\begin{equation}\label{eq:h-leftover-1}
  h_{i'}(u) = \frac{1}{h_i(u+\kappa-\sum_{k=1}^i (-1)^{\ol{k}})} \,
  \prod_{j=i+1}^{n+m-1} \frac{h_j(u+\kappa-\sum_{k=1}^{j-1} (-1)^{\ol{k}})}{h_j(u+\kappa-\sum_{k=1}^{j} (-1)^{\ol{k}})}
  \cdot h_{n+m}(u) h_{n+m+1}(u) \,.
\end{equation}
\end{Lem}

\begin{proof}
For $i=1$, this formula follows immediately from the equality $h_{1'}(u)=h_1(u+\kappa)^{-1} c_V(u+\kappa)$
of~\eqref{eq:mat.coef.4} combined with the explicit formula~\eqref{eq:c-evenN-even} for $c_V(u)$ as
$\kappa-n+m+1=0$. The case $1<i<n+m$ follows now by applying Theorem~\ref{thm:embedding} and evoking
the formula~\eqref{eq:kappa truncated}.
\end{proof}

\noindent
$\bullet$ $N=2n$ and $|v_{n+m}|=\bar{1}$.

This case generalizes (from $n=0$ case) the $C_{m}$-type formula of~\cite[Theorem 5.8]{jlm}:

\begin{Lem}\label{lem:cseries-evenN-odd}
The central series $c_V(u)$ from~\eqref{eq:central-c} can be factorized as follows:
\begin{equation}\label{eq:c-evenN-odd}
  c_V(u) \, =
  \prod_{i=1}^{n+m-1} \frac{h_i(u-\sum_{k=1}^{i-1} (-1)^{\ol{k}})}{h_i(u-\sum_{k=1}^{i} (-1)^{\ol{k}})}
  \cdot h_{n+m}(u-n+m-1) h_{n+m+1}(u-n+m+1) \,.
\end{equation}
\end{Lem}

\begin{proof}
The proof is precisely the same as that of Lemma~\ref{lem:cseries-evenN-even} except that now
$\kappa^{[n+m-1]}=-2$ and so one rather plugs $c^{[n+m-1]}_{V}(u)=h_{n+m}(u)h_{n+m+1}(u+2)$ into
the formula~\eqref{eq:c-derive-8}.
\end{proof}

Analogously to Lemma~\ref{lem:h-recover-1}, we also obtain the following generalization
(from $n=0$ case) of~\cite[Lemma 3.11(a)]{ft}:

\begin{Lem}\label{lem:h-recover-2}
For $1\leq i<n+m$, we have
\begin{equation}\label{eq:h-leftover-2}
  h_{i'}(u) = \frac{1}{h_i(u+\kappa-\sum_{k=1}^i (-1)^{\ol{k}})} \,
  \prod_{j=i+1}^{n+m-1} \frac{h_j(u+\kappa-\sum_{k=1}^{j-1} (-1)^{\ol{k}})}{h_j(u+\kappa-\sum_{k=1}^{j} (-1)^{\ol{k}})}
  \cdot h_{n+m}(u-2) h_{n+m+1}(u) \,.
\end{equation}
\end{Lem}

\noindent
$\bullet$ $N=2n+1$.

This case generalizes (from $m=0$ case) the $B_{n}$-type formulas of~\cite[Theorem 5.8]{jlm}:

\begin{Lem}\label{lem:cseries-oddN}
The central series $c_V(u)$ from~\eqref{eq:central-c} can be factorized as follows:
\begin{equation}\label{eq:c-oddN}
  c_V(u) =
  \prod_{i=1}^{n+m} \frac{h_i(u-\sum_{k=1}^{i-1} (-1)^{\ol{k}})}{h_i(u-\sum_{k=1}^{i} (-1)^{\ol{k}})}
  \cdot h_{n+m+1}\left(u-n+m+\sfrac{1}{2}\right) h_{n+m+1}\left(u-n+m\right) \,.
\end{equation}
\end{Lem}

\begin{proof}
The proof is precisely the same as that of Lemma~\ref{lem:cseries-evenN-even}.
Specifically, the formula~\eqref{eq:c-derive-8} is now replaced by
\begin{equation}\label{eq:c-derive-9}
  c_V(u) =
  \prod_{i=1}^{n+m} \frac{h_i(u-\sum_{k=1}^{i-1} (-1)^{\ol{k}})}{h_i(u-\sum_{k=1}^{i} (-1)^{\ol{k}})}
  \cdot c^{[n+m]}_{V}\left(u-\sum_{k=1}^{n+m} (-1)^{\ol{k}}\right) \,.
\end{equation}
But $T^{[n+m]}_V(u)$ is a $1\times 1$ matrix $(h_{n+m+1}(u))$, so that
$c^{[n+m]}_{V}(u)=h_{n+m+1}(u)h_{n+m+1}(u+\frac{1}{2})$. Plugging this equality
into~\eqref{eq:c-derive-9} recovers the desired formula~\eqref{eq:c-oddN}.
\end{proof}

Analogously to Lemma~\ref{lem:h-recover-1}, we also obtain the following generalization
(from $m=0$ case) of~\cite[Lemma 4.10(a)]{ft}:

\begin{Lem}\label{lem:h-recover-3}
For $1\leq i\leq n+m$, we have
\begin{equation}\label{eq:h-leftover-3}
  h_{i'}(u) =  \frac{1}{h_i(u+\kappa-\sum_{k=1}^i (-1)^{\ol{k}})} \,
  \prod_{j=i+1}^{n+m} \frac{h_j(u+\kappa-\sum_{k=1}^{j-1} (-1)^{\ol{k}})}{h_j(u+\kappa-\sum_{k=1}^{j} (-1)^{\ol{k}})}
  \cdot h_{n+m+1}(u) h_{n+m+1}(u-\sfrac{1}{2}) \,.
\end{equation}
\end{Lem}


\subsection{Higher order relations for orthosymplectic super Yangians}
\label{ssec:Serre-Yangian}
\

The aim of this subsection is to detect degree $3$, $4$, $6$, $7$ relations in $X^\rtt(\fosp(V))$ that
quantize the loop version of the corresponding Serre relations from Subsection~\ref{ssec:Serre-classical}.
Due to Theorem~\ref{thm:embedding}, it suffices to establish these relations at the smallest possible ranks
$3$, $3$, $3$, $4$, respectively. Here, we note that sub-diagrams~\eqref{eq:pic-Serre-new4a} always arise
through a super $A$-type sub-diagram, and therefore the corresponding degree $4$ Serre relations follow
from~\eqref{eq:Atype-superserre}, due to Corollaries~\ref{cor:A-type relations} and~\ref{cor:other-A-type relations}.

\medskip
\noindent
$\bullet$ $\dim(V)=6$ and $\Parity=(\ast,\bar{1},\bar{0})$ with $\ast\in\{\bar{0},\bar{1}\}$.
Thus the Dynkin diagram is as in~\eqref{eq:pic-Serre-new3}.

\begin{Lem}\label{lem:Serre-deg3}
Under the above assumptions, the following relations hold in $X^\rtt(\fosp(V))$:
\begin{equation}\label{eq:Serre-Y-new3}
\begin{split}
  & \big[e_{3}^{(1)},[e_{2}^{(1)},e_{1}(u)]\big] -
    \big[e_{2}^{(1)},[e_{3}^{(1)},e_{1}(u)]\big] = 0 \,, \\
  & \big[f_{3}^{(1)},[f_{2}^{(1)},f_{1}(u)]\big] -
    \big[f_{2}^{(1)},[f_{3}^{(1)},f_{1}(u)]\big] = 0 \,.
\end{split}
\end{equation}
\end{Lem}

\begin{proof}
Evaluating the $v^{-1}$-coefficients in the defining relation
\begin{equation*}
  [t_{12}(u),t_{23}(v)]=\frac{(-1)^{\ol{2}}}{u-v} \Big(t_{22}(u)t_{13}(v)-t_{22}(v)t_{13}(u)\Big) \,,
\end{equation*}
we get:
\begin{equation}\label{eq:ser3-1}
  t_{13}(u)=-[t_{12}(u),t_{23}^{(1)}] \,.
\end{equation}
On the other hand, comparing the $v^{-1}$-coefficients of both sides of the defining relation
\begin{equation*}
  [t_{13}(u),t_{24}(v)]=\frac{(-1)^{\sharp}(t_{23}(u)t_{14}(v)-t_{23}(v)t_{14}(u))}{u-v} +
  \frac{\sum_{p=1}^6 t_{2p'}(v)t_{1p}(u)(-1)^{\ol{1}\cdot \ol{2}+\ol{3}\cdot \ol{2}+\ol{1}\cdot \ol{p}}\theta_4\theta_{p'}}
       {u-v-\kappa} \,,
\end{equation*}
where we use $\sharp$ whenever the exact value is irrelevant, we obtain:
\begin{equation}\label{eq:ser3-2}
  t_{15}(u)=-[t_{13}(u),t_{24}^{(1)}] \,.
\end{equation}
Combining~\eqref{eq:ser3-1} and~\eqref{eq:ser3-2}, we thus get:
\begin{equation}\label{eq:ser3-3}
  t_{15}(u) = \big[[t_{12}(u),t_{23}^{(1)}],t_{24}^{(1)}\big] = \big[[t_{12}(u),e_{23}^{(1)}],e_{24}^{(1)}\big] \,.
\end{equation}

Likewise, comparing the $v^{-1}$-coefficients of both sides of the defining relation~\eqref{eq:RTT-termwise}
applied to the commutators $[t_{12}(u),t_{24}(v)]$ and $[t_{14}(u),t_{23}(v)]$, we obtain:
\begin{align}
  t_{14}(u) & = -[t_{12}(u),t_{24}^{(1)}] \,,
    \label{eq:ser3-4} \\
  t_{15}(u) & = -[t_{14}(u),t_{23}^{(1)}] \,.
    \label{eq:ser3-5}
\end{align}
Combining~\eqref{eq:ser3-4} and~\eqref{eq:ser3-5}, we thus get:
\begin{equation}\label{eq:ser3-6}
  t_{15}(u) = \big[[t_{12}(u),t_{24}^{(1)}],t_{23}^{(1)}\big] = \big[[t_{12}(u),e_{24}^{(1)}],e_{23}^{(1)}\big] \,.
\end{equation}

Comparing the above equalities~\eqref{eq:ser3-3} and~\eqref{eq:ser3-6}, we conclude that
\begin{equation}\label{eq:ser3-7}
  \big[[t_{12}(u),e_{23}^{(1)}],e_{24}^{(1)}\big] = \big[[t_{12}(u),e_{24}^{(1)}],e_{23}^{(1)}\big] \,.
\end{equation}
As $t_{12}(u)=h_1(u)e_{12}(u)$ and $h_1(u)$ commutes with $e_{23}^{(1)},e_{24}^{(1)}$ by
Corollary~\ref{cor:commutativity}, we get:
\begin{equation}\label{eq:ser3-8}
  h_1(u)\big[[e_{12}(u),e_{23}^{(1)}],e_{24}^{(1)}\big] =
  h_1(u)\big[[e_{12}(u),e_{24}^{(1)}],e_{23}^{(1)}\big] \,.
\end{equation}
Multiplying both sides of~\eqref{eq:ser3-8} by $h_1(u)^{-1}$ on the left, we obtain the first relation
of~\eqref{eq:Serre-Y-new3}.

Applying the anti-automorphism $\tau$ of $X^\rtt(\fosp(V))$ given by~\eqref{eq:tau-t} to the first
relation of~\eqref{eq:Serre-Y-new3} and using the formulas~\eqref{eq:tau-efh} establishes the second
relation of~\eqref{eq:Serre-Y-new3}.
\end{proof}

\begin{Rem}
(a) The relations~\eqref{eq:Serre-Y-new3} still hold when $\Parity=(\ast,\bar{0},\bar{0})$
with $\ast\in \{\bar{0},\bar{1}\}$, due to the super Jacobi identity and Serre relations
$[e_{2}^{(1)},e_{3}^{(1)}]=0=[f_{2}^{(1)},f_{3}^{(1)}]$, cf.~Remark~\ref{rem:Serre-parities}(b).

\medskip
\noindent
(b) Evaluating the $u^{-1}$-coefficients in~\eqref{eq:Serre-Y-new3}, we recover precisely
the cubic Serre relations~\eqref{eq:Lie-Serre-new3}.
\end{Rem}

\medskip
\noindent
$\bullet$ $\dim(V)=7$ and $\Parity=(\ol{1},\ol{2},\ol{3})$ with $\ol{2}\ne \ol{3}$.
Thus, the Dynkin diagram is as in~\eqref{eq:pic-Serre-new4b}.

\begin{Lem}\label{lem:Serre-deg4b}
Under the above assumptions, the following relations hold in $X^\rtt(\fosp(V))$:
\begin{equation}\label{eq:Serre-Y-new4b}
\begin{split}
  & \big[[e_1(u),e^{(1)}_2],[e^{(1)}_2,e^{(1)}_3]\big]=0 \,, \\
  & \big[[f_1(u),f^{(1)}_2],[f^{(1)}_2,f^{(1)}_3]\big]=0 \,.
\end{split}
\end{equation}
\end{Lem}

\begin{proof}
Evaluating the $v^{-1}$-coefficients in the defining relation
\begin{equation*}
  [t_{12}(u),t_{23}(v)]=\frac{(-1)^{\ol{2}}}{u-v} \Big(t_{22}(u)t_{13}(v)-t_{22}(v)t_{13}(u)\Big) \,,
\end{equation*}
we get:
\begin{equation}\label{eq:ser4b-1}
  t_{13}(u)=(-1)^{\ol{2}}[t_{12}(u),e_{23}^{(1)}] \,.
\end{equation}
Likewise, evaluating the $v^{-1}$-coefficients in the defining relation
\begin{equation*}
  [t_{23}(u),t_{34}(v)]=\frac{(-1)^{\ol{3}}}{u-v} \Big(t_{33}(u)t_{24}(v)-t_{33}(v)t_{24}(u)\Big) \,,
\end{equation*}
we obtain $t_{24}(u)=(-1)^{\ol{3}}[t_{23}(u),e^{(1)}_{34}]$, so that
\begin{equation}\label{eq:ser4b-2}
  e^{(1)}_{24}=(-1)^{\ol{3}}[e^{(1)}_{23},e^{(1)}_{34}] \,.
\end{equation}
Finally, comparing the $v^{-1}$-coefficients of both sides of the defining relation
\begin{equation*}
  [t_{13}(u),t_{24}(v)]=\frac{(-1)^{\sharp}}{u-v} \Big(t_{23}(u)t_{14}(v)-t_{23}(v)t_{14}(u)\Big) \,,
\end{equation*}
we get
\begin{equation}\label{eq:ser4b-3}
  [t_{13}(u),e^{(1)}_{24}]=0 \,.
\end{equation}

Combining the equalities~(\ref{eq:ser4b-1},~\ref{eq:ser4b-2},~\ref{eq:ser4b-3}), we obtain:
\begin{equation*}
  \big[[t_{12}(u),e^{(1)}_{23}],[e^{(1)}_{23},e^{(1)}_{34}]\big]=0 \,,
\end{equation*}
which implies the first relation of~\eqref{eq:Serre-Y-new4b} as $h_1(u)$ commutes with
both $e^{(1)}_2=e^{(1)}_{23}$ and $e^{(1)}_3=e^{(1)}_{34}$.

Applying the anti-automorphism $\tau$ of $X^\rtt(\fosp(V))$ given by~\eqref{eq:tau-t} to the first
relation of~\eqref{eq:Serre-Y-new4b} and using the formulas~\eqref{eq:tau-efh} establishes the second
relation of~\eqref{eq:Serre-Y-new4b}.
\end{proof}

\begin{Rem}
(a) The relations~\eqref{eq:Serre-Y-new4b} still hold for an arbitrary $\Parity=(\ast,\ast,\ast)$
with $\ast\in \{\bar{0},\bar{1}\}$.

\medskip
\noindent
(b) Evaluating the $u^{-1}$-coefficients in~\eqref{eq:Serre-Y-new4b}, we recover the Serre
relations~\eqref{eq:Lie-Serre-superA}.
\end{Rem}

\medskip
\noindent
$\bullet$ $\dim(V)=8$ and $\Parity=(\ast,\bar{0},\bar{0},\bar{1})$ with $\ast\in \{\bar{0},\bar{1}\}$.
Thus the Dynkin diagram is as in~\eqref{eq:pic-Serre-new7}.

\begin{Lem}\label{lem:Serre-deg7}
Under the above assumptions, the following relations hold in $X^\rtt(\fosp(V))$:
\begin{equation}\label{eq:Serre-Y-new7}
\begin{split}
  & \Big[\big[e_{1}(u), [e_{2}^{(1)},e_{3}^{(1)}]\big],
    \big[[e_{2}^{(1)},e^{(1)}_{3}],[e_{3}^{(1)},e_{4}^{(1)}]\big]\Big]=0 \,, \\
  & \Big[\big[f_{1}(u), [f_{2}^{(1)},f_{3}^{(1)}]\big],
    \big[[f_{2}^{(1)},f^{(1)}_{3}],[f_{3}^{(1)},f_{4}^{(1)}]\big]\Big]=0 \,.
\end{split}
\end{equation}
\end{Lem}

\begin{proof}
Evaluating the $v^{-1}$-coefficients in the defining relation
\begin{equation*}
  [t_{23}(u),t_{34}(v)]=\frac{(-1)^{\ol{3}}}{u-v} \Big(t_{33}(u)t_{24}(v)-t_{33}(v)t_{24}(u) \Big) \,,
\end{equation*}
we obtain:
\begin{equation}\label{eq:ser7-1}
  t_{24}(u)=[t_{23}(u),t_{34}^{(1)}]=[t_{23}(u),e_{34}^{(1)}] \,.
\end{equation}
Likewise, evaluating the $v^{-1}$-coefficients in the defining relation
\begin{equation*}
  [t_{12}(u),t_{24}(v)]=\frac{(-1)^{\ol{2}}}{u-v} \Big(t_{22}(u)t_{14}(v)-t_{22}(v)t_{14}(u)\Big) \,,
\end{equation*}
we obtain:
\begin{equation}\label{eq:ser7-2}
  t_{14}(u)=[t_{12}(u),t_{24}^{(1)}]=[t_{12}(u),e_{24}^{(1)}] \,.
\end{equation}
Combining the above formulas, we thus get:
\begin{equation}\label{eq:ser7-3}
  t_{14}(u)=\big[t_{12}(u),[e_{23}^{(1)},e_{34}^{(1)}]\big] \,.
\end{equation}

Comparing the $v^{-1}$-coefficients of both sides of the defining relation
\begin{equation*}
  [t_{34}(u),t_{45}(v)]=\frac{(-1)^{\ol{4}}(t_{44}(u)t_{35}(v)-t_{44}(v)t_{35}(u))}{u-v} +
  \frac{\sum_{p=1}^8 t_{4p'}(v)t_{3p}(u)(-1)^{\ol{3}\cdot \ol{4}+\ol{4}+\ol{3}\cdot \ol{p}}\theta_5\theta_{p'}}
       {u-v-\kappa} \,,
\end{equation*}
we obtain:
\begin{equation}\label{eq:ser7-4}
  -2t_{35}(u)=[t_{34}(u),t_{45}^{(1)}]=[t_{34}(u),e_{45}^{(1)}] \,.
\end{equation}
Likewise, comparing the $v^{-1}$-coefficients of both sides of the defining relation
\begin{equation*}
  [t_{24}(u),t_{35}(v)]=\frac{(-1)^{\sharp}(t_{34}(u)t_{25}(v)-t_{34}(v)t_{25}(u))}{u-v} +
  \frac{\sum_{p=1}^8 t_{3p'}(v)t_{2p}(u)(-1)^{\ol{2}\cdot \ol{3}+\ol{3}\cdot \ol{4}+\ol{2}\cdot \ol{p}}\theta_5\theta_{p'}}
       {u-v-\kappa} \,,
\end{equation*}
we obtain:
\begin{equation}\label{eq:ser7-5}
  t_{26}(u)=[t_{24}(u),t_{35}^{(1)}]=[t_{24}(u),e_{35}^{(1)}] \,.
\end{equation}
Combining~\eqref{eq:ser7-4} and~\eqref{eq:ser7-5}, we thus get:
\begin{equation}\label{eq:ser7-6}
  t_{26}(u)=-\sfrac{1}{2} \big[t_{24}(u),[e^{(1)}_{34},e^{(1)}_{45}]\big] \,.
\end{equation}

Finally, evaluating the $v^{-1}$-coefficients in the defining relation
\begin{equation*}
  [t_{14}(u),t_{26}(v)]=\frac{(-1)^{\sharp}}{u-v} \Big(t_{24}(u)t_{16}(v)-t_{24}(v)t_{16}(u)\Big) \,,
\end{equation*}
we obtain:
\begin{equation}\label{eq:ser7-7}
  [t_{14}(u),t^{(1)}_{26}]=0 \,.
\end{equation}
Combining all the formulas above, we get the following equality:
\begin{equation}\label{eq:ser7-8}
  \Big[\big[t_{12}(u), [e_{23}^{(1)},e_{34}^{(1)}]\big],
       \big[[e_{23}^{(1)},e^{(1)}_{34}],[e_{34}^{(1)},e_{45}^{(1)}]\big]\Big]=0 \,.
\end{equation}
As $t_{12}(u)=h_1(u)e_{12}(u)$ and $h_1(u)$ commutes with $e_{23}^{(1)},e_{34}^{(1)},e^{(1)}_{45}$
by Corollary~\ref{cor:commutativity}, we get:
\begin{equation}\label{eq:ser7-9}
  h_1(u)\Big[\big[e_{12}(u), [e_{23}^{(1)},e_{34}^{(1)}]\big],
             \big[[e_{23}^{(1)},e^{(1)}_{34}],[e_{34}^{(1)},e_{45}^{(1)}]\big]\Big]=0 \,.
\end{equation}
Multiplying both sides of~\eqref{eq:ser7-9} by $h_1(u)^{-1}$ on the left, we obtain the first
relation of~\eqref{eq:Serre-Y-new7}.

Applying the anti-automorphism $\tau$ of $X^\rtt(\fosp(V))$ given by~\eqref{eq:tau-t} to the first
relation of~\eqref{eq:Serre-Y-new7} and using the formulas~\eqref{eq:tau-efh} establishes the second
relation of~\eqref{eq:Serre-Y-new7}.
\end{proof}

\begin{Rem}
(a) The relations~\eqref{eq:Serre-Y-new7} hold for all parity sequences $\Parity$:
for even $v_4$ we actually have
  $\big[[e_{2}^{(1)},e^{(1)}_{3}],[e_{3}^{(1)},e_{4}^{(1)}]\big] = 0 =
   \big[[f_{2}^{(1)},f^{(1)}_{3}],[f_{3}^{(1)},f_{4}^{(1)}]\big]$,
while for odd $v_4$ one can apply the same argument as above, cf.~Remark~\ref{rem:Serre-parities}(a).

\medskip
\noindent
(b) Evaluating the $u^{-1}$-coefficients in~\eqref{eq:Serre-Y-new7}, we recover precisely
the Serre relations~\eqref{eq:Lie-Serre-new7}.
\end{Rem}

\medskip
\noindent
$\bullet$ $\dim(V)=6$ and $\Parity=(\bar{1},\bar{0},\bar{1})$.
Thus the Dynkin diagram is as in~\eqref{eq:pic-Serre-new6}.

\begin{Lem}\label{lem:Serre-deg6}
Under the above assumptions, the following relations hold in $X^\rtt(\fosp(V))$:
\begin{equation}\label{eq:Serre-Y-new6}
\begin{split}
  & \Big[[e_{1}(u), e_{2}^{(1)}],
    \big[[e_{1}^{(1)},e^{(1)}_{2}],[e_{2}^{(1)},e_{3}^{(1)}]\big]\Big] =
    \big[[e_{1}(u),e^{(1)}_{2}],[e_{2}^{(1)},e_{3}^{(1)}]\big] \cdot [e_{1}(u), e_{2}^{(1)}] \,, \\
  & \Big[[f_{1}(u), f_{2}^{(1)}],
    \big[[f_{1}^{(1)},f^{(1)}_{2}],[f_{2}^{(1)},f_{3}^{(1)}]\big]\Big] =
    [f_{1}(u), f_{2}^{(1)}] \cdot \big[[f_{1}(u),f^{(1)}_{2}],[f_{2}^{(1)},f_{3}^{(1)}]\big] \,.
\end{split}
\end{equation}
\end{Lem}

\begin{Rem}
Evaluating the $u^{-1}$-coefficients in~\eqref{eq:Serre-Y-new6}, we recover the Serre
relations~\eqref{eq:Lie-Serre-new6}.
\end{Rem}

\begin{proof}
Evaluating the $v^{-1}$-coefficients in the defining relation~\eqref{eq:RTT-termwise} for
$[t_{12}(u),t_{23}(v)]$ and using $[h_1(u),e^{(1)}_{23}]=0$ from Corollary~\ref{cor:commutativity},
we obtain:
\begin{equation}\label{eq:ser6-1}
  t_{13}(u)=[t_{12}(u),e_{23}^{(1)}] \,, \qquad
  e_{13}(u)=[e_{12}(u),e_{23}^{(1)}] \,,
\end{equation}
cf.~\eqref{eq:ser3-1}. Comparing the $v^{-1}$-coefficients of both sides of the
defining relation~\eqref{eq:RTT-termwise} for $[t_{23}(u),t_{34}(v)]$, we get:
\begin{equation}\label{eq:ser6-2}
  -2t_{24}(u)=[t_{23}(u),t_{34}^{(1)}]=[t_{23}(u),e_{34}^{(1)}] \,,
\end{equation}
cf.~\eqref{eq:ser7-4}. Likewise, comparing the $v^{-1}$-coefficients of both sides of the
defining relation~\eqref{eq:RTT-termwise} for $[t_{13}(u),t_{24}(v)]$, we also obtain:
\begin{equation}\label{eq:ser6-3}
  t_{15}(u)=[t_{13}(u),t_{24}^{(1)}]=[t_{13}(u),e_{24}^{(1)}] \,,
\end{equation}
cf.~\eqref{eq:ser3-2}.

Let us now consider the defining relation
\begin{equation}\label{eq:ser6-4}
  [t_{13}(u),t_{15}(v)]=\frac{(-1)^{\ol{1}}}{u-v} \Big(t_{13}(u)t_{15}(v)-t_{13}(v)t_{15}(u)\Big) \,.
\end{equation}
Evaluating the $v^{-1}$-coefficients in~\eqref{eq:ser6-4} and using the formulas above, we obtain:
\begin{equation*}
  \Big[[h_{1}(u)e_1(u), e_{2}^{(1)}],
       \big[[e_{1}^{(1)},e^{(1)}_{2}],[e_{2}^{(1)},e_{3}^{(1)}]\big]\Big]=0 \,.
\end{equation*}
However, we \underline{can not} pull $h_1(u)$ to the left of the brackets, as we did in the cases of
degree $3$ and $7$ relations above, due to the presence of non-commuting $e_{1}^{(1)}$.
Instead, let us rewrite~\eqref{eq:ser6-4} as
\begin{equation}\label{eq:ser6-6}
  (u-v+1)h_1(u)e_{13}(u)h_1(v)e_{15}(v) = h_1(v)e_{13}(v)h_1(u)e_{15}(u) + (u-v) h_1(v)e_{15}(v)h_1(u)e_{13}(u) \,.
\end{equation}
We shall next pull all $h_1$-currents to the left. To this end, multiplying both sides of the relation
\begin{equation*}
  [t_{11}(u),t_{13}(v)]=\frac{(-1)^{\ol{1}}}{u-v} \Big(t_{11}(u)t_{13}(v)-t_{11}(v)t_{13}(u)\Big)
\end{equation*}
by $h_1(v)^{-1}$ on the left, we obtain:
\begin{equation}\label{eq:ser6-7}
  e_{13}(v)h_1(u)=h_1(u)\left(\frac{u-v+1}{u-v}e_{13}(v)-\frac{1}{u-v}e_{13}(u)\right) \,.
\end{equation}
Completely analogously, we also get:
\begin{equation}\label{eq:ser6-8}
  e_{15}(v)h_1(u)=h_1(u)\left(\frac{u-v+1}{u-v}e_{15}(v)-\frac{1}{u-v}e_{15}(u)\right) \,.
\end{equation}
Plugging~(\ref{eq:ser6-7},~\ref{eq:ser6-8}) into~\eqref{eq:ser6-6} and multiplying both sides by
$(u-v)h_1(u)^{-1}h_1(v)^{-1}$ on the left, we obtain:
\begin{multline}\label{eq:ser6-9}
  \left((u-v)^2-1\right)e_{13}(u)e_{15}(v) + (u-v+1)e_{13}(v)e_{15}(v) = -(u-v)e_{15}(u)e_{13}(u) \, + \\
  \left((u-v)^2+(u-v)\right)e_{15}(v)e_{13}(u)+(u-v+1)e_{13}(v)e_{15}(u)-e_{13}(u)e_{15}(u) \,.
\end{multline}
Evaluating the $v^1$-coefficients in this relation, we get:
\begin{equation}\label{eq:ser6-10}
  [e_{13}(u),e^{(1)}_{15}] = e_{15}(u)e_{13}(u) \,.
\end{equation}
Here, $e_{13}(u)$ and $e_{15}(u)$ can be expressed via~\eqref{eq:ser6-1}--\eqref{eq:ser6-3} as follows:
\begin{equation}\label{eq:ser6-11}
  e_{13}(u) = [e_{1}(u),e_{2}^{(1)}] \,, \qquad
  e_{15}(u) = -\sfrac{1}{2} \big[ [e_{1}(u),e_{2}^{(1)}],[e_{2}^{(1)},e_{3}^{(1)}] \big] \,.
\end{equation}
Plugging~\eqref{eq:ser6-11} into the equality~\eqref{eq:ser6-10} recovers precisely the first
degree $6$ relation of~\eqref{eq:Serre-Y-new6}.

Applying the anti-automorphism $\tau$ of $X^\rtt(\fosp(V))$ given by~\eqref{eq:tau-t} to the first
relation of~\eqref{eq:Serre-Y-new6} and using the formulas~\eqref{eq:tau-efh} establishes the second
relation of~\eqref{eq:Serre-Y-new6}.
\end{proof}

\begin{Rem}\label{rem:Serre-deg6-general}
As follows from the above proof, the relations~\eqref{eq:Serre-Y-new6} admit more general versions.
To this end, we note that~\eqref{eq:ser6-9} can be equivalently written as:
\begin{multline*}
  [e_{13}(u), e_{15}(v)]=
  \frac{1}{(u-v)^2}e_{13}(u)e_{15}(v) + \frac{1}{u-v}e_{15}(v)e_{13}(u)-\frac{1}{u-v}e_{15}(u)e_{13}(u) \, - \\
  \left(\frac{1}{(u-v)^2} + \frac{1}{u-v}\right)e_{13}(v)e_{15}(v) +
  \left(\frac{1}{(u-v)^2} + \frac{1}{u-v}\right)e_{13}(v)e_{15}(u) -
  \frac{1}{(u-v)^2} e_{13}(u)e_{15}(u) \,,
\end{multline*}
with $e_{13}(u)$ and $e_{15}(u)$ expressed via~\eqref{eq:ser6-11}. Applying the anti-automorphism $\tau$ of
$X^\rtt(\fosp(V))$ given by~\eqref{eq:tau-t} to the relation above and using the formulas~\eqref{eq:tau-efh},
we also obtain:
\begin{multline*}
  [f_{31}(u), f_{51}(v)]=
  \frac{1}{(u-v)^2}f_{31}(u)f_{51}(v) - \frac{1}{u-v}f_{51}(v)f_{31}(u)+\frac{1}{u-v}f_{51}(v)f_{31}(v) \, - \\
  \left(\frac{1}{(u-v)^2} - \frac{1}{u-v}\right)f_{31}(u)f_{51}(u) +
  \left(\frac{1}{(u-v)^2} - \frac{1}{u-v}\right)f_{31}(v)f_{51}(u) -
  \frac{1}{(u-v)^2} f_{31}(v)f_{51}(v) \,.
\end{multline*}
\end{Rem}

\begin{Rem}
The analogues of degree $6$ relations~\eqref{eq:Serre-Y-new6}, with both right-hand sides been multiplied by
$-(-1)^{\ol{1}}$, hold for all parity sequences $\Parity$. Indeed, for even $v_3$ both sides vanish as we have
  $$\big[[e_{1}(u),e^{(1)}_{2}],[e_{2}^{(1)},e_{3}^{(1)}]\big] = 0 =
    \big[[f_{1}(u),f^{(1)}_{2}],[f_{2}^{(1)},f_{3}^{(1)}]\big] \,,$$
while for odd $v_3$ one can apply the same argument as above, cf.~Remark~\ref{rem:Serre-parities}(a).
\end{Rem}


\section{Rank 1 and 2 relations}
\label{sec:rank 1-2}

In this section, we establish quadratic relations between the generating currents $e_i(u),f_i(u),h_\imath(u)$ of
$X^\rtt(\fosp(V))$ in rank $\leq 2$ cases (corresponding to $N+2m\leq 5$). The arguments are straightforward
though a bit tedious. While our treatment is case-by-case, we try to present them in a rather uniform way
(in particular, eliminating the smaller rank reduction of~\cite{jlm} for non-super types).


\subsection{Rank 1 cases}
\label{ssec:rank-1}
\

In this subsection, we establish quadratic relations for rank $1$ orthosymplectic Yangians which do not
follow from Corollary~\ref{cor:A-type relations}. There are four cases that we consider separately:
$(N=2,m=0)$, $(N=0,m=1)$, $(N=3,m=0)$, and $(N=1,m=1)$. The first three were treated in~\cite{jlm}.


\subsubsection{Relations for $\fosp(2|0)$ case}
\label{ssec:osp-20}
\

We note that $X^\rtt(\fosp(2|0))\simeq X^\rtt(\fso_2)$ by Remark~\ref{rem:super-to-nonsuper}.

\begin{Prop}\label{prop:osp20}
The following relations hold in $X^\rtt(\fosp(2|0))$:
\begin{equation}\label{eq:ef-vanishing-rk1}
  e_{12}(u)=0=f_{21}(u) \,.
\end{equation}
\end{Prop}

\begin{Rem}
This result follows from the relations~\eqref{eq:ef-vanishing-1} established in~\cite[Lemma 5.3]{jlm} using
the low rank isomorphism of~\cite{amr} by evoking the embedding $X^\rtt(\fso_2)\hookrightarrow X^\rtt(\fso_4)$ of
Theorem~\ref{thm:embedding} which maps $e_{12}(u)\mapsto e_{23}(u)$ and $f_{21}(u)\mapsto f_{32}(u)$. However, for
the rest of this section, it is instructive to present a direct self-contained proof of~\eqref{eq:ef-vanishing-rk1}.
\end{Rem}

\begin{proof}
Consider the defining relation~\eqref{eq:RTT-termwise} for $[t_{11}(u),t_{12}(v)]$ (note that $\kappa=0$):
\begin{equation*}
  [t_{11}(u),t_{12}(v)]=\frac{1}{u-v}t_{11}(u)t_{12}(v)+\frac{1}{u-v}t_{12}(v)t_{11}(u) \,,
\end{equation*}
where we readily cancelled two terms containing $t_{11}(v)t_{12}(u)$ in the right-hand side.
Multiplying both sides by $(u-v)h_1(v)^{-1}$ on the left, we get:
\begin{equation*}
  (u-v-1)h_1(u)e_{12}(v)=(u-v+1)e_{12}(v)h_1(u) \,.
\end{equation*}
Plugging $u=v-1$ above, we obtain $h_1(v-1)e_{12}(v)=0$. Multiplying further by $h_1(v-1)^{-1}$ on the left,
we get the desired relation $e_{12}(v)=0$. Applying the anti-automorphism $\tau$ of $X^\rtt(\fso_2)$ given
by~\eqref{eq:tau-t} to $e_{12}(v)=0$, we obtain $f_{21}(v)=0$, due to Remark~\ref{rem:tau-efh}.
\end{proof}


\subsubsection{Relations for $\fosp(0|2)$ case}
\label{ssec:osp-02}
\

We note that $X^\rtt(\fosp(0|2))\simeq X^\rtt(\fsp_2)$ by Remark~\ref{rem:super-to-nonsuper}.

\begin{Prop}\label{prop:osp02}
The currents $h_1(-2u)$, $h_2(-2u)$, $e_{1}(-2u)$, $f_{1}(-2u)$ satisfy the relations
of Theorem~\ref{thm:Drinfeld-A} for the parity sequence $\parity=(\bar{0},\bar{0})$.
\end{Prop}

\begin{proof}
This result follows from the fact that the assignment $T(u)\mapsto \sfT(-u/2)$ gives rise to the superalgebra
isomorphism $X^\rtt(\fosp(0|2)) \iso Y^\rtt(\gl_2)$. This map can be viewed as a composition of the aforementioned
isomorphism $X^\rtt(\fosp(0|2)) \iso X^\rtt(\fsp_2)$, given by $T(u)\mapsto T(-u)$, and the isomorphism
$X^\rtt(\fsp_2) \iso Y^\rtt(\gl_2)$ of~\cite[Proposition~4.1]{amr}, given by $T(u)\mapsto \sfT(u/2)$.
The latter follows from the observation that $P+Q=\ID$ for $\fsp_2$-case, which allows to relate the
corresponding $R$-matrices of $\fsp_2$ and $\gl_2$ types via $R(u)=\frac{u-1}{u-2}\sfR(u/2)$.
\end{proof}


\subsubsection{Relations for $\fosp(3|0)$ case}
\label{ssec:osp-30}
\

We note that $X^\rtt(\fosp(3|0))\simeq X^\rtt(\fso_3)$ by Remark~\ref{rem:super-to-nonsuper}.
In this case, the only relation directly implied by Corollary~\ref{cor:A-type relations} is
the obvious commutativity $[h_1(u),h_1(v)]=0$.

\begin{Prop}\label{prop:osp30}
The following relations hold in $X^\rtt(\fso_3)$:
\begin{equation}\label{eq:osp30-1}
  [h_i(u),h_j(v)]=0 \quad \mathrm{for\ all} \quad 1\leq i,j\leq 2 \,,
\end{equation}
\begin{equation}\label{eq:osp30-2}
  [h_1(u),e_{12}(v)]=\frac{h_1(u)\big(e_{12}(v)-e_{12}(u)\big)}{u-v} \,,\qquad
  [h_1(u),f_{21}(v)]=\frac{\big(f_{21}(u)-f_{21}(v)\big)h_1(u)}{u-v} \,,
\end{equation}
\begin{equation}\label{eq:osp30-3}
  [h_2(u),e_{12}(v)]=
  \frac{h_2(u)\big(e_{12}(u)-e_{12}(v)\big)}{2(u-v)} - \frac{\big(e_{12}(u-1)-e_{12}(v)\big)h_2(u)}{2(u-v-1)} \,,
\end{equation}
\begin{equation}\label{eq:osp30-4}
  [h_2(u),f_{21}(v)]=
  \frac{\big(f_{21}(v)-f_{21}(u)\big)h_2(u)}{2(u-v)} - \frac{h_2(u)\big(f_{21}(v)-f_{21}(u-1)\big)}{2(u-v-1)} \,,
\end{equation}
\begin{equation}\label{eq:osp30-5}
  [e_{12}(u),f_{21}(v)]=\frac{1}{u-v}\left(h_1(u)^{-1}h_2(u)-h_1(v)^{-1}h_2(v)\right) \,,
\end{equation}
\begin{equation}\label{eq:osp30-6}
  [e_{12}(u),e_{12}(v)]=\frac{\big(e_{12}(u)-e_{12}(v)\big)^2}{u-v} \,,
\end{equation}
\begin{equation}\label{eq:osp30-7}
  [f_{21}(u),f_{21}(v)]=-\frac{\big(f_{21}(u)-f_{21}(v)\big)^2}{u-v} \,.
\end{equation}
\end{Prop}

\begin{Rem}\label{rem:so3-remark}
(a) The relation~\eqref{eq:osp30-4} corrects a typo in~\cite[(5.4)]{jlm}.

\medskip
\noindent
(b) We note that these relations were established in~\cite[Proposition 5.4]{jlm} using the low rank isomorphism
$X^\rtt(\fso_3)\simeq Y^\rtt(\gl_2)$ of~\cite[Proposition~4.4]{amr}, see Proposition~\ref{prop:so3=gl2}(a) below.
However, for the rest of this section, it is instructive to establish all these relations directly.
\end{Rem}

\begin{proof}
In view of Remark~\ref{rem:so3-remark}, we shall only present a direct proof of~\eqref{eq:osp30-4},
though it can be also derived from~\eqref{eq:osp30-3} by applying the anti-automorphism $\tau$ of
$X^\rtt(\fso_3)$. The relations (\ref{eq:osp30-1}--\ref{eq:osp30-2},~\ref{eq:osp30-5}--\ref{eq:osp30-7})
can be proved similarly to analogous relations from Proposition~\ref{prop:osp12} below.

Our proof of~\eqref{eq:osp30-4} shall closely follow that of~\eqref{eq:osp12-3} presented below.
First, let us express $h_2(u)$ via the $h_1$-current and the central current $\sz_V(u)$ from
Remark~\ref{rem:tau-generators} defined through the difference equation $c_V(u)=\sz_V(u-1/2)\sz_V(u)$,
see~\eqref{eq:c-squareroot}. Evoking $c_V(u)=\frac{h_1(u)h_2(u-1/2)h_2(u-1)}{h_1(u-1)}$, due to
Lemma~\ref{lem:cseries-oddN}, we get $\sz_V(u-1/2)=\frac{h_1(u-1/2)h_2(u-1)}{h_1(u-1)}$, so that
\begin{equation}\label{eq:h2-via-h1c-so3}
  h_2(u)=\sz_V(u+\sfrac{1}{2})h_1(u)h_1(u+\sfrac{1}{2})^{-1} \,.
\end{equation}
Combining~\eqref{eq:h2-via-h1c-so3} with the following commutation rules between $h_1(u)$ and $f_{21}(v)$,
recovered from the defining relation~\eqref{eq:RTT-termwise} applied to $[t_{11}(u),t_{21}(v)]$:
\begin{equation*}
\begin{split}
  & h_1(u)f_{21}(v)=\left(\frac{u-v-1}{u-v}f_{21}(v)+\frac{1}{u-v}f_{21}(u)\right)h_1(u) \,, \\
  & h_1(u)^{-1}f_{21}(v)=\left(\frac{u-v}{u-v-1}f_{21}(v)-\frac{1}{u-v-1}f_{21}(u-1)\right) h_1(u)^{-1} \,,
\end{split}
\end{equation*}
we obtain:
\begin{multline}\label{eq:h2f21-so3}
  h_2(u)f_{21}(v)=
  h_1(u+\sfrac{1}{2})^{-1} \left(\frac{u-v-1}{u-v}f_{21}(v)+\frac{1}{u-v}f_{21}(u)\right) h_1(u)\sz_V(u+\sfrac{1}{2}) = \\
  \left(\frac{(u-v+1/2)(u-v-1)}{(u-v)(u-v-1/2)}f_{21}(v)+\frac{1}{u-v-1/2}f_{21}(u-\sfrac{1}{2})-\frac{1}{u-v}f_{21}(u)\right)
  h_2(u) \,.
\end{multline}
In particular, plugging $v=u-1$ into~\eqref{eq:h2f21-so3}, we find:
\begin{equation}\label{eq:h2f21-so3-spec}
  f_{21}(u-\sfrac{1}{2})h_2(u)=\frac{h_2(u)f_{21}(u-1)+f_{21}(u)h_2(u)}{2} \,.
\end{equation}
Plugging the formula~\eqref{eq:h2f21-so3-spec} into the equality~\eqref{eq:h2f21-so3}, multiplying by
$\frac{2u-2v-1}{2u-2v-2}$, and rearranging the terms, we obtain the desired relation~\eqref{eq:osp30-4}.
\end{proof}


\subsubsection{Relations for $\fosp(1|2)$ case}
\label{ssec:osp-12}
\

Finally, let us treat the remaining rank $1$ case of $X^\rtt(\fosp(V))=X^\rtt(\fosp(1|2))$ which can not
be reduced to non-super setup unlike the previous three cases. The corresponding relations also appeared
very recently in~\cite{mr}.

\begin{Prop}\label{prop:osp12}
The following relations hold in $X^\rtt(\fosp(1|2))$:
\begin{equation}\label{eq:osp12-1}
  [h_i(u),h_j(v)]=0 \quad \mathrm{for\ all} \quad 1\leq i,j\leq 2 \,,
\end{equation}
\begin{equation}\label{eq:osp12-2}
  [h_1(u),e_{12}(v)]=\frac{h_1(u)\big(e_{12}(u)-e_{12}(v)\big)}{u-v} \,,\qquad
  [h_1(u),f_{21}(v)]=\frac{\big(f_{21}(v)-f_{21}(u)\big)h_1(u)}{u-v} \,,
\end{equation}
\begin{equation}\label{eq:osp12-3}
  [h_2(u),e_{12}(v)]=h_2(u)\left(\frac{e_{12}(u)-e_{12}(v)}{u-v}+\frac{e_{12}(v)-e_{12}(u-1/2)}{u-v-1/2}\right) \,,
\end{equation}
\begin{equation}\label{eq:osp12-4}
  [h_2(u),f_{21}(v)]=\left(\frac{f_{21}(v)-f_{21}(u)}{u-v}+\frac{f_{21}(u-1/2)-f_{21}(v)}{u-v-1/2}\right)h_2(u) \,,
\end{equation}
\begin{equation}\label{eq:osp12-5}
  [e_{12}(u),f_{21}(v)]=\frac{1}{u-v}\left(h_1(u)^{-1}h_2(u)-h_1(v)^{-1}h_2(v)\right)
\end{equation}
as well as
\begin{multline}\label{eq:osp12-6}
  [e_{12}(u),e_{12}(v)]=\frac{e_{13}(u)-e_{13}(v)}{u-v} + \frac{e_{12}(u)^2-e_{12}(v)^2}{u-v} \, + \\
  \frac{e_{12}(u)e_{12}(v)-e_{12}(v)e_{12}(u)}{2(u-v)} - \frac{\big(e_{12}(u)-e_{12}(v)\big)^2}{2(u-v)^2} \,,
\end{multline}
\begin{multline}\label{eq:osp12-7}
  [f_{21}(v),f_{21}(u)]=\frac{f_{31}(v)-f_{31}(u)}{u-v} + \frac{f_{21}(u)^2-f_{21}(v)^2}{u-v} \, + \\
  \frac{f_{21}(v)f_{21}(u)-f_{21}(u)f_{21}(v)}{2(u-v)} - \frac{\big(f_{21}(v)-f_{21}(u)\big)^2}{2(u-v)^2} \,,
\end{multline}
\begin{multline}\label{eq:Serre-osp12-e}
  (u-v-1)(u-v+1/2)e_{12}(u)e_{13}(v) + (u-v+1/2)e_{12}(v)e_{13}(v) - (u-v+1/2)e_{12}(v)e_{13}(u) - \\
  (u-v)(u-v+3/2)e_{13}(v)e_{12}(u) + (2u-2v+1/2)e_{13}(u)e_{12}(u) - (u-v)e_{12}(v)e_{12}(u)^2 - e_{12}(u)^3=0 \,,
\end{multline}
\begin{multline}\label{eq:Serre-osp12-f}
  (u-v-1)(u-v+1/2)f_{31}(v)f_{21}(u) + (u-v+1/2)f_{31}(v)f_{21}(v) - (u-v+1/2)f_{31}(u)f_{21}(v) - \\
  (u-v)(u-v+3/2)f_{21}(u)f_{31}(v) + (2u-2v+1/2)f_{21}(u)f_{31}(u) + (u-v)f_{21}(u)^2f_{21}(v) + f_{21}(u)^3=0 \,,
\end{multline}
where $e_{13}(u)$ and $f_{31}(u)$ can be further expressed via
\begin{equation}\label{eq:osp12-8}
  e_{13}(u)=-e_{12}(u)^2-[e_{12}(u),e^{(1)}_{12}] \,, \qquad
  f_{31}(u)=f_{21}(u)^2+[f_{21}(u),f^{(1)}_{21}] \,.
\end{equation}

Furthermore, the remaining entries of the matrices $E(u),F(u),H(u)$ are given by:
\begin{equation}\label{eq:osp12-9}
  e_{23}(u)=-e_{12}(u-\sfrac{1}{2}) \,,\
  f_{32}(u)=f_{21}(u-\sfrac{1}{2}) \,,\
  h_3(u)= h_1(u-\sfrac{1}{2})^{-1}h_2(u-\sfrac{1}{2})h_2(u) \,.
\end{equation}
\end{Prop}

\begin{proof}
The defining relation~\eqref{eq:RTT-termwise} applied to $[t_{11}(u),t_{11}(v)]$ implies
$(u-v+1)h_1(u)h_1(v)=(u-v+1)h_1(v)h_1(u)$, hence, $[h_1(u),h_1(v)]=0$. Likewise, both relations
of~\eqref{eq:osp12-2} follow directly by applying the defining relation~\eqref{eq:RTT-termwise}
to the commutators $[t_{11}(u),t_{12}(v)]$ and $[t_{11}(u),t_{21}(v)]$.

We note that the relations~\eqref{eq:osp12-2} allow one to pull $h_1(u)$ past $e_{12}(v)$ and $f_{21}(v)$
either to the left or to the right. To this end, let us first rewrite~\eqref{eq:osp12-2} as follows:
\begin{equation}\label{eq:h-past-ef-osp12-1}
\begin{split}
  & (u-v)e_{12}(v)h_1(u)=h_1(u)\big((u-v+1)e_{12}(v)-e_{12}(u)\big) \,, \\
  & (u-v)h_1(u)f_{21}(v)=\big((u-v+1)f_{21}(v)-f_{21}(u)\big)h_1(u) \,.
\end{split}
\end{equation}
Plugging $v=u+1$ into these relations, we obtain, cf.~(\ref{eq:h-to-e-1},~\ref{eq:h-to-f-1}):
\begin{equation}\label{eq:h-past-ef-osp12-2}
  h_1(u)e_{12}(u)=e_{12}(u+1)h_1(u) \,, \qquad f_{21}(u)h_1(u)=h_1(u)f_{21}(u+1) \,.
\end{equation}
Finally, plugging~\eqref{eq:h-past-ef-osp12-2} back into the equalities~\eqref{eq:h-past-ef-osp12-1},
we also obtain:
\begin{equation}\label{eq:h-past-ef-osp12-3}
\begin{split}
  & (u-v+1)h_1(u)e_{12}(v)=\big((u-v)e_{12}(v)+e_{12}(u+1)\big)h_1(u) \,, \\
  & (u-v+1)f_{21}(v)h_1(u)=h_1(u)\big((u-v)f_{21}(v)+f_{21}(u+1)\big) \,.
\end{split}
\end{equation}

The commutativity $[h_1(u),h_2(v)]=0$ is a direct consequence of Corollary~\ref{cor:commutativity}.
For an alternative direct proof, let us apply the defining relation~\eqref{eq:RTT-termwise} to
$[t_{11}(u),t_{22}(v)]$:
\begin{equation*}
  (u-v)[h_1(u),h_2(v)+f_{21}(v)h_1(v)e_{12}(v)]=f_{21}(v)h_1(v)h_1(u)e_{12}(u)-f_{21}(u)h_1(u)h_1(v)e_{12}(v) \,.
\end{equation*}
Using the equalities~\eqref{eq:h-past-ef-osp12-1} to pull $h_1(u)$ and $h_1(v)$ to the middle
in the left-hand side, we get:
\begin{equation*}
  (u-v)[h_1(u),h_2(v)]+(u-v+1)f_{21}(v)[h_1(u),h_1(v)]e_{12}(v)=0 \,,
\end{equation*}
so that $[h_1(u),h_2(v)]=0$ as claimed.

Finally, the commutativity $[h_2(u),h_2(v)]=0$ of~\eqref{eq:osp12-1} follows from the formula
$c_V(u)=\frac{h_1(u)}{h_1(u+1)}h_2(u+1)h_2(u+3/2)$ for the central current $c_V(u)$ of~\eqref{eq:central-c},
due to Lemma~\ref{lem:cseries-oddN}.

According to Lemma~\ref{lem:E-entries-oddN}(b,d), we have $e_{13}(u)=-e_{12}(u)^2-[e_{12}(u),e^{(1)}_{12}]$,
$e_{23}(u)=-e_{12}(u-1/2)$, thus recovering the first formulas of~(\ref{eq:osp12-8},~\ref{eq:osp12-9}).
The latter implies $e_{23}^{(1)}=-e^{(1)}_{12}$. Likewise, due to Lemma~\ref{lem:F-entries-oddN}(b,d),
we have $f_{31}(u)=f_{21}(u)^2+[f_{21}(u),f^{(1)}_{21}]$, $f_{32}(u)=f_{21}(u-1/2)$, thus recovering the
second formulas of~(\ref{eq:osp12-8},~\ref{eq:osp12-9}). The latter implies $f_{32}^{(1)}=f^{(1)}_{21}$.
Finally, we have $h_3(u)=h_1(u-1/2)^{-1}h_2(u-1/2)h_2(u)$ due to Lemma~\ref{lem:h-recover-3}, recovering
the last formula of~\eqref{eq:osp12-9}.

\medskip
Let us prove~\eqref{eq:osp12-5}.
Applying the defining relation~\eqref{eq:RTT-termwise} to $[t_{21}(u),t_{12}(v)]$, we get:
\begin{multline}\label{eq:t21t12-osp12}
  (u-v)f_{21}(u)h_1(u)h_1(v)e_{12}(v)+ (u-v)h_1(v)e_{12}(v)f_{21}(u)h_1(u) = \\
  h_1(v)h_2(u)-h_1(u)h_2(v)+h_1(v)f_{21}(u)h_1(u)e_{12}(u)-h_1(u)f_{21}(v)h_1(v)e_{12}(v) \,.
\end{multline}
Using the equalities~(\ref{eq:h-past-ef-osp12-1})--(\ref{eq:h-past-ef-osp12-3}) we can pull both
$h_1(u),h_1(v)$ to the leftmost part in all summands of~\eqref{eq:t21t12-osp12}, and multiplying
further both sides by $h_1(u)^{-1}h_1(v)^{-1}$ on the left, we obtain:
\begin{equation}\label{eq:ef-commutator-osp12-1}
  (u-v+1)[e_{12}(v),f_{21}(u+1)]-[e_{12}(u),f_{21}(u+1)]=h_1(u)^{-1}h_2(u)-h_1(v)^{-1}h_2(v) \,.
\end{equation}
Plugging $v=u+1$ into~\eqref{eq:ef-commutator-osp12-1}, we get:
\begin{equation}\label{eq:ef-commutator-osp12-2}
  -[e_{12}(u),f_{21}(u+1)]=h_1(u)^{-1}h_2(u)-h_1(u+1)^{-1}h_2(u+1) \,.
\end{equation}
Subtracting~\eqref{eq:ef-commutator-osp12-2} from~\eqref{eq:ef-commutator-osp12-1} and
renaming $v\rightsquigarrow u, u+1\rightsquigarrow v$, we obtain the relation~\eqref{eq:osp12-5}.

\medskip
Let us prove~\eqref{eq:osp12-3}. One way to establish it is to consider the defining relation
\begin{equation}\label{eq:t12t22-osp12}
  [t_{12}(u),t_{22}(v)]=\frac{t_{22}(u)t_{12}(v)-t_{22}(v)t_{12}(u)}{u-v} +
  \frac{t_{23}(v)t_{11}(u)+t_{22}(v)t_{12}(u)-t_{21}(v)t_{13}(u)}{u-v+3/2} \,.
\end{equation}
Here, the left-hand side may be written as follows:
\begin{multline}\label{eq:t12t22-osp12-left}
  [t_{12}(u),t_{22}(v)]=h_1(u)[e_{12}(u),h_2(v)]+[t_{12}(u),t_{21}(v)]e_{12}(v)-t_{21}(v)[t_{12}(u),e_{12}(v)] = \\
  h_1(u)[e_{12}(u),h_2(v)]+\frac{t_{22}(u)t_{11}(v)-t_{22}(v)t_{11}(u)}{u-v}e_{12}(v)-t_{21}(v)[t_{12}(u),e_{12}(v)] \,.
\end{multline}
Plugging~\eqref{eq:t12t22-osp12-left} into the left-hand side of~\eqref{eq:t12t22-osp12}, rearranging the terms,
and using the defining relation~\eqref{eq:RTT-termwise} for $[t_{12}(u),t_{12}(v)]$ and $[t_{12}(u),t_{11}(v)]$,
we eventually obtain:
\begin{equation}\label{eq:t12t22-osp12-2}
  h_1(u)[e_{12}(u),h_2(v)] =
  \frac{h_1(u)h_2(v)\big(e_{12}(v)-e_{12}(u)\big)}{u-v} +
  \frac{h_2(v)\big(e_{23}(v)h_1(u)+h_1(u)e_{12}(u)\big)}{u-v+3/2} \,.
\end{equation}
Evoking the first equalities of~\eqref{eq:osp12-9} and~\eqref{eq:h-past-ef-osp12-1}, we get:
\begin{equation*}
  e_{23}(v)h_1(u)=-e_{12}(v-\sfrac{1}{2})h_1(u)=
  -h_1(u)\left( \frac{u-v+3/2}{u-v+1/2}e_{12}(v-\sfrac{1}{2})-\frac{1}{u-v+1/2}e_{12}(u) \right) \,,
\end{equation*}
so that
\begin{equation*}
  \frac{h_2(v)\big(e_{23}(v)h_1(u)+h_1(u)e_{12}(u)\big)}{u-v+3/2}=
  \frac{h_1(u)h_2(v)\big(e_{12}(u)-e_{12}(v-1/2)\big)}{u-v+1/2} \,.
\end{equation*}
Plugging this into~\eqref{eq:t12t22-osp12-2}, multiplying by $h_1(u)^{-1}$ on the left,
and renaming $u \leftrightsquigarrow v$, we get~\eqref{eq:osp12-3}.

Another proof of~\eqref{eq:osp12-3} is based on the expression of $h_2(u)$ via the $h_1$-current and a
central current $\wt{\sz}_V(u)$ defined via the following difference equation (cf.~\eqref{eq:c-squareroot}):
\begin{equation*}
  c_V(u)=\wt{\sz}_V(u+\sfrac{1}{2})\wt{\sz}_V(u) \,.
\end{equation*}
Evoking $c_V(u)=\frac{h_1(u)h_2(u+1)h_2(u+3/2)}{h_1(u+1)}$, we get\footnote{In fact, the difference
equation defining $\wt{\sz}_V(u)$ is specifically engineered to allow for such an expression.}
$\wt{\sz}_V(u)=\frac{h_1(u)h_2(u+1)}{h_1(u+1/2)}$, so that
\begin{equation}\label{eq:h2-via-h1c}
  h_2(u)=\wt{\sz}_V(u-1)h_1(u-\sfrac{1}{2})h_1(u-1)^{-1} \,.
\end{equation}
Combining~\eqref{eq:h2-via-h1c} with the relation
  $e_{12}(v)h_1(u)^{-1}=h_1(u)^{-1}\left(\frac{u-v}{u-v+1}e_{12}(v)+\frac{1}{u-v+1}e_{12}(u+1)\right)$
which follows from~\eqref{eq:h-past-ef-osp12-3}, and evoking~\eqref{eq:h-past-ef-osp12-1}, we obtain:
\begin{multline}\label{eq:e12h2-osp12}
  e_{12}(v)h_2(u) =
  \wt{\sz}_V(u-1)h_1(u-\sfrac{1}{2})
  \left(\frac{u-v+1/2}{u-v-1/2}e_{12}(v)-\frac{1}{u-v-1/2}e_{12}(u-\sfrac{1}{2})\right) h_1(u-1)^{-1}=\\
  h_2(u)\left(\frac{(u-v+1/2)(u-v-1)}{(u-v)(u-v-1/2)}e_{12}(v)+
              \frac{1}{u-v-1/2}e_{12}(u-\sfrac{1}{2})-\frac{1}{u-v}e_{12}(u)\right) \,.
\end{multline}
Subtracting $h_2(u)e_{12}(v)$ from both sides of~\eqref{eq:e12h2-osp12}, we obtain the desired
relation~\eqref{eq:osp12-3}, due to the equality
  $\frac{(u-v+1/2)(u-v-1)}{(u-v)(u-v-1/2)}-1=\frac{1}{u-v}-\frac{1}{u-v-1/2}$.


\medskip
Let us prove~\eqref{eq:osp12-6}. Applying the defining relation~\eqref{eq:RTT-termwise} to
$[t_{12}(u),t_{12}(v)]$, we get:
\begin{multline}\label{eq:t12t12-osp12}
  t_{12}(u)t_{12}(v)+t_{12}(v)t_{12}(u) + \frac{1}{u-v}\Big(t_{12}(u)t_{12}(v)-t_{12}(v)t_{12}(u)\Big) \, - \\
  \frac{1}{u-v+3/2} \Big(t_{11}(v)t_{13}(u)-t_{12}(v)t_{12}(u)-t_{13}(v)t_{11}(u)\Big) = 0 \,.
\end{multline}
Using~\eqref{eq:h-past-ef-osp12-1} let us pull both $h_1(u)$ and $h_1(v)$ to the leftmost part
in all terms but $t_{13}(v)t_{11}(u)$:
\begin{equation*}
\begin{split}
  & t_{12}(u)t_{12}(v)=h_1(u)h_1(v)\left(\frac{u-v-1}{u-v}e_{12}(u)e_{12}(v)+\frac{1}{u-v}e_{12}(v)^2\right) \,, \\
  & t_{12}(v)t_{12}(u)=h_1(u)h_1(v)\left(\frac{u-v+1}{u-v}e_{12}(v)e_{12}(u)-\frac{1}{u-v}e_{12}(u)^2\right) \,, \\
  & t_{11}(v)t_{13}(u)=h_1(u)h_1(v)e_{13}(u) \,,\quad
    t_{11}(u)t_{13}(v)=h_1(u)h_1(v)e_{13}(v) \,.
\end{split}
\end{equation*}
To treat the remaining summand $t_{13}(v)t_{11}(u)$ in~\eqref{eq:t12t12-osp12}, we recall the defining relation
\begin{multline}\label{eq:t11t13-osp12}
  [t_{11}(u),t_{13}(v)] = \frac{-1}{u-v}\Big(t_{11}(u)t_{13}(v)-t_{11}(v)t_{13}(u)\Big) \, + \\
  \frac{1}{u-v+3/2}\Big(t_{11}(v)t_{13}(u)-t_{12}(v)t_{12}(u)-t_{13}(v)t_{11}(u)\Big) \,.
\end{multline}
Rearranging the terms in~\eqref{eq:t11t13-osp12}, we obtain:
\begin{multline}\label{eq:t13t11-osp12}
  \frac{1}{u-v+3/2}t_{13}(v)t_{11}(u)= - \frac{2u-2v+3/2}{(u-v)(u-v+1/2)(u-v+3/2)}t_{11}(v)t_{13}(u) \, + \\
  \frac{u-v+1}{(u-v)(u-v+1/2)}t_{11}(u)t_{13}(v) + \frac{1}{(u-v+1/2)(u-v+3/2)}t_{12}(v)t_{12}(u) \,.
\end{multline}
Thus, using the equality~\eqref{eq:t13t11-osp12} for the last term in the right-hand side
of~\eqref{eq:t12t12-osp12}, then pulling both $h_1(u)$ and $h_1(v)$ to the leftmost part
as outlined above, and finally multiplying further by $h_1(u)^{-1}h_1(v)^{-1}$ on the left, we get:
\begin{multline}\label{eq:osp12-6harder}
  \frac{u-v+1}{u-v}\left(\frac{u-v-1}{u-v}e_{12}(u)e_{12}(v)+\frac{1}{u-v}e_{12}(v)^2\right) + \\
  \left(\frac{u-v+3/2}{u-v+1/2}-\frac{1}{u-v}\right)\left(\frac{u-v+1}{u-v}e_{12}(v)e_{12}(u)-\frac{1}{u-v}e_{12}(u)^2\right) + \\
  \frac{u-v+1}{(u-v)(u-v+1/2)}e_{13}(v) - \frac{u-v+1}{(u-v)(u-v+1/2)}e_{13}(u) = 0 \,.
\end{multline}
Note that $\frac{u-v+3/2}{u-v+1/2}-\frac{1}{u-v}=\frac{(u-v+1)(u-v-1/2)}{(u-v)(u-v+1/2)}$.
Therefore, multiplying~\eqref{eq:osp12-6harder} by $\frac{(u-v)^2(u-v+1/2)}{u-v+1}$, we obtain
an equivalent relation:
\begin{multline}\label{eq:osp12-6harder2}
  (u-v+1/2)(u-v-1)e_{12}(u)e_{12}(v) + (u-v+1/2)e_{12}(v)^2 + (u-v+1)(u-v-1/2)e_{12}(v)e_{12}(u) - \\
  (u-v-1/2)e_{12}(u)^2 + (u-v)e_{13}(v) - (u-v)e_{13}(u) = 0 \,.
\end{multline}
Rearranging the terms in~\eqref{eq:osp12-6harder2} and multiplying by $\frac{1}{(u-v)^{2}}$,
we recover the desired relation~\eqref{eq:osp12-6}.


\medskip
Let us finally prove~\eqref{eq:Serre-osp12-e}. Applying the defining relation~\eqref{eq:RTT-termwise}
to $[t_{12}(u),t_{13}(v)]$, we get:
\begin{equation}\label{eq:t12t13-osp12}
  (u-v+1)h_1(u)e_{12}(u)h_1(v)e_{13}(v)-h_1(v)e_{12}(v)h_1(u)e_{13}(u)-(u-v)h_1(v)e_{13}(v)h_1(u)e_{12}(u)=0 \,.
\end{equation}
Using~\eqref{eq:h-past-ef-osp12-1}, let us pull both $h_1(u)$ and $h_1(v)$ to the leftmost part
in the first two terms:
\begin{equation*}
\begin{split}
  & e_{12}(u)h_1(v)=h_1(v)\left(\frac{u-v-1}{u-v}e_{12}(u)+\frac{1}{u-v}e_{12}(v)\right) \,, \\
  & e_{12}(v)h_1(u)=h_1(u)\left(\frac{u-v+1}{u-v}e_{12}(v)-\frac{1}{u-v}e_{12}(u)\right) \,.
\end{split}
\end{equation*}
On the other hand, $h_1(v)e_{13}(v)h_1(u)=t_{13}(v)t_{11}(u)$ has been already evaluated
in~\eqref{eq:t13t11-osp12} above. Thus, first using the equality~\eqref{eq:t13t11-osp12} for
the last term in~\eqref{eq:t12t13-osp12}, then pulling both $h_1(u),h_1(v)$ to the leftmost part
as outlined above, and finally multiplying by $h_1(u)^{-1}h_1(v)^{-1}$ on the left, we get:
\begin{multline}\label{eq:e12e13-osp12}
  \frac{(u-v+1)(u-v-1)}{u-v}e_{12}(u)e_{13}(v) + \frac{u-v+1}{u-v}e_{12}(v)e_{13}(v) -
  \frac{u-v+1}{u-v}e_{12}(v)e_{13}(u) \, + \\
  \frac{1}{u-v}e_{12}(u)e_{13}(u)- \frac{(u-v+1)(u-v+3/2)}{u-v+1/2}e_{13}(v)e_{12}(u) +
  \frac{2u-2v+3/2}{u-v+1/2}e_{13}(u)e_{12}(u) \, - \\
  \frac{u-v+1}{u-v+1/2}e_{12}(v)e_{12}(u)^2+\frac{1}{u-v+1/2}e_{12}(u)^3=0 \,.
\end{multline}
Plugging $v=u+1$ into~\eqref{eq:e12e13-osp12}, we obtain:
\begin{equation}\label{eq:e12e13-osp12-2}
  e_{12}(u)e_{13}(u)=e_{13}(u)e_{12}(u)-2e_{12}(u)^3 \,.
\end{equation}
Therefore, replacing $e_{12}(u)e_{13}(u)$ in~\eqref{eq:e12e13-osp12} with the right-hand side
of~\eqref{eq:e12e13-osp12-2} and multiplying further by $\frac{(u-v)(u-v+1/2)}{u-v+1}$, we get
the desired relation~\eqref{eq:Serre-osp12-e}.


We note that relations~(\ref{eq:osp12-4},~\ref{eq:osp12-7},~\ref{eq:Serre-osp12-f}) follow directly by
applying the anti-automorphism $\tau$ of $X^\rtt(\fosp(1|2))$ given by~\eqref{eq:tau-t} to the
relations~(\ref{eq:osp12-3},~\ref{eq:osp12-6},~\ref{eq:Serre-osp12-e}) and using the formulas~\eqref{eq:tau-efh}.

\medskip
This completes our proof of Proposition~\ref{prop:osp12}.
\end{proof}

\begin{Rem}
Evaluating the $u^{1}$-coefficients in the relations~\eqref{eq:Serre-osp12-e}
and~\eqref{eq:Serre-osp12-f}, we obtain:
\begin{equation*}
  [e_{13}(v),e_{12}^{(1)}]-e_{12}(v)e_{13}(v)=0 \,,\qquad
  [f_{31}(v),f_{21}^{(1)}]+f_{31}(v)f_{21}(v)=0 \,.
\end{equation*}
Plugging above the formulas for $e_{13}(v)$ and $f_{31}(v)$ from~\eqref{eq:osp12-8}, we obtain
the following cubic relations for the currents $e_{12}(v)$ and $f_{21}(v)$, cf.~\cite[(3.7, 3.8)]{acfr}:
\begin{equation}\label{eq:acfr-like Serre}
\begin{split}
  & e_{12}(v)^3=[e_{12}(v),(e_{12}^{(1)})^2]-[e_{12}(v),e_{12}^{(1)}]e_{12}(v) \,, \\
  & f_{21}(v)^3=-[f_{21}(v),(f_{21}^{(1)})^2]-f_{21}(v)[f_{21}(v),f_{21}^{(1)}] \,.
\end{split}
\end{equation}
\end{Rem}

\begin{Rem}
We note that the cubic relations~(\ref{eq:Serre-osp12-e},~\ref{eq:Serre-osp12-f}) differ slightly
from~\cite[(4.9,~4.10)]{mr}, which is not surprising as one can add linear multiples of the quadratic
relations~(\ref{eq:osp12-6},~\ref{eq:osp12-7}). However, the key feature of both choices is that at
the associated graded algebra level they yield:
\begin{equation}\label{eq:e12e13-assgr}
  [\wt{e}^{(r)}_{12},\wt{e}^{(s)}_{13}]=0 \,,\qquad [\wt{f}^{(r)}_{21},\wt{f}^{(s)}_{31}]=0
  \qquad \mathrm{for\ any} \quad r,s\geq 1 \,.
\end{equation}
Indeed, evaluating $u^{-k}v^{-\ell}$-coefficients in~\eqref{eq:Serre-osp12-e} and passing to
their associated graded, we get:
\begin{equation}\label{eq:e12e13-assgr-direct}
  [\wt{e}^{(k+2)}_{12},\wt{e}^{(\ell)}_{13}] - 2[\wt{e}^{(k+1)}_{12},\wt{e}^{(\ell+1)}_{13}] +
  [\wt{e}^{(k)}_{12},\wt{e}^{(\ell+2)}_{13}] = 0
    \quad \mathrm{for\ any} \quad k,\ell\in \BZ \,,
\end{equation}
with $\wt{e}^{(\leq 0)}_{\bullet \bullet}=0$. In particular, we get $[\wt{e}^{(k)}_{12},\wt{e}^{(1)}_{13}]=0$
(by plugging $\ell=-1$ into~\eqref{eq:e12e13-assgr-direct}), $[\wt{e}^{(k)}_{12},\wt{e}^{(2)}_{13}]=0$
(by plugging $\ell=0$ into~\eqref{eq:e12e13-assgr-direct}), and then we get the first equality
of~\eqref{eq:e12e13-assgr} by induction on $s$.
\end{Rem}

\begin{Rem}
We note that the $\BZ_2$-grading of $V$ in~\cite{acfr} is $|v_1|=\bar{0}, |v_2|=\bar{1}, |v_3|=\bar{0}$,
which is opposite to ours, and hence the $R$-matrix of~\cite[(2.4)]{acfr} slightly differs from ours
(besides for the common prefactor). The main isomorphism $\phi\colon \mathcal{A}^+\iso \mathcal{Y}(R)$
of~\cite[Theorem~3.1]{acfr} between the (new) Drinfeld and RTT realizations of the super Yangian of $\fosp(1|2)$
is best restated using the opposite Gauss decomposition of the generator matrix $T(u)$
(denoted by $L(u)$ in \emph{loc.~cit.}):
  $$\phi\colon e(u)\mapsto \widehat{e}_{23}(-u) \,,\quad
    f(u) \mapsto \widehat{f}_{32}(-u) \,,\quad
    h(u)\mapsto \widehat{h}_2(-u-1)\widehat{h}_3(-u-1)^{-1} \,.$$
Here, the \emph{opposite} Gauss decomposition of $T(u)$ refers to the unique factorization
\begin{equation}\label{eq:gauss-opposite}
  T(u)=\widehat{E}(u) \cdot \widehat{H}(u) \cdot \widehat{F}(u)
\end{equation}
with
\begin{enumerate}

\item[--]
an upper-triangular matrix $\widehat{E}(u)=(\widehat{e}_{ij}(u))$ with $\widehat{e}_{ii}(u)=1$,

\item[--]
a diagonal matrix $\widehat{H}(u)=\mathrm{diag}(\widehat{h}_1(u),\ldots,\widehat{h}_{1'}(u))$,

\item[--]
a lower-triangular matrix $\widehat{F}(u)=(\widehat{f}_{ji}(u))$ with $\widehat{f}_{ii}(u)=1$.

\end{enumerate}
One may wonder how the two Gauss decompositions are related, and if the defining relations for our
conventions~\eqref{eq:gauss-osp} imply those for the generating series in the opposite Gauss
decomposition~\eqref{eq:gauss-opposite}. In fact, the composition of the anti-automorphism $\tau$ from
Remark~\ref{rem:tau-antiautom} and the antipode anti-automorphism $S$ give by $S(T(u))=T(u)^{-1}$ gives
rise to an algebra automorphism of $X^\rtt(\fosp(V))$ that intertwines our Gauss decomposition and the
opposite one. Thus, it is just a matter of preference which one to use, and we follow the previous
literature~\cite{bk,jlm,m} on the subject.
\end{Rem}


\subsection{Rank 2 cases}
\label{ssec:rank-2}
\

In this subsection, we establish quadratic relations for rank $2$ orthosymplectic Yangians which do not
follow from Corollaries~\ref{cor:A-type relations},~\ref{cor:other-A-type relations} and from rank $1$
cases treated in Subsection~\ref{ssec:rank-1}. There are eight cases that we consider separately:
$(N=4,m=0)$, $(N=0,m=2)$, $(N=2,m=1)$ with the parity sequence $\Parity=(\bar{1},\bar{0})$ or
$\Parity=(\bar{0},\bar{1})$, $(N=5,m=0)$, $(N=1,m=2)$, $(N=3,m=1)$ with the parity sequence
$\Parity=(\bar{1},\bar{0})$ or $\Parity=(\bar{0},\bar{1})$. We note that the first, second, and fifth
cases were already treated in~\cite{jlm}, while the sixth case was treated very recently in~\cite{mr}.


\subsubsection{Relations for $\fosp(4|0)$ case}
\label{ssec:osp-22-3}
\

In this case, we have $X^\rtt(\fosp(V))\simeq X^\rtt(\fso_4)$ by Remark~\ref{rem:super-to-nonsuper}. Some of
the relations among the generating currents $e_{12}(u),e_{13}(u),f_{21}(u),f_{31}(u),h_1(u),h_2(u),h_3(u)$
already follow from those for $Y^\rtt(\gl_2)$, as specified in Corollaries~\ref{cor:A-type relations}
and~\ref{cor:other-A-type relations}. On the other hand, we also have
\begin{equation}\label{eq:ef-vanishing-1}
  e_{23}(u)=0=f_{32}(u) \,,
\end{equation}
due to Theorem~\ref{thm:embedding} and Proposition~\ref{prop:osp20}.

\begin{Prop}\label{prop:osp40}
The following relations hold in $X^\rtt(\fosp(4|0))$:
\begin{equation}\label{eq:osp40-23}
  [h_3(u),e_{12}(v)]=\frac{h_3(u)\big(e_{12}(v)-e_{12}(u)\big)}{u-v} \,,\quad
  [h_3(u),f_{21}(v)]=\frac{\big(f_{21}(u)-f_{21}(v)\big)h_3(u)}{u-v} \,,
\end{equation}
\begin{equation}\label{eq:osp40-25}
  [h_2(u),e_{13}(v)]=\frac{h_2(u)\big(e_{13}(v)-e_{13}(u)\big)}{u-v} \,,\quad
  [h_2(u),f_{31}(v)]=\frac{\big(f_{31}(u)-f_{31}(v)\big)h_2(u)}{u-v} \,,
\end{equation}
\begin{equation}\label{eq:osp40-33}
  [e_{12}(u),f_{31}(v)]=0 \,,\qquad [e_{13}(u),f_{21}(v)]=0 \,,
\end{equation}
\begin{equation}\label{eq:osp40-5}
  [e_{12}(u),e_{13}(v)]=0 \,,\qquad [f_{21}(u),f_{31}(v)]=0 \,.
\end{equation}
\end{Prop}

This result was established in~\cite{jlm} using the embedding
  $X^\rtt(\fso_4)\hookrightarrow Y^\rtt(\gl_2)\otimes Y^\rtt(\gl_2)$
of~\cite{amr}. However, it is instructive to prove these relations directly, which can
be done completely analogously to our proof of Proposition~\ref{prop:osp22-10} below
(we leave details to the interested reader).


\subsubsection{Relations for $\fosp(0|4)$ case}
\label{ssec:osp-22-4}
\

In this case, we have $X^\rtt(\fosp(V))\simeq X^\rtt(\fsp_4)$ by Remark~\ref{rem:super-to-nonsuper},
with the isomorphism given by $T(u)\mapsto T(-u)$. The relations on the generating currents
$e_{12}(u),f_{21}(u),h_1(u),h_2(u)$ already follow from those for $Y^\rtt(\gl(0|2))\simeq Y^\rtt(\gl_2)$
as specified in Corollary~\ref{cor:A-type relations}. On the other hand, the relations on the generating
currents $e_{23}(u),f_{32}(u),h_2(u),h_3(u)$ readily follow from those for
$X^\rtt(\fosp(V^{[1]}))\simeq X^\rtt(\fosp(0|2))\simeq X^\rtt(\fsp_2)$ as specified
in Proposition~\ref{prop:osp02}.

\begin{Prop}\label{prop:osp04}
The following relations hold in $X^\rtt(\fosp(0|4))$:
\begin{equation}\label{eq:osp04-1}
  [h_1(u),h_3(v)]=0 \,,
\end{equation}
\begin{equation}\label{eq:osp04-23}
  [h_3(u),e_{12}(v)]=\frac{h_3(u)(e_{12}(u-2)-e_{12}(v))}{u-v-2} \,,\qquad
  [h_3(u),f_{21}(v)]=\frac{(f_{21}(v)-f_{21}(u-2))h_3(u)}{u-v-2} \,,
\end{equation}
\begin{equation}\label{eq:osp04-24}
  [h_1(u),e_{23}(v)]=0 \,,\qquad
  [h_1(u),f_{32}(v)]=0 \,,
\end{equation}
\begin{equation}\label{eq:osp04-33}
  [e_{12}(u),f_{32}(v)]=0 \,,\qquad [e_{23}(u),f_{21}(v)]=0 \,,
\end{equation}
\begin{equation}\label{eq:osp04-43}
  [e_{12}(u),e_{23}(v)]=\frac{2}{u-v} \Big(e_{13}(u)-e_{13}(v)-e_{12}(u)e_{23}(v)+e_{12}(v)e_{23}(v)\Big) \,,
\end{equation}
\begin{equation}\label{eq:osp04-44}
  [f_{21}(u),f_{32}(v)]=\frac{2}{u-v} \Big(f_{31}(v)-f_{31}(u)+f_{32}(v)f_{21}(u)-f_{32}(v)f_{21}(v)\Big) \,.
\end{equation}
\end{Prop}

This result goes back to~\cite{jlm}. We note however that~\eqref{eq:osp04-44} corrects
a typo in~\cite[(5.34)]{jlm}.


\subsubsection{Relations for $\fosp(2|2)$ case with the parity sequence $(\bar{1},\bar{0})$}
\label{ssec:osp-22-1}
\

In this case, we have the generating currents $e_{12}(u),e_{13}(u),f_{21}(u),f_{31}(u),h_1(u),h_2(u),h_3(u)$.
Some of the relations among them already follow from those for $Y^\rtt(\gl(1|1))$ with the parity sequence
$\parity=\Parity=(\bar{1},\bar{0})$, as specified in Corollaries~\ref{cor:A-type relations},~\ref{cor:other-A-type relations}.
On the other hand, we also have
\begin{equation}\label{eq:ef-vanishing-2}
  e_{23}(u)=0=f_{32}(u) \,,
\end{equation}
due to Theorem~\ref{thm:embedding} and Proposition~\ref{prop:osp20}.

\begin{Prop}\label{prop:osp22-10}
The following relations hold in the corresponding $X^\rtt(\fosp(V))$:
\begin{equation}\label{eq:osp221-2}
  [h_3(u),e_{12}(v)]=\frac{h_3(u)(e_{12}(v)-e_{12}(u))}{u-v} \,,\quad
  [h_3(u),f_{21}(v)]=\frac{(f_{21}(u)-f_{21}(v))h_3(u)}{u-v} \,,
\end{equation}
\begin{equation}\label{eq:osp221-1}
  [h_2(u),e_{13}(v)]=\frac{h_2(u)(e_{13}(v)-e_{13}(u))}{u-v} \,,\quad
  [h_2(u),f_{31}(v)]=\frac{(f_{31}(u)-f_{31}(v))h_2(u)}{u-v} \,,
\end{equation}
\begin{equation}\label{eq:osp221-3}
  [e_{12}(u),f_{31}(v)]=0 \,,\qquad [e_{13}(u),f_{21}(v)]=0
\end{equation}
as well as
\begin{multline}\label{eq:osp221-4}
  [e_{12}(u),e_{13}(v)]=\frac{1}{u-v}\Big(e_{12}(u)e_{13}(v)-e_{13}(v)e_{12}(u)\Big) \, + \\
  \frac{1}{u-v}\Big(-e_{12}(u)e_{13}(u)+e_{13}(u)e_{12}(u)+[e_{13}(v),e^{(1)}_{12}]-[e_{13}(u),e^{(1)}_{12}]\Big) \,,
\end{multline}
\begin{multline}\label{eq:osp221-5}
  [f_{21}(u),f_{31}(v)]=\frac{1}{u-v}\Big(f_{31}(v)f_{21}(u)-f_{21}(u)f_{31}(v)\Big) \, + \\
  \frac{1}{u-v}\Big(-f_{31}(u)f_{21}(u)+f_{21}(u)f_{31}(u)+[f^{(1)}_{21},f_{31}(v)]-[f^{(1)}_{21},f_{31}(u)]\Big) \,.
\end{multline}
\end{Prop}

\begin{Rem}
As a direct consequence of the relations~(\ref{eq:osp221-4},~\ref{eq:osp221-5}),
we obtain more familiar relations, cf.~(\ref{eq:Atype-ee-2},~\ref{eq:Atype-ff-2}):
\begin{equation}\label{eq:osp221-4simplify}
  u[e^\circ_{12}(u),e_{13}(v)]-v[e_{12}(u),e^\circ_{13}(v)]=e_{12}(u)e_{13}(v)-e_{13}(v)e_{12}(u) \,,
\end{equation}
\begin{equation}\label{eq:osp221-5simplify}
  u[f^\circ_{21}(u),f_{31}(v)]-v[f_{21}(u),f^\circ_{31}(v)]=f_{31}(v)f_{21}(u)-f_{21}(u)f_{31}(v) \,,
\end{equation}
with the currents $e^\circ_{1k}(u)=\sum_{r\geq 2} e^{(r)}_{1k}u^{-r}$ and
$f^\circ_{k1}(u)=\sum_{r\geq 2} f^{(r)}_{k1}u^{-r}$.
\end{Rem}

\begin{proof}
First, as follows from~\eqref{eq:Atype-eh} and Corollaries~\ref{cor:A-resonance},~\ref{cor:A-type relations},
we have the following relations:
\begin{equation}\label{eq:he12-osp221}
\begin{split}
  & e_{12}(v)h_1(u)=h_1(u)\left(\frac{u-v+1}{u-v}e_{12}(v)-\frac{1}{u-v}e_{12}(u)\right) \,, \\
  & e_{12}(v)h_2(u)=h_2(u)\left(\frac{u-v+1}{u-v}e_{12}(v)-\frac{1}{u-v}e_{12}(u)\right) \,, \\
  & h_1(u)e_{12}(v)=\left(\frac{u-v}{u-v+1}e_{12}(v)+\frac{1}{u-v+1}e_{12}(u+1)\right)h_1(u) \,, \\
  & h_2(u)e_{12}(v)=\left(\frac{u-v}{u-v+1}e_{12}(v)+\frac{1}{u-v+1}e_{12}(u+1)\right)h_2(u) \,,
\end{split}
\end{equation}
which allow one to pull $h_1(u)^{\pm 1}$ and $h_2(u)^{\pm 1}$ past $e_{12}(v)$ either to the left or
to the right. According to Corollary~\ref{cor:other-A-type relations}, we get analogous relations with
  $h_1(u)\rightsquigarrow h_1(u), h_2(u)\rightsquigarrow h_3(u), e_{12}(v)\rightsquigarrow e_{13}(v)$.

Let us prove the first relations of~(\ref{eq:osp221-2},~\ref{eq:osp221-1}). As
$h_3(u)=c_V(u-1)h_1(u)h_2(u)^{-1}h_1(u-1)^{-1}$ by Lemma~\ref{lem:cseries-evenN-even}, we have:
\begin{multline}\label{eq:e12h3-osp221}
  e_{12}(v)h_3(u) =
  c_V(u-1)h_1(u)\left(\frac{u-v+1}{u-v}e_{12}(v)-\frac{1}{u-v}e_{12}(u)\right)h_2(u)^{-1}h_1(u-1)^{-1} = \\
  c_V(u-1)h_1(u)h_2(u)^{-1}e_{12}(v)h_1(u-1)^{-1} =
  h_3(u) \left(\frac{u-v-1}{u-v}e_{12}(v)+\frac{1}{u-v}e_{12}(u)\right) \,,
\end{multline}
where we pull all the $h_\bullet$-currents to the left of $e_{12}(v)$ using~\eqref{eq:he12-osp221}.
Subtracting $h_3(u)e_{12}(v)$ from both sides, we get the first relation of~\eqref{eq:osp221-2}.
The proof of the first relation of~\eqref{eq:osp221-1} is analogous with the indices $2\leftrightarrow 3$
swapped, in particular, we use $h_2(u)=c_V(u-1)h_1(u)h_3(u)^{-1}h_1(u-1)^{-1}$.

We note that the second relations of~\eqref{eq:osp221-2} and~\eqref{eq:osp221-1} follow directly by
applying the anti-automorphism $\tau$ given by~\eqref{eq:tau-t} to the corresponding first relations
and using the formulas~\eqref{eq:tau-efh}.

\medskip
Let us prove~\eqref{eq:osp221-3}. Applying the defining relation~\eqref{eq:RTT-termwise} to
$[t_{13}(u),t_{21}(v)]$, we get:
\begin{equation}\label{eq:t13t21-osp221-0}
  [t_{13}(u),t_{21}(v)]=\frac{1}{u-v}\Big( t_{23}(u)t_{11}(v) - t_{23}(v)t_{11}(u) \Big) \,.
\end{equation}
As $e_{23}(u)=0$ by~\eqref{eq:ef-vanishing-2} and $h_1(u)e_{13}(u)=e_{13}(u+1)h_1(u)$
by~\eqref{eq:h-to-e-1} and Corollary~\ref{cor:other-A-type relations}, we actually have
$t_{23}(u)=f_{21}(u)e_{13}(u+1)h_1(u)$. Hence, the relation~\eqref{eq:t13t21-osp221-0}
can be written as:
\begin{multline}\label{eq:t13t21-osp221}
  e_{13}(u+1)h_1(u)f_{21}(v)h_1(v) + f_{21}(v)h_1(v)e_{13}(u+1)h_1(u) = \\
  \frac{1}{u-v} f_{21}(u)e_{13}(u+1)h_1(u)h_1(v) - \frac{1}{u-v} f_{21}(v)e_{13}(v+1)h_1(u)h_1(v) \,.
\end{multline}
Pulling both $h_1(u)$ and $h_1(v)$ to the right in the left-hand side of~\eqref{eq:t13t21-osp221} by using
\begin{equation*}
\begin{split}
  & h_1(u)f_{21}(v)=\left(\frac{u-v+1}{u-v}f_{21}(v)-\frac{1}{u-v}f_{21}(u)\right)h_1(u) \,, \\
  & h_1(v)e_{13}(u+1)=\left(\frac{u-v+1}{u-v}e_{13}(u+1)-\frac{1}{u-v}e_{13}(v+1)\right)h_1(v) \,,
\end{split}
\end{equation*}
and multiplying further by $(u-v)h_1(u)^{-1}h_1(v)^{-1}$ on the right, we obtain
\begin{multline*}
  e_{13}(u+1)\big((u-v+1)f_{21}(v)-f_{21}(u)\big) + f_{21}(v)\big((u-v+1)e_{13}(u+1)-e_{13}(v+1)\big) = \\
  f_{21}(u)e_{13}(u+1)-f_{21}(v)e_{13}(v+1) \,,
\end{multline*}
which can be further simplified to:
\begin{equation}\label{eq:e13f21-osp221}
  (u-v+1)[e_{13}(u+1),f_{21}(v)] = [e_{13}(u+1),f_{21}(u)] \,.
\end{equation}
Plugging $v=u+1$ into~\eqref{eq:e13f21-osp221}, we get $[e_{13}(u+1),f_{21}(u)]=0$ and so
$(u-v+1)[e_{13}(u+1),f_{21}(v)]=0$. This implies the second relation of~\eqref{eq:osp221-3}.
Meanwhile, the first relation of~\eqref{eq:osp221-3} follows directly by applying the
anti-automorphism $\tau$ given by~\eqref{eq:tau-t} to the second relation and using~\eqref{eq:tau-efh}.

\medskip
Let us prove~\eqref{eq:osp221-4}. Applying the defining relation~\eqref{eq:RTT-termwise} to
$[t_{12}(u),t_{13}(v)]$, we get:
\begin{multline*}
  [t_{12}(u),t_{13}(v)]=\frac{-1}{u-v}\Big( t_{12}(u)t_{13}(v) - t_{12}(v)t_{13}(u) \Big) \, + \\
  \frac{1}{u-v+1}\Big(t_{11}(v)t_{14}(u)-t_{12}(v)t_{13}(u)-t_{13}(v)t_{12}(u)-t_{14}(v)t_{11}(u)\Big) \,.
\end{multline*}
The relation above can be rearranged as:
\begin{multline*}
  \frac{u-v+1}{u-v}h_1(u)e_{12}(u)h_1(v)e_{13}(v) - \frac{1}{(u-v)(u-v+1)}h_1(v)e_{12}(v)h_1(u)e_{13}(u) \, + \\
  \frac{u-v+2}{u-v+1}h_1(v)e_{13}(v)h_1(u)e_{12}(u)- \frac{1}{u-v+1}h_1(v)h_1(u)e_{14}(u) +
  \frac{1}{u-v+1}h_1(v)e_{14}(v)h_1(u) = 0 \,.
\end{multline*}
Let us first evaluate the last summand above. To this end, evoking the defining
relation~\eqref{eq:RTT-termwise} applied to $[t_{11}(u),t_{14}(v)]$, we obtain:
\begin{multline}\label{eq:osp221-e14-h1}
  \frac{1}{u-v+1}e_{14}(v)h_1(u)=
  \frac{u-v+1}{(u-v)^2}h_1(u)e_{14}(v)-\frac{1}{u-v}\left(\frac{1}{u-v}+\frac{1}{u-v+1}\right)h_1(u)e_{14}(u) \, + \\
  \frac{1}{(u-v)(u-v+1)}e_{13}(v)h_1(u)e_{12}(u)+\frac{1}{(u-v)(u-v+1)}e_{12}(v)h_1(u)e_{13}(u) \,.
\end{multline}
Plugging~\eqref{eq:osp221-e14-h1} into the formula above, let us now pull both $h_1(u)$ and $h_1(v)$
to the leftmost part using the following equalities, cf.~\eqref{eq:he12-osp221}:
\begin{equation*}
  e_{13}(v)h_1(u)=h_1(u)\frac{(u-v+1)e_{13}(v)-e_{13}(u)}{u-v} \,,\
  e_{12}(v)h_1(u)=h_1(u)\frac{(u-v+1)e_{12}(v)-e_{12}(u)}{u-v} \,.
\end{equation*}
Multiplying further by $h_1(u)^{-1}h_1(v)^{-1}$ on the left and rearranging terms, we obtain:
\begin{multline}\label{eq:osp221-4-harder}
  \frac{(u-v)^2-1}{(u-v)^2}e_{12}(u)e_{13}(v)+\frac{(u-v+1)^2}{(u-v)^2}e_{13}(v)e_{12}(u)+
  \frac{u-v+1}{(u-v)^2}e_{12}(v)e_{13}(v) \, - \\
  \frac{u-v+1}{(u-v)^2}e_{13}(u)e_{12}(u)+\frac{u-v+1}{(u-v)^2}e_{14}(v)-\frac{u-v+1}{(u-v)^2}e_{14}(u)=0 \,.
\end{multline}
Plugging the formula $e_{14}(u)=-e_{12}(u)e_{13}(u)-[e_{13}(u),e_{12}^{(1)}]$ from
Lemma~\ref{lem:E-entries-evenN}(f) into the last two summands of~\eqref{eq:osp221-4-harder},
and multiplying both sides by $\frac{(u-v)^2}{u-v+1}$, we obtain precisely the relation~\eqref{eq:osp221-4}.

We note that the relation~\eqref{eq:osp221-5} follows directly by applying the anti-automorphism $\tau$
given by~\eqref{eq:tau-t} to the relation~\eqref{eq:osp221-4} and using the formulas~\eqref{eq:tau-efh}.

\medskip
This completes our proof of Proposition~\ref{prop:osp22-10}.
\end{proof}


\subsubsection{Relations for $\fosp(2|2)$ case with the parity sequence $(\bar{0},\bar{1})$}
\label{ssec:osp-22-2}
\

The relations on the generating currents $e_{12}(u),f_{21}(u),h_1(u),h_2(u)$ already follow from those of
Theorem~\ref{thm:Drinfeld-A} for $Y^\rtt(\gl(1|1))$ with the parity sequence $\parity=\Parity=(\bar{0},\bar{1})$,
due to Corollary~\ref{cor:A-type relations}. On the other hand, the relations on the generating currents
$e_{23}(u),f_{32}(u),h_2(u),h_3(u)$ readily follow from those for
  $X^\rtt(\fosp(V^{[1]}))\simeq X^\rtt(\fosp(0|2))\simeq X^\rtt(\fsp_2)$
as specified in Proposition~\ref{prop:osp02}.

\begin{Prop}\label{prop:osp22-01}
The following relations hold in the corresponding $X^\rtt(\fosp(V))$:
\begin{equation}\label{eq:osp222-1}
  [h_1(u),h_3(v)]=0 \,,
\end{equation}
\begin{equation}\label{eq:osp222-23}
  [h_3(u),e_{12}(v)]=\frac{h_3(u)(e_{12}(u-2)-e_{12}(v))}{u-v-2} \,,\qquad
  [h_3(u),f_{21}(v)]=\frac{(f_{21}(v)-f_{21}(u-2))h_3(u)}{u-v-2} \,,
\end{equation}
\begin{equation}\label{eq:osp222-24}
  [h_1(u),e_{23}(v)]=0 \,,\qquad
  [h_1(u),f_{32}(v)]=0 \,,
\end{equation}
\begin{equation}\label{eq:osp222-33}
  [e_{12}(u),f_{32}(v)]=0 \,,\qquad
  [e_{23}(u),f_{21}(v)]=0 \,,
\end{equation}
\begin{equation}\label{eq:osp222-43}
  [e_{12}(u),e_{23}(v)]=\frac{2}{u-v} \Big(e_{13}(u)-e_{13}(v)-e_{12}(u)e_{23}(v)+e_{12}(v)e_{23}(v) \Big) \,,
\end{equation}
\begin{equation}\label{eq:osp222-44}
  [f_{21}(u),f_{32}(v)]=\frac{2}{u-v} \Big(f_{31}(v)-f_{31}(u)+f_{32}(v)f_{21}(u)-f_{32}(v)f_{21}(v) \Big) \,.
\end{equation}
\end{Prop}

\begin{proof}
The relation~\eqref{eq:osp222-1} follows directly from Corollary~\ref{cor:commutativity}.
Alternatively, it follows from the commutativity $[h_1(u),h_1(v)]=[h_1(u),h_2(v)]=[h_2(u),h_2(v)]=0$
and the equality of Lemma~\ref{lem:cseries-evenN-odd}:
\begin{equation}\label{eq:h3-osp222}
  h_3(u)=c_V(u-1)h_1(u-2)h_2(u-2)^{-1}h_1(u-1)^{-1} \,.
\end{equation}

According to~\eqref{eq:Atype-eh} and Corollaries~\ref{cor:A-resonance},~\ref{cor:A-type relations},
we have the following relations:
\begin{equation}\label{eq:he12-osp222}
\begin{split}
  & e_{12}(v)h_1(u)=h_1(u)\left(\frac{u-v-1}{u-v}e_{12}(v)+\frac{1}{u-v}e_{12}(u)\right) \,, \\
  & e_{12}(v)h_2(u)=h_2(u)\left(\frac{u-v-1}{u-v}e_{12}(v)+\frac{1}{u-v}e_{12}(u)\right) \,, \\
  & h_1(u)e_{12}(v)=\left(\frac{u-v}{u-v-1}e_{12}(v)-\frac{1}{u-v-1}e_{12}(u-1)\right)h_1(u) \,, \\
  & h_2(u)e_{12}(v)=\left(\frac{u-v}{u-v-1}e_{12}(v)-\frac{1}{u-v-1}e_{12}(u-1)\right)h_2(u) \,,
\end{split}
\end{equation}
which allow one to pull currents $h_1(u)^{\pm 1}$ and $h_2(u)^{\pm 1}$ past $e_{12}(v)$ either
to the left or to the right. In particular, evoking~\eqref{eq:h3-osp222}, we obtain:
\begin{multline*}
  e_{12}(v)h_3(u)=c_V(u-1)e_{12}(v)h_1(u-2)h_2(u-2)^{-1}h_1(u-1)^{-1}= \\
  c_V(u-1)h_1(u-2)\left(\frac{u-v-3}{u-v-2}e_{12}(v)+\frac{1}{u-v-2}e_{12}(u-2)\right)h_2(u-2)^{-1}h_1(u-1)^{-1}= \\
  c_V(u-1)h_1(u-2)h_2(u-2)^{-1}e_{12}(v)h_1(u-1)^{-1}=h_3(u)\left(\frac{u-v-1}{u-v-2}e_{12}(v)-\frac{1}{u-v-2}e_{12}(u-2)\right) \,,
\end{multline*}
where we pull all the $h_\bullet$-currents to the left of $e_{12}(v)$ using~\eqref{eq:he12-osp222}. Subtracting
$h_3(u)e_{12}(v)$ from both sides of the equality above, we get the first relation of~\eqref{eq:osp222-23}.

We note that the second relation of~\eqref{eq:osp222-23} follows directly by applying the anti-automorphism
$\tau$ of $X^\rtt(\fosp(V))$ given by~\eqref{eq:tau-t} to the first relation of~\eqref{eq:osp222-23} and
using the formulas~\eqref{eq:tau-efh}.

\medskip
The relations~\eqref{eq:osp222-24} follow immediately from Corollary~\ref{cor:commutativity}.
Alternatively, to prove the first relation of~\eqref{eq:osp222-24}, one can rewrite the defining
relation~\eqref{eq:RTT-termwise} for $[t_{11}(u),t_{23}(v)]$ in the form
\begin{multline*}
  h_2(v)[h_1(u),e_{23}(v)]=f_{21}(v)h_1(v)e_{13}(v)h_1(u)-h_1(u)f_{21}(v)h_1(v)e_{13}(v) \, + \\
  \frac{1}{u-v} \Big(f_{21}(u)h_1(u)h_1(v)e_{13}(v)-f_{21}(v)h_1(v)h_1(u)e_{13}(u)\Big) \,,
\end{multline*}
and then pull all the $h_\bullet$-currents in the right-hand side to the right to deduce $[h_1(u),e_{23}(v)]=0$.

\medskip
The shortest proof of~\eqref{eq:osp222-33} is based on Lemma~\ref{lem:ef-com-tl}. To this end, let us
consider the corresponding relation~\eqref{eq:useful-commutator-1} for $\ell=1$ and $k=2,i=3,j=2$:
\begin{equation}\label{eq:e12t32-osp222}
  [e_{12}(u),t^{[1]}_{32}(v)]=\frac{-1}{u-v} t^{[1]}_{32}(v) \Big(e_{12}(v) - e_{12}(u)\Big) \,.
\end{equation}
As $t^{[1]}_{32}(v)=f_{32}(v)h_2(v)$, we have
  $[e_{12}(u),t^{[1]}_{32}(v)]=[e_{12}(u),f_{32}(v)]h_2(v)+f_{32}(v)[e_{12}(u),h_2(v)]$.
Combining this with  $[e_{12}(u),h_2(v)]=\frac{1}{u-v}h_2(v)(e_{12}(u)-e_{12}(v))$, due
to~\eqref{eq:Atype-eh} and Corollary~\ref{cor:A-type relations}, we immediately obtain the commutativity
$[e_{12}(u),f_{32}(v)]=0$. Applying further the anti-automorphism $\tau$ of $X^\rtt(\fosp(V))$ given
by~\eqref{eq:tau-t}, we also obtain $[e_{23}(v),f_{21}(u)]=0$, due to the formulas~\eqref{eq:tau-efh}.

\medskip
Let us finally prove~\eqref{eq:osp222-43}. Applying the defining relation~\eqref{eq:RTT-termwise}
to $[t_{12}(u),t_{23}(v)]$, we get:
\begin{multline}\label{eq:t12t23-osp222}
  [t_{12}(u),t_{23}(v)]=\frac{-1}{u-v} \Big(t_{22}(u)t_{13}(v)-t_{22}(v)t_{13}(u)\Big) \, + \\
  \frac{1}{u-v+1} \Big(t_{24}(v)t_{11}(u)-t_{23}(v)t_{12}(u)+t_{22}(v)t_{13}(u)+t_{21}(v)t_{14}(u)\Big) \,.
\end{multline}
As $[h_1(u),h_2(v)]=0=[h_1(u),e_{23}(v)]$, the left-hand side of~\eqref{eq:t12t23-osp222} can be
expressed as follows:
\begin{multline}\label{eq:lhs-t12t23}
  [t_{12}(u),t_{23}(v)]=h_1(u)[e_{12}(u),h_2(v)e_{23}(v)] \, + \\
  [h_1(u)e_{12}(u),f_{21}(v)h_1(v)]e_{13}(v)-f_{21}(v)h_1(v)[h_1(u)e_{12}(u),e_{13}(v)] \,.
\end{multline}
The second summand in the right-hand side of~\eqref{eq:lhs-t12t23} can be simplified using~\eqref{eq:RTT-termwise}:
\begin{multline}\label{eq:2nd-t12t13-osp222}
  [h_1(u)e_{12}(u),f_{21}(v)h_1(v)]=[t_{12}(u),t_{21}(v)]=\frac{-1}{u-v}\Big(t_{22}(u)t_{11}(v)-t_{22}(v)t_{11}(u)\Big)=\\
  -\frac{h_2(u)h_1(v)+f_{21}(u)h_1(u)e_{12}(u)h_1(v)-h_2(v)h_1(u)-f_{21}(v)h_1(v)e_{12}(v)h_1(u)}{u-v} \,.
\end{multline}
Likewise, the third summand in the right-hand side of~\eqref{eq:lhs-t12t23} can also be simplified
using~\eqref{eq:RTT-termwise}:
\begin{multline}\label{eq:3rd-t12t13-osp222}
  [h_1(u)e_{12}(u),e_{13}(v)]=[t_{12}(u),t_{11}(v)^{-1}t_{13}(v)]= \\
  -t_{11}(v)^{-1}[t_{12}(u),t_{11}(v)]t_{11}(v)^{-1}t_{13}(v)+t_{11}(v)^{-1}[t_{12}(u),t_{13}(v)]= \\
  -\frac{1}{u-v} h_1(v)^{-1} \Big(t_{12}(u)t_{11}(v)-t_{12}(v)t_{11}(u)\Big)e_{13}(v)+
  \frac{1}{u-v} h_1(v)^{-1} \Big(t_{12}(u)t_{13}(v)-t_{12}(v)t_{13}(u)\Big) \, - \\
  \frac{1}{u-v+1}h_1(v)^{-1} \Big(t_{14}(v)t_{11}(u)-t_{13}(v)t_{12}(u)+t_{12}(v)t_{13}(u)+t_{11}(v)t_{14}(u)\Big) \,.
\end{multline}
Expressing all the $t_{\bullet \bullet}$-currents in terms of the Gauss coordinates in the right-hand side
of~\eqref{eq:3rd-t12t13-osp222} and plugging the resulting formula together with~\eqref{eq:2nd-t12t13-osp222}
into~\eqref{eq:lhs-t12t23}, we obtain:
\begin{multline}\label{eq:lhs2-t12t23}
  [t_{12}(u),t_{23}(v)]=
  h_1(u)h_2(v)[e_{12}(u),e_{23}(v)]+\frac{1}{u-v}h_1(u)h_2(v)\big(e_{12}(u)e_{23}(v)-e_{12}(v)e_{23}(v)\big) \, + \\
  \frac{\big(h_2(v)h_1(u)-h_2(u)h_1(v)-f_{21}(u)h_1(u)e_{12}(u)h_1(v)\big)e_{13}(v)+f_{21}(v)h_1(v)e_{12}(v)h_1(u)e_{13}(u)}
       {u-v} \, + \\
  \frac{f_{21}(v)h_1(v) \big( e_{14}(v)h_1(u)-e_{13}(v)h_1(u)e_{12}(u)+e_{12}(v)h_1(u)e_{13}(u)+h_1(u)e_{14}(u)\big)}
       {u-v+1} \,.
\end{multline}

Next, expressing all the $t_{\bullet \bullet}$-currents in the right-hand side of~\eqref{eq:t12t23-osp222} via
the Gauss coordinates, and canceling common terms with those that appear in~\eqref{eq:lhs2-t12t23}, we obtain:
\begin{multline}\label{eq:e12e23-osp222-one}
  h_1(u)h_2(v)[e_{12}(u),e_{23}(v)]=
  \frac{1}{u-v} h_1(u)h_2(v) \Big(e_{13}(u)-e_{13}(v)+e_{12}(v)e_{23}(v)-e_{12}(u)e_{23}(v) \Big) \, + \\
  \frac{1}{u-v+1} h_2(v) \Big( e_{24}(v)h_1(u)-e_{23}(v)h_1(u)e_{12}(u)+h_1(u)e_{13}(u) \Big) \,.
\end{multline}
Multiplying~\eqref{eq:e12e23-osp222-one} by $h_1(u)^{-1}h_2(v)^{-1}$ on the left and evoking $[h_1(u),e_{23}(v)]=0$, we get:
\begin{multline}\label{eq:e12e23-osp222-two}
  [e_{12}(u),e_{23}(v)]=\frac{1}{u-v} \Big(e_{13}(u)-e_{13}(v)+e_{12}(v)e_{23}(v)-e_{12}(u)e_{23}(v) \Big) \, + \\
  \frac{1}{u-v+1} \Big( h_1(u)^{-1}e_{24}(v)h_1(u)-e_{23}(v)e_{12}(u)+e_{13}(u) \Big) \,.
\end{multline}
It thus remains to evaluate the summand $h_1(u)^{-1}e_{24}(v)h_1(u)$ from the right-hand side of~\eqref{eq:e12e23-osp222-two}.
To this end, let us consider the defining relation~\eqref{eq:RTT-termwise} for $[t_{11}(u),t_{24}(v)]$:
\begin{multline}\label{eq:t11t24-osp222}
  [t_{11}(u),t_{24}(v)]=\frac{1}{u-v} \Big( t_{21}(u)t_{14}(v) - t_{21}(v)t_{14}(u)\Big) \, + \\
  \frac{1}{u-v+1} \Big(t_{24}(v)t_{11}(u)-t_{23}(v)t_{12}(u)+t_{22}(v)t_{13}(u)+t_{21}(v)t_{14}(u)\Big) \,.
\end{multline}
The left-hand side of~\eqref{eq:t11t24-osp222} can be expanded as follows:
\begin{multline}\label{eq:lhs-t11t24-osp222}
  [t_{11}(u),t_{24}(v)]=h_2(v)[h_1(u),e_{24}(v)]+[h_1(u),f_{21}(v)]h_1(v)e_{14}(v)+f_{21}(v)[h_1(u),h_1(v)e_{14}(v)] \,.
\end{multline}
Evoking the equality $[h_1(u),f_{21}(v)]=\frac{1}{u-v} \big(f_{21}(u)-f_{21}(v)\big)h_1(u)$, due to~\eqref{eq:Atype-fh}
and Corollary~\ref{cor:A-type relations}, applying further the defining relation~\eqref{eq:RTT-termwise} to
\begin{multline*}
  [h_1(u),h_1(v)e_{14}(v)]=[t_{11}(u),t_{14}(v)]=\frac{1}{u-v} \Big(t_{11}(u)t_{14}(v)-t_{11}(v)t_{14}(u)\Big) \, + \\
  \frac{1}{u-v+1} \Big( t_{14}(v)t_{11}(u)-t_{13}(v)t_{12}(u)+t_{12}(v)t_{13}(u)+t_{11}(v)t_{14}(u) \Big) \,,
\end{multline*}
and rearranging the terms, we obtain:
\begin{multline*}
  [t_{11}(u),t_{24}(v)]=h_2(v)[h_1(u),e_{24}(v)] +
  \frac{f_{21}(u)h_1(u)h_1(v)e_{14}(v)-f_{21}(v)h_1(v)h_1(u)e_{14}(u)}{u-v} \, + \\
  \frac{f_{21}(v)h_1(v)\big(e_{14}(v)h_1(u)-e_{13}(v)h_1(u)e_{12}(u)+e_{12}(v)h_1(u)e_{13}(u)+h_1(u)e_{14}(u)\big)}{u-v+1} .
\end{multline*}
Comparing this with the right-hand side of~\eqref{eq:t11t24-osp222}, where all the $t_{\bullet \bullet}$-currents
are expanded via the Gauss coordinates, and canceling common terms, we get:
\begin{equation}\label{eq:h1e24-osp222}
  [h_1(u),e_{24}(v)]=\frac{1}{u-v+1}\Big(e_{24}(v)h_1(u)-e_{23}(v)h_1(u)e_{12}(u)+h_1(u)e_{13}(u) \Big) \,.
\end{equation}
The equality~\eqref{eq:h1e24-osp222} is equivalent to:
\begin{equation}\label{eq:e24h11-osp222}
  h_1(u)^{-1}e_{24}(v)h_1(u)=\frac{u-v+1}{u-v+2}e_{24}(v)+\frac{1}{u-v+2}e_{23}(v)e_{12}(u)-\frac{1}{u-v+2}e_{13}(u) \,.
\end{equation}
Plugging this formula back into~\eqref{eq:e12e23-osp222-two}, we obtain:
\begin{multline}\label{eq:e12e23-osp222-three}
  [e_{12}(u),e_{23}(v)]=\frac{1}{u-v} \Big(e_{13}(u)-e_{13}(v)-e_{12}(u)e_{23}(v)+e_{12}(v)e_{23}(v) \Big) \, + \\
  \frac{1}{u-v+2} \Big( e_{24}(v)-e_{23}(v)e_{12}(u)+e_{13}(u) \Big) \,.
\end{multline}
Multiplying both sides of~\eqref{eq:e12e23-osp222-three} by $\frac{u-v+2}{u-v+1}$ and rearranging terms, we get:
\begin{multline}\label{eq:e12e23-osp222-four}
  [e_{12}(u),e_{23}(v)]=\frac{2}{u-v}\Big(e_{13}(u)-e_{13}(v)-e_{12}(u)e_{23}(v)+e_{12}(v)e_{23}(v) \Big) \, + \\
  \frac{1}{u-v+1} \Big(e_{13}(v)-e_{12}(v)e_{23}(v)+e_{24}(v)\Big) \,.
\end{multline}
Multiplying both sides of~\eqref{eq:e12e23-osp222-four} by $u-v+1$ and setting $u=v-1$ afterwards, we find
\begin{equation}\label{eq:vanishing-term2}
  e_{13}(v)-e_{12}(v)e_{23}(v)+e_{24}(v)=0 \,.
\end{equation}
Thus, plugging~\eqref{eq:vanishing-term2} into the equality~\eqref{eq:e12e23-osp222-four}, we obtain
precisely the desired relation~\eqref{eq:osp222-43}.

We note that the relation~\eqref{eq:osp222-44} follows directly by applying the anti-automorphism
$\tau$ of $X^\rtt(\fosp(V))$ given by~\eqref{eq:tau-t} to~\eqref{eq:osp222-43} and using
the formulas~\eqref{eq:tau-efh}.

\medskip
This completes our proof of Proposition~\ref{prop:osp22-01}.
\end{proof}


\subsubsection{Relations for $\fosp(5|0)$ case}
\label{ssec:osp-50}
\

In this case, we have $X^\rtt(\fosp(V))\simeq X^\rtt(\fso_5)$ by Remark~\ref{rem:super-to-nonsuper}.
The relations on the generating currents $e_{12}(u),f_{21}(u),h_1(u),h_2(u)$ already follow from those
for $Y^\rtt(\gl_2)$ from Theorem~\ref{thm:Drinfeld-A}, due to Corollary~\ref{cor:A-type relations}.
On the other hand, the relations on the currents $e_{23}(u),f_{32}(u),h_2(u),h_3(u)$ follow from
those for $X^\rtt(\fosp(V^{[1]}))\simeq X^\rtt(\fosp(3|0))\simeq X^\rtt(\fso_3)$ as specified in
Proposition~\ref{prop:osp30}.

\begin{Prop}\label{prop:osp50}
The following relations hold in $X^\rtt(\fosp(5|0))$:
\begin{equation}\label{eq:osp50-1}
  [h_1(u),h_3(v)]=0 \,,
\end{equation}
\begin{equation}\label{eq:osp50-21}
  [h_3(u),e_{12}(v)]=0 \,,\qquad
  [h_3(u),f_{21}(v)]=0 \,,
\end{equation}
\begin{equation}\label{eq:osp50-22}
  [h_1(u),e_{23}(v)]=0 \,,\qquad  [h_1(u),f_{32}(v)]=0 \,,
\end{equation}
\begin{equation}\label{eq:osp50-3}
  [e_{12}(u),f_{32}(v)]=0 \,,\qquad [e_{23}(u),f_{21}(v)]=0 \,,
\end{equation}
\begin{equation}\label{eq:osp50-41}
  [e_{12}(u),e_{23}(v)]=\frac{1}{u-v} \Big(e_{13}(v)-e_{13}(u)+e_{12}(u)e_{23}(v)-e_{12}(v)e_{23}(v) \Big) \,,
\end{equation}
\begin{equation}\label{eq:osp50-42}
  [f_{21}(u),f_{32}(v)]=\frac{1}{u-v} \Big(f_{31}(u)-f_{31}(v)-f_{32}(v)f_{21}(u)+f_{32}(v)f_{21}(v) \Big) \,.
\end{equation}
\end{Prop}

This result goes back to~\cite{jlm}. We note however that~\eqref{eq:osp50-42} corrects
a typo in~\cite[(5.29)]{jlm}.


\subsubsection{Relations for $\fosp(3|2)$ case with the parity sequence $(\bar{1},\bar{0})$}
\label{ssec:osp-32-1}
\

In this case, the relations on the generating currents $e_{12}(u),f_{21}(u),h_1(u),h_2(u)$ already
follow from those of Theorem~\ref{thm:Drinfeld-A} for $Y^\rtt(\gl(1|1))$ with the parity sequence
$\parity=\Parity=(\bar{1},\bar{0})$, due to Corollary~\ref{cor:A-type relations}. On the other hand,
the relations on the currents $e_{23}(u),f_{32}(u),h_2(u),h_3(u)$ readily follow from those for
$X^\rtt(\fosp(V^{[1]}))\simeq X^\rtt(\fosp(3|0))\simeq X^\rtt(\fso_3)$ as specified in
Proposition~\ref{prop:osp30}.

\begin{Prop}\label{prop:osp321}
The relations~\eqref{eq:osp50-1}--\eqref{eq:osp50-42} hold in $X^\rtt(\fosp(V))$.
\end{Prop}

\begin{proof}
The relations~(\ref{eq:osp50-1})--(\ref{eq:osp50-22}) follow directly from Corollary~\ref{cor:commutativity}.
The relations~\eqref{eq:osp50-3} can be proved alike~\eqref{eq:osp222-33} by using Lemma~\ref{lem:ef-com-tl}.
To do so, we consider the corresponding relation
\begin{equation}\label{eq:e12t32-osp321}
  [e_{12}(u),t^{[1]}_{32}(v)]=\frac{1}{u-v} t^{[1]}_{32}(v) \Big(e_{12}(v) - e_{12}(u)\Big) \,.
\end{equation}
As $t^{[1]}_{32}(v)=f_{32}(v)h_2(v)$, we have
  $[e_{12}(u),t^{[1]}_{32}(v)]=[e_{12}(u),f_{32}(v)]h_2(v)+f_{32}(v)[e_{12}(u),h_2(v)]$.
Combining this with $[e_{12}(u),h_2(v)]=\frac{1}{u-v}h_2(v)(e_{12}(v)-e_{12}(u))$, due to~\eqref{eq:Atype-eh}
and Corollary~\ref{cor:A-type relations}, we immediately obtain the commutativity $[e_{12}(u),f_{32}(v)]=0$.
Applying further the anti-automorphism $\tau$ of $X^\rtt(\fosp(V))$ given by~\eqref{eq:tau-t}, we also obtain
$[e_{23}(v),f_{21}(u)]=0$, due to the formulas~\eqref{eq:tau-efh}.

\medskip
The relations~(\ref{eq:osp50-41},~\ref{eq:osp50-42}) can be established similarly to~\eqref{eq:osp50-3}.
To this end, let us consider the corresponding relation~\eqref{eq:useful-commutator-1} for $\ell=1$ and
$k=2,i=2,j=3$:
\begin{equation}\label{eq:e12t23-osp321}
  [e_{12}(u),t^{[1]}_{23}(v)]=\frac{1}{u-v} t^{[1]}_{22}(v) \Big(e_{13}(v) - e_{13}(u)\Big) \,.
\end{equation}
As $t^{[1]}_{23}(v)=h_2(v)e_{23}(v)$, we have
  $[e_{12}(u),t^{[1]}_{23}(v)]=[e_{12}(u),h_2(v)]e_{23}(v)+h_2(v)[e_{12}(u),e_{23}(v)]$.
Combining this with $t^{[1]}_{22}(v)=h_2(v)$ and $[e_{12}(u),h_2(v)]=\frac{1}{u-v}h_2(v)(e_{12}(v)-e_{12}(u))$
from above, we obtain the desired relation~\eqref{eq:osp50-41}. Applying the anti-automorphism $\tau$ of
$X^\rtt(\fosp(V))$ given by~\eqref{eq:tau-t} to~\eqref{eq:osp50-41}, we also obtain~\eqref{eq:osp50-42},
due to the formulas~\eqref{eq:tau-efh}.

\medskip
This completes our proof of Proposition~\ref{prop:osp321}.
\end{proof}


\subsubsection{Relations for $\fosp(1|4)$ and for $\fosp(3|2)$ with the parity sequence $(\bar{0},\bar{1})$}
\label{ssec:osp-32-2}
\

In these cases, the relations on the generating currents $e_{12}(u),f_{21}(u),h_1(u),h_2(u)$ already follow
from those of Theorem~\ref{thm:Drinfeld-A} for $Y^\rtt(\gl(\VV))$ with the parity sequence $\parity=\Parity$
being $(\bar{0},\bar{1})$ or $(\bar{1},\bar{1})$, due to Corollary~\ref{cor:A-type relations}. On the
other hand, the relations on the currents $e_{23}(u),f_{32}(u),h_2(u),h_3(u)$ readily follow from those
for $X^\rtt(\fosp(V^{[1]}))\simeq X^\rtt(\fosp(1|2))$ as specified in Proposition~\ref{prop:osp12}.

\begin{Prop}\label{prop:osp322}
The following relations hold in $X^\rtt(\fosp(V))$:
\begin{equation}\label{eq:osp322-1}
  [h_1(u),h_3(v)]=0 \,,
\end{equation}
\begin{equation}\label{eq:osp322-21}
  [h_3(u),e_{12}(v)]=0 \,,\qquad
  [h_3(u),f_{21}(v)]=0 \,,
\end{equation}
\begin{equation}\label{eq:osp322-22}
  [h_1(u),e_{23}(v)]=0 \,,\qquad
  [h_1(u),f_{32}(v)]=0 \,,
\end{equation}
\begin{equation}\label{eq:osp322-3}
  [e_{12}(u),f_{32}(v)]=0 \,,\qquad [e_{23}(u),f_{21}(v)]=0 \,,
\end{equation}
\begin{equation}\label{eq:osp322-41}
  [e_{12}(u),e_{23}(v)]=\frac{1}{u-v} \Big(e_{13}(u)-e_{13}(v)-e_{12}(u)e_{23}(v)+e_{12}(v)e_{23}(v) \Big) \,,
\end{equation}
\begin{equation}\label{eq:osp322-42}
  [f_{21}(u),f_{32}(v)]=\frac{1}{u-v} \Big(f_{31}(v)-f_{31}(u)+f_{32}(v)f_{21}(u)-f_{32}(v)f_{21}(v) \Big) \,.
\end{equation}
Additionally, we also have the following relations:
\begin{equation}\label{eq:osp322-weird1}
  [e^{(1)}_{12},e_{24}(v)]=-e_{14}(v)-e_{14}(v-\sfrac{3}{2})+
  e_{12}(v)e_{24}(v)+e_{24}(v)e_{12}(v-\sfrac{3}{2})-(-1)^{\ol{1}}e_{23}(v)e_{13}(v-\sfrac{3}{2}) \,,
\end{equation}
\begin{equation}\label{eq:osp322-weird2}
  [f^{(1)}_{21},f_{42}(v)]=f_{41}(v)+f_{41}(v-\sfrac{3}{2})-
  f_{42}(v)f_{21}(v)-f_{21}(v-\sfrac{3}{2})f_{42}(v)-f_{31}(v-\sfrac{3}{2})f_{32}(v) \,.
\end{equation}
\end{Prop}

\begin{proof}
The proof of~\eqref{eq:osp322-1}--\eqref{eq:osp322-42} is completely analogous to that
of Proposition~\ref{prop:osp321}; we leave details to the interested reader.

\medskip
Let us now prove~(\ref{eq:osp322-weird1},~\ref{eq:osp322-weird2}).
To this end, we start with the equality from Lemma~\ref{lem:E-entries-oddN}(e):
\begin{equation*}
  e_{25}(v)=(-1)^{\ol{1}}\left(e_{14}(v)-e_{12}(v)e_{24}(v)-[e_{24}(v),e^{(1)}_{12}]\right) \,.
\end{equation*}
We can rewrite it in the form:
\begin{equation}\label{eq:wd-1}
  [e_{12}^{(1)},e_{24}(v)] = -e_{14}(v)+e_{12}(v)e_{24}(v)+(-1)^{\ol{1}}e_{25}(v) \,.
\end{equation}
Thus, it remains to re-express $e_{25}(v)$. To do so, we recall the equality $T^t(v+\kappa)=T(v)^{-1}c_V(v+\kappa)$
of~\eqref{eq:T-reln-used}. In particular, comparing the $(4,5)$ matrix coefficients, we obtained
Lemma~\ref{lem:E-entries-oddN}(b):
\begin{equation*}
  e_{12}(v-\sfrac{3}{2})=(-1)^{\ol{1}}e_{45}(v) \,,
\end{equation*}
cf.~our proof of Lemma~\ref{lem:E-entries-evenN}(d). Here, we used the equality
$h_1(v+\kappa)=h_{5}(v)^{-1}c_V(v+\kappa)$ of~\eqref{eq:mat.coef.4}, the equality~\eqref{eq:h-to-e-1},
and finally the identity
\begin{equation*}
  \kappa-(-1)^{\ol{1}}=-\sfrac{3}{2} \,.
\end{equation*}
Likewise, comparing the $(3,5)$ matrix coefficients, we obtain:
\begin{equation*}
  e_{13}(v-\sfrac{3}{2})=(E(v)^{-1})_{35}=e_{34}(v)e_{45}(v)-e_{35}(v) \,,
\end{equation*}
cf.~our proof of Lemma~\ref{lem:E-entries-evenN}(i).
Finally, comparing the $(2,5)$ matrix coefficients, we obtain:
\begin{equation*}
  (-1)^{\ol{1}}e_{14}(v-\sfrac{3}{2})=(E(v)^{-1})_{25}=
  -\big(e_{25}(v)-e_{24}(v)e_{45}(v)-e_{23}(v)e_{35}(v)+e_{23}(v)e_{34}(v)e_{45}(v)\big) \,.
\end{equation*}
Combining the above formulas for $e_{14}(v-\sfrac{3}{2})$, $e_{13}(v-\sfrac{3}{2})$,
and $e_{12}(v-\sfrac{3}{2})$, we obtain:
\begin{equation}\label{eq:wd-2}
  (-1)^{\ol{1}}e_{25}(v)=
  -e_{14}(v-\sfrac{3}{2})+e_{24}(v)e_{12}(v-\sfrac{3}{2})-(-1)^{\ol{1}}e_{23}(v)e_{13}(v-\sfrac{3}{2}) \,.
\end{equation}
Plugging the right-hand side of~\eqref{eq:wd-2} instead of $(-1)^{\ol{1}}e_{25}(v)$ in~\eqref{eq:wd-1},
we obtain precisely~\eqref{eq:osp322-weird1}.

Applying the anti-automorphism $\tau$ of $X^\rtt(\fosp(V))$ given by~\eqref{eq:tau-t} to~\eqref{eq:osp322-weird1},
we also obtain~\eqref{eq:osp322-weird2}, due to the formulas~\eqref{eq:tau-efh}.
%
\end{proof}

%
%
%
%
%
%


\section{Drinfeld orthosymplectic Yangians}\label{sec:Drinfeld orthosymplectic}

In this section, we introduce the Drinfeld (extended) orthosymplectic Yangians of $\fosp(V)$
and identify them with their RTT counterparts from Section~\ref{sec:RTT orthosymplectic}.


\subsection{Drinfeld extended orthosymplectic super Yangian}
\label{ssec:Dr-ext-Yangian}
\

We fix $N,m$, and $V$ as in Subsection~\ref{ssec:setup}. Let $n=\lfloor N/2\rfloor$, so that
$N=2n$ or $N=2n+1$, and recall the notation $\ol{i}$ of~\eqref{eq:index-parity}. We define the
\emph{Drinfeld extended Yangian of $\fosp(V)$}, denoted by $X(\fosp(V))$, to be the associative
$\BC$-superalgebra generated by
  $\{e_{i,r},f_{i,r}\}_{1\leq i\leq n+m}^{r\geq 1}\cup \{h_{\imath,r}\}_{1\leq \imath\leq  n+m+1}^{r\geq 1}$
with the $\BZ_2$-grading given by
\begin{equation*}
\begin{split}
  & |e_{i,r}|=|f_{i,r}|=\ol{i}+\ol{i+1} \,, \quad |h_{\imath,r}|=\bar{0}
    \qquad \forall\ i<n+m \,,\, \imath\leq n+m+1 \,,\, r\geq 1 \,, \\
  & |e_{n+m,r}|=|f_{n+m,r}|=
  \begin{cases}
    \ol{n+m-1}+\ol{n+m} & \text{if} \quad N=2n \,,\, \ol{n+m}=\bar{0} \\
    \ol{n+m}+\ol{n+m+1} & \text{otherwise}
  \end{cases} \,,
\end{split}
\end{equation*}
and subject to the defining relations~\eqref{eq:Dr-osp-rel-ext1}--\eqref{eq:Dr-osp-rel-ext13.2}.
To state the relations, form the generating~series:
\begin{equation*}
  e_i(u)=\sum_{r\geq 1} e_{i,r}u^{-r} \,,\quad
  f_i(u)=\sum_{r\geq 1} f_{i,r}u^{-r} \,,\quad
  h_\imath(u)=1+\sum_{r\geq 1} h_{\imath,r}u^{-r}
\end{equation*}
for all $1\leq i\leq n+m$ and $1\leq \imath\leq n+m+1$, as well as
\begin{equation*}
\begin{split}
  & e^\circ_i(u)=\sum_{r\geq 2} e_{i,r}u^{-r} \,,\quad
    f^\circ_i(u)=\sum_{r\geq 2} f_{i,r}u^{-r} \,, \\
  & k_i(u)=
    \begin{cases}
      h_{n+m-1}(u)^{-1}h_{n+m+1}(u) & \text{if} \ \ N=2n, \ol{n+m}=\bar{0}, i=n+m \\
      h_{i}(u)^{-1}h_{i+1}(u) & \text{otherwise,\ with} \ 1\leq i\leq n+m
    \end{cases} \,.
\end{split}
\end{equation*}
Recall the basis $e^*_i$ of $\fh^*$ (dual to the basis $\{F_{ii}\}_{i=1}^{n+m}$ of the Cartan
subalgebra $\fh$ of $\fosp(V)$) from Subsection~\ref{ssec:osp-definition}, the bilinear form
$(\cdot,\cdot)$ on $\fh^*$ determined by~\eqref{eq:form-dual-Cartan}, the specific simple roots
$\{\alpha_1,\ldots,\alpha_{n+m}\}$ as specified in Subsection~\ref{ssec:Dynkin diagrams},
and the resulting Cartan matrix $A=(a_{ij})$ of~\eqref{eq:nonsymm-Cartan-matrix}.

\medskip
\noindent
\underline{Commutator of $h_i(u)$ and $h_j(v)$}
\begin{equation}\label{eq:Dr-osp-rel-ext1}
  [h_i(u),h_j(v)]=0 \qquad \forall\ 1\leq i,j\leq n+m+1 \,.
\end{equation}

\medskip
\noindent
\underline{Commutator of $e_i(u)$ and $f_j(v)$}
\begin{equation}\label{eq:Dr-osp-rel-ext2}
  [e_i(u),f_j(v)]=\delta_{ij}(-1)^{\ol{i+1}}2^{\varrho}\, \frac{k_i(u)-k_i(v)}{u-v}
  \qquad \forall\ 1\leq i,j\leq n+m \,,
\end{equation}
where
  $\varrho=
   \begin{cases}
     1 & \text{if} \quad N=2n \,,\, \ol{n+m}=\bar{1} \,,\, i=n+m \\
     0 & \text{otherwise}
   \end{cases}$.

\medskip
\noindent
\underline{Commutator of $h_i(u)$ and $e_j(v)$}
\begin{equation}\label{eq:Dr-osp-rel-ext3.1}
  [h_i(u),e_j(v)]=-(e^*_i,\alpha_j)h_i(u)\, \frac{e_j(u)-e_j(v)}{u-v}
    \qquad \forall\ 1\leq i,j\leq n+m \,,
\end{equation}
\begin{equation}\label{eq:Dr-osp-rel-ext3.2}
  [h_{n+m+1}(u),e_j(v)]=0 \qquad \forall\ 1\leq j<n+m-1 \,,
\end{equation}
\begin{equation}\label{eq:Dr-osp-rel-ext3.3}
  [h_{n+m+1}(u),e_{n+m-1}(v)]=
  \begin{cases}
    -h_{n+m+1}(u)\, \frac{e_{n+m-1}(u)-e_{n+m-1}(v)}{u-v} & \text{if} \quad N=2n \,,\, \ol{n+m}=\bar{0} \\
    h_{n+m+1}(u)\, \frac{e_{n+m-1}(u-2)-e_{n+m-1}(v)}{u-v-2} & \text{if} \quad N=2n \,,\, \ol{n+m}=\bar{1} \\
    0 & \text{if} \quad N=2n+1
  \end{cases} \,,
\end{equation}
\begin{equation}\label{eq:Dr-osp-rel-ext3.4}
  [h_{n+m+1}(u),e_{n+m}(v)]=
  \begin{cases}
    h_{n+m+1}(u)\, \frac{e_{n+m}(u)-e_{n+m}(v)}{u-v}
      \qquad \qquad \   \text{if} \quad N=2n \,,\, \ol{n+m}=\bar{0} \\
    -2h_{n+m+1}(u)\, \frac{e_{n+m}(u)-e_{n+m}(v)}{u-v}
      \qquad \quad  \text{if} \quad N=2n \,,\, \ol{n+m}=\bar{1} \\
    h_{n+m+1}(u)\, \frac{e_{n+m}(u)-e_{n+m}(v)}{2(u-v)} -
      \frac{e_{n+m}(u-1)-e_{n+m}(v)}{2(u-v-1)}\, h_{n+m+1}(u) \\
      \qquad \qquad \qquad \qquad \qquad \qquad \qquad \qquad
      \text{if} \quad N=2n+1 \,,\, \ol{n+m}=\bar{0} \\
    h_{n+m+1}(u) \left(\frac{e_{n+m}(u)-e_{n+m}(v)}{u-v} -
      \frac{e_{n+m}(u-1/2)-e_{n+m}(v)}{u-v-1/2}\right) \\
      \qquad \qquad \qquad \qquad \qquad \qquad \qquad \qquad
      \text{if} \quad N=2n+1 \,,\, \ol{n+m}=\bar{1}
  \end{cases} \,.
\end{equation}

\medskip
\noindent
\underline{Commutator of $h_i(u)$ and $f_j(v)$}
\begin{equation}\label{eq:Dr-osp-rel-ext4.1}
  [h_i(u),f_j(v)]=(e^*_i,\alpha_j)\frac{f_j(u)-f_j(v)}{u-v}\, h_i(u)
    \qquad \forall\ 1\leq i,j\leq n+m \,,
\end{equation}
\begin{equation}\label{eq:Dr-osp-rel-ext4.2}
  [h_{n+m+1}(u),f_j(v)]=0 \qquad \forall\ 1\leq j<n+m-1 \,,
\end{equation}
\begin{equation}\label{eq:Dr-osp-rel-ext4.3}
  [h_{n+m+1}(u),f_{n+m-1}(v)]=
  \begin{cases}
    \frac{f_{n+m-1}(u)-f_{n+m-1}(v)}{u-v}\, h_{n+m+1}(u) & \text{if} \quad N=2n \,,\, \ol{n+m}=\bar{0} \\
    -\frac{f_{n+m-1}(u-2)-f_{n+m-1}(v)}{u-v-2}\, h_{n+m+1}(u) & \text{if} \quad N=2n \,,\, \ol{n+m}=\bar{1} \\
    0 & \text{if} \quad N=2n+1
  \end{cases} \,,
\end{equation}
\begin{equation}\label{eq:Dr-osp-rel-ext4.4}
  [h_{n+m+1}(u),f_{n+m}(v)]=
  \begin{cases}
    -\frac{f_{n+m}(u)-f_{n+m}(v)}{u-v}\, h_{n+m+1}(u)
      \qquad \qquad \    \text{if} \quad N=2n \,,\, \ol{n+m}=\bar{0} \\
    2\, \frac{f_{n+m}(u)-f_{n+m}(v)}{u-v}\, h_{n+m+1}(u)
      \qquad \quad \quad \ \ \text{if} \quad N=2n \,,\, \ol{n+m}=\bar{1} \\
    -\frac{f_{n+m}(u)-f_{n+m}(v)}{2(u-v)}\, h_{n+m+1}(u) +
      h_{n+m+1}(u)\, \frac{f_{n+m}(u-1)-f_{n+m}(v)}{2(u-v-1)} \\
      \qquad \qquad \qquad \qquad \qquad \qquad \qquad \qquad  \quad
      \text{if} \quad N=2n+1 \,,\, \ol{n+m}=\bar{0} \\
    \left(-\frac{f_{n+m}(u)-f_{n+m}(v)}{u-v} + \frac{f_{n+m}(u-1/2)-f_{n+m}(v)}{u-v-1/2}\right) h_{n+m+1}(u) \\
      \qquad \qquad \qquad \qquad \qquad \qquad \qquad \qquad  \quad
      \text{if} \quad N=2n+1 \,,\, \ol{n+m}=\bar{1}
  \end{cases} \,.
\end{equation}

\medskip
\noindent
\underline{Commutator of $e_i(u)$ and $e_i(v)$}
\

\noindent
Unless $N=2n+1$, $\ol{n+m}=\bar{1}$, and $i=n+m$, we impose:
\begin{equation}\label{eq:Dr-osp-rel-ext5.1}
   [e_i(u),e_i(v)]=\frac{(\alpha_i,\alpha_i)}{2} \, \frac{(e_i(u)-e_i(v))^2}{u-v} \,.
\end{equation}
For the remaining case $N=2n+1$, $\ol{n+m}=\bar{1}$, and $i=n+m$,
following~(\ref{eq:osp12-6},~\ref{eq:osp12-8}), we impose:
\begin{multline}\label{eq:Dr-osp-rel-ext5.2}
   [e_{n+m}(u),e_{n+m}(v)]=
   \frac{e'_{n+m}(u)-e'_{n+m}(v)}{u-v} + \frac{e_{n+m}(u)^2-e_{n+m}(v)^2}{u-v} \, + \\
   \frac{e_{n+m}(u)e_{n+m}(v)-e_{n+m}(v)e_{n+m}(u)}{2(u-v)} - \frac{(e_{n+m}(u)-e_{n+m}(v))^2}{2(u-v)^2} \,,
\end{multline}
where we define $e'_{n+m}(u)=-e_{n+m}(u)^2-[e_{n+m}(u),e_{n+m,1}]$.

\medskip
\noindent
\underline{Commutator of $f_i(u)$ and $f_i(v)$}
\

\noindent
Unless $N=2n+1$, $\ol{n+m}=\bar{1}$, and $i=n+m$, we impose:
\begin{equation}\label{eq:Dr-osp-rel-ext6.1}
   [f_i(u),f_i(v)]=-\frac{(\alpha_i,\alpha_i)}{2} \, \frac{(f_i(u)-f_i(v))^2}{u-v} \,.
\end{equation}
For the remaining case $N=2n+1$, $\ol{n+m}=\bar{1}$, and $i=n+m$,
following~(\ref{eq:osp12-7},~\ref{eq:osp12-8}), we impose:
\begin{multline}\label{eq:Dr-osp-rel-ext6.2}
  [f_{n+m}(u),f_{n+m}(v)]=
  \frac{f'_{n+m}(v)-f'_{n+m}(u)}{u-v} + \frac{f_{n+m}(u)^2-f_{n+m}(v)^2}{u-v} \, + \\
  \frac{f_{n+m}(v)f_{n+m}(u)-f_{n+m}(u)f_{n+m}(v)}{2(u-v)} - \frac{(f_{n+m}(v)-f_{n+m}(u))^2}{2(u-v)^2} \,,
\end{multline}
where we define $f'_{n+m}(u)=f_{n+m}(u)^2+[f_{n+m}(u),f_{n+m,1}]$.

\medskip
\noindent
\underline{Commutator of $e_i(u)$ and $e_j(v)$ for $i<j$}
\

\noindent
Unless $N=2n, \ol{n+m}=\bar{0}, \ol{n+m-1}=\bar{1}$, and $j=i+1=n+m$, we impose:
\begin{equation}\label{eq:Dr-osp-rel-ext7.1}
   u[e^\circ_i(u),e_j(v)]-v[e_i(u),e^\circ_j(v)] = -(\alpha_i,\alpha_j) e_i(u)e_j(v) \,.
\end{equation}
For $N=2n, \ol{n+m}=\bar{0}, \ol{n+m-1}=\bar{1}$, and $j=i+1=n+m$, following~\eqref{eq:osp221-4simplify} we impose:
\begin{equation}\label{eq:Dr-osp-rel-ext7.2}
   u[e^\circ_{n+m-1}(u),e_{n+m}(v)]-v[e_{n+m-1}(u),e^\circ_{n+m}(v)] = e_{n+m-1}(u)e_{n+m}(v)-e_{n+m}(v)e_{n+m-1}(u) \,.
\end{equation}

\medskip
\noindent
\underline{Commutator of $f_i(u)$ and $f_j(v)$ for $i<j$}
\

\noindent
Unless $N=2n, \ol{n+m}=\bar{0}, \ol{n+m-1}=\bar{1}$, and $j=i+1=n+m$, we impose:
\begin{equation}\label{eq:Dr-osp-rel-ext8.1}
   u[f^\circ_i(u),f_j(v)]-v[f_i(u),f^\circ_j(v)] = (\alpha_i,\alpha_j) f_j(v)f_i(u) \,.
\end{equation}
For $N=2n, \ol{n+m}=\bar{0}, \ol{n+m-1}=\bar{1}$, and $j=i+1=n+m$, following~\eqref{eq:osp221-5simplify} we impose:
\begin{equation}\label{eq:Dr-osp-rel-ext8.2}
   u[f^\circ_{n+m-1}(u),f_{n+m}(v)]-v[f_{n+m-1}(u),f^\circ_{n+m}(v)] = -f_{n+m-1}(u)f_{n+m}(v)+f_{n+m}(v)f_{n+m-1}(u) \,.
\end{equation}

\medskip
\noindent
\underline{``Additional'' relations for $N=2n+1$ and $\ol{n+m}=\bar{1}$}
\

For $N=2n+1$ and $\ol{n+m}=\bar{1}$, following~(\ref{eq:osp322-weird1},~\ref{eq:osp322-weird2}), we impose:
\begin{multline}\label{eq:Dr-osp-rel-ext-weird1}
  [e_{n+m-1,1},e'_{n+m}(v)]=-e'''_{n+m}(v)-e'''_{n+m}(v-\sfrac{3}{2}) \, + \\
  e_{n+m-1}(v)e'_{n+m}(v)+e'_{n+m}(v)e_{n+m-1}(v-\sfrac{3}{2})-(-1)^{\ol{n+m-1}}e_{n+m}(v)e''_{n+m}(v-\sfrac{3}{2}) \,,
\end{multline}
\begin{multline}\label{eq:Dr-osp-rel-ext-weird2}
  [f_{n+m-1,1},f'_{n+m}(v)]=f'''_{n+m}(v)+f'''_{n+m}(v-\sfrac{3}{2}) \, - \\
  f'_{n+m}(v)f_{n+m-1}(v)-f_{n+m-1}(v-\sfrac{3}{2})f'_{n+m}(v)-f''_{n+m}(v-\sfrac{3}{2})f_{n+m}(v) \,,
\end{multline}
where $e'_{n+m}(u),f'_{n+m}(u)$ are as above, and following
Lemmas~\ref{lem:E-entries-oddN}(a,b),~\ref{lem:F-entries-oddN}(a,b) we also define:
\begin{equation*}
\begin{split}
  & e''_{n+m}(v)=-[e_{n+m-1}(v),e_{n+m,1}] \,, \qquad
    e'''_{n+m}(v)=\big[[e_{n+m-1}(v),e_{n+m,1}],e_{n+m,1}\big] \,, \\
  & f''_{n+m}(v)=-[f_{n+m,1},f_{n+m-1}(v)]  \,, \qquad
    f'''_{n+m}(v)=-\big[f_{n+m,1},[f_{n+m,1},f_{n+m-1}(v)]\big] \,.
\end{split}
\end{equation*}

\medskip
\noindent
\underline{Standard Serre relations}
\

\noindent
For $1\leq i\ne j\leq n+m$ such that $a_{ii}\ne 0$ or $a_{ij}=0$, we impose:
\begin{equation}\label{eq:Dr-osp-rel-ext9.1}
  (\mathrm{ad}_{e_{i,1}})^{1-a_{ij}}(e_{j,1}) = 0 \,,
\end{equation}
\begin{equation}\label{eq:Dr-osp-rel-ext9.2}
  (\mathrm{ad}_{f_{i,1}})^{1-a_{ij}}(f_{j,1})=0 \,.
\end{equation}
For $1\leq i\leq n+m$ such that $a_{ii}=0$, we impose:
\begin{equation}\label{eq:Dr-osp-rel-ext9.3}
  [e_{i,1},e_{i,1}] = 0 \,,
\end{equation}
\begin{equation}\label{eq:Dr-osp-rel-ext9.4}
  [f_{i,1},f_{i,1}]=0 \,.
\end{equation}

\medskip
\noindent
\underline{Higher order Serre relations of degree $4$}
\

\noindent
For any of the sub-diagrams~\eqref{eq:pic-Serre-new4a}--\eqref{eq:pic-Serre-new4b}, we impose:
\begin{equation}\label{eq:Dr-osp-rel-ext10.1}
  \big[[e_{j,1},e_{t,1}],[e_{t,1},e_{k,1}]\big] = 0 \,,
\end{equation}
\begin{equation}\label{eq:Dr-osp-rel-ext10.2}
  \big[[f_{j,1},f_{t,1}],[f_{t,1},f_{k,1}]\big] = 0 \,,
\end{equation}
cf.~\eqref{eq:Lie-Serre-superA}.

\smallskip
\noindent
\underline{Higher order Serre relations of degree $3$}
\

\noindent
For the sub-diagram~\eqref{eq:pic-Serre-new3} (corresponding to $N=2n, n+m\geq 3$,
and $\Parity$ ending $\bar{1}\bar{0}$), we impose:
\begin{equation}\label{eq:Dr-osp-rel-ext11.1}
  \big[e_{t,1},[e_{s,1},e_{i,1}]\big]-\big[e_{s,1},[e_{t,1},e_{i,1}]\big]=0 \,,
\end{equation}
\begin{equation}\label{eq:Dr-osp-rel-ext11.2}
  \big[f_{t,1},[f_{s,1},f_{i,1}]\big]-\big[f_{s,1},[f_{t,1},f_{i,1}]\big]=0 \,,
\end{equation}
cf.~\eqref{eq:Lie-Serre-new3}.

\smallskip
\noindent
\underline{Higher order Serre relations of degree $6$}
\

\noindent
For the sub-diagram~\eqref{eq:pic-Serre-new6} (corresponding to $N=2n$, $n+m\geq 3$,
and $\Parity$ ending $\bar{1} \bar{0} \bar{1}$), we impose:
\begin{equation}\label{eq:Dr-osp-rel-ext12.1}
  \Big[[e_{j,1},e_{t,1}],\big[[e_{j,1},e_{t,1}],[e_{t,1},e_{k,1}]\big]\Big]=0 \,,
\end{equation}
\begin{equation}\label{eq:Dr-osp-rel-ext12.2}
  \Big[[f_{j,1},f_{t,1}],\big[[f_{j,1},f_{t,1}],[f_{t,1},f_{k,1}]\big]\Big]=0 \,,
\end{equation}
cf.~\eqref{eq:Lie-Serre-new6}.

\smallskip
\noindent
\underline{Higher order Serre relations of degree $7$}
\

\noindent
For the sub-diagram~\eqref{eq:pic-Serre-new7} (corresponding to $N=2n$, $n+m\geq 4$,
and $\Parity$ ending $\bar{0} \bar{0} \bar{1}$), we impose:
\begin{equation}\label{eq:Dr-osp-rel-ext13.1}
  \Big[\big[e_{i,1},[e_{j,1},e_{t,1}]\big],\big[[e_{j,1},e_{t,1}],[e_{t,1},e_{k,1}]\big]\Big]=0 \,,
\end{equation}
\begin{equation}\label{eq:Dr-osp-rel-ext13.2}
  \Big[\big[f_{i,1},[f_{j,1},f_{t,1}]\big],\big[[f_{j,1},f_{t,1}],[f_{t,1},f_{k,1}]\big]\Big]=0 \,,
\end{equation}
cf.~\eqref{eq:Lie-Serre-new7}.

\smallskip
Recall the generators
  $\{e_i^{(r)},f_i^{(r)}\}_{1\leq i\leq n+m}^{r\geq 1}\cup \{h_\imath^{(r)}\}_{1\leq \imath\leq n+m+1}^{r\geq 1}$
of $X^\rtt(\fosp(V))$, see~\eqref{eq:rtt-generators}. The following relation between $X(\fosp(V))$ and
$X^\rtt(\fosp(V))$ is the main result of the present~section.

\begin{Thm}\label{thm:Main-Theorem-ext}
The assignment
\begin{equation}\label{eq:Dr-to-RTT-assignment-ext}
   e_{i,r}\mapsto e_i^{(r)} \,,\quad f_{i,r}\mapsto f_i^{(r)} \,, \quad h_{\imath,r}\mapsto h_\imath^{(r)}
     \qquad \forall\ i\,,\, \imath \,,\, r
\end{equation}
gives rise to a superalgebra isomorphism
  $$\Upsilon\colon X(\fosp(V))\iso X^\rtt(\fosp(V)) \,.$$
\end{Thm}

\begin{proof}
First, we verify that the series $e_i(u),f_i(u),h_\imath(u)$ satisfy the defining
relations~\eqref{eq:Dr-osp-rel-ext1}--\eqref{eq:Dr-osp-rel-ext13.2}, so that the
assignment~\eqref{eq:Dr-to-RTT-assignment-ext} gives rise to a superalgebra homomorphism
\begin{equation}\label{eq:Upsilon-extended}
  \Upsilon\colon X(\fosp(V))\longrightarrow X^\rtt(\fosp(V)) \,.
\end{equation}
For $1\leq i,j<n+m$ and $1\leq \imath\leq n+m$, all these relations follow from Corollary~\ref{cor:A-type relations}
combined with the corresponding super $A$-type relations of Theorem~\ref{thm:Drinfeld-A}. In the remaining cases,
the relations follow from the commutativity of Corollary~\ref{cor:commutativity} and the rank $\leq 2$ relations
of Section~\ref{sec:rank 1-2}. The surjectivity of the homomorphism $\Upsilon$ from~\eqref{eq:Upsilon-extended}
follows from the results of Subsections~\ref{ssec:e-currents}--\ref{ssec:c-current}.

To prove the injectivity of~\eqref{eq:Upsilon-extended}, we follow the classical argument of~\cite{bk}.
First, we note that Corollary~\ref{cor:pbw-thm} implies in the standard way (cf.~\cite[\S6]{m}) that
the set of ordered monomials in
\begin{equation}\label{eq:pbw-monomials}
  \big\{ h^{(r)}_{\imath},e^{(r)}_{ij},f^{(r)}_{ji} \,\big|\,
         1\leq \imath\leq n+m+1 \,, i<j\leq i'-\delta_{\ol{i},\bar{0}} \,, r\geq 1 \big\} \,,
\end{equation}
with the powers of odd generators not exceeding $1$, form a basis of $X^\rtt(\fosp(V))$. We define the elements
$\{e_{ij}^{(r)},f_{ji}^{(r)}\}$ with $i<j\leq i'-\delta_{\ol{i},\bar{0}}$ and $r\geq 1$ in the algebra $X(\fosp(V))$,
so that the series $e_{ij}(u)=\sum_{r\geq 1} e_{ij}^{(r)} u^{-r}$ and $f_{ji}(u)=\sum_{r\geq 1} f_{ji}^{(r)} u^{-r}$
are expressed through $e_i(u),f_i(u)$ as in Subsections~\ref{ssec:e-currents}--\ref{ssec:f-currents}. These notations
are compatible with those in $X^\rtt(\fosp(V))$ as we clearly have $\Upsilon(e_{ij}(u))=e_{ij}(u)$ and
$\Upsilon(f_{ji}(u))=f_{ji}(u)$. Thus, to prove the injectivity of~\eqref{eq:Upsilon-extended} it suffices to show
that $X(\fosp(V))$ is spanned by the ordered monomials in~\eqref{eq:pbw-monomials}, with the powers of odd generators
not exceeding $1$.

Let $X^>(\fosp(V))$ denote the \emph{positive subalgebra} of $X(\fosp(V))$ generated by all $\{e_{i,r}\}$.
We consider a filtration on $X^>(\fosp(V))$ defined via $\deg\, e_{i,r}=r-1$, cf.~\eqref{eq:gradings}.
Likewise, let $X^{\geq }(\fosp(V))$ denote the \emph{non-negative subalgebra} of $X(\fosp(V))$ generated
by all $\{e_{i,r},h_{\imath,r}\}$, and consider a filtration on $X^{\geq}(\fosp(V))$ defined via
$\deg\, e_{i,r}=\deg\, h_{\imath,r}=r-1$. Let $\Gr\, X^>(\fosp(V))$, $\Gr\, X^\geq(\fosp(V))$ denote
the corresponding associated graded algebras. Similarly to Subsection~\ref{ssec:quasiclassical},
let $\hat{e}^{(r)}_{ij}:=(-1)^{\ol{i}}\, e^{(r)}_{ij}$. We shall denote the images of $\hat{e}^{(r)}_{ij}$
in $\Gr_{r-1} X^>(\fosp(V))$ or $\Gr_{r-1} X^{\geq}(\fosp(V))$ simply by
$\bar{e}^{(r)}_{ij}$\footnote{Instead of a more confusing notation $\wt{\hat{e}}^{(r)}_{ij}$ as if using notations
from Subsection~\ref{ssec:quasiclassical}.}. Let also $\bar{h}^{(r)}_\imath$ denote the image of $h_{\imath,r}$
in $\Gr_{r-1} X^{\geq}(\fosp(V))$. Finally, we extend $\bar{e}^{(r)}_{ij}$ to all $1\leq i<j\leq 1'$ via
\begin{equation}\label{eq:bare-skew}
  \bar{e}^{(r)}_{ij}=-(-1)^{\ol{i}\cdot \ol{j}+\ol{i}}\theta_i\theta_j\, \bar{e}^{(r)}_{j'i'} \,,
\end{equation}
similarly to the relation satisfied by $F_{ij}\in \fosp(V)$. To establish the aforementioned spanning property
of $X^>(\fosp(V))$, it suffices to show that $\bar{e}^{(r)}_{ij}$ satisfy the commutation relations
alike~\eqref{eq:F-commutation}:
\begin{multline}\label{eq:key-commut-rel}
  [\bar{e}^{(r)}_{ij},\bar{e}^{(s)}_{k\ell}] =
  \delta_{kj} \bar{e}^{(r+s-1)}_{i \ell} - \delta_{\ell i} (-1)^{(\ol{i}+\ol{j})(\ol{k}+\ol{\ell})}\, \bar{e}^{(r+s-1)}_{kj} \, - \\
  \delta_{k i'} (-1)^{\ol{i}\cdot \ol{j}+\ol{i}}\theta_i\theta_j\, \bar{e}^{(r+s-1)}_{j' \ell} +
  \delta_{\ell j'} (-1)^{\ol{i}\cdot \ol{k} + \ol{\ell}\cdot \ol{k}} \theta_{i'}\theta_{j'}\, \bar{e}^{(r+s-1)}_{k i'} \,.
\end{multline}

We prove~\eqref{eq:key-commut-rel} by induction on $r+s$. The base of induction $r=s=1$ is trivial as our
relations~\eqref{eq:Dr-osp-rel-ext1}--\eqref{eq:Dr-osp-rel-ext13.2} are compatible with the defining relations
of $\fosp(V)\oplus \BC\cdot \sfc$, cf.~Theorem~\ref{thm:Chevalley-Serre-Zhang}. The proof of the induction step
relies on Lemmas~\ref{lem:ind-claim-1} and~\ref{lem:ind-claim-2}. First, we define
$\{\alpha_{ij}\}_{1\leq i<j\leq 1'}\subset \fh^*$:
\begin{equation}\label{eq:roots}
\begin{split}
  & \alpha_{ij}=\alpha_{j'i'}=e^*_i-e^*_j \,,\quad
    \alpha_{ij'}=\alpha_{ji'}=e^*_i+e^*_j \qquad \forall\ 1\leq i<j\leq n+m \,,\\
  & \alpha_{i,n+m+1}=\alpha_{n+m+1,i'}=e^*_i  \qquad \forall\ 1\leq i\leq n+m \quad \mathrm{if} \quad N=2n+1 \,, \\
  & \alpha_{ii'}=
    \begin{cases}
      2e^*_i & \text{if} \quad \ol{i}=\bar{1} \\
      0 & \text{otherwise}
    \end{cases} \,.
\end{split}
\end{equation}
According to~\eqref{eq:Dr-osp-rel-ext3.1}, we have
$[h_{\imath,2},e_{j,r}]=(e^*_\imath,\alpha_j)\, (e_{j,r+1}+h_{\imath,1}e_{j,r})$, so that
\begin{equation}\label{eq:bar-h2}
  [\bar{h}^{(2)}_\imath, \bar{e}^{(r)}_{j}]=(e^*_\imath,\alpha_j) \bar{e}^{(r+1)}_{j}
    \qquad \forall\ 1\leq \imath,j\leq n+m \,.
\end{equation}
This result can be generalized as follows:

\begin{Lem}\label{lem:ind-claim-1}
For any $1\leq i<j\leq 1'$, $1\leq \imath\leq n+m$, and $r\geq 1$, we have
\begin{equation}\label{bar-h2-general}
  [\bar{h}^{(2)}_\imath, \bar{e}^{(r)}_{ij}]=(e^*_\imath,\alpha_{ij}) \bar{e}^{(r+1)}_{ij} \,.
\end{equation}
\end{Lem}

Applying $\mathrm{ad}_{\bar{h}^{(2)}_\imath}$ to~\eqref{eq:key-commut-rel}, we thus obtain:
\begin{equation}\label{eq:ind-step}
\begin{split}
  & (e^*_\imath,\alpha_{ij})[\bar{e}^{(r+1)}_{ij},\bar{e}^{(s)}_{k\ell}] +
    (e^*_\imath,\alpha_{k\ell})[\bar{e}^{(r)}_{ij},\bar{e}^{(s+1)}_{k\ell}] \, = \\
  & \qquad \qquad \delta_{kj} (e^*_\imath,\alpha_{ij}+\alpha_{k\ell}) \bar{e}^{(r+s)}_{i \ell} -
    \delta_{\ell i} (e^*_\imath,\alpha_{ij}+\alpha_{k\ell}) (-1)^{(\ol{i}+\ol{j})(\ol{k}+\ol{\ell})}\, \bar{e}^{(r+s)}_{kj} \, - \\
  & \qquad \qquad
    \delta_{k i'} (e^*_\imath,\alpha_{ij}+\alpha_{k\ell})
      (-1)^{\ol{i}\cdot \ol{j}+\ol{i}}\theta_i\theta_j\, \bar{e}^{(r+s)}_{j' \ell} +
    \delta_{\ell j'} (e^*_\imath,\alpha_{ij}+\alpha_{k\ell})
      (-1)^{\ol{i}\cdot \ol{k} + \ol{\ell}\cdot \ol{k}} \theta_{i'}\theta_{j'}\, \bar{e}^{(r+s)}_{k i'} \,,
\end{split}
\end{equation}
where we used the equalities
\begin{equation*}
\begin{split}
  & \delta_{kj} (e^*_\imath,\alpha_{i \ell}) = \delta_{kj} (e^*_\imath,\alpha_{ij}+\alpha_{k\ell}) \,,\quad
    \delta_{\ell i} (e^*_\imath,\alpha_{kj}) = \delta_{\ell i} (e^*_\imath,\alpha_{ij}+\alpha_{k\ell}) \,, \\
  & \delta_{k i'} (e^*_\imath,\alpha_{j' \ell}) = \delta_{k i'} (e^*_\imath,\alpha_{ij}+\alpha_{k\ell}) \,,\quad
    \delta_{\ell j'} (e^*_\imath,\alpha_{k i'}) = \delta_{\ell j'} (e^*_\imath,\alpha_{ij}+\alpha_{k\ell}) \,,
\end{split}
\end{equation*}
which follow by comparing $\fh$-eigenvalues of all summands in~\eqref{eq:F-commutation}. Note that if
$\alpha_{ij}\ne \alpha_{k\ell}$, then we can find $1\leq \imath\ne \jmath\leq n+m$ such that the matrix
  $\Big(\begin{smallmatrix}
      (e^*_\imath,\alpha_{ij}) & (e^*_\imath,\alpha_{k\ell}) \\
      (e^*_\jmath,\alpha_{ij}) & (e^*_\jmath,\alpha_{k\ell})
   \end{smallmatrix}\Big)$
is non-degenerate. Then, combining~\eqref{eq:ind-step} for $\imath,\jmath$, we obtain the desired
formulas~\eqref{eq:key-commut-rel} for both commutators $[\bar{e}^{(r+1)}_{ij},\bar{e}^{(s)}_{k\ell}]$
and $[\bar{e}^{(r)}_{ij},\bar{e}^{(s+1)}_{k\ell}]$, completing the induction step. It thus remains
to prove~\eqref{eq:key-commut-rel} for $(i,j)=(k,\ell)$.

The proof of the latter result as well as the proof of Lemma~\ref{lem:ind-claim-1} rely on Lemma~\ref{lem:ind-claim-2}.
To state this result, let us first summarize the inductive definition of $\bar{e}^{(r)}_{ij}$:
\begin{equation}\label{eq:bare-iterative-1}
  \bar{e}^{(r)}_{i,j+1}=[\bar{e}^{(r)}_{ij},\bar{e}^{(1)}_{j,j+1}]
\end{equation}
for $1\leq i<j\leq n+m$ if $N=2n+1$ or $1\leq i<j<n+m$ if $N=2n$,
\begin{equation}\label{eq:bare-iterative-2}
  \bar{e}^{(r)}_{i,n+m+1}=
  \begin{cases}
     [\bar{e}^{(r)}_{i,n+m-1},\bar{e}^{(1)}_{n+m-1,n+m+1}]
       & \text{if} \quad N=2n \ \ \mathrm{and} \ \ \ol{n+m}=\bar{0} \\
     \sfrac{1}{2}[\bar{e}^{(r)}_{i,n+m},\bar{e}^{(1)}_{n+m,n+m+1}]
       & \text{if} \quad N=2n \ \ \mathrm{and} \ \ \ol{n+m}=\bar{1}
  \end{cases} \,,
\end{equation}
as well as
\begin{equation}\label{eq:bare-iterative-3}
  \bar{e}^{(r)}_{ij'}=(-1)^{1+\ol{j}+\ol{j}\cdot \ol{j+1}}\, [\bar{e}^{(r)}_{i(j+1)'},\bar{e}^{(1)}_{j,j+1}]
\end{equation}
for $1\leq i<j\leq n+m$ if $N=2n+1$ or $1\leq i<j<n+m$ if $N=2n$, as well as $i=j$ if $\ol{i}=\bar{1}$.

\begin{Lem}\label{lem:ind-claim-2}
(a) For $1\leq i<j\leq \lfloor \sfrac{N-1}{2} \rfloor+m$ and $r,s\geq 1$, we have:
\begin{equation}\label{eq:bare-iterative-1-general}
  \bar{e}^{(r+s-1)}_{i,j+1}=[\bar{e}^{(r)}_{ij},\bar{e}^{(s)}_{j,j+1}] \,.
\end{equation}

\noindent
(b) For $N=2n$ and $r,s\geq 1$, we have:
\begin{equation}\label{eq:bare-iterative-2-general}
  \bar{e}^{(r+s-1)}_{i,n+m+1}=
  \begin{cases}
     [\bar{e}^{(r)}_{i,n+m-1},\bar{e}^{(s)}_{n+m-1,n+m+1}]
       & \text{if} \quad N=2n \ \ \mathrm{and} \ \ \ol{n+m}=\bar{0} \\
     \sfrac{1}{2}[\bar{e}^{(r)}_{i,n+m},\bar{e}^{(s)}_{n+m,n+m+1}]
       & \text{if} \quad N=2n \ \ \mathrm{and} \ \ \ol{n+m}=\bar{1}
  \end{cases} \,.
\end{equation}

\noindent
(c) For $1\leq i<j\leq \lfloor \sfrac{N-1}{2} \rfloor+m$ as well as $i=j$ if $\ol{i}=\bar{1}$,
and $r,s\geq 1$, we have:
\begin{equation}\label{eq:bare-iterative-3-general}
  \bar{e}^{(r+s-1)}_{ij'}=
  (-1)^{1+\ol{j}+\ol{j}\cdot \ol{j+1}}\, [\bar{e}^{(r)}_{i(j+1)'},\bar{e}^{(s)}_{j,j+1}] \,.
\end{equation}
\end{Lem}

\begin{proof}[Proofs of Lemma~\ref{lem:ind-claim-1} and Lemma~\ref{lem:ind-claim-2}]
\

We shall prove Lemma~\ref{lem:ind-claim-2} by induction on $i,j$,
while at the same time also proving Lemma~\ref{lem:ind-claim-1}.

\medskip
(a) We prove~\eqref{eq:bare-iterative-1-general} by induction on $j-i$. According to the defining
relations~(\ref{eq:Dr-osp-rel-ext7.1},~\ref{eq:Dr-osp-rel-ext7.2}), we have
  $[\bar{e}^{(r+1)}_{j-1,j},\bar{e}^{(s)}_{j,j+1}]=[\bar{e}^{(r)}_{j-1,j},\bar{e}^{(s+1)}_{j,j+1}]$,
establishing the base of induction. As for the induction step:
\begin{multline}\label{eq:proof-lem2a}
  [\bar{e}^{(r)}_{ij},\bar{e}^{(s)}_{j,j+1}]=
  \big[[\bar{e}^{(r)}_{i,j-1},\bar{e}^{(1)}_{j-1,j}],\bar{e}^{(s)}_{j,j+1}\big]=
  \big[\bar{e}^{(r)}_{i,j-1},[\bar{e}^{(1)}_{j-1,j},\bar{e}^{(s)}_{j,j+1}]\big]= \\
  \big[\bar{e}^{(r)}_{i,j-1},[\bar{e}^{(s)}_{j-1,j},\bar{e}^{(1)}_{j,j+1}]\big]=
  \big[[\bar{e}^{(r)}_{i,j-1},\bar{e}^{(s)}_{j-1,j}],\bar{e}^{(1)}_{j,j+1}\big]=
  [\bar{e}^{(r+s-1)}_{ij},\bar{e}^{(1)}_{j,j+1}]\overset{\eqref{eq:bare-iterative-1}}{=}\bar{e}^{(r+s-1)}_{i,j+1} \,.
\end{multline}
Here, we used the induction hypothesis in the first, third, and fifth equalities, while the second and fourth
equalities relied on the commutativity $[\bar{e}^{(\sharp)}_{i,j-1},\bar{e}^{(\sharp')}_{j,j+1}]=0$, which
follows from~\eqref{eq:Dr-osp-rel-ext9.1}.

We can now also prove~\eqref{bar-h2-general} for $1\leq i<j\leq \lfloor \sfrac{N-1}{2} \rfloor+m$
arguing by induction on $j-i$:
\begin{multline}\label{eq:proof-lem1a}
  [\bar{h}^{(2)}_\imath, \bar{e}^{(r)}_{ij}]\overset{\eqref{eq:bare-iterative-1}}{=}
  \big[\bar{h}^{(2)}_\imath, [\bar{e}^{(r)}_{i,j-1},\bar{e}^{(1)}_{j-1,j}]\big]=
  \big[[\bar{h}^{(2)}_\imath, \bar{e}^{(r)}_{i,j-1}],\bar{e}^{(1)}_{j-1,j}]\big] +
  \big[\bar{e}^{(r)}_{i,j-1},[\bar{h}^{(2)}_\imath,\bar{e}^{(1)}_{j-1,j}]\big] = \\
  (e^*_\imath,\alpha_{i,j-1}) [\bar{e}^{(r+1)}_{i,j-1},\bar{e}^{(1)}_{j-1,j}] +
  (e^*_\imath,\alpha_{j-1,j}) [\bar{e}^{(r)}_{j-1,j},\bar{e}^{(2)}_{j-1,j}]
  \overset{\eqref{eq:bare-iterative-1-general}}{=} (e^*_\imath,\alpha_{ij}) \bar{e}^{(r+1)}_{ij} \,.
\end{multline}

(b) The proofs of part (b) and of Lemma~\ref{lem:ind-claim-1} in that case are completely analogous to part (a).

\medskip
(c) We prove~\eqref{eq:bare-iterative-3-general} by a decreasing induction on $j$ (with an inner decreasing
induction on $i$). Let us note that once~\eqref{eq:bare-iterative-3-general} is established for specific $i,j$
and any $r,s$, the validity of~\eqref{bar-h2-general} for the same $i,j$ and arbitrary $r,\imath$ is derived
exactly as explained in the proof of (a) above. For the base of induction, we shall consider the cases
$N=2n$ and $N=2n+1$ separately.

\medskip
\noindent
\underline{Case~1:} $N=2n$ and $j=n+m-1$.
\

First, we treat the case $i=n+m-1$ with $\ol{n+m-1}=\bar{1}$. In this case,~\eqref{eq:bare-iterative-3-general} is equivalent to:
\begin{equation}\label{eq:missed-case}
  [\bar{e}^{(r)}_{n+m-1,n+m+1},\bar{e}^{(s)}_{n+m-1,n+m}]=[\bar{e}^{(r+s-1)}_{n+m-1,n+m+1},\bar{e}^{(1)}_{n+m-1,n+m}] \,.
\end{equation}
If $\ol{n+m}=\bar{0}$, then~\eqref{eq:missed-case} follows from~\eqref{eq:Dr-osp-rel-ext7.2}. On the other hand,
for $\ol{n+m}=\bar{1}$, we have $\bar{e}^{(r)}_{n+m-1,n+m+1}=\frac{1}{2}[\bar{e}^{(r)}_{n+m-1,n+m},\bar{e}^{(1)}_{n+m,n+m+1}]$
by~\eqref{eq:bare-iterative-2}, $[\bar{e}^{(\sharp)}_{n+m-1,n+m},\bar{e}^{(\sharp')}_{n+m-1,n+m}]=0$
by~\eqref{eq:Dr-osp-rel-ext5.1}. Therefore, we get:
\begin{multline*}
  [\bar{e}^{(r)}_{n+m-1,n+m+1},\bar{e}^{(s)}_{n+m-1,n+m}]=
  \sfrac{1}{2}\big[[\bar{e}^{(r)}_{n+m-1,n+m},\bar{e}^{(1)}_{n+m,n+m+1}],\bar{e}^{(s)}_{n+m-1,n+m}\big]=\\
  \sfrac{1}{2}\big[\bar{e}^{(r)}_{n+m-1,n+m},[\bar{e}^{(1)}_{n+m,n+m+1},\bar{e}^{(s)}_{n+m-1,n+m}]\big]\overset{(b)}{=}
  \sfrac{1}{2}\big[\bar{e}^{(r)}_{n+m-1,n+m},[\bar{e}^{(s)}_{n+m,n+m+1},\bar{e}^{(1)}_{n+m-1,n+m}]\big]=\\
  \sfrac{1}{2}\big[[\bar{e}^{(r)}_{n+m-1,n+m},\bar{e}^{(s)}_{n+m,n+m+1}],\bar{e}^{(1)}_{n+m-1,n+m}\big]\overset{(b)}{=}
  [\bar{e}^{(r+s-1)}_{n+m-1,n+m+1},\bar{e}^{(1)}_{n+m-1,n+m}] \,.
\end{multline*}
This completes our proof of~\eqref{eq:missed-case}.

Next, we treat the case $i=n+m-2$. There are two cases to consider: $\ol{n+m}=\bar{1}$ and $\ol{n+m}=\bar{0}$.
If $\ol{n+m}=\bar{1}$, then
  $\bar{e}^{(r)}_{n+m-2,n+m+1}=\sfrac{1}{2}[\bar{e}^{(r)}_{n+m-2,n+m},\bar{e}^{(1)}_{n+m,n+m+1}]$
and so we have:
\begin{multline*}
  [\bar{e}^{(r)}_{n+m-2,n+m+1},\bar{e}^{(s)}_{n+m-1,n+m}]=
  \sfrac{1}{2}\big[[\bar{e}^{(r)}_{n+m-2,n+m},\bar{e}^{(1)}_{n+m,n+m+1}],\bar{e}^{(s)}_{n+m-1,n+m}\big]= \\
  \sfrac{1}{2}\big[\bar{e}^{(r)}_{n+m-2,n+m},[\bar{e}^{(1)}_{n+m,n+m+1},\bar{e}^{(s)}_{n+m-1,n+m}]\big]\overset{(b)}{=}
  \sfrac{1}{2}\big[\bar{e}^{(r)}_{n+m-2,n+m},[\bar{e}^{(s)}_{n+m,n+m+1},\bar{e}^{(1)}_{n+m-1,n+m}]\big]= \\
  \sfrac{1}{2}\big[[\bar{e}^{(r)}_{n+m-2,n+m},\bar{e}^{(s)}_{n+m,n+m+1}],\bar{e}^{(1)}_{n+m-1,n+m}\big]\overset{(b)}{=}
  [\bar{e}^{(r+s-1)}_{n+m-2,n+m+1},\bar{e}^{(1)}_{n+m-1,n+m}]\overset{\eqref{eq:bare-iterative-3}}{=} \\
  (-1)^{1+\ol{n+m-1}+\ol{n+m-1}\cdot \ol{n+m}}\, \bar{e}^{(r+s-1)}_{n+m-2,n+m+2} \,,
\end{multline*}
where we used an already established $[\bar{e}^{(\sharp)}_{n+m-2,n+m},\bar{e}^{(\sharp')}_{n+m-1,n+m}]=0$
in the second and fourth equalities. If $\ol{n+m}=\bar{0}$, then instead we have
  $\bar{e}^{(r)}_{n+m-2,n+m+1}=[\bar{e}^{(r)}_{n+m-2,n+m-1},\bar{e}^{(1)}_{n+m-1,n+m+1}]$
as well as
  $[\bar{e}^{(\sharp)}_{n+m-1,n+m+1},\bar{e}^{(\sharp')}_{n+m-1,n+m}]=0$,
due to~\eqref{eq:Dr-osp-rel-ext9.1}, cf.~\eqref{eq:Dr-osp-rel-ext9.1-series}. Therefore, we obtain:
\begin{multline*}
  [\bar{e}^{(r)}_{n+m-2,n+m+1},\bar{e}^{(s)}_{n+m-1,n+m}]=
  \big[[\bar{e}^{(r)}_{n+m-2,n+m-1},\bar{e}^{(1)}_{n+m-1,n+m+1}],\bar{e}^{(s)}_{n+m-1,n+m}\big]= \\
  -(-1)^{\ol{n+m-1}(1+\ol{n+m-2})}\,
    \big[\bar{e}^{(1)}_{n+m-1,n+m+1},[\bar{e}^{(r)}_{n+m-2,n+m-1},\bar{e}^{(s)}_{n+m-1,n+m}]\big]\overset{(a)}{=} \\
  -(-1)^{\ol{n+m-1}(1+\ol{n+m-2})}\,
    \big[\bar{e}^{(1)}_{n+m-1,n+m+1},[\bar{e}^{(r+s-1)}_{n+m-2,n+m-1},\bar{e}^{(1)}_{n+m-1,n+m}]\big]= \\
  \big[[\bar{e}^{(r+s-1)}_{n+m-2,n+m-1},\bar{e}^{(1)}_{n+m-1,n+m+1}],\bar{e}^{(1)}_{n+m-1,n+m}\big]
  \overset{\eqref{eq:bare-iterative-3}}{=} (-1)^{1+\ol{n+m-1}}\, \bar{e}^{(r+s-1)}_{n+m-2,n+m+2} \,.
\end{multline*}

The rest proceeds by a decreasing induction on $i$ (with the base $i=n+m-2$ established above).
To this end, we note:
\begin{multline*}
  [\bar{e}^{(r)}_{i,n+m+1},\bar{e}^{(s)}_{n+m-1,n+m}]=
  \big[[\bar{e}^{(r)}_{i,i+1},\bar{e}^{(1)}_{i+1,n+m+1}],\bar{e}^{(s)}_{n+m-1,n+m}\big]= \\
  \big[\bar{e}^{(r)}_{i,i+1},[\bar{e}^{(1)}_{i+1,n+m+1},\bar{e}^{(s)}_{n+m-1,n+m}]\big]=
  \big[\bar{e}^{(r)}_{i,i+1},[\bar{e}^{(s)}_{i+1,n+m+1},\bar{e}^{(1)}_{n+m-1,n+m}]\big]= \\
  \big[[\bar{e}^{(r)}_{i,i+1},\bar{e}^{(s)}_{i+1,n+m+1}],\bar{e}^{(1)}_{n+m-1,n+m}\big]=
  [\bar{e}^{(r+s-1)}_{i,n+m+1},\bar{e}^{(1)}_{n+m-1,n+m}]\overset{\eqref{eq:bare-iterative-3}}{=}
  (-1)^{1+\ol{n+m-1}(1+\ol{n+m})}\bar{e}^{(r+s-1)}_{i,n+m+2} \,,
\end{multline*}
where in the first and fifth equalities we used already established cases of~\eqref{eq:key-commut-rel},
while the second and fourth equalities relied on the commutativity
  $[\bar{e}^{(\sharp)}_{i,i+1},\bar{e}^{(\sharp')}_{n+m-1,n+m}]=0$,
due to~\eqref{eq:Dr-osp-rel-ext9.1}.

\medskip
\noindent
\underline{Case~2:} $N=2n+1$ and $j=n+m$.
\

The proof is by a decreasing induction on $i$. We shall only give details for the base of induction
($i=n+m$ or $i=n+m-1$), as the step of induction is identical to the above one for even $N$.

If $i=n+m$ with $\ol{n+m}=\bar{1}$, then according to~\eqref{eq:Dr-osp-rel-ext5.2} we get:
\begin{equation*}
  [\bar{e}^{(r)}_{n+m,n+m+1},\bar{e}^{(s)}_{n+m,n+m+1}]=
  [\bar{e}^{(r+s-1)}_{n+m,n+m+1},\bar{e}^{(1)}_{n+m,n+m+1}]=
  \bar{e}^{(r+s-1)}_{n+m,n+m+2} \,.
\end{equation*}
If $i=n+m$ with $\ol{n+m}=\bar{0}$, then $[\bar{e}^{(r)}_{n+m,n+m+1},\bar{e}^{(s)}_{n+m,n+m+1}]=0$
according to~\eqref{eq:Dr-osp-rel-ext5.1}.

Let us now treat the case $i=n+m-1$. If $\ol{n+m}=\bar{0}$, then
$[\bar{e}^{(\sharp)}_{n+m,n+m+1},\bar{e}^{(\sharp')}_{n+m,n+m+1}]=0$ as just shown.
Therefore:
\begin{multline*}
  [\bar{e}^{(r)}_{n+m-1,n+m+1},\bar{e}^{(s)}_{n+m,n+m+1}]=
  \big[[\bar{e}^{(r)}_{n+m-1,n+m},\bar{e}^{(1)}_{n+m,n+m+1}],\bar{e}^{(s)}_{n+m,n+m+1}\big]= \\
  -\big[\bar{e}^{(1)}_{n+m,n+m+1},[\bar{e}^{(r)}_{n+m-1,n+m},\bar{e}^{(s)}_{n+m,n+m+1}]\big]\overset{(a)}{=}
  -[\bar{e}^{(1)}_{n+m,n+m+1},\bar{e}^{(r+s-1)}_{n+m-1,n+m+1}]= \\
  [\bar{e}^{(r+s-1)}_{n+m-1,n+m+1},\bar{e}^{(1)}_{n+m,n+m+1}]
  \overset{\eqref{eq:bare-iterative-3}}{=}-\bar{e}^{(r+s-1)}_{n+m-1,n+m+2} \,.
\end{multline*}
If $\ol{n+m}=\bar{1}$, then according to~\eqref{eq:bare-iterative-3} it suffices to verify:
\begin{equation}\label{eq:weird-used}
  [\ol{e}^{(r)}_{n+m-1,n+m},\ol{e}^{(s)}_{n+m,n+m+2}] = [\ol{e}^{(r+s-1)}_{n+m-1,n+m},\ol{e}^{(1)}_{n+m,n+m+2}] \,.
\end{equation}
To prove the latter, we recall~\eqref{eq:Dr-osp-rel-ext-weird1} which implies
  $[\bar{e}^{(1)}_{n+m-1,n+m}, \bar{e}^{(s)}_{n+m,n+m+2}]=2\bar{e}^{(s)}_{n+m-1,n+m+2}$
for any $s\geq 1$. We also recall that
  $$e_{n+m-1,n+m+2}(v)=e'''_{n+m}(v)=\big[[e_{n+m-1}(v),e^{(1)}_{n+m,n+m+1}],e^{(1)}_{n+m,n+m+1}\big] \,.$$
Applying the super Jacobi identity to the latter, we find
\begin{equation*}
  e_{n+m-1,n+m+2}(v)=\sfrac{1}{2}\big[e_{n+m-1,n+m}(v),[e^{(1)}_{n+m,n+m+1},e^{(1)}_{n+m,n+m+1}]\big] \,,
\end{equation*}
so that
  $\bar{e}^{(s)}_{n+m-1,n+m+2}=\sfrac{1}{2}[\bar{e}^{(s)}_{n+m-1,n+m}, \bar{e}^{(1)}_{n+m,n+m+2}]$.
This establishes~\eqref{eq:weird-used} for $r=1$, $s\geq 1$. Commuting this further with
$\bar{h}^{(2)}_{n+m-1}$ several times, we derive the equality~\eqref{eq:weird-used} for any $r,s\geq 1$.

\medskip
The above completes the base of induction on $j$. For the step of induction, we argue as follows:
\begin{multline*}
  [\bar{e}^{(r)}_{i(j+1)'},\bar{e}^{(s)}_{j,j+1}]=
  (-1)^{1+\ol{j+1}+\ol{j+1}\cdot \ol{j+2}}\, \big[[\bar{e}^{(r)}_{i(j+2)'},\bar{e}^{(1)}_{j+1,j+2}],\bar{e}^{(s)}_{j,j+1}\big]= \\
  (-1)^{1+\ol{j+1}+\ol{j+1}\cdot \ol{j+2}}\, \big[\bar{e}^{(r)}_{i(j+2)'},[\bar{e}^{(1)}_{j+1,j+2},\bar{e}^{(s)}_{j,j+1}]\big]
  \overset{(a)}{=}
  (-1)^{1+\ol{j+1}+\ol{j+1}\cdot \ol{j+2}}\, \big[\bar{e}^{(r)}_{i(j+2)'},[\bar{e}^{(s)}_{j+1,j+2},\bar{e}^{(1)}_{j,j+1}]\big]= \\
  (-1)^{1+\ol{j+1}+\ol{j+1}\cdot \ol{j+2}}\, \big[[\bar{e}^{(r)}_{i(j+2)'},\bar{e}^{(s)}_{j+1,j+2}],\bar{e}^{(1)}_{j,j+1}]\big]=
  [\bar{e}^{(r+s-1)}_{i(j+1)'},\bar{e}^{(1)}_{j,j+1}]\overset{\eqref{eq:bare-iterative-3}}{=}
  (-1)^{1+\ol{j}+\ol{j}\cdot \ol{j+1}}\, \bar{e}^{(r+s-1)}_{ij'} \,.
\end{multline*}
Here, we used the induction hypothesis in the first and fifth equalities, while the second and fourth equalities
used the commutativity $[\bar{e}^{(\sharp)}_{i(j+2)'},\bar{e}^{(\sharp')}_{j,j+1}]=0$, due to already established
cases of~\eqref{eq:key-commut-rel}.

This completes our proof of part (c).
\end{proof}

It remains to treat the cases $(i,j)=(k,\ell)$. The case $j=n+m+1$ for $N=2n+1$ has been already treated in the proof
of Lemma~\ref{lem:ind-claim-2}(c) above. Otherwise, we need to show that $[\bar{e}^{(r)}_{ij},\bar{e}^{(s)}_{ij}]=0$,
assuming $1\leq i<j\leq i' - \delta_{\ol{i},\bar{0}}$. For $j=i+1$ (as well as for $j=i+2=n+m+1$ when $N=2n$ and
$\ol{n+m}=\bar{0}$), this commutativity follows from~\eqref{eq:Dr-osp-rel-ext5.1}. Otherwise, let us use already
established cases of~\eqref{eq:key-commut-rel} to write $\bar{e}^{(s)}_{ij}=[\bar{e}^{(s)}_{ik},\bar{e}^{(1)}_{kj}]$
for any $i<k<j$ with $k\ne j'$. Then, $[\bar{e}^{(r)}_{ij},\bar{e}^{(s)}_{ij}]=0$ follows from already established
equalities $[\bar{e}^{(r)}_{ij},\bar{e}^{(s)}_{ik}]=0$, $[\bar{e}^{(r)}_{ij},\bar{e}^{(1)}_{kj}]=0$. The only case when
such $k$ may not exist is for $N=2n$ with $i=n+m-1,j=n+m+1$, and $\ol{n+m}=\bar{1}$ (as the case $\ol{n+m}=\bar{0}$
has been already treated above). However,
  $\bar{e}^{(s)}_{n+m-1,n+m+1}=\sfrac{1}{2}[\bar{e}^{(s)}_{n+m-1,n+m},\bar{e}^{(1)}_{n+m,n+m+1}]$
in this case, and thus the desired commutativity $[\bar{e}^{(r)}_{n+m-1,n+m+1},\bar{e}^{(s)}_{n+m-1,n+m+1}]=0$ follows
from already established equalities
  $[\bar{e}^{(r)}_{n+m-1,n+m+1},\bar{e}^{(1)}_{n+m,n+m+1}]=0$,
  $[\bar{e}^{(r+s-1)}_{n+m-1,n+m+2},\bar{e}^{(1)}_{n+m,n+m+1}]=0$.

\medskip
This completes our proof of the equality~\eqref{eq:key-commut-rel}, hence also of Theorem~\ref{thm:Main-Theorem-ext}.
\end{proof}

\begin{Rem}
We note that the ``additional'' relations~(\ref{eq:Dr-osp-rel-ext-weird1},~\ref{eq:Dr-osp-rel-ext-weird2})
were used in the proof of~\eqref{eq:weird-used}.
\end{Rem}

\begin{Rem}\label{rem:Serre-extosp-series-general}
While the Serre relations~\eqref{eq:Dr-osp-rel-ext9.1}--\eqref{eq:Dr-osp-rel-ext13.2} are literally the same as
those for $\fosp(V)$ in Theorem~\ref{thm:Chevalley-Serre-Zhang}, the classical argument allows to deduce more
general Serre relations by commuting the above further with the Cartan series $h_i(u)$, cf.~\cite[Remark 2.61(b)]{t}.
Explicitly, we have:

\medskip
\noindent
(a) Generalizing~(\ref{eq:Dr-osp-rel-ext9.1},~\ref{eq:Dr-osp-rel-ext9.2}), the following relations hold:
\begin{equation}\label{eq:Dr-osp-rel-ext9.1-series}
  \mathrm{Sym}\, \Big[e_i(u_{1}),\big[e_i(u_{2}),\cdots,[e_i(u_{1-a_{ij}}),e_j(v)]\cdots \big]\Big] = 0 \,,
\end{equation}
\begin{equation}\label{eq:Dr-osp-rel-ext9.2-series}
  \mathrm{Sym}\, \Big[f_i(u_{1}),\big[f_i(u_{2}),\cdots,[f_i(u_{1-a_{ij}}),f_j(v)]\cdots \big]\Big] = 0 \,,
\end{equation}
where  $\mathrm{Sym}$ denotes the symmetrization with respect to all permutations of $\{u_1,\ldots,u_{1-a_{ij}}\}$.

\medskip
\noindent
(b) Generalizing~(\ref{eq:Dr-osp-rel-ext10.1},~\ref{eq:Dr-osp-rel-ext10.2}), the following relations hold
(cf.~\eqref{eq:Atype-superserre}):
\begin{equation}\label{eq:Dr-osp-rel-ext10.1-series}
  \big[[e_{j}(u),e_t(v_1)],[e_t(v_2),e_{k}(w)]\big] +
  \big[[e_{j}(u),e_t(v_2)],[e_t(v_1),e_{k}(w)]\big] = 0 \,,
\end{equation}
\begin{equation}\label{eq:Dr-osp-rel-ext10.2-series}
  \big[[f_{j}(u),f_t(v_1)],[f_t(v_2),f_{k}(w)]\big] +
  \big[[f_{j}(u),f_t(v_2)],[f_t(v_1),f_{k}(w)]\big] = 0 \,.
\end{equation}

\medskip
\noindent
(c) Generalizing~(\ref{eq:Dr-osp-rel-ext11.1},~\ref{eq:Dr-osp-rel-ext11.2}), the following relations hold:
\begin{equation}\label{eq:Dr-osp-rel-ext11.1-series}
  \big[e_{t}(u),[e_{s}(v),e_{i}(w)]\big]-\big[e_{s}(v),[e_{t}(u),e_{i}(w)]\big]=0 \,,
\end{equation}
\begin{equation}\label{eq:Dr-osp-rel-ext11.2-series}
  \big[f_{t}(u),[f_{s}(v),f_{i}(w)]\big]-\big[f_{s}(v),[f_{t}(u),f_{i}(w)]\big]=0 \,.
\end{equation}

\medskip
\noindent
(d) Generalizing~(\ref{eq:Dr-osp-rel-ext12.1},~\ref{eq:Dr-osp-rel-ext12.2}), the following relations hold:
\begin{equation}\label{eq:Dr-osp-rel-ext12.1-series}
  \mathrm{Sym}\, \Big[[e_{j}(u_1),e_{t}(v_1)],\big[[e_{j}(u_2),e_{t}(v_2)],[e_{t}(v_3),e_{k}(w)]\big]\Big]=0 \,,
\end{equation}
\begin{equation}\label{eq:Dr-osp-rel-ext12.2-series}
  \mathrm{Sym}\, \Big[[f_{j}(u_1),f_{t}(v_1)],\big[[f_{j}(u_2),f_{t}(v_2)],[f_{t}(v_3),f_{k}(w)]\big]\Big]=0 \,,
\end{equation}
where $\mathrm{Sym}$ denotes the symmetrization with respect to all permutations of $\{u_1,u_2\}$, $\{v_1,v_2,v_3\}$.

\medskip
\noindent
(e) Generalizing~(\ref{eq:Dr-osp-rel-ext13.1},~\ref{eq:Dr-osp-rel-ext13.2}), the following relations hold:
\begin{equation}\label{eq:Dr-osp-rel-ext13.1-series}
  \mathrm{Sym}\, \Big[\big[e_{i}(z),[e_{j}(u_1),e_{t}(v_1)]\big],\big[[e_{j}(u_2),e_{t}(v_2)],[e_{t}(v_3),e_{k}(w)]\big]\Big]=0 \,,
\end{equation}
\begin{equation}\label{eq:Dr-osp-rel-ext13.2-series}
  \mathrm{Sym}\, \Big[\big[f_{i}(z),[f_{j}(u_1),f_{t}(v_1)]\big],\big[[f_{j}(u_2),f_{t}(v_2)],[f_{t}(v_3),f_{k}(w)]\big]\Big]=0 \,,
\end{equation}
where $\mathrm{Sym}$ denotes the symmetrization with respect to all permutations of $\{u_1,u_2\}$, $\{v_1,v_2,v_3\}$.
\end{Rem}

\begin{Rem}
We note that we presently derived~(\ref{eq:Dr-osp-rel-ext10.1-series},~\ref{eq:Dr-osp-rel-ext10.2-series}) from
their simplest cases~(\ref{eq:Dr-osp-rel-ext10.1},~\ref{eq:Dr-osp-rel-ext10.2}), unlike the super $A$-type
of~\cite{t} where we rather derived the former from the more general relations
\begin{equation}\label{eq:osp-vs-sl-deg4}
  \big[[e_{j,r+1},e_{t,1}],[e_{t,1},e_{k,s+1}]\big] = 0 = \big[[f_{j,r+1},f_{t,1}],[f_{t,1},f_{k,s+1}]\big]
  \qquad \forall\ r,s\geq 0 \,.
\end{equation}
In fact, the only reason we used this more general form~\eqref{eq:osp-vs-sl-deg4} in~\cite{t} instead of
just~(\ref{eq:Dr-osp-rel-ext10.1},~\ref{eq:Dr-osp-rel-ext10.2}) is to treat the special case of $\gl(2|2)$
with the parity sequence $(\bar{0},\bar{0},\bar{1},\bar{1})$ or $(\bar{1},\bar{1},\bar{0},\bar{0})$.
\end{Rem}


\subsection{Drinfeld orthosymplectic super Yangian}
\label{ssec:Dr-nonext-Yangian}
\

Following the above notations, we define the \emph{Drinfeld Yangian of $\fosp(V)$}, denoted by $Y(\fosp(V))$,
to be the associative $\BC$-superalgebra generated by $\{\sfx^\pm_{i,r},k_{i,r}\}_{1\leq i\leq n+m}^{r\geq 0}$
with the $\BZ_2$-grading given by
\begin{equation*}
\begin{split}
  & |\sfx^\pm_{i,r}|=\ol{i}+\ol{i+1} \,, \quad |k_{\imath,r}|=\bar{0}
      \qquad \forall\ i<n+m \,,\, \imath\leq n+m \,,\, r\geq 0 \,, \\
  & |\sfx^\pm_{n+m,r}|=
    \begin{cases}
      \ol{n+m-1}+\ol{n+m} & \text{if} \quad N=2n \,,\, \ol{n+m}=\bar{0} \\
      \ol{n+m}+\ol{n+m+1} & \text{otherwise}
    \end{cases} \,,
\end{split}
\end{equation*}
and subject to the defining relations~\eqref{eq:Dr-osp-rel-1}--\eqref{eq:Dr-osp-rel-9}.
To state the relations, form the generating~series:
\begin{equation}
  \sfx^\pm_i(u)=\sum_{r\geq 0} \sfx^\pm_{i,r}u^{-r-1} \,,\qquad
  k_i(u)=1+\sum_{r\geq 0} k_{i,r}u^{-r-1} \,.
\end{equation}
We also recall the symmetrized Cartan matrix $B=(b_{ij})$ of~\eqref{eq:symm-Cartan-matrix} with
$b_{ij}=(\alpha_i,\alpha_j)$ and the Cartan matrix $A=(a_{ij})$ of~\eqref{eq:nonsymm-Cartan-matrix}.
The defining relations of $Y(\fosp(V))$ are as follows:
\begin{equation}\label{eq:Dr-osp-rel-1}
  [k_{i,r},k_{j,s}]=0 \qquad \forall\ 1\leq i,j\leq n+m \,,\, r,s\geq 0 \,,
\end{equation}
\begin{equation}\label{eq:Dr-osp-rel-2}
  [\sfx^+_{i,r},\sfx^-_{j,s}]=\delta_{ij}k_{i,r+s}
    \qquad \forall\ 1\leq i,j\leq n+m \,,\, r,s\geq 0 \,,
\end{equation}
\begin{equation}\label{eq:Dr-osp-rel-3.1}
  [k_{i,0},\sfx^\pm_{j,s}]=\pm b_{ij}\, \sfx^\pm_{j,s}
    \qquad \forall\ 1\leq i,j\leq n+m \,,\, s\geq 0 \,,
\end{equation}
\begin{equation}\label{eq:Dr-osp-rel-3.2}
  [k_{i,r+1},\sfx^\pm_{j,s}]-[k_{i,r},\sfx^\pm_{j,s+1}]=\pm \frac{b_{ij}}{2}\, \{k_{i,r},\sfx^\pm_{j,s}\}
    \qquad \mathrm{unless}\ i=j\ \mathrm{and}\ |\alpha_i| =\bar{1}\,,
\end{equation}
\begin{equation}\label{eq:Dr-osp-rel-3.3}
  [k_{i,r},\sfx^\pm_{i,s}]=0  \qquad \mathrm{for}\  |\alpha_i| =\bar{1}
    \quad \mathrm{unless}\ N=2n+1, \overline{n+m}=\bar{1},i=n+m \,,
\end{equation}
and in the latter case of $N=2n+1,\overline{n+m}=\bar{1}, i=n+m$, we rather impose:
\begin{equation}\label{eq:Dr-osp-rel-3-spec}
\begin{split}
  & [k_{n+m}(u),\sfx^-_{n+m}(v)]=-k_{n+m}(u)
    \left(\frac{1}{3}\, \frac{\sfx^-_{n+m}(u-1/2)-\sfx^-_{n+m}(v)}{u-v-1/2}+
           \frac{2}{3}\, \frac{\sfx^-_{n+m}(u+1)-\sfx^-_{n+m}(v)}{u-v+1} \right) \,, \\
  & [k_{n+m}(u),\sfx^+_{n+m}(v)]=
    \left(\frac{1}{3}\, \frac{\sfx^+_{n+m}(u-1/2)-\sfx^+_{n+m}(v)}{u-v-1/2}+
          \frac{2}{3}\, \frac{\sfx^+_{n+m}(u+1)-\sfx^+_{n+m}(v)}{u-v+1} \right)k_{n+m}(u) \,,
\end{split}
\end{equation}
\begin{equation}\label{eq:Dr-osp-rel-4}
  [\sfx^\pm_{i,r+1},\sfx^\pm_{j,s}]-[\sfx^\pm_{i,r},\sfx^\pm_{j,s+1}]=
  \pm \frac{b_{ij}}{2}\, \{\sfx^\pm_{i,r},\sfx^\pm_{j,s}\}
    \qquad \mathrm{unless}\ N=2n+1,\overline{n+m}=\bar{1},i=j=n+m\,,
\end{equation}
and in the latter case of $N=2n+1, \overline{n+m}=\bar{1}, i=j=n+m$, we rather impose:
\begin{equation}\label{eq:Dr-osp-rel-4-spec}
\begin{split}
  & [\sfx^+_{n+m}(u),\sfx^+_{n+m}(v)]=
    \frac{\sfx^{'+}_{n+m}(v)-\sfx^{'+}_{n+m}(u)}{u-v} + \frac{\sfx^+_{n+m}(u)^2-\sfx^+_{n+m}(v)^2}{u-v} \, + \\
  & \qquad \qquad \qquad \qquad \qquad
    \frac{\sfx^+_{n+m}(v)\sfx^+_{n+m}(u)-\sfx^+_{n+m}(u)\sfx^+_{n+m}(v)}{2(u-v)} -
    \frac{(\sfx^+_{n+m}(v)-\sfx^+_{n+m}(u))^2}{2(u-v)^2} \,, \\
  & [\sfx^-_{n+m}(u),\sfx^-_{n+m}(v)]=
    \frac{\sfx^{'-}_{n+m}(u)-\sfx^{'-}_{n+m}(v)}{u-v} + \frac{\sfx^-_{n+m}(u)^2-\sfx^-_{n+m}(v)^2}{u-v} \, + \\
  & \qquad \qquad \qquad \qquad \qquad
    \frac{\sfx^-_{n+m}(u)\sfx^-_{n+m}(v)-\sfx^-_{n+m}(v)\sfx^-_{n+m}(u)}{2(u-v)} -
    \frac{(\sfx^-_{n+m}(u)-\sfx^-_{n+m}(v))^2}{2(u-v)^2} \,,
\end{split}
\end{equation}
where we set
\begin{equation*}
  \sfx^{'+}_{n+m}(u)=\sfx^+_{n+m}(u)^2+[\sfx^+_{n+m}(u),\sfx^+_{n+m,0}] \,,\qquad
  \sfx^{'-}_{n+m}(u)=-\sfx^-_{n+m}(u)^2-[\sfx^-_{n+m}(u),\sfx^-_{n+m,0}] \,,
\end{equation*}
for $N=2n+1$ and $\ol{n+m}=\bar{1}$, we also impose:
\begin{multline}\label{eq:Dr-osp-rel-weird1}
  [\sfx^-_{n+m-1,0},\sfx^{'-}_{n+m}(v)]=-\sfx^{'''-}_{n+m}(v+\sfrac{1}{2})-\sfx^{'''-}_{n+m}(v-1) \, + \\
  \sfx^-_{n+m-1}(v+\sfrac{1}{2})\sfx^{'-}_{n+m}(v)+\sfx^{'-}_{n+m}(v)\sfx^-_{n+m-1}(v-1)-
  (-1)^{\ol{n+m-1}}\sfx^-_{n+m}(v)\sfx^{''-}_{n+m}(v-1) \,,
\end{multline}
\begin{multline}\label{eq:Dr-osp-rel-weird2}
  [\sfx^+_{n+m-1,0},\sfx^{'+}_{n+m}(v)]=\sfx^{'''+}_{n+m}(v+\sfrac{1}{2})+\sfx^{'''+}_{n+m}(v-1) \, - \\
  \sfx^{'+}_{n+m}(v) \sfx^+_{n+m-1}(v+\sfrac{1}{2})-\sfx^{+}_{n+m-1}(v-1)\sfx^{'+}_{n+m}(v)-
  \sfx^{''+}_{n+m}(v-1)\sfx^+_{n+m}(v) \,,
\end{multline}
with
\begin{equation*}
\begin{split}
  & \sfx^{''-}_{n+m}(v)=-[\sfx^{-}_{n+m-1}(v),\sfx^-_{n+m,0}] \,, \qquad
    \sfx^{'''-}_{n+m}(v)=\big[[\sfx^{-}_{n+m-1}(v),\sfx^-_{n+m,0}],\sfx^-_{n+m,0}\big] \,, \\
  & \sfx^{''+}_{n+m}(v)=-[\sfx^+_{n+m,0},\sfx^{+}_{n+m-1}(v)]  \,, \qquad
     \sfx^{'''+}_{n+m}(v)=-\big[\sfx^+_{n+m,0},[\sfx^+_{n+m,0},\sfx^{+}_{n+m-1}(v)]\big] \,,
\end{split}
\end{equation*}
as well as the \emph{standard Serre relations}
\begin{equation}\label{eq:Dr-osp-rel-5.1}
  (\mathrm{ad}_{\sfx^\pm_{i,0}})^{1-a_{ij}}(\sfx^\pm_{j,0}) = 0
  \qquad \mathrm{for} \quad i\ne j \,,\quad \mathrm{with}\ a_{ii}\ne 0 \ \mathrm{or}\ a_{ij}=0 \,,
\end{equation}
\begin{equation}\label{eq:Dr-osp-rel-5.2}
  [\sfx^\pm_{i,0},\sfx^\pm_{i,0}] = 0  \qquad \mathrm{if} \quad a_{ii}=0 \,,
\end{equation}
and the following \emph{higher order Serre relations}:
\begin{equation}\label{eq:Dr-osp-rel-6}
  \big[[\sfx^\pm_{j,0},\sfx^\pm_{t,0}],[\sfx^\pm_{t,0},\sfx^\pm_{k,0}]\big] = 0
    \qquad \mathrm{for\ subdiagrams\ \eqref{eq:pic-Serre-new4a}-\eqref{eq:pic-Serre-new4b}} \,,
\end{equation}
\begin{equation}\label{eq:Dr-osp-rel-7}
  \big[\sfx^\pm_{t,0},[\sfx^\pm_{s,0},\sfx^\pm_{i,0}]\big]-\big[\sfx^\pm_{s,0},[\sfx^\pm_{t,0},\sfx^\pm_{i,0}]\big] = 0
    \qquad \mathrm{for\ subdiagram\ \eqref{eq:pic-Serre-new3}} \,,
\end{equation}
\begin{equation}\label{eq:Dr-osp-rel-8}
  \Big[[\sfx^\pm_{j,0},\sfx^\pm_{t,0}],\big[[\sfx^\pm_{j,0},\sfx^\pm_{t,0}],[\sfx^\pm_{t,0},\sfx^\pm_{k,0}]\big]\Big] = 0
    \qquad \mathrm{for\ subdiagram\ \eqref{eq:pic-Serre-new6}} \,,
\end{equation}
\begin{equation}\label{eq:Dr-osp-rel-9}
  \Big[\big[\sfx^\pm_{i,0},[\sfx^\pm_{j,0},\sfx^\pm_{t,0}]\big],
       \big[[\sfx^\pm_{j,0},\sfx^\pm_{t,0}],[\sfx^\pm_{t,0},\sfx^\pm_{k,0}]\big]\Big] = 0
   \qquad \mathrm{for\ subdiagram\ \eqref{eq:pic-Serre-new7}} \,.
\end{equation}

\begin{Rem}\label{rem:relations-via-series}
(a) The relation~\eqref{eq:Dr-osp-rel-1} can be equivalently written via the generating series as:
\begin{equation}\label{eq:Dr-osp-rel-1-series}
  [k_i(u),k_j(v)]=0 \qquad \forall\ 1\leq i,j\leq n+m  \,.
\end{equation}

\noindent
(b) The relation~\eqref{eq:Dr-osp-rel-2} can be equivalently written via the generating series as:
\begin{equation}\label{eq:Dr-osp-rel-2-series}
  [\sfx^+_i(u),\sfx^-_j(v)]=-\delta_{ij}\, \frac{k_i(u)-k_i(v)}{u-v} \qquad \forall\ 1\leq i,j\leq n+m \,.
\end{equation}

\noindent
(c) The relations~\eqref{eq:Dr-osp-rel-3.1}--\eqref{eq:Dr-osp-rel-3.3} can be equivalently and
uniformly written via the generating series:
\begin{equation}\label{eq:Dr-osp-rel-3-series}
  [k_i(u),\sfx^\pm_j(v)]=\mp \frac{b_{ij}}{2}\, \frac{\{k_i(u),\sfx^\pm_j(u)-\sfx^\pm_j(v)\}}{u-v}
    \qquad \forall\ 1\leq i,j\leq n+m \,.
\end{equation}

\noindent
(d) The relations~\eqref{eq:Dr-osp-rel-4} imply the following equality on the generating series:
\begin{equation}\label{eq:Dr-osp-rel-4-series}
  [\sfx^\pm_i(u),\sfx^\pm_j(v)]-[\sfx^\pm_i(v),\sfx^\pm_j(u)]=
  \mp \frac{b_{ij}}{2}\, \frac{\{\sfx^\pm_i(u)-\sfx^\pm_i(v),\sfx^\pm_j(u)-\sfx^\pm_j(v)\}}{u-v} \,.
\end{equation}
The left-hand side above is usually written as $[\sfx^\pm_i(u),\sfx^\pm_j(v)]+[\sfx^\pm_j(u),\sfx^\pm_i(v)]$
in non-super case, but it rather becomes $[\sfx^\pm_i(u),\sfx^\pm_j(v)]-[\sfx^\pm_j(u),\sfx^\pm_i(v)]$ if
both simple roots $\alpha_i,\alpha_j$ are odd.

\medskip
\noindent
(e) It is not clear to us if~\eqref{eq:Dr-osp-rel-4-series} alone imply~\eqref{eq:Dr-osp-rel-4} unless
$i=j$ or $b_{ij}=0$. In non-super case, one can first derive the $r=s=0$ case of~\eqref{eq:Dr-osp-rel-4}
from~\eqref{eq:Dr-osp-rel-4-series}, and then establish the general case of~\eqref{eq:Dr-osp-rel-4} by
utilizing~\eqref{eq:Dr-osp-rel-3-series}, see e.g.~\cite[Remark~2.61(b)]{t}. In the present setup,
since~\eqref{eq:Dr-osp-rel-3-series} holds always except for $N=2n,\ol{n+m}=\bar{1}, i=j=n+m$,
one can thus derive \eqref{eq:Dr-osp-rel-4} from~\eqref{eq:Dr-osp-rel-4-series} combined
with~\eqref{eq:Dr-osp-rel-3-series} for all cases but $N+2m=5, |v_2|=\bar{1},i\ne j$.
\end{Rem}

\begin{Rem}\label{rem:osp12-alt}
We note that~\eqref{eq:Dr-osp-rel-3-spec} can be equivalently written as follows
(see~\eqref{eq:k-e-osp12-alternative} below):
\begin{equation}\label{eq:Dr-osp-rel-3-spec-alt}
\begin{split}
  & [k_{n+m}(u),\sfx^-_{n+m}(v)]=-
    \left(\frac{1}{3}\, \frac{\sfx^-_{n+m}(u+1/2)-\sfx^-_{n+m}(v)}{u-v+1/2}+
           \frac{2}{3}\, \frac{\sfx^-_{n+m}(u-1)-\sfx^-_{n+m}(v)}{u-v-1} \right)k_{n+m}(u) \,, \\
  & [k_{n+m}(u),\sfx^+_{n+m}(v)]=k_{n+m}(u)
    \left(\frac{1}{3}\, \frac{\sfx^+_{n+m}(u+1/2)-\sfx^+_{n+m}(v)}{u-v+1/2}+
          \frac{2}{3}\, \frac{\sfx^+_{n+m}(u-1)-\sfx^+_{n+m}(v)}{u-v-1} \right) \,.
\end{split}
\end{equation}
\end{Rem}

Let us now relate the above algebra $Y(\fosp(V))$ to $Y^\rtt(\fosp(V))$ of Subsection~\ref{ssec:center}.
To do so, we follow the same strategy as in $A$-type, see~\cite[\S2.5]{t}. First, we define a sequence
$u_1,\ldots,u_{n+m}$ via
\begin{equation}\label{u-shifts}
  u_1:=u \qquad \mathrm{and} \qquad
  u_{i+1}=u_i+\frac{b_{i,i+1}}{2} \qquad \mathrm{for} \quad 1\leq i<n+m \,.
\end{equation}
Thus, $u_{i}=u_{i-1}-\frac{(-1)^{\ol{i}}}{2}$ for $1\leq i<n+m$, while $u_{n+m}$ satisfies
\begin{equation}
  u_{n+m}-u_{n+m-1}=
  \begin{cases}
    0 & \text{if} \quad N=2n \,,\, \ol{n+m}=\bar{0} \\
    1 & \text{if} \quad N=2n \,,\, \ol{n+m}=\bar{1} \\
    -\sfrac{1}{2} & \text{if} \quad N=2n+1 \,,\, \ol{n+m}=\bar{0} \\
    \sfrac{1}{2} & \text{if} \quad N=2n+1 \,,\, \ol{n+m}=\bar{1}
  \end{cases} \,.
\end{equation}
We also consider the following generating series with coefficients in $X^\rtt(\fosp(V))$:
\begin{equation}\label{eq:xk-series-1}
  X^+_i(u)=f_{i+1,i}(u_i) \,,\quad
  X^-_i(u)=(-1)^{\ol{i}}e_{i,i+1}(u_i) \,,\quad
  K_i(u)=h_i(u_i)^{-1}h_{i+1}(u_i)
  \quad \forall\ 1\leq i<n+m \,,
\end{equation}
while $X^\pm_{n+m}(u), K_{n+m}(u)$ are defined by~\eqref{eq:xk-series-1} for odd $N$,
and otherwise are given by:
\begin{equation}\label{eq:xk-series-last-f}
  X^+_{n+m}(u)=
  \begin{cases}
    f_{n+m+1,n+m-1}(u_{n+m-1}) & \text{if} \quad N=2n \,,\, \ol{n+m}=\bar{0} \\
    f_{n+m+1,n+m}(u_{n+m}) & \text{if} \quad N=2n \,,\, \ol{n+m}=\bar{1}
  \end{cases} \,,
\end{equation}
\begin{equation}\label{eq:xk-series-last-e}
  X^-_{n+m}(u)=
  \begin{cases}
    (-1)^{\ol{n+m}}\, e_{n+m-1,n+m+1}(u_{n+m-1}) & \text{if} \quad N=2n \,,\, \ol{n+m}=\bar{0} \\
    \sfrac{1}{2} (-1)^{\ol{n+m}}\, e_{n+m,n+m+1}(u_{n+m}) & \text{if} \quad N=2n \,,\, \ol{n+m}=\bar{1}
  \end{cases} \,,
\end{equation}
\begin{equation}\label{eq:xk-series-last-k}
  K_{n+m}(u)=
  \begin{cases}
    h_{n+m-1}(u_{n+m-1})^{-1}h_{n+m+1}(u_{n+m-1}) & \text{if} \quad N=2n \,,\, \ol{n+m}=\bar{0} \\
    h_{n+m}(u_{n+m})^{-1}h_{n+m+1}(u_{n+m}) & \text{if} \quad N=2n \,,\, \ol{n+m}=\bar{1}
  \end{cases} \,.
\end{equation}
We shall denote their coefficients by $\{X^+_{i,r}, X^-_{i,r}, K_{i,r}\}_{1\leq i\leq n+m}^{r\geq 0}$,
respectively, so that
\begin{equation}\label{eq:XK-coefficients}
  X^\pm_i(u)=\sum_{r\geq 0} X^\pm_{i,r} u^{-r-1} \,, \qquad
  K_i(u)=1+\sum_{r\geq 0} K_{i,r} u^{-r-1} \,.
\end{equation}
We note right away that all these elements actually belong to $Y^\rtt(\fosp(V))$ of~\eqref{eq:RTT-non-extended}.

\medskip
The following is the main result of this subsection:

\begin{Thm}\label{thm:Main-Theorem-nonext}
The assignment
\begin{equation}\label{eq:Dr-to-RTT-assignment}
  \sfx^\pm_{i,r}\mapsto X^\pm_{i,r} \,,\quad k_{i,r}\mapsto K_{i,r}
    \qquad \forall\ 1\leq i\leq n+m \,,\, r\geq 0
\end{equation}
gives rise to a superalgebra isomorphism
  $$\Upsilon\colon Y(\fosp(V))\iso Y^\rtt(\fosp(V)) \,.$$
\end{Thm}

\begin{proof}
First, we verify that the currents $X^\pm_i(u),K_i(u)$ satisfy the defining
relations~\eqref{eq:Dr-osp-rel-1}--\eqref{eq:Dr-osp-rel-9}, so that the
assignment~\eqref{eq:Dr-to-RTT-assignment} gives rise to a superalgebra homomorphism
  $$\Upsilon\colon Y(\fosp(V))\to Y^\rtt(\fosp(V)) \,.$$
For $1\leq i,j<n+m$ (respectively, $i,j\in \{1,\ldots,n+m-2,n+m\}$ for $N=2n$, $\ol{n+m}=\bar{0}$), all these
relations follow from Corollary~\ref{cor:A-type relations} (respectively, Corollary~\ref{cor:other-A-type relations})
combined with the corresponding super $A$-type relations of~\cite[Theorem 2.67]{t}. In the remaining cases with
$\max\{i,j\}=n+m$ and $|i-j|\geq 2$, all the above relations follow from the commutativity statement of
Corollary~\ref{cor:commutativity}. It~thus remains to treat the cases $i=j=n+m$ or $\{i,j\}=\{n+m-1,n+m\}$.
Evoking Theorem~\ref{thm:embedding}, these actually reduce to the corresponding relations in rank 1 (four cases
treated in Subsection~\ref{ssec:rank-1}) and rank 2 (eight cases treated in Subsection~\ref{ssec:rank-2}),
which are verified case-by-case.

A uniform way to check the commutation formulas between $K_i(u)$ and $X^\pm_j(v)$ with $i,j\in \{n+m,n+m-1\}$ is
to pull $h_i(u)^{-1}$ and $h_{i+1}(u)$ to the leftmost and rightmost sides (in fact, only one of the two options
works, as the other produces poles) in both the left-hand and right-hand sides of~\eqref{eq:Dr-osp-rel-3-series}.
The only exception from this rule are the cases $i=j=n+m$ for odd $N=2n+1$. The latter essentially reduces to the
rank $1$ cases of $\fosp(3|0)$ and $\fosp(1|2)$, which we treat next:

\medskip
\noindent
$\bullet$
\emph{Verification of~\eqref{eq:Dr-osp-rel-3-series} for $\fosp(3|0)$ (see also~\cite{jlm})}.

According to~(\ref{eq:osp30-2},~\ref{eq:osp30-3}), we have
  $[h_2(u),e_{12}(v)]=\frac{h_2(u)(e_{12}(u)-e_{12}(v))}{2(u-v)} - \frac{(e_{12}(u-1)-e_{12}(v))h_2(u)}{2(u-v-1)}$
and $[h_1(u),e_{12}(v)]=-\frac{h_1(u)(e_{12}(u)-e_{12}(v))}{u-v}$. The latter equality implies:
\begin{equation}\label{eq:h1-e12-thm2}
  h_1(u)^{-1}e_{12}(v)=\left(\frac{u-v-1}{u-v}\, e_{12}(v)+\frac{1}{u-v}\, e_{12}(u)\right)h_1(u)^{-1} \,.
\end{equation}
Therefore, we obtain:
\begin{multline}
  [h_1(u)^{-1}h_2(u),e_{12}(v)]=
  h_1(u)^{-1}[h_2(u),e_{12}(v)]-h_1(u)^{-1}[h_1(u),e_{12}(v)]h_1(u)^{-1}h_2(u)=\\
  \frac{1}{2(u-v)}h_1(u)^{-1}h_2(u)\big(e_{12}(u)-e_{12}(v)\big) -
    \frac{1}{2(u-v-1)}h_1(u)^{-1}\big(e_{12}(u-1)-e_{12}(v)\big)h_2(u) \, + \\
  \frac{1}{u-v}\big(e_{12}(u)-e_{12}(v)\big)h_1(u)^{-1}h_2(u) \,.
\end{multline}
Using~\eqref{eq:h1-e12-thm2}, we see that the second summand above simplifies to:
\begin{equation*}
   -\frac{1}{2(u-v-1)}h_1(u)^{-1}\big(e_{12}(u-1)-e_{12}(v)\big)h_2(u)=
   -\frac{1}{2(u-v)}\big(e_{12}(u)-e_{12}(v)\big)h_1(u)^{-1}h_2(u) \,.
\end{equation*}
Combining the above two equalities, we obtain the desired relation (cf.~\eqref{eq:Dr-osp-rel-3-series}):
\begin{equation*}
  [h_1(u)^{-1}h_2(u),e_{12}(v)]=\frac{1}{2}\frac{\big\{h_1(u)^{-1}h_2(u),e_{12}(u)-e_{12}(v)\big\}}{u-v} \,.
\end{equation*}

\medskip
\noindent
$\bullet$
\emph{Verification of~\eqref{eq:Dr-osp-rel-3-spec} for $\fosp(1|2)$}.

According to~(\ref{eq:osp12-2},~\ref{eq:osp12-3}), we have
  $[h_2(u),e_{12}(v)]=h_2(u)\left(\frac{e_{12}(u)-e_{12}(v)}{u-v}+\frac{e_{12}(v)-e_{12}(u-1/2)}{u-v-1/2}\right)$
and $[h_1(u),e_{12}(v)]=\frac{h_1(u)(e_{12}(u)-e_{12}(v))}{u-v}$. The latter equality also implies:
\begin{equation}\label{eq:e12-h1inv-thm2}
  e_{12}(v)h_1(u)^{-1}=h_1(u)^{-1}\left(\frac{u-v}{u-v+1}e_{12}(v)+\frac{1}{u-v+1}e_{12}(u+1)\right) \,.
\end{equation}
Therefore, we obtain:
\begin{multline}
  [h_2(u)h_1(u)^{-1},e_{12}(v)] = h_2(u)[h_1(u)^{-1},e_{12}(v)]+[h_2(u),e_{12}(v)]h_1(u)^{-1} = \\
  h_2(u)h_1(u)^{-1}\, \frac{e_{12}(v)-e_{12}(u+1)}{u-v+1} +
  h_2(u) \left(\frac{e_{12}(u)-e_{12}(v)}{u-v} - \frac{e_{12}(u-1/2)-e_{12}(v)}{u-v-1/2} \right) h_1(u)^{-1} \,.
\end{multline}
Using~\eqref{eq:e12-h1inv-thm2} to move $h_1(u)^{-1}$ to the leftmost part,
we obtain the desired relation (cf.~\eqref{eq:Dr-osp-rel-3-spec}):
\begin{equation*}
  [h_2(u)h_1(u)^{-1},e_{12}(v)] = h_2(u)h_1(u)^{-1}
  \left(-\frac{1}{3}\, \frac{e_{12}(u-1/2)-e_{12}(v)}{u-v-1/2}
        -\frac{2}{3}\, \frac{e_{12}(u+1)-e_{12}(v)}{u-v+1}\right) \,.
\end{equation*}
One could alternatively move both $h_1(u)^{-1}, h_2(u)$ to the rightmost part,
thus deriving (cf.~\eqref{eq:Dr-osp-rel-3-spec-alt}):
\begin{equation}\label{eq:k-e-osp12-alternative}
  [h_1(u)^{-1}h_2(u),e_{12}(v)] =
  \left(-\frac{1}{3}\, \frac{e_{12}(u+1/2)-e_{12}(v)}{u-v+1/2}
        -\frac{2}{3}\, \frac{e_{12}(u-1)-e_{12}(v)}{u-v-1}\right) h_1(u)^{-1}h_2(u) \,.
\end{equation}

\medskip
Let us also comment on the commutation formulas~(\ref{eq:Dr-osp-rel-4},~\ref{eq:Dr-osp-rel-4-spec})
between $X^\pm_i(u)$ and $X^\pm_j(v)$ for $i,j\in \{n+m,n+m-1\}$. For $i=j=n+m$ with $N=2n$, the result
follows from the commutator formulas~(\ref{eq:Atype-ee-1},~\ref{eq:Atype-ff-1}) through
Corollaries~\ref{cor:A-type relations},~\ref{cor:other-A-type relations}, see also
Remark~\ref{rem:relations-via-series}(e). For $i=j=n+m$, $N=2n+1$, $\ol{n+m}=\bar{0}$, the relations follow
from the similar relations~(\ref{eq:osp30-6},~\ref{eq:osp30-7}) in the rank $1$ case of $\fosp(3|0)$.
Likewise, for $i=j=n+m$, $N=2n+1$, $\ol{n+m}=\bar{1}$, the relation~\eqref{eq:Dr-osp-rel-4-spec} follows
from the similar relations~(\ref{eq:osp12-6},~\ref{eq:osp12-7}) in the rank $1$ case of $\fosp(1|2)$.
Finally, verification of~\eqref{eq:Dr-osp-rel-4} for $\{i,j\}=\{n+m-1,n+m\}$ reduces to the rank $2$ cases.
Unless $N=2n$ and $\ol{n+m}=\bar{0}$, the corresponding relations always had the form:
\begin{equation*}
\begin{split}
  & [e_{12}(u),e_{23}(v)]=\frac{\sharp}{u-v}\Big(e_{13}(u)-e_{13}(v)-e_{12}(u)e_{23}(v)+e_{12}(v)e_{23}(v)\Big) \,, \\
  & [f_{21}(u),f_{32}(v)]=\frac{\sharp}{u-v}\Big(f_{31}(v)-f_{31}(u)+f_{32}(v)f_{21}(u)-f_{32}(v)f_{21}(v)\Big) \,,
\end{split}
\end{equation*}
with $\sharp\in \{-1,1,2\}$. These relations imply~\eqref{eq:Dr-osp-rel-4}: this is explained
in~\cite[End of \S5]{bk} for $\sharp=-1$. If $N=2n$, $\ol{n+m}=\bar{0}$, $\ol{n+m-1}=\bar{0}$,
then~\eqref{eq:Dr-osp-rel-4} follows from~\eqref{eq:osp40-5}. In the remaining case $N=2n$,
$\ol{n+m}=\bar{0}$, $\ol{n+m-1}=\bar{1}$, the relation~\eqref{eq:Dr-osp-rel-4} follows in turn
from~(\ref{eq:osp221-4simplify},~\ref{eq:osp221-5simplify}).

\medskip
Combining the fact that the coefficients of $\{e_i(u),f_i(u),h_\imath(u)\}_{1\leq i\leq n+m}^{1\leq \imath\leq n+m+1}$
generate $X^\rtt(\fosp(V))$ with the tensor product decomposition~\eqref{eq:extended-vs-nonextended}, description of
the center $ZX^\rtt(\fosp(V))$, and the factorization of the central generating series $c_V(u)$ from
Lemmas~\ref{lem:cseries-evenN-even},~\ref{lem:cseries-evenN-odd},~\ref{lem:cseries-oddN}, we conclude
that the homomorphism $\Upsilon$ is surjective. The injectivity of $\Upsilon$ follows from the injectivity
of~\eqref{eq:Upsilon-extended}.

Alternatively, one can use~\eqref{eq:extended-vs-nonextended} and identify $Y(\fosp(V))$ with the preimage
of $Y^\rtt(\fosp(V))$ under~\eqref{eq:Upsilon-extended}. This amounts to checking that the subalgebra of
$X(\fosp(V))$ generated by the same-named currents \eqref{eq:xk-series-1}--\eqref{eq:xk-series-last-k} is
isomorphic to $Y(\fosp(V))$ defined via generators and relations.
\end{proof}

\begin{Rem}
The Serre relations~\eqref{eq:Dr-osp-rel-5.1}--\eqref{eq:Dr-osp-rel-9} can be generalized exactly as
in Remark~\ref{rem:Serre-extosp-series-general}.
\end{Rem}


%
%
%

\newpage
\appendix
\section{Low rank identification through 6-fold fusion}\label{sec:6-fold fusion}


For $m=0$ (respectively, $N=0$), our straightforward treatment of the corresponding RTT orthogonal
(respectively, symplectic) Yangians is slightly different from the one in~\cite{jlm}. More specifically,
the arguments of~\cite{jlm} crucially utilized (see the proof of~\cite[Proposition 5.4]{jlm}) the low level
isomorphisms established in~\cite[Section~4]{amr}. The aim of this appendix is thus twofold. Starting from
the \emph{6-fold $R$-matrix fusion} argument of~\cite{amr}, used to explicitly construct isomorphisms
  $X^\rtt(\fso_3)\simeq Y^\rtt(\gl_2)$ and $Y^\rtt(\fso_3)\simeq Y^\rtt(\ssl_2)$,
we construct analogous isomorphisms\footnote{These isomorphisms are known to experts, but we did not
find explicit RTT-type realizations in the literature.}
  $X^\rtt(\fso_6)\simeq Y^\rtt(\gl_4)$ and $Y^\rtt(\fso_6)\simeq Y^\rtt(\ssl_4)$.
Finally, we explain why applying this approach to $Y^\rtt(\gl(1|2))$ recovers an algebra that looks
surprisingly different\footnote{We thank A.~Molev who noted that there is actually an algebra isomorphism
$X(\fosp(2|2))\simeq Y(\gl(1|2))$ between the Drinfeld realizations of these Yangians, which however
does not admit any nice RTT-type interpretation.} from $X^\rtt(\fosp(2|2))$.

\medskip
\noindent
$\bullet$ $\fso_3$ vs $\gl_2$.

Consider the Yangian $Y^\rtt(\gl_2)=Y^\rtt(\gl(\BC^2))$ associated with the $R$-matrix
$\sfR(u)=\ID-\frac{\Pop}{u}$, where $\Pop\in \End\, (\BC^2\otimes \BC^2)$ is the permutation operator.
Here, we choose a basis $\{\sfv_1,\sfv_2\}$ of $\BC^2$ and use $\sfT(u)$ to denote the corresponding
$2\times 2$ generator matrix of $Y^\rtt(\gl_2)$, see Subsection~\ref{ssec:RTT gl-Yangian}.

The symmetric square $V=S^2(\BC^2)=\sfR(-1)(\BC^2\otimes \BC^2)$ has a basis
  $$v_1=\sfv_1\otimes \sfv_1 \,,\quad
    v_2=\sfrac{1}{\sqrt{2}}(\sfv_1\otimes \sfv_2+\sfv_2\otimes \sfv_1) \,,\quad
    v_3=-\sfv_2\otimes \sfv_2 \,.$$
Let $X^\rtt(\fso_3)$ be the corresponding RTT extended orthogonal Yangian of
Subsection~\ref{ssec:RTT osp-Yangian}. Here, $N=3,m=0$, $\kappa=1/2$, $\theta_1=\theta_2=\theta_3=1$,
$P,Q$ are as in~(\ref{eq:P},~\ref{eq:Q}), and $R(u)$ is defined in~\eqref{eq:osp-Rmatrix}.

\begin{Rem}\label{rem:so3=gl2-Lie}
The above choice of $V$, its basis $\{v_1,v_2,v_3\}$, and the key RTT-type construction of
Proposition~\ref{prop:so3=gl2} are all crucially based on the following two simple observations:
\begin{enumerate}

\item[(a)]
the assignment $e\mapsto \sqrt{2}F_{12}, f\mapsto \sqrt{2}F_{21}, h\mapsto 2F_{11}$, where $\{h,e,f\}$
denotes the standard basis of $\ssl_2$ and $F_{ij}$ are as in~\eqref{eq:F-elements}, gives rise to
a Lie algebra isomorphism $\rho\colon \ssl_2\iso \fso_3$;

\item[(b)]
the vector space isomorphism $\rho\colon S^2(\BC^2)\iso \BC^3$ mapping $v_1,v_2,v_3$ to the
standard basis of $\BC^3$  is compatible with the above Lie algebra isomorphism, that is:
$\rho(x(v))=\rho(x)(\rho(v))$.

\end{enumerate}
\end{Rem}

Consider the tensor product space $(\BC^2)^{\otimes 4}$, and we shall view $V\otimes V$ as a natural
subspace of $(\BC^2)^{\otimes 2}\otimes (\BC^2)^{\otimes 2}=(\BC^2)^{\otimes 4}$. Moreover, the operator
  $\frac{1+\Pop_{12}}{2}\cdot \frac{1+\Pop_{34}}{2}=\frac{1}{4}\sfR_{12}(-1)\sfR_{34}(-1)$
defines a projection of $(\BC^2)^{\otimes 2}\otimes (\BC^2)^{\otimes 2}$ onto this subspace
$V\otimes V$. Let us consider the following
\begin{equation}\label{eq:RV-of-amr}
  \mathrm{\textbf{6-fold\ fusion}} \quad
  R_V(u):=\frac{1+\Pop_{12}}{2}\cdot \frac{1+\Pop_{34}}{2}\cdot
  \sfR_{14}(2u-1) \sfR_{13}(2u) \sfR_{24}(2u) \sfR_{23}(2u+1) \,,
\end{equation}
which can be equivalently written as
\begin{equation*}
  R_V(u) =
  \sfR_{23}(2u+1) \sfR_{13}(2u) \sfR_{24}(2u) \sfR_{14}(2u-1) \cdot
  \frac{1+\Pop_{12}}{2}\cdot \frac{1+\Pop_{34}}{2} \,,
\end{equation*}
since the $R$-matrix $\sfR(u)$ satisfies the Yang–Baxter equation~\eqref{eq:YBE-gl}. The subspace
$V\otimes V$ is clearly stable under the operator $R_V(u)$. The following observation first appeared
in~\cite[Lemma 4.5]{amr}:

\begin{Lem}\label{lem:amr-fusion}
We have the equality of operators in $V\otimes V$:
\begin{equation}\label{eq:Rfused-1}
  R_V(u) =
  \frac{2u-1}{2u+1}\cdot \left(\ID - \frac{P}{u} + \frac{Q}{u-1/2} \right) = \frac{2u-1}{2u+1}\cdot R(u) \,.
\end{equation}
\end{Lem}

Thus, $R_V(u)\in \End\,V\otimes\End\,V$ coincides with the $R$-matrix $R(u)$ for $\fso_3=\fso(V)$, up to a scalar
factor. Combining this result with the repeated application of the defining RTT relation~\eqref{eq:RTT relation-gl}
and the PBW theorem for $X^\rtt(\fso_3)$, one easily obtains~\cite[Proposition 4.4, Corollary 4.6]{amr}:

\begin{Prop}\label{prop:so3=gl2}
(a) The assignment
\begin{equation*}
  T(u)\mapsto
  \frac{1+\Pop}{2}\cdot \sfT_1(2u) \sfT_2(2u+1) = \sfT_2(2u+1) \sfT_1(2u) \cdot \frac{1+\Pop}{2}
\end{equation*}
gives rise to an algebra isomorphism $\phi\colon X^\rtt(\fso_3) \iso Y^\rtt(\gl_2)$.

\medskip
\noindent
(b) The restriction of the isomorphism from (a) to the subalgebra $Y^\rtt(\fso_3)$ of $X^\rtt(\fso_3)$
gives rise to an algebra isomorphism $\phi\colon Y^\rtt(\fso_3) \iso Y^\rtt(\ssl_2)$.
\end{Prop}

We refer the interested reader to~\cite{amr} for more details and the explicit formulas for $\phi(t_{ij}(u))$.

\medskip
\noindent
$\bullet$ $\fso_6$ vs $\gl_4$.

Consider the Yangian $Y^\rtt(\gl_4)=Y^\rtt(\gl(\BC^4))$ associated with the $R$-matrix
$\sfR(u)=\ID-\frac{\Pop}{u}$ of~\eqref{eq:gl-Rmatrix}. Here, we apply the construction
of Subsection~\ref{ssec:RTT gl-Yangian} to $\VV=\BC^4$, and fix its specific basis
$\{\sfv_1, \sfv_2, \sfv_3,\sfv_4\}$. We shall use $\sfT(u)$ to denote the corresponding
$4\times 4$ generator matrix of $Y^\rtt(\gl_4)$.

The second exterior power $V=\Lambda^2(\BC^4)=\sfR(1)(\BC^4\otimes \BC^4)$ has a basis
\begin{equation}\label{eq:6-basis}
  v_1=\sfv_1 \wedge \sfv_2 \,,\, v_2=\sfv_1 \wedge \sfv_3 \,,\, v_3=\sfv_2 \wedge \sfv_3 \,,\,
  v_4=\sfv_1 \wedge \sfv_4 \,,\, v_5=\sfv_4 \wedge \sfv_2 \,,\, v_6=\sfv_3 \wedge \sfv_4 \,.
\end{equation}
Let $X^\rtt(\fso_6)$ be the corresponding RTT extended orthogonal Yangian of
Subsection~\ref{ssec:RTT osp-Yangian}. Here, $N=6,m=0$, $\kappa=2$, $\theta_1=\dots=\theta_6=1$,
$P,Q$ are as in~(\ref{eq:P},~\ref{eq:Q}), and $R(u)$ is defined in~\eqref{eq:osp-Rmatrix}.

\begin{Rem}\label{rem:so6=gl4-Lie}
The above choice of $V$, its basis $\{v_k\}_{k=1}^6$, and the key RTT-type construction of
Proposition~\ref{prop:so6=gl4} are all crucially based on the following two simple observations:
\begin{enumerate}

\item[(a)]
the assignment $E_{12}\mapsto F_{23}, E_{23}\mapsto F_{12}, E_{34}\mapsto F_{24}$,
$E_{21}\mapsto F_{32}, E_{32}\mapsto F_{21}, E_{43}\mapsto F_{42}$, with $F_{ij}\in \gl(V)$
from~\eqref{eq:F-elements}, gives rise to a Lie algebra isomorphism $\rho\colon \ssl_4\iso \fso_6$;

\item[(b)]
the vector space isomorphism $\rho\colon \Lambda^2(\BC^4)\iso \BC^6$ mapping $v_1,\ldots,v_6$
to the standard basis of $\BC^6$  is compatible with the above Lie algebra isomorphism, that is:
$\rho(x(v))=\rho(x)(\rho(v))$.

\end{enumerate}
\end{Rem}

Consider the tensor product space $(\BC^4)^{\otimes 4}$, and we shall view $V\otimes V$ as
a natural subspace of $(\BC^4)^{\otimes 2}\otimes (\BC^4)^{\otimes 2}=(\BC^4)^{\otimes 4}$.
Moreover, the operator
  $\frac{1-\Pop_{12}}{2}\cdot \frac{1-\Pop_{34}}{2}=\frac{1}{4}\sfR_{12}(1)\sfR_{34}(1)$
defines a projection of $(\BC^4)^{\otimes 2}\otimes (\BC^4)^{\otimes 2}$ onto this subspace
$V\otimes V$. Let us consider the following
\begin{equation}\label{eq:RV-new}
  \mathrm{\textbf{6-fold fusion}} \quad
  R_V(u):=\frac{1-\Pop_{12}}{2}\cdot \frac{1-\Pop_{34}}{2}\cdot
  \sfR_{14}(u+1) \sfR_{13}(u) \sfR_{24}(u) \sfR_{23}(u-1) \,,
\end{equation}
which can be equivalently written as
\begin{equation*}
  R_V(u) =
  \sfR_{23}(u-1) \sfR_{13}(u) \sfR_{24}(u) \sfR_{14}(u+1) \cdot
  \frac{1-\Pop_{12}}{2}\cdot \frac{1-\Pop_{34}}{2} \,,
\end{equation*}
since the $R$-matrix $\sfR(u)$ satisfies the Yang–Baxter equation~\eqref{eq:YBE-gl}.
The subspace $V\otimes V$ is clearly stable under the operator $R_V(u)$.
The following result is analogous to Lemma~\ref{lem:amr-fusion}:

\begin{Lem}\label{lem:new-fusion}
We have the equality of operators in $V\otimes V$:
\begin{equation}\label{eq:Rfused-2}
  R_V(u) =
  \frac{u-2}{u-1}\cdot \left(\ID - \frac{P}{u} + \frac{Q}{u-2} \right) = \frac{u-2}{u-1}\cdot R(u) \,.
\end{equation}
\end{Lem}

\begin{proof}
Straightforward computation.
\end{proof}

Thus, $R_V(u)\in \End\,V\otimes\End\,V$ coincides with the $R$-matrix $R(u)$ for $\fso_6=\fso(V)$, up to a scalar
factor. Combining this with the repeated application of the defining RTT relation~\eqref{eq:RTT relation-gl}
and the PBW theorem for $X^\rtt(\fso_6)$, one obtains the following analogue of Proposition~\ref{prop:so3=gl2}:

\begin{Prop}\label{prop:so6=gl4}
(a) The assignment
\begin{equation*}
  T(u)\mapsto
  \frac{1-\Pop}{2}\cdot \sfT_1(u+1) \sfT_2(u) = \sfT_2(u) \sfT_1(u+1) \cdot \frac{1-\Pop}{2}
\end{equation*}
gives rise to an algebra isomorphism $\phi\colon X^\rtt(\fso_6) \iso Y^\rtt(\gl_4)$.

\medskip
\noindent
(b) The restriction of the isomorphism from (a) to the subalgebra $Y^\rtt(\fso_6)$ of $X^\rtt(\fso_6)$
gives rise to an algebra isomorphism $\phi\colon Y^\rtt(\fso_6) \iso Y^\rtt(\ssl_4)$.
\end{Prop}

\begin{Rem}
(a) As for any $f(u)\in 1+u^{-1}\BC[[u^{-1}]]$ there exists $g(u)\in 1+u^{-1}\BC[[u^{-1}]]$
satisfying $f(u)=g(u)g(u+1)$, we have $\mu_g\circ \phi=\phi\circ \mu_f$, so that part (b)
follows immediately from part (a).

\medskip
\noindent
(b) Combining $\phi$ of Proposition~\ref{prop:so6=gl4}(b) with the evaluation homomorphism
$Y^\rtt(\ssl_4)\twoheadrightarrow U(\ssl_4)$ (given by
  $\sft_{ij}(u)\mapsto \delta_{ij}+(E_{ij}-\delta_{ij}\frac{E_{11}+E_{22}+E_{33}+E_{44}}{4})u^{-1}$)
and the isomorphism $U(\ssl_4)\simeq U(\fso_6)$ of Remark~\ref{rem:so6=gl4-Lie}(a), we obtain
an algebra epimorphism $Y^\rtt(\fso_6)\twoheadrightarrow U(\fso_6)$, cf.~\cite[Corollary~4.7]{amr}.

\medskip
\noindent
(c) The images $\phi(t_{k\ell}(u))$ can be explicitly described as follows:
\begin{equation*}
  \phi(t_{k\ell}(u))=
  \frac{1}{2} \Big( \sft_{ap}(u+1)\sft_{bq}(u) - \sft_{aq}(u+1)\sft_{bp}(u) -
  \sft_{bp}(u+1)\sft_{aq}(u) + \sft_{bq}(u+1)\sft_{ap}(u) \Big) \,,
\end{equation*}
for unique indices $1\leq a,b,p,q\leq 4$ satisfying $v_k=\sfv_a\wedge \sfv_b$ and $v_\ell=\sfv_p\wedge \sfv_q$,
see~\eqref{eq:6-basis}.
\end{Rem}

\medskip
\noindent
$\bullet$ $\fosp(2|2)$ vs $\gl(1|2)$.

Consider a superspace $\VV=\BC^{1|2}$ with a basis $\{\sfv_1, \sfv_2, \sfv_3\}$ whose parity is
$|\sfv_1|=\bar{1}, |\sfv_2|=\bar{0}, |\sfv_3|=\bar{1}$. Let $Y^\rtt(\gl(\BC^{1|2}))$ be the corresponding
RTT Yangian associated with the $R$-matrix $\sfR(u)=\ID-\frac{\Pop}{u}$ and let $\sfT(u)$ denote the
corresponding $3\times 3$ generator matrix of $Y^\rtt(\gl(\VV))$, see Subsection~\ref{ssec:RTT gl-Yangian}.

We note that the $\gl(\BC^{1|2})$-module\footnote{Recall that in super-case the action on the tensor product
is given by $x(v\otimes w)=x(v)\otimes w + (-1)^{|x|\cdot |v|}v\otimes x(w)$.} $\BC^{1|2}\otimes \BC^{1|2}$
decomposes into the direct sum of 4-dimensional $S^2(\BC^{1|2})=\sfR(-1)(\BC^{1|2}\otimes \BC^{1|2})$
and 5-dimensional $\Lambda^2(\BC^{1|2})=\sfR(1)(\BC^{1|2}\otimes \BC^{1|2})$ submodules.
The symmetric square $V=S^2(\VV)=S^2(\BC^{1|2})$ has a basis
  $$v_1=\sfv_1 \otimes \sfv_2 + \sfv_1 \otimes \sfv_2 \,,\quad
    v_2=\sfv_2 \otimes \sfv_2 \,,\quad
    v_3=\sfv_1 \otimes \sfv_3 - \sfv_3 \otimes \sfv_1 \,,\quad
    v_4=\sfv_2 \otimes \sfv_3 + \sfv_3 \otimes \sfv_2 \,,$$
with a parity $|v_1|=|v_4|=\bar{1}, |v_2|=|v_3|=\bar{0}$. Let $X^\rtt(\fosp(V))$ be the corresponding
RTT extended orthosymplectic Yangian of Subsection~\ref{ssec:RTT osp-Yangian}. Here, $N=2,m=1$, $\kappa=-1$
by~\eqref{eq:kappa}, $\theta_1=\theta_2=\theta_3=1, \theta_4=-1$ according to~\eqref{eq:theta},
$P,Q$ are as in~(\ref{eq:P},~\ref{eq:Q}), and $R(u)$ is as in~\eqref{eq:osp-Rmatrix}.

\begin{Rem}\label{rem:sop22=gl12-Lie}
(a) The Dynkin diagram of $\ssl(\BC^{1|2})=A(\BC^{1|2})$ is
  $\begin{tikzpicture}
   \foreach \a in {1} {
     \begin{scope}[shift={(0.7*\a,0)}]
       \draw[fill=gray] (0.3*\a,0) circle (0.3cm);
       \draw[black,thick] (0.3*\a+0.3,0)--(0.3*\a+2.2,0);
     \end{scope}
    }
      \draw[fill=gray] (3,0) circle (0.3cm);
   \end{tikzpicture}$
which coincides with the Dynkin diagram of $\fosp(V)$ for the parity sequence $\Parity=(\bar{1},\bar{0})$,
see Subsection~\ref{ssec:Dynkin diagrams}. Therefore, one has an abstract isomorphism of Lie superalgebras
$\ssl(\VV)\simeq \fosp(V)$.

\medskip
\noindent
(b) The assignment
\begin{equation*}
\begin{split}
  & E_{12}\mapsto \sfrac{1}{\sqrt{2}}F_{12} \,,\ \
    E_{23}\mapsto \sfrac{1}{\sqrt{2}}F_{13} \,,\ \
    E_{13}\mapsto \sfrac{1}{2}F_{14} \,, \\
  & E_{21}\mapsto \sfrac{1}{\sqrt{2}}F_{21} \,,\ \
    E_{32}\mapsto -\sfrac{1}{\sqrt{2}}F_{31} \,,\ \
    E_{31}\mapsto \sfrac{1}{2}F_{41} \,, \\
  & E_{11}+E_{22} \mapsto \sfrac{1}{2}(F_{11}+F_{22}) \,,\ \
    E_{22}+E_{33} \mapsto -\sfrac{1}{2}(F_{11}-F_{22}) \,,
\end{split}
\end{equation*}
with $F_{ij}\in \gl(V)$ of~\eqref{eq:F-elements}, gives rise to a Lie superalgebra isomorphism
$\rho\colon \ssl(\VV) \iso \fosp(V)$, cf.~(a).

\medskip
\noindent
(c) However, in contrast to Remarks~\ref{rem:so3=gl2-Lie}(b),~\ref{rem:so6=gl4-Lie}(b), there is
\underline{no isomorphism} between $\ssl(\VV)$-module $S^2(\VV)$ and the natural $\fosp(V)$-module $V$,
intertwined by the isomorphism $\rho$ from part~(b).

\medskip
\noindent
(d) According to~\cite{ma}, the Lie superalgebra $\ssl(\BC^{1|2})$ admits a $1$-parameter family of
non-isomorphic 4-dimensional modules, denoted by $[b,1/2]$. The generators $S_\pm, V_\pm, \ol{V}_\pm$
of~\cite[\S2.1]{ma} may be related to ours via:
\begin{equation}\label{eq:ma-to-us}
  V_+ \leftrightarrow \sfrac{1}{\sqrt{2}} E_{12} \,,\ \ol{V}_+ \leftrightarrow \sfrac{1}{\sqrt{2}} E_{23} \,,\
  V_- \leftrightarrow \sfrac{1}{\sqrt{2}} E_{32} \,,\ \ol{V}_- \leftrightarrow -\sfrac{1}{\sqrt{2}} E_{21} \,,\
  S_+ \leftrightarrow E_{13} \,,\, S_- \leftrightarrow E_{31} \,.
\end{equation}
The explicit action of $\ssl(\BC^{1|2})$ on $[b,1/2]$ is provided in~\cite[\S4.1]{ma}.
In particular, combining~\cite[(21,~22)]{ma} with~\eqref{eq:ma-to-us}, the lower-triangular generators
can be represented by the following matrices:
\begin{equation*}
  E_{21}\mapsto
  \left(\begin{array}{cccc}
        0 & 0 & 0 & 0 \\
        \sqrt{2}\beta & 0 & 0 & 0 \\
        0 & 0 & 0 & 0 \\
        0 & 0 & -\sqrt{2}\gamma & 0
       \end{array}\right) \,,\
  E_{32}\mapsto
  \left(\begin{array}{cccc}
        0 & 0 & 0 & 0 \\
        0 & 0 & 0 & 0 \\
        -\sqrt{2}\alpha & 0 & 0 & 0 \\
        0 & \sqrt{2}\epsilon & 0 & 0
       \end{array}\right) \,,\
  E_{31}\mapsto
  \left(\begin{array}{cccc}
        0 & 0 & 0 & 0 \\
        0 & 0 & 0 & 0 \\
        0 & 0 & 0 & 0 \\
        1 & 0 & 0 & 0
       \end{array}\right) \,,
\end{equation*}
with the constants $\alpha,\beta,\gamma,\epsilon$ satisfying $4\alpha\gamma=1+2b, 4\beta\epsilon=1-2b$.
It is now straightforward to check that the $4$-dimensional $\ssl(\VV)$-module $S^2(\VV)$ corresponds
to $b=-3/2$, while the pull-back of the $4$-dimensional $\fosp(V)$-module $V$ under the isomorphism
$\rho$ of part (b) corresponds to $b=0$.
\end{Rem}

Consider the tensor product space $(\BC^{1|2})^{\otimes 4}$. We shall view $V\otimes V$ as a natural
subspace of $(\BC^{1|2})^{\otimes 2}\otimes (\BC^{1|2})^{\otimes 2}=(\BC^{1|2})^{\otimes 4}$, while
the operator $\frac{1+\Pop_{12}}{2}\cdot \frac{1+\Pop_{34}}{2}=\frac{1}{4}\sfR_{12}(-1)\sfR_{34}(-1)$
defines a projection of $(\BC^{1|2})^{\otimes 2}\otimes (\BC^{1|2})^{\otimes 2}$ onto this subspace
$V\otimes V$. Similarly to~\eqref{eq:RV-of-amr}, we consider
\begin{multline}\label{eq:fused-osp}
  \mathrm{\textbf{6-fold fusion}}  \quad
  R_V(u):=\frac{1+\Pop_{12}}{2}\cdot \frac{1+\Pop_{34}}{2}\cdot
  \sfR_{14}(u-1) \sfR_{13}(u) \sfR_{24}(u) \sfR_{23}(u+1) = \\
  \sfR_{23}(u+1) \sfR_{13}(u) \sfR_{24}(u) \sfR_{14}(u-1) \cdot \frac{1+\Pop_{12}}{2}\cdot \frac{1+\Pop_{34}}{2} \,.
\end{multline}
The subspace $V\otimes V$ is clearly stable under $R_V(u)$. Moreover, this operator satisfies
the Yang-Baxter equation according to our next result:

\begin{Lem}
The operator $R_V(u)\in \End\,V \otimes \End\,V$ satisfies the Yang-Baxter equation~\eqref{eq:YB intro}.
\end{Lem}

\begin{proof}
First, let us note the following equalities of operators in $(\End\, \BC^{1|2})^{\otimes 4}$:
\begin{equation}\label{eq:projections}
\begin{split}
   &  \frac{1+\Pop_{12}}{2}\cdot \frac{1+\Pop_{34}}{2}\cdot P_{14}P_{13} =
      \frac{1+\Pop_{12}}{2}\cdot \frac{1+\Pop_{34}}{2}\cdot P_{14} \,, \\
   &  \frac{1+\Pop_{12}}{2}\cdot \frac{1+\Pop_{34}}{2}\cdot P_{14}P_{24} =
      \frac{1+\Pop_{12}}{2}\cdot \frac{1+\Pop_{34}}{2}\cdot P_{14} \,, \\
   &  \frac{1+\Pop_{12}}{2}\cdot \frac{1+\Pop_{34}}{2}\cdot P_{14}P_{23}=
      \frac{1+\Pop_{12}}{2}\cdot \frac{1+\Pop_{34}}{2}\cdot P_{13}P_{24} \,, \\
   &  \frac{1+\Pop_{12}}{2}\cdot \frac{1+\Pop_{34}}{2}\cdot P_{13}P_{23}=
      \frac{1+\Pop_{12}}{2}\cdot \frac{1+\Pop_{34}}{2}\cdot P_{13} \,, \\
   &  \frac{1+\Pop_{12}}{2}\cdot \frac{1+\Pop_{34}}{2}\cdot P_{24}P_{23}=
      \frac{1+\Pop_{12}}{2}\cdot \frac{1+\Pop_{34}}{2}\cdot P_{24} \,.
\end{split}
\end{equation}
Using~\eqref{eq:projections}, we obtain the following simplified formula for $R_V(u)$ of~\eqref{eq:fused-osp}:
\begin{equation}\label{eq:RV-simplified-1}
   R_V(u)=\frac{1+\Pop_{12}}{2}\cdot \frac{1+\Pop_{34}}{2}\cdot
  \left(1-\frac{P_{14}+P_{24}+P_{13}+P_{23}}{u+1}+\frac{2P_{13}P_{24}}{u(u+1)}\right) \,.
\end{equation}
Therefore, the restriction of $R_V(u)$ to $V\otimes V$ is simply given by:
\begin{equation}\label{eq:RV-simplified-2}
   R_V(u)=1-\frac{P_{14}+P_{24}+P_{13}+P_{23}}{u+1}+\frac{2P_{13}P_{24}}{u(u+1)} \,,
\end{equation}
cf.~\cite[(4.21)]{amr}.

\medskip
Using the formula~\eqref{eq:RV-simplified-2}, it is easy now to compute the corresponding
$16\times 16$ matrix for the action of $R_V(u)$ in the ordered basis
$\{v_1\otimes v_1, v_1\otimes v_2, \ldots, v_4\otimes v_3, v_4\otimes v_4\}$ of $V\otimes V$:
\begin{equation}\label{eq:fusedmatrix-explicit}
  R_V(u) =
  \left(\begin{smallmatrix}
    a(u) & 0 & 0 & 0 & 0 & 0 & 0 & 0 & 0 & 0 & 0 & 0 & 0 & 0 & 0 & 0 \\
    0 & b(u) & 0 & 0 & c(u) & 0 & 0 & 0 & 0 & 0 & 0 & 0 & 0 & 0 & 0 & 0 \\
    0 & 0 & d(u) & 0 & 0 & 0 & 0 & 0 & e(u) & 0 & 0 & 0 & 0 & 0 & 0 & 0 \\
    0 & 0 & 0 & f(u) & 0 & 0 & -k(u) & 0 & 0 & -k(u) & 0 & 0 & g(u) & 0 & 0 & 0 \\
    0 & c(u) & 0 & 0 & b(u) & 0 & 0 & 0 & 0 & 0 & 0 & 0 & 0 & 0 & 0 & 0 \\
    0 & 0 & 0 & 0 & 0 & i(u) & 0 & 0 & 0 & 0 & 0 & 0 & 0 & 0 & 0 & 0 \\
    0 & 0 & 0 & h(u) & 0 & 0 & 1 & 0 & 0 & j(u) & 0 & 0 & -h(u) & 0 & 0 & 0 \\
    0 & 0 & 0 & 0 & 0 & 0 & 0 & b(u) & 0 & 0 & 0 & 0 & 0 & c(u) & 0 & 0 \\
    0 & 0 & e(u) & 0 & 0 & 0 & 0 & 0 & d(u) & 0 & 0 & 0 & 0 & 0 & 0 & 0 \\
    0 & 0 & 0 & h(u) & 0 & 0 & j(u) & 0 & 0 & 1 & 0 & 0 & -h(u) & 0 & 0 & 0 \\
    0 & 0 & 0 & 0 & 0 & 0 & 0 & 0 & 0 & 0 & l(u) & 0 & 0 & 0 & 0 & 0 \\
    0 & 0 & 0 & 0 & 0 & 0 & 0 & 0 & 0 & 0 & 0 & d(u) & 0 & 0 & e(u) & 0 \\
    0 & 0 & 0 & g(u) & 0 & 0 & k(u) & 0 & 0 & k(u) & 0 & 0 & f(u) & 0 & 0 & 0 \\
    0 & 0 & 0 & 0 & 0 & 0 & 0 & c(u) & 0 & 0 & 0 & 0 & 0 & b(u) & 0 & 0 \\
    0 & 0 & 0 & 0 & 0 & 0 & 0 & 0 & 0 & 0 & 0 & e(u) & 0 & 0 & d(u) & 0 \\
    0 & 0 & 0 & 0 & 0 & 0 & 0 & 0 & 0 & 0 & 0 & 0 & 0 & 0 & 0 & a(u)
  \end{smallmatrix}\right)
\end{equation}
where
\begin{equation}\label{eq:fusedmatrix-entries}
\begin{split}
  & a(u)=\frac{(u-1)(u+2)}{u(u+1)} \,, \quad
    c(u)=\frac{-2(u-1)}{u(u+1)} \,, \quad
    e(u)=\frac{u+2}{u(u+1)} \,, \\
  & g(u)=\frac{u-2}{u(u+1)} \,, \quad
    i(u)=\frac{(u-1)(u-2)}{u(u+1)} \,, \quad
    j(u)=\frac{2}{u(u+1)} \,, \\
  & b(u)=\frac{u-1}{u+1} \,, \quad
    d(u)=\frac{u+2}{u+1} \,, \quad
    f(u)=\frac{u}{u+1} \,, \\
  & h(u)=\frac{-2}{u+1} \,, \quad
    k(u)=\frac{1}{u+1} \,, \quad
    l(u)=\frac{u+2}{u} \,.
\end{split}
\end{equation}

Finally, we have verified on the computer that the above
matrix~(\ref{eq:fusedmatrix-explicit},~\ref{eq:fusedmatrix-entries}) indeed satisfies
the Yang-Baxter equation~\eqref{eq:YB intro}.
\end{proof}

In view of Remark~\ref{rem:sop22=gl12-Lie}(c,d), it is not surprising that $R_V(u)$ is
\underline{not} a scalar multiple of the orthosymplectic $R$-matrix $R(au)$ of $\fosp(2|2)=\fosp(V)$
for any $a\in \BC$, in contrast to Lemma~\ref{lem:amr-fusion}.

\begin{Rem}\label{rem:comparison-to-RM}
Let us match both the $6$-fold fusion $R_V(u)$ and the orthosymplectic $R$-matrix $R(u)$ with the
special cases of the $R$-matrix from~\cite{rm}. We use $\check R_{RM}(u,b)$ to denote the $16\times 16$
matrix of~\cite[(2)]{rm}, which at $u=0$ reduces to the identity and not to the permutation operator.

\medskip
\noindent
(a) We have
\begin{equation}
  \frac{u(u+1)}{(u-1)(u+2)}\, R_V(u)= S\, \check R_{RM}(u,-\sfrac{3}{2})\, PS^{-1}
\end{equation}
with
\begin{equation}\label{eq:P-app}
  P =
  \left(\begin{smallmatrix}
    -1 & 0 & 0 & 0 & 0 & 0 & 0 & 0 & 0 & 0 & 0 & 0 & 0 & 0 & 0 & 0 \\
    0 & 0 & 0 & 0 & 1 & 0 & 0 & 0 & 0 & 0 & 0 & 0 & 0 & 0 & 0 & 0 \\
    0 & 0 & 0 & 0 & 0 & 0 & 0 & 0 & 1 & 0 & 0 & 0 & 0 & 0 & 0 & 0 \\
    0 & 0 & 0 & 0 & 0 & 0 & 0 & 0 & 0 & 0 & 0 & 0 & -1 & 0 & 0 & 0 \\
    0 & 1 & 0 & 0 & 0 & 0 & 0 & 0 & 0 & 0 & 0 & 0 & 0 & 0 & 0 & 0 \\
    0 & 0 & 0 & 0 & 0 & 1 & 0 & 0 & 0 & 0 & 0 & 0 & 0 & 0 & 0 & 0 \\
    0 & 0 & 0 & 0 & 0 & 0 & 0 & 0 & 0 & 1 & 0 & 0 & 0 & 0 & 0 & 0 \\
    0 & 0 & 0 & 0 & 0 & 0 & 0 & 0 & 0 & 0 & 0 & 0 & 0 & 1 & 0 & 0 \\
    0 & 0 & 1 & 0 & 0 & 0 & 0 & 0 & 0 & 0 & 0 & 0 & 0 & 0 & 0 & 0 \\
    0 & 0 & 0 & 0 & 0 & 0 & 1 & 0 & 0 & 0 & 0 & 0 & 0 & 0 & 0 & 0 \\
    0 & 0 & 0 & 0 & 0 & 0 & 0 & 0 & 0 & 0 & 1 & 0 & 0 & 0 & 0 & 0 \\
    0 & 0 & 0 & 0 & 0 & 0 & 0 & 0 & 0 & 0 & 0 & 0 & 0 & 0 & 1 & 0 \\
    0 & 0 & 0 & -1 & 0 & 0 & 0 & 0 & 0 & 0 & 0 & 0 & 0 & 0 & 0 & 0 \\
    0 & 0 & 0 & 0 & 0 & 0 & 0 & 1 & 0 & 0 & 0 & 0 & 0 & 0 & 0 & 0 \\
    0 & 0 & 0 & 0 & 0 & 0 & 0 & 0 & 0 & 0 & 0 & 1 & 0 & 0 & 0 & 0 \\
    0 & 0 & 0 & 0 & 0 & 0 & 0 & 0 & 0 & 0 & 0 & 0 & 0 & 0 & 0 & -1
  \end{smallmatrix}\right)
\end{equation}
and
\begin{equation}\label{eq:S-app}
  S =
  \left(\begin{smallmatrix}
    0 & \frac{1}{2} & 0 & 0 & \frac{1}{2} & 0 & 0 & 0 & 0 & 0 & 0 & 0 & 0 & 0 & 0 & 0 \\
    1 & 0 & 0 & 0 & 0 & -1 & 0 & 0 & 0 & 0 & 0 & 0 & 0 & 0 & 0 & 0 \\
    0 & 0 & -\frac{1}{2} & \frac{1}{3} & 0 & 0 & \frac{i}{3 \sqrt{2}} & 0 & -\frac{1}{2} & -\frac{i}{3 \sqrt{2}}
      & 0 & 0 & \frac{1}{3} & 0 & 0 & 0 \\
    0 & 0 & \frac{1}{2} & 0 & 0 & 0 & 0 & 0 & -\frac{1}{2} & 0 & 1 & \frac{1}{2} & 0 & 0 & -\frac{1}{2} & 0 \\
    1 & 0 & 0 & 0 & 0 & 1 & 0 & 0 & 0 & 0 & 0 & 0 & 0 & 0 & 0 & 0 \\
    0 & -\frac{1}{2} & 0 & 0 & \frac{1}{2} & 0 & 0 & 0 & 0 & 0 & 0 & 0 & 0 & 0 & 0 & 0 \\
    0 & 0 & \frac{1}{2} & 0 & 0 & 0 & -\frac{1}{2} & 0 & -\frac{1}{2} & -\frac{1}{2} & 0 & -1 & 0 & 0 & 1 & 0 \\
    0 & 0 & 0 & \frac{2}{3} & 0 & 0 & -\frac{i}{3 \sqrt{2}} & 0 & 0 & \frac{i}{3 \sqrt{2}} & 0 & 0 & -\frac{1}{3} & 0 & 0 & 0 \\
    0 & 0 & \frac{1}{2} & \frac{1}{3} & 0 & 0 & \frac{i}{3 \sqrt{2}} & 0 & \frac{1}{2} & -\frac{i}{3 \sqrt{2}}
      & 0 & 0 & \frac{1}{3} & 0 & 0 & 0 \\
    0 & 0 & \frac{1}{2} & 0 & 0 & 0 & \frac{1}{2} & 0 & -\frac{1}{2} & \frac{1}{2} & 0 & -1 & 0 & 0 & 1 & 0 \\
    0 & 0 & 0 & 0 & 0 & 0 & 0 & -\frac{1}{2} & 0 & 0 & 0 & 0 & 0 & \frac{1}{2} & 0 & 0 \\
    0 & 0 & 0 & 0 & 0 & 0 & 0 & -\frac{1}{2} & 0 & 0 & 0 & 0 & 0 & -\frac{1}{2} & 0 & 1 \\
    0 & 0 & -\frac{1}{2} & 0 & 0 & 0 & 0 & 0 & \frac{1}{2} & 0 & 1 & -\frac{1}{2} & 0 & 0 & \frac{1}{2} & 0 \\
    0 & 0 & 0 & -\frac{1}{3} & 0 & 0 & -\frac{i}{3 \sqrt{2}} & 0 & 0 & \frac{i}{3 \sqrt{2}} & 0 & 0 & \frac{2}{3} & 0 & 0 & 0 \\
    0 & 0 & 0 & 0 & 0 & 0 & 0 & \frac{1}{2} & 0 & 0 & 0 & 0 & 0 & \frac{1}{2} & 0 & 1 \\
    0 & 0 & 0 & 0 & 0 & 0 & 0 & 0 & 0 & 0 & 0 & \frac{1}{2} & 0 & 0 & \frac{1}{2} & 0
  \end{smallmatrix}\right) \,.
\end{equation}
We note that it is $\check R_{RM}(u,b)\, P$ and not $\check R_{RM}(u,b)$ that satisfy the
Yang-Baxter equation~\eqref{eq:YB intro}.

\medskip
\noindent
(b) Likewise, the orthosymplectic $R$-matrix $R(u)$ of~\eqref{eq:osp-Rmatrix} for $N=2,m=1$ (so that $\kappa=-1$)
with the parity sequence $\Parity=(\bar{1},\bar{0})$ is explicitly given by the following matrix:
\begin{equation}
  R(u) =
  \left(\begin{smallmatrix}
    \frac{u+1}{u} & 0 & 0 & 0 & 0 & 0 & 0 & 0 & 0 & 0 & 0 & 0 & 0 & 0 & 0 & 0 \\
    0 & 1 & 0 & 0 & -\frac{1}{u} & 0 & 0 & 0 & 0 & 0 & 0 & 0 & 0 & 0 & 0 & 0 \\
    0 & 0 & 1 & 0 & 0 & 0 & 0 & 0 & -\frac{1}{u} & 0 & 0 & 0 & 0 & 0 & 0 & 0 \\
    0 & 0 & 0 & \frac{u-2}{u-1} & 0 & 0 & -\frac{1}{u-1} & 0 & 0 & \frac{1}{u-1} & 0 & 0 & -\frac{1}{(u-1) u} & 0 & 0 & 0 \\
    0 & -\frac{1}{u} & 0 & 0 & 1 & 0 & 0 & 0 & 0 & 0 & 0 & 0 & 0 & 0 & 0 & 0 \\
    0 & 0 & 0 & 0 & 0 & \frac{u-1}{u} & 0 & 0 & 0 & 0 & 0 & 0 & 0 & 0 & 0 & 0 \\
    0 & 0 & 0 & \frac{1}{u-1} & 0 & 0 & \frac{u}{u-1} & 0 & 0 & -\frac{2 u-1}{(u-1) u} & 0 & 0 & \frac{1}{u-1} & 0 & 0 & 0 \\
    0 & 0 & 0 & 0 & 0 & 0 & 0 & 1 & 0 & 0 & 0 & 0 & 0 & -\frac{1}{u} & 0 & 0 \\
    0 & 0 & -\frac{1}{u} & 0 & 0 & 0 & 0 & 0 & 1 & 0 & 0 & 0 & 0 & 0 & 0 & 0 \\
    0 & 0 & 0 & -\frac{1}{u-1} & 0 & 0 & -\frac{2 u-1}{(u-1) u} & 0 & 0 & \frac{u}{u-1} & 0 & 0 & -\frac{1}{u-1} & 0 & 0 & 0 \\
    0 & 0 & 0 & 0 & 0 & 0 & 0 & 0 & 0 & 0 & \frac{u-1}{u} & 0 & 0 & 0 & 0 & 0 \\
    0 & 0 & 0 & 0 & 0 & 0 & 0 & 0 & 0 & 0 & 0 & 1 & 0 & 0 & -\frac{1}{u} & 0 \\
    0 & 0 & 0 & -\frac{1}{(u-1) u} & 0 & 0 & -\frac{1}{u-1} & 0 & 0 & \frac{1}{u-1} & 0 & 0 & \frac{u-2}{u-1} & 0 & 0 & 0 \\
    0 & 0 & 0 & 0 & 0 & 0 & 0 & -\frac{1}{u} & 0 & 0 & 0 & 0 & 0 & 1 & 0 & 0 \\
    0 & 0 & 0 & 0 & 0 & 0 & 0 & 0 & 0 & 0 & 0 & -\frac{1}{u} & 0 & 0 & 1 & 0 \\
    0 & 0 & 0 & 0 & 0 & 0 & 0 & 0 & 0 & 0 & 0 & 0 & 0 & 0 & 0 & \frac{u+1}{u}
  \end{smallmatrix}\right) \,.
\end{equation}
It is related to that of~\cite[(2)]{rm} via the following equality:
\begin{equation}
  \frac{u}{u-1}\, R(u) = \check R_{RM}(-\sfrac{u}{2},0)\, P \,,
\end{equation}
with $P$ as in~\eqref{eq:P-app}.
\end{Rem}


\end{document}